\documentclass[11pt]{article}
\usepackage[margin=1in]{geometry}

\usepackage{amsfonts,amsmath,amssymb,amsthm}
\usepackage{aliascnt}
\usepackage{mathtools,bm,mathrsfs,mleftright}
\usepackage{xcolor,enumitem,autobreak}
\allowdisplaybreaks[4]
\usepackage{graphicx}
\usepackage{placeins}
\usepackage[authoryear,round]{natbib}
\setcitestyle{authoryear,round}
\usepackage[colorlinks,citecolor=blue,linkcolor=blue,urlcolor=blue]{hyperref}
\usepackage[nameinlink,noabbrev]{cleveref}

\crefname{theorem}{Theorem}{Theorems}
\crefname{lemma}{Lemma}{Lemmas}
\crefname{corollary}{Corollary}{Corollaries}
\crefname{proposition}{Proposition}{Propositions}
\crefname{definition}{Definition}{Definitions}
\crefname{remark}{Remark}{Remarks}
\crefname{example}{Example}{Examples}
\crefname{assumption}{Assumption}{Assumptions}
\crefname{condition}{Condition}{Conditions}
\crefname{section}{Section}{Sections}
\crefname{subsection}{Subsection}{Subsections}
\crefname{subsubsection}{Subsection}{Subsections}

\crefformat{equation}{(#2#1#3)}
\crefrangeformat{equation}{(#3#1#4)--(#5#2#6)}
\crefmultiformat{equation}{(#2#1#3)}{ and (#2#1#3)}{, (#2#1#3)}{, and (#2#1#3)}

\newcommand{\R}{\mathbb{R}}
\newcommand{\E}{\mathbb{E}}

\newcommand{\bbS}{\mathbb{S}}
\newcommand{\bbZ}{\mathbb{Z}}

\DeclareMathOperator{\Var}{Var}
\DeclareMathOperator{\Cov}{Cov}

\newcommand{\mr}{\mathrm}

\newcommand{\caA}{\mathcal{A}}
\newcommand{\caE}{\mathcal{E}}

\newcommand{\caL}{\mathcal{L}}
\newcommand{\caR}{\mathcal{R}}

\newcommand{\xk}[1]{\left(#1\right)}
\newcommand{\zk}[1]{\left[#1\right]}
\newcommand{\dk}[1]{\left\{#1\right\}}

\newcommand{\zkx}[1]{[#1]}

\providecommand{\ang}[1]{\left\langle{#1}\right\rangle}
\providecommand{\angx}[1]{\langle{#1}\rangle}
\providecommand{\abs}[1]{\left\lvert{#1}\right\rvert}
\providecommand{\norm}[1]{\left\lVert{#1}\right\rVert}
\providecommand{\absx}[1]{\lvert {#1}\rvert}
\providecommand{\normx}[1]{\lVert{#1}\rVert}
\providecommand{\floor}[1]{\left\lfloor{#1}\right\rfloor}
\providecommand{\dd}{~\mathrm{d}}

\newcommand{\caEelig}[1]{\caE_{#1}^{\mr{elig}}}
\newcommand{\caEthr}[1]{\caE_{#1}^{\mr{thr}}}
\newcommand{\caEkeep}[1]{\caE_{#1}^{\mr{keep}}}

\theoremstyle{plain}
\newtheorem{theorem}{Theorem}[section]
\newaliascnt{lemma}{theorem}
\newtheorem{lemma}[lemma]{Lemma}
\aliascntresetthe{lemma}
\newaliascnt{corollary}{theorem}
\newtheorem{corollary}[corollary]{Corollary}
\aliascntresetthe{corollary}

\theoremstyle{definition}
\newaliascnt{proposition}{theorem}
\newtheorem{proposition}[proposition]{Proposition}
\aliascntresetthe{proposition}
\newaliascnt{definition}{theorem}

\aliascntresetthe{definition}
\newaliascnt{remark}{theorem}

\aliascntresetthe{remark}

\newtheorem{assumption}{Assumption}

\hypersetup{colorlinks=true,linkcolor=blue,citecolor=blue,urlcolor=blue}

\begin{document}

\title{The Cost of Discretization in Functional Linear Regression: Minimax Rates and Adaptation}
\author{
T. Tony Cai\thanks{Department of Statistics and Data Science, The Wharton School, University of Pennsylvania, Philadelphia, PA, USA. Email: \texttt{tcai@wharton.upenn.edu}.}
\and
Yicheng Li\thanks{KLATASDS-MOE, School of Statistics, East China Normal University, Shanghai, China. Email: \texttt{ycli@sfs.ecnu.edu.cn}.}
}
\date{\today}
\maketitle

\begin{abstract}
      We study scalar-on-function linear regression when each covariate curve is observed only through finitely many noisy point evaluations. Our goal is to characterize the minimax estimation and prediction risks as joint functions of the number of trajectories \(n\) and the within-trajectory resolution \(m\). Working in a fixed trigonometric eigenbasis, with covariance eigenvalues decaying at rate \(\alpha\) and slope function of Sobolev smoothness \(s\), we derive matching minimax upper and lower bounds under two canonical sampling schemes.

Under an independent random design, the minimax prediction rate is
\[
  n^{-\frac{2\alpha+2s}{2\alpha+2s+1}}
  +
  (nm)^{-\frac{2\alpha+2s}{4\alpha+2s+1}} .
\]
The first term is the fully observed functional linear regression benchmark, while the second term captures the cost of noisy point evaluations after amplification by the inverse covariance operator. Under a common design on an equally spaced grid, the shared sampling geometry introduces additional obstructions, and the minimax prediction rate becomes
\[
  n^{-\frac{2\alpha+2s}{2\alpha+2s+1}}
  +
  (nm)^{-\frac{2\alpha+2s}{4\alpha+2s+1}}
  +
  m^{-(2\alpha+2s)}
  +
  m^{-4\alpha}.
\]
Here the third term represents discretization error induced by the fixed grid, whereas the fourth reflects the cost of identifying unknown eigenvalues from observations on a common grid.

We further construct data-driven adaptive estimators that screen the covariance scale and threshold blockwise prediction energy, attaining these rates without prior knowledge of the eigenvalue sequence or the smoothness indices. The results reveal a sharp phase transition that depends on the sampling resolution under independent design and a richer phase diagram under common design. Numerical simulations and a real data example illustrate the theoretical findings.

     \end{abstract}


  \section{Introduction}
\label{sec:intro}

Functional linear regression (FLR) is a central model for scalar-on-function regression in functional data analysis
\citep{ramsay2005_FunctionalDataAnalysis,wang2016_FunctionalData,hsing2015_TheoreticalFoundationsFDA}.
It provides a basic framework for studying how an entire trajectory, curve, image profile, or time-varying signal predicts a scalar outcome.
When the functional covariate is fully observed, estimation and prediction are supported by a mature minimax theory.
For example, \citet{yuan2010_RKHSFLR} provided an RKHS framework with smoothness regularization that yields minimax optimal rates for both slope estimation and prediction.
Other related work covers functional PCA and regularization for slope estimation
\citep{hall2007_MethodologyConvergenceFLR},
adaptive estimation
\citep{cai2012_MinimaxAdaptivePrediction,cai2018_AdaptiveFLR},
and recent general regularization methods
\citep{fan2024_SpectralAlgorithmsFLR,gupta2025_OptimalRatesFLRGeneralRegularization}.

In realistic applications, however, the predictor trajectory is rarely observed as an ideal function on a continuum.
Instead, each subject is typically measured only at finitely many locations, and each measurement may be contaminated by noise
\citep{cai2011_OptimalMeanFunctionDiscrete,zhang2016_SparseDense,zhou2023_DiscretelyObservedFLR}.
This regime arises in longitudinal studies, wearable sensing, imaging, environmental monitoring, and other acquisition systems where functional predictors are available only through discrete noisy measurements.
The distinction is not a technical nuisance.
Discrete noisy observation changes the statistical experiment itself: the sampling frequency, measurement noise, and sampling geometry all affect how much information about the latent trajectory and the slope function is available.
Consequently, procedures that are optimal for fully observed FLR need not remain optimal, or even rate-optimal, when the covariates are observed only through noisy point evaluations.
A fundamental question is therefore how the minimax risks for slope estimation and prediction are affected by discretization, and how the within-trajectory resolution \(m\) interacts with the number of trajectories \(n\) to determine the effective information in the data.

This paper studies this ``cost of discretization'' for functional linear regression.
Concretely, the data consist of scalar responses and noisy point evaluations
\[
  Y_i = \ang{X_i,\beta}_{L^2} + \varepsilon_i,
  \qquad
  Z_{ij} = X_i(t_{ij}) + \delta_{ij},
  \quad j=1,\dots,m,\quad i=1,\dots,n.
\]
Here \(X_i\) are independent latent trajectories, \(\beta\) is the unknown slope function, and \(t_{ij}\) are the sampling locations.
For estimation, performance is measured by the squared \(L^2\) risk
\[
  \E \normx{\hat{\beta}-\beta}_{L^2}^2
\]
of the slope function itself.
For prediction, performance is measured by the excess prediction risk
\[
  \E_\star \zkx{\angx{X_\star,\hat{\beta}-\beta}_{L^2}^2},
\]
where \(X_\star\) is an independent test trajectory drawn from the same distribution as the \(X_i\)'s.
The estimation and prediction problems are closely related but not identical:
estimation measures recovery of the slope function in \(L^2\), whereas prediction weights errors through the covariance structure of the future trajectory.
In this paper, the two analyses are parallel, so we focus on prediction risk in the main theorems for clarity, while the corresponding rates for estimation error are also discussed.
The full setting is given in \cref{sec:setup}.

The sampling scheme itself is a central part of the statistical problem.
We focus on two standard schemes:
\emph{independent design} and \emph{common design}
\citep{cai2010_NonparametricCovarianceFunction}.
Under independent design, each subject is observed at its own random sampling locations;
under common design, all trajectories share the same deterministic grid.
These two designs represent common data-acquisition mechanisms and lead to fundamentally different forms of information loss.
In independent design, random sampling locations can be pooled across subjects, so increasing \(n\) also improves coverage of the domain.
In common design, all subjects are observed through the same grid, so increasing \(n\) reduces stochastic variation but cannot remove geometric aliasing or fixed-grid approximation errors.
Thus the central question is not only how the minimax rates depend on \((n,m)\), but also how the answer changes with the sampling geometry.
Identifying this dependence is essential for understanding when discretely observed FLR behaves like the fully observed problem, when the total number of scalar measurements \(nm\) is the limiting resource, and when the grid itself creates an irreducible obstruction.

A related line of work studies the cost of discretization for functional mean and covariance estimation,
where the sampling frequency already changes the statistical problem before the inverse structure of FLR enters.
Sparse methods for functional data developed by \citet{yao2005_SparseLongitudinalFDA} and \citet{hall2006_PrincipalComponentLongitudinal} provide foundational tools for recovering mean and covariance structure from partial noisy trajectories.
For mean estimation, \citet{cai2011_OptimalMeanFunctionDiscrete} made this cost explicit by deriving distinct minimax rates and phase transitions in \(m\) under common and independent designs respectively.
Analogous rates have been obtained for covariance estimation \citep{cai2010_NonparametricCovarianceFunction} and
high-dimensional functional data \citep{petersen2024_DiscreteFunctionalCovariance}.
The sparse to dense viewpoint was further systematized by \citet{zhang2016_SparseDense} and has also been studied through predictive distributions \citep{gajardo2021_PredictiveDistributions}.
These works often identify a transition between two regimes: when \( m \) is large, the fully observed rate is recovered, whereas when \( m \) is small, the rate is dominated by a sampling term that depends on \( m \).

The FLR problem considered here is nevertheless different in an essential way because it also has an inverse structure.
One must recover not only the sample path \(X\), but also the slope function \(\beta\) and the interaction between the two.
As a result, discretization terms can arise in FLR that have no analogue in mean or covariance estimation.
For discretely observed scalar-on-function regression, the most relevant prior work is \citet{zhou2023_DiscretelyObservedFLR},
which focuses on independent random design and gives sufficient conditions on \(m\) under which the fully observed dense rate is recovered.
However, it does not provide a minimax lower bound that depends on \(m\), so it is unclear whether its threshold for the dense regime is sharp, and it also does not address sampling on a common grid.
What left open by these works is a complete minimax theory for the FLR targets that treats \(n\) and \(m\) as joint statistical resources,
separates the term due to noisy point evaluations from the fully observed benchmark,
and distinguishes these statistical terms from losses created by a common grid.

We address this problem by establishing minimax rates with respect to both \( n \) and \(m\) for the two sampling schemes.
To isolate the cost of discretization, we work in a fixed trigonometric basis and assume that the covariance operator of \(X\) is diagonalized by this basis, which includes periodic stationary covariance kernels.
This formulation sets aside the additional effects of covariance eigenbasis estimation and of the alignment between the covariance geometry and the regularity scale of \(\beta\)~\citep{cai2012_MinimaxAdaptivePrediction},
which are present even in the fully observed FLR model and are not the focus of this paper.
The resulting rates show that discretely observed FLR has its own phase transition structure, not merely the sparse to dense transition known from mean and covariance estimation.
Under independent design, the cost of discretization appears through a sampling term with a distinct exponent from the dense FLR term due to amplification by covariance inversion.
Under common design, the rate is richer: the same statistical branch remains,
but the shared grid creates additional geometric obstructions through fixed grid discretization and the identification of unknown eigenvalues.

\begin{table}[t]
  \centering
  \caption{Rate terms appearing in the minimax prediction risk under the main observation settings.
  Define \( \nu\coloneqq 2\alpha+2s \) and \( \kappa\coloneqq 4\alpha+2s+1 \).
  The known \(\lambda\) common design row treats the eigenvalue sequence as fixed and available, whereas the last row treats it as unknown.
  }
  \label{tab:intro-rate-terms}
  \begin{tabular}{l c@{\hspace{3pt}}c@{\hspace{3pt}}c@{\hspace{3pt}}c@{\hspace{3pt}}c@{\hspace{3pt}}c@{\hspace{3pt}}c}
    \hline
    Setting                         & Dense (I)            &       & Sampling (II)          &       & Grid (III)   &       & \shortstack{Eigenvalue\\ identification} (IV) \\
    \hline
    Fully observed                  & \(n^{-\nu/(\nu+1)}\) &       &                        &       &              &       &                          \\
    Independent design              & \(n^{-\nu/(\nu+1)}\) & \(+\) & \((nm)^{-\nu/\kappa}\) &       &              &       &                          \\
    Common design, known \(\lambda\) & \(n^{-\nu/(\nu+1)}\) & \(+\) & \((nm)^{-\nu/\kappa}\) & \(+\) & \(m^{-\nu}\) &       &                  \\
    Common design, unknown \(\lambda\) & \(n^{-\nu/(\nu+1)}\) & \(+\) & \((nm)^{-\nu/\kappa}\) & \(+\) & \(m^{-\nu}\) & \(+\) & \(m^{-4\alpha}\) \\
    \hline
  \end{tabular}
\end{table}
 
To illustrate our results in detail,
let $\alpha$ be the decay rate of the covariance eigenvalues of \(X\) and \(s\) be the smoothness of the slope function \(\beta\) (see \cref{sec:setup} for precise definitions).
In the case of independent design, the random sampling locations can be pooled across subjects, so the discretization cost appears as a single sampling term caused by noisy point evaluations.
In the range \(\alpha>1/2\) and \(s\ge0\) of minimum regularity requirement, the minimax prediction risk is of order
\[
  n^{-\frac{2\alpha+2s}{2\alpha+2s+1}}
  +
  (nm)^{-\frac{2\alpha+2s}{4\alpha+2s+1}}.
\]
The first term is the minimax rate for fully observed or dense FLR~\citep{cai2012_MinimaxAdaptivePrediction,zhou2023_DiscretelyObservedFLR}.
The exponent $2\alpha+2s$ reflects the effective smoothness of the problem as the sum of the slope smoothness and the covariance decay.
The second term does not appear in the fully observed theory and represents the sampling price paid for noisy point evaluations.
Its exponent has a larger denominator because measurement noise is amplified by the inverse covariance eigenvalues before it enters prediction.
The phase transition between the two terms occurs at \(m\asymp n^{2\alpha/(2\alpha+2s+1)}\).
Below this threshold, the effective information is governed by the total number of scalar measurements \(nm\);
above it, further refinement within each curve no longer improves the rate, and the effective information is governed by the number of curves \(n\) alone.

Common design is different because all trajectories are viewed through the same grid.
The two statistical terms above are still present, but a common grid leaves a geometric error that cannot be averaged away by increasing the number of curves.
When the eigenvalue sequence \((\lambda_k)_{k \in \bbZ}\) is known, the oracle rate for an equal grid consists of the two statistical terms and an additional term \(m^{-(2\alpha+2s)}\), representing the reconstruction error left by the fixed grid.
When \((\lambda_k)_{k \in \bbZ}\) is unknown, estimation must also contend with an obstruction from eigenvalue identification,
and the sharp rate has the four term decomposition
\[
  n^{-\frac{2\alpha+2s}{2\alpha+2s+1}}
  +
  (nm)^{-\frac{2\alpha+2s}{4\alpha+2s+1}}
  +
  m^{-(2\alpha+2s)}
  +
  m^{-4\alpha}.
\]
The additional term $m^{-4\alpha}$ depends solely on the covariance decay and sampling frequency, and represents the cost of identifying the eigenvalues from the common grid.
Our lower bound shows that this term is unavoidable when the eigenvalues are unknown and is not merely a byproduct of the adaptive construction.
This contrasts with independent design, where estimating the eigenvalue sequence does not change the minimax rate.
In this case, the phase diagram is more complex, with multiple transitions between the four terms as \(m\) increases.
See \Cref{fig:common-design-phase-diagram} for an illustration.

Our analysis also requires several new technical ingredients beyond the fully observed FLR theory, as discretization cannot be analyzed as a separate error.
On the lower bound side, we establish a delicate Fisher information control that accounts for noisy point evaluations.
In addition, under common design, we develop two distinct indistinguishability constructions to capture the effects of fixed grid discretization and unknown eigenvalue identification.
On the upper bound side, the interplay between the slope function, the covariance function, and sampling geometry creates a nontrivial tradeoff between bias and variance that cannot be resolved by a single truncation or thresholding step for adaptation.
To this end, we construct adaptive procedures that combine eigenvalue screening with blockwise thresholding of prediction energy.
We believe that these techniques will be of independent interest for other functional data problems with discretization and measurement noise.

\subsection{Related work}
\label{subsec:related-work}

Fully observed scalar-on-function regression has a well-developed minimax and regularization theory.
Classical results develop generalized, principal component, smoothing, and prediction-oriented methods for the functional linear model; in particular, \citet{muller2005_GeneralizedFunctional} studied generalized functional linear models, \citet{cai2006_PredictionFLR} separated prediction from slope recovery, and \citet{hall2007_MethodologyConvergenceFLR} established convergence rates for principal component estimators.
The RKHS formulation of \citet{yuan2010_RKHSFLR} and the adaptive prediction theory of \citet{cai2012_MinimaxAdaptivePrediction,cai2018_AdaptiveFLR} provide benchmark rates when the covariate trajectories are observed as functions.
Related inverse problem viewpoints include the Le Cam equivalence result of \citet{meister2011_AsymptoticEquivalenceFLR}, and recent spectral, general regularization, or unified model analyses include \citet{fan2024_SpectralAlgorithmsFLR,gupta2025_OptimalRatesFLRGeneralRegularization,balasubramanian2025_FunctionalLinear}.

For sparsely or discretely observed functional data, a central theme is that the sampling frequency within each curve changes the statistical experiment.
The PACE methodology of \citet{yao2005_SparseLongitudinalFDA} and the principal component theory of \citet{hall2006_PrincipalComponentLongitudinal} provide foundational tools for recovering latent functional structure from noisy partial trajectories.
For mean and covariance structure, \citet{cai2011_OptimalMeanFunctionDiscrete,cai2010_NonparametricCovarianceFunction} established sharp rates under discrete sampling, and \citet{zhang2016_SparseDense} clarified the transition from sparse to dense regimes.
More recent work extends this regime analysis to predictive distributions and high-dimensional functional data \citep{gajardo2021_PredictiveDistributions,petersen2024_DiscreteFunctionalCovariance,guo2025_SparseDenseHighDimensions}.

Several papers address functional regression with sparse, discrete, or noisy covariates directly.
\citet{yao2005_FLRLongitudinal} developed a functional regression procedure for sparse longitudinal predictor and response processes using conditional principal component scores.
\citet{li2007_RatesConvergenceFLR} and \citet{cardot2007_SmoothingSplineErrorsInVariables} studied slope estimation when functional predictors are observed on discrete grids with measurement error, and \citet{ferraty2012_PresmoothingFLR} investigated presmoothing before functional linear regression.
The closest minimax comparator is \citet{zhou2023_DiscretelyObservedFLR}, which proves that the fully observed rate is attainable under independent random design once \(m\) is sufficiently large.
However, none of these papers gives a matching \(m\)-dependent minimax lower bound for the prediction risk.

\subsection{Organization of the paper}
\label{subsec:organization}

The rest of the paper is organized as follows.
\Cref{sec:setup} introduces the problem formulation, including the model, parameter spaces, and prediction and estimation risks.
\Cref{sec:indep} studies independent design: it states the sampling assumption,
gives the oracle upper bound, proves the lower bound, and constructs the
adaptive estimator and upper bound.
\Cref{sec:common} follows the same structure for the common design problem on an equal grid, where the phase diagram includes the additional terms from fixed grid discretization and the identification of unknown eigenvalues.
\Cref{sec:numerics} presents simulation studies and a real data example.
\Cref{sec:discussion} returns to the main conclusions, provides the corresponding \(L^2\) estimation rates, and discusses the current limitations and next open questions.
All proofs and technical results are collected in the appendices.

We use the following notation throughout the paper.
Let $L^2([0,1])$ be the space of square integrable functions on $[0,1]$ with inner product $\ang{\cdot,\cdot}_{L^2}$ and norm $\norm{\cdot}_{L^2}$.
The expectation operator is denoted by \(\E\), with subscripts indicating the variable of integration when needed.
For real numbers \(a\) and \(b\), write \(a\wedge b=\min\{a,b\}\) and \(a\vee b=\max\{a,b\}\), and let \(\floor{x}\) be the integer part of \(x\).
For a finite set \(A\), \(\absx{A}\) denotes its cardinality, and \(\mathbf{1}\{\cdot\}\) denotes the indicator function.
For two nonnegative quantities \(a\) and \(b\), the notation \(a\lesssim b\) means that \(a\le Cb\) for a constant \(C\) independent of \(n\), \(m\), and the tuning parameters under discussion; \(a\gtrsim b\) and \(a\asymp b\) are defined analogously.
The constants hidden in these comparisons may depend on fixed model parameters such as \(c_\lambda,C_\lambda,R,\sigma_\varepsilon,\sigma_\delta\), and on fixed compact ranges of smoothness parameters when uniform adaptive statements are considered.
   \section{Problem Formulation}
\label{sec:setup}

We use the following functional linear regression model and loss criteria.
Let \( X(\cdot) \) be a centered Gaussian process on the unit interval \([0,1]\) with continuous path.
Let \(K(s,t)=\E[X(s)X(t)]\) be its covariance function and $\Sigma$ be the covariance operator defined as
\(
(\Sigma f)(t) = \int_0^1 K(s,t) f(s) \dd s
\)
for any \(f\in L^2[0,1]\).
Let $\dk{X_i(\cdot)}_{i=1}^n$ be independent copies of \(X\).
The response \(Y_i\) follows the scalar-on-function regression model
\begin{equation}
  \label{eq:FLR}
  Y_i = \int_0^1 X_i(t)\beta(t) \dd t + \varepsilon_i = \ang{X_i, \beta}_{L^2} + \varepsilon_i,
\end{equation}
where \(\beta\) is an unknown slope function and \(\varepsilon_i\) are independent \(\mathcal{N}(0,\sigma_\varepsilon^2)\) errors.

The predictor trajectories \(X_i\) are observed only through noisy point evaluations.
Let \(m\) be the number of sampling locations per trajectory and let \(\{t_{ij}:j=1,\dots,m\}\) be the sampling locations for the \(i\)-th trajectory.
The observed point evaluations are
\begin{equation}
  \label{eq:sampling-model}
  Z_{ij}=X_i(t_{ij})+\delta_{ij},
\end{equation}
where \(\delta_{ij}\) are independent \(\mathcal{N}(0,\sigma_\delta^2)\) measurement errors with \(\sigma_\delta^2 > 0\).
The families \(\{\delta_{ij}\}\), \(\{\varepsilon_i\}\), and \(\{X_i\}\) are mutually independent.
Throughout the paper, \(\varepsilon_i\) denotes the scalar response error and \(\delta_{ij}\) denotes the pointwise measurement error.
The two sampling schemes for \(t_{ij}\) are specified later in \cref{assum:indep-design} and \cref{assum:common-design}.

Given an estimator \(\hat{\beta}\) of the slope function \(\beta\), our primary performance criterion is the excess prediction risk
\begin{equation}
  \label{eq:generalization-error}
  \E_{\star} \zk{\angx{X_{\star}, \hat{\beta}-\beta}_{L^2}^2},
\end{equation}
where \(X_\star\) is an independent test trajectory drawn from the same distribution as the \(X_i\)'s and \(\E_{\star}\) denotes the expectation with respect to \(X_\star\).
We also consider slope recovery under the squared \(L^2\) estimation loss \(\normx{\hat{\beta}-\beta}_{L^2}^2\).
The main theorems below are stated for prediction risk; the corresponding rates for estimation risk are discussed in \cref{sec:discussion}.

To focus on the impact of discrete noisy sampling and avoid technicalities related to the choice of basis, we work with the fixed orthonormal trigonometric basis on \([0,1]\)
\begin{equation}
  \label{eq:trig-basis}
  e_0(t)\coloneqq 1,\qquad
  e_r(t)\coloneqq \sqrt{2}\cos(2\pi r t),\qquad
  e_{-r}(t)\coloneqq \sqrt{2}\sin(2\pi r t),
  \quad r\ge1.
\end{equation}
The choice of the index set \(\bbZ\) is for notational convenience, with \(r=0\) corresponding to the constant function, positive indices labeling cosine functions, and negative indices labeling sine functions.
Then, every square integrable function \(f\) is represented as
\[
  f=\sum_{r\in\bbZ} a_r e_r,\qquad
  a_r=\ang{f,e_r}_{L^2}=\int_0^1 f(t)e_r(t)\dd t,
  \qquad a_r \in\R.
\]

Throughout the paper, we assume that \(\Sigma\) is diagonalized by the trigonometric basis \((e_r)_{r\in\bbZ}\) with eigenvalues \((\lambda_r)_{r\in\bbZ}\).
It is satisfied, for example, by periodic stationary covariance kernels \(K(s,t)=\kappa(s-t \bmod 1)\), which arise naturally for many periodic data settings, such as daily or weekly trajectories, when covariance depends on circular lag.
This assumption lets us focus on the statistical cost created by noisy point evaluations and by the sampling geometry, rather than on the additional impact of estimating the covariance eigenbasis or the alignment between the slope function and the covariance geometry.
In a principal component formulation, the covariance operator is diagonal in its own population eigenbasis, but that basis is unknown and covariance dependent.
The need of estimating the covariance eigenbasis from discretely observed data may introduce an additional source of error, which is not the main focus of this paper.
Moreover, imposing coefficient decay or smoothness of \(\beta\) in that basis ties the slope class to the covariance geometry, thereby introducing an alignment issue between covariance and slope regularity~\citep{cai2012_MinimaxAdaptivePrediction,gupta2025_OptimalRatesFLRGeneralRegularization} that is already present in the fully observed case.
Therefore, we instead use the fixed trigonometric basis both to diagonalize the covariance operator and to unify the smoothness classes for the predictor trajectories \(X_i\) and for the slope function \(\beta\).

Consequently, the predictor trajectories \(X_i\) have the Karhunen--Loève expansion
\begin{equation}
  \label{eq:predictor-fourier-expansion}
  X_i(t)=\sum_{r\in\bbZ} x_{ir} e_r(t),
  \qquad
  \E x_{ir}^2=\lambda_r,
\end{equation}
where the Fourier scores are independent centered Gaussian variables.
We also assume that the unknown slope function \(\beta\) has the Fourier expansion
\begin{equation}
  \label{eq:beta-fourier}
  \beta(t)=\sum_{r\in\bbZ} \theta_r e_r(t),
  \qquad
  \theta_r \in\R.
\end{equation}
We introduce the cross-covariance function as
\begin{equation}
  \label{eq:g-fourier}
  g(t)\coloneqq \E[Y_1 X_1(t)] = (\Sigma\beta)(t) = \sum_{r\in\bbZ} \gamma_r e_r(t), \qquad \gamma_r \coloneqq\lambda_r \theta_r.
\end{equation}
With the fixed trigonometric basis, the excess prediction risk can be expressed as a weighted \(\ell^2\) error of the Fourier coefficients.
If \( \hat{\beta} \) has the Fourier expansion \( \hat{\beta}=\sum_{r\in\bbZ} \hat{\theta}_r e_r \), then
\begin{equation}
  \label{eq:generalization-error-fourier}
  \caR(\hat{\beta};\theta,\lambda) \coloneqq  \E_{\star} \left[\angx{X_\star,\hat{\beta}-\beta}^2 \right] = \sum_{r\in\bbZ} \lambda_r \absx{\hat{\theta}_r-\theta_r}^2.
\end{equation}
Since \(\hat{\beta}\) depends on the training data, \(\caR(\hat{\beta};\theta,\lambda)\) is the test risk conditional on that training sample.
The minimax bounds below use the outer expectation \(\E\caR(\hat{\beta};\theta,\lambda)\), taken over all training randomness, including the latent trajectories, design points, response errors, and measurement errors.

For $\alpha > 1/2$,
we define the eigenvalue class
\begin{equation}
  \label{eq:class-lambda}
  \caL_\alpha(c_\lambda,C_\lambda)
  \coloneqq
  \dk{
   (\lambda_r)_{r\in\bbZ}
    :
    c_\lambda(1+\absx{r})^{-2\alpha}
    \le \lambda_r \le
    C_\lambda(1+\absx{r})^{-2\alpha}
  }
\end{equation}
where \(0<c_\lambda < C_\lambda<\infty\) are fixed constants.
The requirement \(\alpha>1/2\) is necessary to ensure that the eigenvalues are summable and the predictor trajectories \(X_i\) are continuous almost surely.
For \(s\ge 0\) and \(R_0>0\), the slope smoothness is indexed by the Sobolev ball
\begin{equation}
  \label{eq:class-theta}
  \Theta_s(R_0)
  \coloneqq
  \dk{
  \theta=(\theta_r)_{r\in\bbZ}
    :
    \sum_{r\in\bbZ}(1+\absx{r})^{2s}\theta_r^2 \le R_0^2
  }.
\end{equation}
The condition \(\theta \in \Theta_s(R_0)\) is equivalent to assuming that the slope function \(\beta\) belongs to the periodic Sobolev space \(H^s([0,1])\).
Throughout the paper, we assume that \(\alpha > 1/2\) and \(s\ge 0\) without further mention,
which is the minimal regularity range for the problem to be well defined.
   \section{Independent Design}
\label{sec:indep}

Under independent design, each subject is observed on an independent random grid.
Formally, we introduce the following assumption on the design points.

\begin{assumption}[Independent design]
  \label{assum:indep-design}
  The design points satisfy \(t_{ij} \overset{\mathrm{iid}}{\sim}\mathrm{Unif}[0,1]\), independently across \(i\) and \(j\), and independently of all other random quantities.
\end{assumption}

The results in this section show that discretization affects the risk only through an additional statistical measurement term.

\subsection{Oracle Estimator}
\label{subsec:indep-oracle-upper}

We begin with an oracle estimator that depends on the eigenvalue sequence \((\lambda_r)_{r\in\bbZ}\) and uses a deterministic cutoff \(d\), which may depend on the smoothness indices \((\alpha,s)\).
Define the cross-covariance coefficient estimator for \(\gamma_r\) by
\begin{equation}
  \label{eq:cross-covariance-coefficient}
  \bar{Z}_{ir}
  \coloneqq
  \frac{1}{m}\sum_{j=1}^m Z_{ij} e_r(t_{ij}),
  \qquad
  \hat{\gamma}_r
  \coloneqq
  \frac{1}{n}\sum_{i=1}^n Y_i \bar{Z}_{ir}.
\end{equation}
Under independent design, \( \hat{\gamma}_r \) is an unbiased estimator of \(\gamma_r\).
Since \(\gamma_r=\lambda_r \theta_r\), define the plug-in estimator
\begin{equation}
  \label{eq:oracle-estimator}
  \tilde{\beta}(t) \coloneqq \sum_{\absx{r}\le d}\lambda_r^{-1} \hat{\gamma}_r e_r(t),
\end{equation}
where the cutoff \(d\) is a tuning parameter.
The oracle estimator satisfies the following upper bound on the excess prediction risk.

\begin{theorem}[Oracle upper bound under independent design]
  \label{thm:indep-oracle-upper}
  Assume that \(\lambda\in\caL_\alpha(c_\lambda,C_\lambda)\) is known.
  Under \cref{assum:indep-design}, for every \(d\ge 1\),
  \begin{equation}
    \label{eq:indep-oracle-terms}
    \sup_{\lambda\in\caL_\alpha(c_\lambda,C_\lambda)}
    \sup_{\theta\in\Theta_s(R_0)}
    \E\caR(\tilde{\beta};\theta,\lambda)
    \lesssim
    \frac{d}{n}
    +
    \frac{d^{2\alpha+1}}{nm}
    +
    d^{-(2\alpha+2s)}.
  \end{equation}
  Consequently, take the optimized cutoff
  \begin{equation}
    \label{eq:indep-oracle-d}
    d^* \asymp n^{\frac{1}{2\alpha+2s+1}} \wedge (nm)^{\frac{1}{4\alpha+2s+1}}
  \end{equation}
  and let \(\tilde{\beta}^*\) be the corresponding estimator.
  Then
  \begin{equation}
    \label{eq:indep-oracle-rate}
    \sup_{\lambda\in\caL_\alpha(c_\lambda,C_\lambda)}
    \sup_{\theta\in\Theta_s(R_0)}
    \E\caR(\tilde{\beta}^*;\theta,\lambda)
    \lesssim
    n^{-\frac{2\alpha+2s}{2\alpha+2s+1}}
    +
    (nm)^{-\frac{2\alpha+2s}{4\alpha+2s+1}}.
  \end{equation}
\end{theorem}

The proof of \cref{thm:indep-oracle-upper} is deferred to Appendix~B.
We remark here that \(d/n\) is the standard estimation contribution from estimating \(d\) coordinates in fully observed FLR, \(d^{-(2\alpha+2s)}\) is the prediction tail, and \(d^{2\alpha+1}/(nm)\) is the cost of estimating \(\gamma_r\) from noisy point evaluations before division by \(\lambda_r\).

\subsection{Lower Bound}
\label{subsec:indep-lower-bound}

We next provide a matching lower bound.

\begin{theorem}[Lower bound under independent design]
  \label{thm:indep-lower}
  Under \cref{assum:indep-design}, we have
  \begin{equation}
    \label{eq:indep-lower}
    \inf_{\hat{\beta}}
    \sup_{\lambda\in\caL_\alpha(c_\lambda,C_\lambda)}
    \sup_{\theta\in\Theta_s(R_0)}
    \E\caR(\hat{\beta};\theta,\lambda)
    \gtrsim
    n^{-\frac{2\alpha+2s}{2\alpha+2s+1}}
    +
    (nm)^{-\frac{2\alpha+2s}{4\alpha+2s+1}}.
  \end{equation}
\end{theorem}

The proof of \cref{thm:indep-lower} is based on a van Trees argument over Gaussian block submodels tailored to the independent random design.
The main difficulty is to identify the impact of discrete noisy observations on the Fisher information.
A straightforward reduction to ordinary nonparametric regression only recovers the dense observation contribution.
To capture the cost of discretization, our refined analysis provides an alternative upper bound of the Fisher information that scales with \( m \), leading to the second term in the lower bound.
In addition, we have to choose an optimized block size \( d \) to balance the two terms and prior information,
which turns out to mirror the decomposition in the upper bound in \cref{eq:indep-oracle-terms}.
The detailed proof is given in Appendix~D.

\subsection{Adaptive Estimator}
\label{subsec:indep-adaptive-upper}

We now construct an adaptive estimator that attains the minimax rate without knowledge of \((\lambda_r)_{r\in\bbZ}\), \(\alpha\), or \(s\).
The main difficulty is that adaptation cannot be separated from covariance estimation.
The oracle risk balance is expressed in the prediction norm \cref{eq:generalization-error-fourier}, so the scale at which a frequency component should be kept or discarded depends on the eigenvalue \(\lambda_r\).
The coefficient itself is also recovered through the inverse relation \(\theta_r=\gamma_r/\lambda_r\).
Thus the eigenvalues enter both the selection step and the final inversion step.
A plug-in adaptive procedure must therefore estimate \((\lambda_r)_{r\in\bbZ}\) accurately enough to calibrate the selection rule, while preventing small or poorly estimated eigenvalues from producing unstable denominators.
This differs from standard block thresholding in fully observed FLR because both the block energy and the local prediction metric are unknown.
The same covariance pilot that stabilizes the inverse estimator \(\hat{\gamma}_r/\check{\lambda}_r\) also determines whether the empirical prediction energy is meaningful on a block.
To this end, we partition the subjects into two disjoint groups
\begin{equation}
  \label{eq:adaptive-n-partition}
  \{1,\dots,n\}
  =
  I_\lambda \cup I_\gamma,\quad
  n_\lambda
  \coloneqq
  \absx{I_\lambda},
  \quad
  n_\gamma
  \coloneqq
  \absx{I_\gamma},
  \quad
  n_\lambda \asymp n_\gamma \asymp n,
\end{equation}
where \(I_\lambda\) is used for the covariance pilot and \(I_\gamma\) is used for the cross-covariance coefficients.
This sample splitting makes the covariance pilot independent of the cross-covariance coefficient estimator, simplifying the analysis of the expected prediction error.
If one only aims for an upper bound in high probability, this independence is not essential and the same upper bound can be given without splitting the subjects.
In practice, it is therefore preferable to use all subjects in both steps, which typically improves the constant factors.

We begin with the eigenvalue pilot.
For each \(i\in I_\lambda\), split the indices within each curve into two disjoint sets
\[
  J_i^{(1)} \cup J_i^{(2)}=\{1,\dots,m\},
  \qquad
  m_i^{(a)} \coloneqq\absx{J_i^{(a)}}\asymp m,\quad a = 1,2.
\]
For \(r\in\bbZ\) and \(a=1,2\), define
\[
  U_{ir}^{(a)}
  \coloneqq
  \frac{1}{m_i^{(a)}}\sum_{j\in J_i^{(a)}}Z_{ij} e_r(t_{ij}),
\]
and set
\begin{equation}
  \label{eq:indep-adaptive-lambda-pilot}
  \check{\lambda}_r
  \coloneqq
  \frac{1}{n_\lambda}\sum_{i\in I_\lambda} U_{ir}^{(1)} U_{ir}^{(2)}.
\end{equation}
The split within each curve makes the two averages conditionally independent given the latent curve, removing the leading bias from measurement noise; in particular, \(\check{\lambda}_r\) is unbiased for \(\lambda_r\).

On the regression group, estimate the cross-covariance coefficients similarly to the oracle case by
\begin{equation}
  \label{eq:indep-adaptive-gamma-estimator}
  \hat{\gamma}_r
  \coloneqq
  \frac{1}{n_\gamma}\sum_{i\in I_\gamma} Y_i \bar{Z}_{ir}.
\end{equation}

The adaptive selection is carried out on a deterministic active frequency band.
Let the largest active dyadic block index and the corresponding symmetric cutoff level be
\begin{equation}
  \label{eq:indep-adaptive-cutoff}
  L_n
  \coloneqq
  \floor{\log_2 \left(C_L \min\{n^{1/2},(nm)^{1/3}\}\right)},
  \qquad
  K_n
  \coloneqq
  2^{L_n},
\end{equation}
where \(C_L>0\) is a fixed numerical constant.
Then \( K_n \asymp n^{1/2} \wedge (nm)^{1/3} \), which is large enough for the oracle comparison in \cref{eq:indep-oracle-d} uniformly on compact subsets of \(\{\alpha>1/2,\ s\ge0\}\).
Define the dyadic blocks by \( B_0=\{0,\pm 1\} \) and \( B_\ell=\{r:2^{\ell-1}<\absx{r}\le 2^\ell\} \), for \(\ell \ge 1\).
For each active block \(B_\ell\), \(0\le \ell\le L_n\), define the empirical prediction energy and its variance proxy by
\begin{equation}
  \label{eq:indep-adaptive-selection-statistics}
  \hat{S}_\ell
  \coloneqq
  \sum_{r\in B_\ell}
  \frac{\absx{\hat{\gamma}_r}^2}{\check{\lambda}_r},
  \qquad
  \hat{V}_\ell
  \coloneqq
  \frac{C_V}{n}
  \sum_{r\in B_\ell}
  \left(1+\frac{1}{m\check{\lambda}_r}\right).
\end{equation}
Here \(C_V>0\) is a fixed sufficiently large constant.
The retention rule requires both a reliable covariance pilot on the block and empirical energy above the variance proxy.
We define the thresholding level for the covariance pilot by
\begin{equation}
  \label{eq:indep-adaptive-lambda-threshold}
  \zeta_{n,m} \coloneqq M_0 m^{-1} \sqrt{\frac{\log(nm)}{n_\lambda}},
\end{equation}
where \(M_0>0\) is a fixed sufficiently large constant.

The thresholded coefficient estimator is
\begin{equation}
  \label{eq:indep-adaptive-thresholded-coefficients}
  \hat{\theta}_r
  \coloneqq
  \frac{\hat{\gamma}_r}{\check{\lambda}_r}
  \mathbf{1}\left\{
              \min_{u\in B_\ell} \check{\lambda}_u \ge 2\zeta_{n,m}
              \ \text{and}\
              \hat{S}_\ell \ge\hat{V}_\ell
  \right\},\quad \text{for }r\in B_\ell,\ 0\le \ell\le L_n.
\end{equation}
The first thresholding is used to avoid unstable division by small or poorly estimated eigenvalues, so that \(\check{\lambda}_r\) is positive on the retained block and the division is well-defined.
The second thresholding keeps only the significant components in terms of the prediction norm for adaptivity.
Set \(\hat{\theta}_r=0\) for \(\absx{r}>K_n\).
Write \( \hat{\theta}_{B_\ell} = (\hat{\theta}_r)_{r\in B_\ell}\) for the estimated coordinates in the block \(B_\ell\).

Finally, we add a projection step for stability.
Fix a sufficiently large constant \(R\ge R_0\).
Let \(\Pi_R\) denote the projection operator \( \Pi_R(u) \coloneqq \min(1,R/\norm{u}_2) u \).
Define the projected coordinates by
\[
  \bar{\theta}_{B_\ell}
  \coloneqq
  \Pi_R \left(\hat{\theta}_{B_\ell} \right),
  \qquad
  0\le \ell\le L_n.
\]
Writing \(\bar{\theta}_r\) for the resulting coordinates, define the final estimator
\begin{equation}
  \label{eq:indep-adaptive-estimator}
  \hat{\beta}(t)\coloneqq
  \sum_{\absx{r}\le K_n}\bar{\theta}_r e_r(t).
\end{equation}

\begin{theorem}[Adaptive upper bound under independent design]
  \label{thm:indep-adaptive-upper}
  Under \cref{assum:indep-design}, uniformly over any compact subset of
  \(\{(\alpha,s):\alpha>1/2,\ s\ge0\}\), the estimator
  \cref{eq:indep-adaptive-estimator} satisfies
  \begin{equation}
    \label{eq:indep-adaptive-rate}
    \sup_{\lambda\in\caL_\alpha(c_\lambda,C_\lambda)}
    \sup_{\theta\in\Theta_s(R_0)}
    \E\caR(\hat{\beta};\theta,\lambda)
    \lesssim
    n^{-\frac{2\alpha+2s}{2\alpha+2s+1}}
    +
    (nm)^{-\frac{2\alpha+2s}{4\alpha+2s+1}}.
  \end{equation}
\end{theorem}

\cref{thm:indep-adaptive-upper} shows that the adaptive estimator under independent design attains the minimax rate without knowledge of \((\lambda_r)_{r\in\bbZ}\), \(\alpha\), or \(s\), with no additional cost in the rate.
The proof controls the error from the covariance pilot, the cross-covariance estimator, and the blockwise selection rule in the prediction norm.
The key point is that the eigenvalue threshold prevents unstable inversion, while the empirical prediction energy threshold retains only blocks whose signal is large relative to the dense variance and the variance due to noisy measurements.
The proof is given in Appendix~C.

\subsection{Discussion}
\label{subsec:indep-discussion}

We now discuss the implications of the results under independent design.
Let
\[
  \caR_{n,m}^{\star,\mathrm{ind}}
  \coloneqq
  \inf_{\hat{\beta}}
  \sup_{\lambda\in\caL_\alpha(c_\lambda,C_\lambda)}
  \sup_{\theta\in\Theta_s(R_0)}
  \E\caR(\hat{\beta};\theta,\lambda),
\]
where the infimum and risk are evaluated under \cref{assum:indep-design}.
Combining the adaptive upper bound with the lower bound yields the following minimax rate:

\begin{corollary}[Minimax rate under independent design]
  \label{cor:indep-minimax}
  Under \cref{assum:indep-design}, \(\alpha>1/2\), and \(s\ge0\),
  \begin{equation}
    \label{eq:indep-minimax-rate}
    \caR_{n,m}^{\star,\mathrm{ind}}
    \asymp
    \underbrace{n^{-\frac{2\alpha+2s}{2\alpha+2s+1}}}_{\mathrm{I}}
    +
    \underbrace{(nm)^{-\frac{2\alpha+2s}{4\alpha+2s+1}}}_{\mathrm{II}}.
  \end{equation}
\end{corollary}

Under independent design, the rate is the larger of the fully observed FLR rate~\citep{yuan2010_RKHSFLR,cai2012_MinimaxAdaptivePrediction}
and a discretization term due to measurement noise.
This yields a sharp phase transition in the resolution \(m\) within each curve.
When $m$ is small, the discretization term dominates;
when \(m\) is large, the fully observed term dominates.
Let $\zeta_* = 2\alpha / (2\alpha + 2s + 1)$ be the critical exponent for \(m\), \(\nu=2\alpha+2s\) and \(\kappa=\nu+2\alpha+1\).
The rate is summarized as
\[
  \caR_{n,m}^{\star,\mathrm{ind}}
  \asymp
  \begin{cases}
    (nm)^{-\nu/\kappa},
  &
  m\le n^{\zeta_*},
    \\[2mm]
    n^{-\nu/(\nu+1)},
    &
    m\ge n^{\zeta_*}.
  \end{cases}
\]

In the sparse regime, the cost of discretization under independent design enters through the number \(nm\) of effective scalar observations.
The exponent \( \frac{2\alpha+2s}{4\alpha+2s+1} \) comes from the interplay between the smoothness of the slope function and the eigenvalue decay, as well as the inverse-covariance amplification incurred when converting the cross-covariance coefficient estimation error to the slope coefficient estimation error.
As a result, the exponent in the discretization term is smaller than that in the fully observed term.

This sparse to dense transition is similar to that of mean function estimation where the optimal rate is given by
\( n^{-1} + (mn)^{-(2\alpha-1)/(2\alpha)} \) under independent design~\citep{cai2011_OptimalMeanFunctionDiscrete}.
For mean function estimation, the fully observed rate is parametric but the discretization term is nonparametric,
whereas for FLR, both terms are nonparametric.
In both settings, the exponents in the two terms differ, and the discretization term has the smaller exponent.

We also compare our result to the most closely related work of \citet{zhou2023_DiscretelyObservedFLR}, who study the independent random design and provide sufficient conditions for attaining the fully observed benchmark rate in the dense regime with noisy discrete covariates.
Translating their notation to ours, their sufficient sampling condition per curve is \(m\gtrsim n^{\zeta_{\mathrm{ZYZ}}}\) with
\(
\zeta_{\mathrm{ZYZ}} = \frac{4\alpha+2}{2\alpha+2s+1} \vee \frac{2\alpha+s+1/2}{2\alpha+2s+1}.
\)
This exponent satisfies \(\zeta_{\mathrm{ZYZ}}>\zeta_*\) for all \(\alpha>1/2\) and \(s\ge0\), and the gap can be arbitrarily close to 1 when \(\alpha\) is large and \(s\) is small,
so the sufficient condition in \citet{zhou2023_DiscretelyObservedFLR} is not tight.
In contrast, by showing a matching lower bound, particularly with respect to $m$,
our result establishes the sharp phase transition in the resolution within each curve and provides a complete characterization of the cost of discretization under independent design.

Finally, we consider optimal sampling allocation under a fixed total sampling budget.
This question is relevant when the total sampling budget is fixed by cost or other constraints, while the allocation between the number of subjects and the number of measurements per subject remains flexible.
Under a fixed total scalar sampling budget \(nm=N\), where \(N\) denotes the total number of scalar point measurements and \(\nu,\kappa\) are as above, the rate under independent design with fixed budget is
\[
  \inf_{nm=N} \caR_{n,m}^{\star,\mathrm{ind}}
  \asymp
  N^{-\nu/\kappa},
\]
which is attained by any allocation satisfying \( m\lesssim N^{2\alpha/\kappa} \).
Thus, under independent design with a fixed total budget, the rate favors more subjects with fewer measurements per subject over fewer subjects with more measurements per subject.

   \section{Common Design}
\label{sec:common}

Under common design, all trajectories are observed on the same deterministic grid.
This setting is natural for many regularly sampled data sets, for example when measurements are recorded every minute or every hour on a common schedule.
We impose the following deterministic design assumption.

\begin{assumption}[Common design]
  \label{assum:common-design}
  All trajectories share the deterministic equal grid
  \[
    t_{ij}=t_j=\frac{j-1}{m},
    \qquad
    j=1,\dots,m.
  \]
\end{assumption}

The shared grid changes the role of discretization.
Unlike independent design, the sampling geometry is not averaged out across subjects, so frequencies with the same trace on the grid remain statistically coupled even as \(n\) increases.
The minimax rate under common design contains the same two statistical terms as in the case of independent design.
The shared grid also creates a term from fixed grid discretization and, when the eigenvalue sequence is unknown, an additional term from eigenvalue identification.

\subsection{Oracle Estimator}
\label{subsec:common-oracle-upper}

We first study an oracle estimator that knows the eigenvalue sequence \((\lambda_r)_{r\in\bbZ}\).
The construction is the same as the oracle estimator under independent design in \cref{eq:oracle-estimator} with a cutoff level \(d\).
However, under common design, the cross-covariance estimator \( \hat{\gamma} \) in \cref{eq:cross-covariance-coefficient}
is no longer an unbiased estimator of the population coefficient \(\gamma_r\),
since bias can arise from aliasing of distinct frequencies on the shared grid.
Here and below, ``known eigenvalues'' means that the estimator is allowed to use the particular eigenvalue sequence \((\lambda_r)_{r\in\bbZ}\).
The following upper bound shows that an additional term from fixed grid discretization appears in the oracle risk.

\begin{theorem}[Oracle upper bound under common design]
  \label{thm:common-oracle-upper}
  Assume that \(\lambda\in\caL_\alpha(c_\lambda,C_\lambda)\) is known.
  Under \cref{assum:common-design}, for every integer \(1\le d\le m/4\),
  \begin{equation}
    \label{eq:common-oracle-terms}
    \sup_{\theta\in\Theta_s(R_0)}
    \E\caR(\tilde{\beta};\theta,\lambda)
    \lesssim
    \frac{d}{n}
    +
    \frac{d^{2\alpha+1}}{nm}
    +
    d^{-(2\alpha+2s)}
    +
    m^{-(2\alpha+2s)}.
  \end{equation}
  Take the optimized cutoff
  \begin{equation}
    \label{eq:common-oracle-d}
    d^*
    \asymp
    m
    \wedge
    n^{\frac{1}{2\alpha+2s+1}}
    \wedge
    (nm)^{\frac{1}{4\alpha+2s+1}}.
  \end{equation}
  Let \(\tilde{\beta}^*\) denote the estimator \(\tilde{\beta}\) constructed with \(d=d^*\).
  Then
  \begin{equation}
    \label{eq:common-oracle-rate}
    \sup_{\theta\in\Theta_s(R_0)}
    \E\caR(\tilde{\beta}^*;\theta,\lambda)
    \lesssim
    n^{-\frac{2\alpha+2s}{2\alpha+2s+1}}
    +
    (nm)^{-\frac{2\alpha+2s}{4\alpha+2s+1}}
    +
    m^{-(2\alpha+2s)}.
  \end{equation}
\end{theorem}

The proof is given in Appendix~E.
The term \(d/n\) is the usual finite-dimensional estimation contribution.
The term \(d^{2\alpha+1}/(nm)\) is the cost due to measurement noise after division by the low frequency eigenvalues, and \(d^{-(2\alpha+2s)}\) is the bias of truncation.
The additional term \(m^{-(2\alpha+2s)}\) represents the reconstruction error left by the fixed grid.
These terms are discussed further in \cref{subsec:common-discussion}.

\subsection{Lower Bound}
\label{subsec:common-lower-bound}

We next prove minimax lower bounds under common design.
The first part is the oracle experiment with a fixed known eigenvalue sequence, and its lower bound contains the two statistical terms from independent design together with the fixed grid discretization term.
The second part is the genuine unknown eigenvalue experiment; in that case an additional identification term from the common grid is unavoidable.

\begin{theorem}[Lower bound under common design]
  \label{thm:common-lower}
  Let \cref{assum:common-design} hold.
  If \( \lambda\in\caL_\alpha(c_\lambda,C_\lambda) \) is fixed and known,
  then
  \begin{equation}
    \label{eq:common-oracle-lower-rate}
    \inf_{\hat{\beta}}
    \sup_{\theta\in\Theta_s(R_0)}
    \E\caR(\hat{\beta};\theta,\lambda)
    \gtrsim
    n^{-\frac{2\alpha+2s}{2\alpha+2s+1}}
    +
    (nm)^{-\frac{2\alpha+2s}{4\alpha+2s+1}}
    +
    m^{-(2\alpha+2s)}.
  \end{equation}
  Moreover, in the case when $\lambda$ is unknown, we have
  \begin{equation}
    \label{eq:common-lower-rate}
    \inf_{\hat{\beta}}
    \sup_{\lambda\in\caL_\alpha(c_\lambda,C_\lambda)}
    \sup_{\theta\in\Theta_s(R_0)}
    \E\caR(\hat{\beta};\theta,\lambda)
    \gtrsim
    n^{-\frac{2\alpha+2s}{2\alpha+2s+1}}
    +
    (nm)^{-\frac{2\alpha+2s}{4\alpha+2s+1}}
    +
    m^{-(2\alpha+2s)}
    +
    m^{-4\alpha}.
  \end{equation}
\end{theorem}

\cref{thm:common-lower} is proved in Appendix~G.
The two statistical terms are obtained through the same van Trees strategy as in the case of independent design, with the Fisher information evaluated for the deterministic grid experiment.
The term \(m^{-(2\alpha+2s)}\) is proved by constructing alternatives whose induced cross-covariance functions vanish on the grid and are therefore indistinguishable no matter how large \(n\) is.
The \(m^{-4\alpha}\) term enters only when the eigenvalues are unknown; it is obtained from alternatives with the same law on the grid but different covariance scales for inversion.
These two lower bound constructions are qualitatively different.
The first keeps the covariance scale fixed and hides regression signal between grid points through the cross-covariance \(g=\Sigma\beta\).
The second varies the covariance scale itself, showing that even when the observed grid distribution is unchanged, the regression target in this inverse problem can differ enough to create the \(m^{-4\alpha}\) term.

\subsection{Adaptive Estimator}
\label{subsec:common-adaptive-upper}

We now construct an adaptive estimator for common design with unknown eigenvalues and unknown smoothness indices.
The construction follows the adaptive estimator under independent design in \cref{subsec:indep-adaptive-upper}, with modifications in the covariance pilot to account for the shared grid.
The main difficulty is that, under a common grid, the eigenvalues cannot be estimated in the same way as in the independent-design case.
When \(m\) is even, one may attempt to use an interlaced partition of the grid to construct an estimator for \(\lambda_r\) similar to \cref{eq:indep-adaptive-lambda-pilot}, but the same approach fails for odd \(m\) because of nonnegligible bias.
We therefore introduce a covariance pilot based on the difference between the active band and a fixed floor at high frequency; this pilot is accurate enough for blockwise adaptation.
The role of this pilot is twofold: it removes the measurement-noise contribution from grid Fourier coefficients and provides a scale for the inverse operation in the blockwise selection rule.

We split the subjects into two groups as before, denoted by \(I_\lambda\cup I_\gamma=\{1,\dots,n\}\) with \(n_\lambda\asymp n_\gamma\asymp n\).
With \(\bar{Z}_{ir}\) defined as in \cref{eq:cross-covariance-coefficient} (so here \(t_{ij}=t_j\)),
we define the pilot estimator for \(\lambda_r\) as the difference between the active band and a fixed floor at high frequency:
\begin{equation}
  \label{eq:common-adaptive-lambda-pilot}
  \check{\lambda}_r
  \coloneqq
  \frac{1}{n_\lambda}\sum_{i\in I_\lambda}
  \left(\bar{Z}_{ir}^2-\bar{Z}_{iH_m}^2 \right),
\end{equation}
where we take \( H_m = \floor{(m-1)/2} \).
This subtraction removes the additional variance contribution in \( \bar{Z}_{ir}^2 \) due to noisy observations.
At the same time, the frequency \(H_m\) is chosen not too high, to avoid excessive variance, and not too low, so that \(\lambda_{H_m}\) is negligible.

Under common design, the active band must be chosen more conservatively to ensure that the bias of \(\hat{\gamma}_r\) does not dominate the variance.
For a sufficiently small constant \(c_L\in(0,1/8)\), we define the active band cutoff as
\[
  L_m
  \coloneqq
  \floor{
    \log_2 \xk{c_L m\wedge n}
  }
  \qquad
  K_m \coloneqq 2^{L_m}.
\]
Additionally, the thresholding level must also be modified to reflect the different balance between bias, variance and additional bias created by the shared grid under common design.
Instead of \cref{eq:indep-adaptive-lambda-threshold}, we use
\begin{equation}
  \label{eq:common-adaptive-lambda-threshold}
  \zeta_{n,m}
  \coloneqq
  M_0 \left(
        \frac{1}{m}\sqrt{\frac{\ell_{n,m}^{\lambda}}{n_\lambda}}
        +
        \frac{\ell_{n,m}^{\lambda}}{n_\lambda}
  \right),
  \quad
  \ell_{n,m}^{\lambda} \coloneqq \log\bigl(n(K_m+1)\bigr),
\end{equation}
where \(M_0>0\) is sufficiently large.
The adaptive estimator for common design, still denoted by \(\hat{\beta}\), is then defined as in \cref{eq:indep-adaptive-estimator} with the above modifications.
The following upper bound states that the adaptive estimator attains the minimax rate.
The proof is given in Appendix~F.

\begin{theorem}[Adaptive upper bound under common design]
  \label{thm:common-adaptive-upper}
  Under \cref{assum:common-design}, uniformly over any compact subset of
  \(\{(\alpha,s):\alpha>1/2,\ s\ge0\}\), we have
  \begin{equation}
    \label{eq:common-adaptive-rate}
    \sup_{\lambda\in\caL_\alpha(c_\lambda,C_\lambda)}
    \sup_{\theta\in\Theta_s(R_0)}
    \E\caR(\hat{\beta};\theta,\lambda)
    \lesssim
    n^{-\frac{2\alpha+2s}{2\alpha+2s+1}}
    +
    (nm)^{-\frac{2\alpha+2s}{4\alpha+2s+1}}
    +
    m^{-(2\alpha+2s)}
    +
    m^{-4\alpha}.
  \end{equation}
\end{theorem}

\cref{thm:common-adaptive-upper} shows that the adaptive procedure matches all four terms in the lower bound when the eigenvalues are unknown under common design.
In particular, the fourth term in \cref{eq:common-adaptive-rate} should not be interpreted as a technical price of the adaptive construction.
It is the minimax obstruction created by unknown covariance scales on a shared grid, as reflected by the lower bound in \cref{eq:common-lower-rate}.
The last two terms have different origins: \(m^{-(2\alpha+2s)}\) is the term from fixed grid discretization already present with known eigenvalues, whereas \(m^{-4\alpha}\) is the additional cost of estimating a covariance scale that can be aliased on the common grid.
Keeping these terms separate is important because they dominate in different smoothness regimes and lead to different phase transitions.

\subsection{Discussion}
\label{subsec:common-discussion}

Denote the minimax risk under common design as
\[
  \caR_{n,m}^{\star,\mathrm{cd}}
  \coloneqq
  \inf_{\hat{\beta}}
  \sup_{\lambda\in\caL_\alpha(c_\lambda,C_\lambda)}
  \sup_{\theta\in\Theta_s(R_0)}
  \E\caR(\hat{\beta};\theta,\lambda).
\]
Combining \cref{thm:common-lower} and \cref{thm:common-adaptive-upper} yields the following minimax rate.

\begin{corollary}[Minimax rate under common design]
  \label{cor:common-minimax}
  Under \cref{assum:common-design},
  \begin{equation}
    \label{eq:common-minimax-rate}
    \caR_{n,m}^{\star,\mathrm{cd}}
    \asymp
    \underbrace{n^{-\frac{2\alpha+2s}{2\alpha+2s+1}}}_{\mathrm{I}}
    +
    \underbrace{(nm)^{-\frac{2\alpha+2s}{4\alpha+2s+1}}}_{\mathrm{II}}
    +
    \underbrace{m^{-(2\alpha+2s)}}_{\mathrm{III}}
    +
    \underbrace{m^{-4\alpha}}_{\mathrm{IV}}.
  \end{equation}
\end{corollary}

The minimax rate in \cref{eq:common-minimax-rate} has four terms.
The first two terms are the same as in the case of independent design: the term (I) is the fully observed FLR benchmark, and the term (II) is the contribution caused by noisy point evaluations.
The last two terms are specific to the common grid.
The term (III) is the approximation error from the fixed grid that is already present in the oracle analysis on an equal grid,
whereas the term (IV) is the additional cost of identifying the covariance scale needed for inversion when the eigenvalues are unknown.
We next discuss several implications of this minimax rate.

\subsubsection{Phase transitions}
\label{subsubsec:common-phase-transitions}

\begin{figure}[t]
  \centering
  \includegraphics[width=1\textwidth]{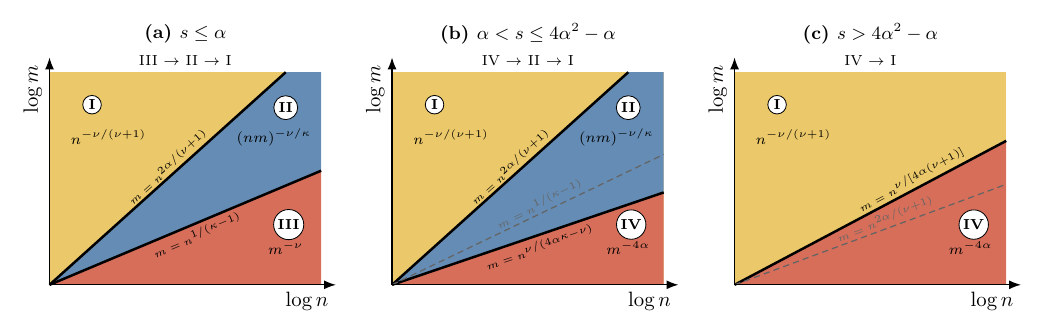}
  \caption{Phase diagram for the minimax rate under common design. The dominant term in \cref{eq:common-minimax-rate} changes with the grid size \(m\) within each trajectory, and the transition points depend on the relative size of \(s\) and \(\alpha\).}
  \label{fig:common-design-phase-diagram}
\end{figure}

We identify the phase transitions between the four rates across different regimes of \(m\) and \((\alpha,s)\).
To simplify the presentation, write
\[
  \nu\coloneqq 2\alpha+2s,
  \qquad
  \kappa\coloneqq 4\alpha+2s+1.
\]
\Cref{fig:common-design-phase-diagram} displays the following regimes:

\begin{enumerate}[label=(\alph*)]
  \item If \(s\le \alpha\), the impact of unknown eigenvalues is negligible compared to the approximation error from the fixed grid, so
  the rate coincides with the oracle rate and has the form
  \[
    \caR_{n,m}^{\star,\mathrm{cd}}
    \asymp
    \begin{cases}
      m^{-\nu},
      &
      m\le n^{1/(\kappa-1)},
      \\[2mm]
      (nm)^{-\nu/\kappa},
      &
      n^{1/(\kappa-1)} \le m\le n^{2\alpha/(\nu+1)},
      \\[2mm]
      n^{-\nu/(\nu+1)},
      &
      m\ge n^{2\alpha/(\nu+1)}.
    \end{cases}
  \]
  \item If \(\alpha<s\le 4\alpha^2-\alpha\), then the term for identifying unknown eigenvalues dominates the approximation error from the fixed grid at low resolution, while the sampling term still determines an intermediate regime.
  We have
  \[
    \caR_{n,m}^{\star,\mathrm{cd}}
    \asymp
    \begin{cases}
      m^{-4\alpha},
      &
      m\le n^{\nu/(4\alpha\kappa-\nu)},
      \\[2mm]
      (nm)^{-\nu/\kappa},
      &
      n^{\nu/(4\alpha\kappa-\nu)} \le m\le n^{2\alpha/(\nu+1)},
      \\[2mm]
      n^{-\nu/(\nu+1)},
      &
      m\ge n^{2\alpha/(\nu+1)}.
    \end{cases}
  \]
  In this range \(m^{-4\alpha}\) dominates \(m^{-\nu}\) at low resolution, but the sampling term still forms a nonempty intermediate regime.
  \item If \(s>4\alpha^2-\alpha\), then \(m^{-4\alpha}\) is the dominant term when $m$ is small.
  The transition is instead between the covariance identification term and the fully observed benchmark.
  We have
  \[
    \caR_{n,m}^{\star,\mathrm{cd}}
    \asymp
    \begin{cases}
      m^{-4\alpha},
      &
      m\le n^{\nu/[4\alpha(\nu+1)]},
      \\[2mm]
      n^{-\nu/(\nu+1)},
      &
      m\ge n^{\nu/[4\alpha(\nu+1)]}.
    \end{cases}
  \]
\end{enumerate}

\subsubsection{Comparison with independent design}

We compare the minimax rate under common design with the rate under independent design in \cref{cor:indep-minimax}.
The two statistical terms (I) and (II) coincide in both settings, showing the fundamental statistical difficulty of the problem.
However, the sampling geometry changes the nature of discretization,
so the discretization cost under common design contains two additional terms (III) and (IV).
Under common design, all trajectories are observed through the same grid projection,
so increasing \(n\) only reduces the stochastic error on this projection,
but it does not remove aliasing or errors in identifying the denominators created by the fixed grid.

The threshold for the dense regime is consequently different.
For independent design, the fully observed benchmark is reached once
\( m \gtrsim n^{\zeta_{\mr{ind}}} \) with \( \zeta_{\mr{ind}} = 2\alpha/(\nu+1) \).
For common design, the dense regime requires
\( m \gtrsim n^{\zeta_{\mr{ind}} \vee \zeta_{\mr{cd}}} \), where \( \zeta_{\mr{cd}} = \nu/[4\alpha(\nu+1)] \).
The second condition is active precisely in the high smoothness range \(s>4\alpha^2-\alpha\).

\subsubsection{Comparison with mean estimation}

It is instructive to compare \cref{eq:common-minimax-rate} with the corresponding minimax theory for functional mean or covariance estimation under discrete noisy observations
\citep{cai2011_OptimalMeanFunctionDiscrete,cai2010_NonparametricCovarianceFunction}.
In mean estimation, the rates under independent and common design take the schematic forms
\[
  n^{-1} + (mn)^{-(2\alpha-1)/(2\alpha)}
  \qquad\text{and}\qquad
  n^{-1}+m^{-(2\alpha-1)},
\]
respectively.
Both rates enter the dense regime at the same resolution scale \(m\asymp n^{1/(2\alpha-1)}\).
Moreover, throughout the range \(m\lesssim n^{1/(2\alpha-1)}\), the sampling term under independent design is no larger than the term under a common grid, while above this scale both sampling terms are dominated by the parametric term \(n^{-1}\).
Thus the phase transition for mean estimation remains essentially a transition between two regimes: a sampling-limited regime and a dense parametric regime.

The FLR rate under common design is more intricate because the regression problem is an inverse problem and the slope smoothness \(s\) enters the balance.
The sampling term is affected by inversion through \(\lambda_r^{-1}\), whereas the terms from the fixed grid and from eigenvalue identification have different dependence on \(s\).
Due to the joint dependence on \((\alpha,s)\), the sampling term (II) is no longer dominated by the additional grid terms (III) and (IV) at low resolution.
Consequently, the phase diagram can have up to four distinct regimes.

The term \(m^{-4\alpha}\) is specific to this FLR setting and has no analogue in mean estimation.
It depends only on the covariance decay and the grid resolution, not on the smoothness \(s\) of the slope.
Consequently, for sufficiently smooth slopes, the leading obstruction under common design is not the approximation of \(\beta\) but the identification of the covariance scale required to transform the cross-covariance into the slope.

\subsubsection{Source of the two grid error terms}

We next give further insight into the two additional grid error terms.
From the perspective of the upper bound, term (III) comes from the bias of the cross-covariance coefficient estimator \(\hat{\gamma}_r\) due to aliasing on the common grid.
Indeed, the shared grid makes \(\hat{\gamma}_r\) a biased estimator of the population coefficient \(\gamma_r\), and the bias can be large enough to dominate the variance at low frequencies.
Term (IV) enters through the need to invert the covariance operator in the regression problem:
under the fixed grid, eigenvalue estimation is biased, which leads to a nonnegligible error in the estimation of the slope function.

From the perspective of the lower bound, the two error terms correspond to two geometric obstructions.
The term (III) reflects the approximation error created by the fixed grid, since small bumps lying between the grid points cannot be detected by the observations on the grid.
When the eigenvalues are unknown, indistinguishability can also arise from the eigenvalues, yielding a separate \(m^{-4\alpha}\) identification cost (IV).

\subsubsection{Optimal sampling allocation}
\label{subsubsec:common-optimal-allocation}

Finally, consider the optimal allocation of the sampling budget \(nm=N\) under common design, where \(N\) denotes the total number of scalar point measurements and \(\nu,\kappa\) are as in the phase-transition display above.
Minimizing \cref{eq:common-minimax-rate} over \((n,m)\) gives the following regimes:
\begin{itemize}
  \item If \(s\le \alpha\), then
  \[
    \inf_{nm=N} \caR_{n,m}^{\star,\mathrm{cd}}
    \asymp
    N^{-\nu/\kappa},
    \quad\text{for}\quad
    N^{1/\kappa} \lesssim m\lesssim N^{2\alpha/\kappa}.
  \]
  \item If \(\alpha<s\le 4\alpha^2-\alpha\), then
  \[
    \inf_{nm=N} \caR_{n,m}^{\star,\mathrm{cd}}
    \asymp
    N^{-\nu/\kappa},
    \quad\text{for}\quad
    N^{\nu/(4\alpha\kappa)} \lesssim m\lesssim N^{2\alpha/\kappa}.
  \]
  \item If \(s>4\alpha^2-\alpha\), then the optimal allocation is given by
  \[
    m^\star \asymp N^{\nu/[4\alpha(\nu+1)+\nu]},
    \quad
    n^\star \asymp N^{4\alpha(\nu+1)/[4\alpha(\nu+1)+\nu]},
  \]
  and
  \[
    \inf_{nm=N} \caR_{n,m}^{\star,\mathrm{cd}}
    \asymp
    N^{-4\alpha\nu/[4\alpha(\nu+1)+\nu]}.
  \]
  In this high smoothness range, the optimal choice balances the fully observed term with the eigenvalue identification term, so a nonnegligible part of the budget must be spent on increasing the grid resolution.
\end{itemize}
Thus, under common design, an optimal allocation must reserve enough budget for the grid resolution itself.
Increasing the number of trajectories cannot compensate for a grid that is too coarse to control the \(m\)-only terms.
This contrasts with independent design, where the optimal rate under a fixed budget can be attained by allocating the budget primarily to the number of trajectories.
   \section{Numerical Experiments}
\label{sec:numerics}
\FloatBarrier

In this section, we present simulations that illustrate the qualitative implications of the theory and a real data example that compares the practical estimator with established benchmarks.
In the numerical implementation, the estimators use all training subjects for eigenvalue and coefficient estimation for better finite-sample performance.

\subsection{Simulations}
\label{subsec:numerics-simulations}

In the simulations, we use the same trigonometric coordinates as in \cref{eq:trig-basis}.
The data-generating model is truncated to \(\absx{r}\le K_0\) with \(K_0=64\), giving a total of 129 basis functions.
For a smoothness pair \((\alpha,s)\), the true covariance and slope coefficients are
\[
  \lambda_r = (1+\absx{r})^{-2\alpha}, \qquad
  \theta_r = c_s (1+\absx{r})^{-(s+1/2)}, \qquad \absx{r}\le K_0,
\]
where the constant \(c_s\) is chosen so that \(\sum_{\absx{r}\le K_0}(1+\absx{r})^{2s}\theta_r^2=1\).
This choice of $\theta_r$ lies on the boundary of the Sobolev ellipsoid in \cref{eq:class-theta} with radius one.
For each subject, we generate independent scores \(\xi_{ir}\sim N(0,\lambda_r)\), set \(X_i(t)=\sum_{\absx{r}\le K_0}\xi_{ir}e_r(t)\), and generate
\[
  Y_i = \sum_{\absx{r}\le K_0}\xi_{ir}\theta_r + \varepsilon_i,\qquad \varepsilon_i\sim N(0,0.5^2).
\]
The observed functional data are \(Z_{ij}=X_i(t_{ij})+\delta_{ij}\), with independent measurement errors \(\delta_{ij}\sim N(0,0.1^2)\) in the qualitative phase panels.
For common design, \(t_{ij}=(j-1)/m\) is a common equispaced grid for all subjects; for independent design, \(t_{ij}\) are independently sampled from the uniform distribution on \([0,1]\).
The reported excess prediction risk is computed against the true coefficients as \(\sum_{\absx{r}\le K_0}\lambda_r(\hat{\theta}_r-\theta_r)^2\).
The implementation under common design uses the high-frequency control pilot on the full grid in \cref{eq:common-adaptive-lambda-pilot}, with fixed tuning constants across all settings.

\begin{figure}[htpb]
  \centering
  \includegraphics[width=1\textwidth]{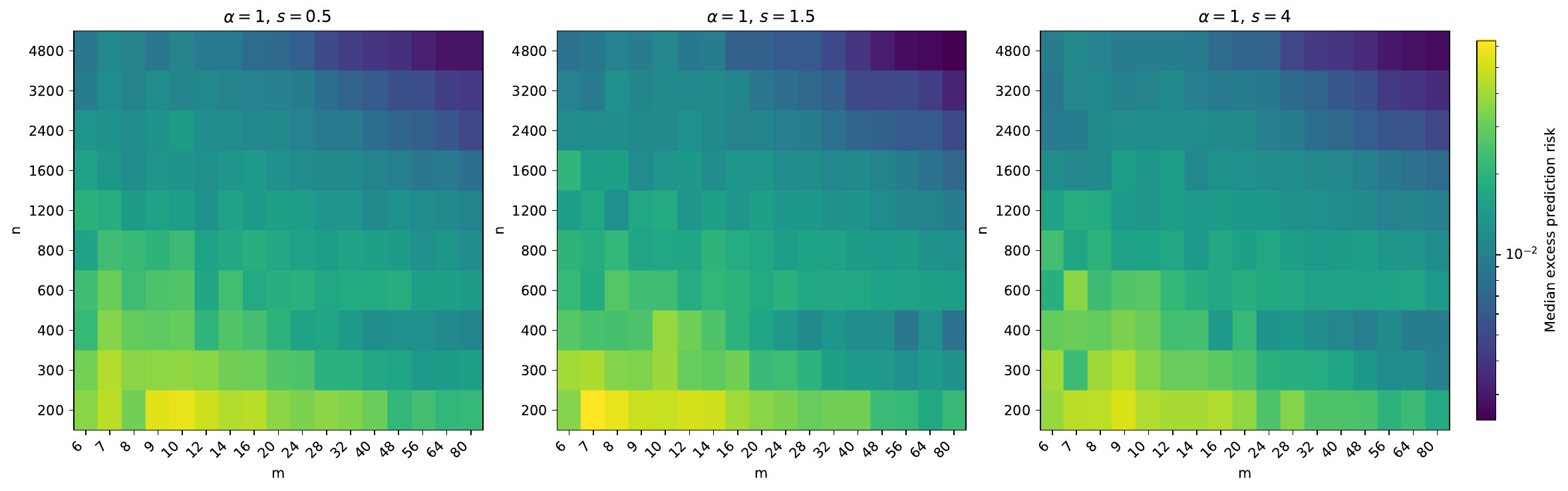}
  \vspace{0.5em}
  \includegraphics[width=1\textwidth]{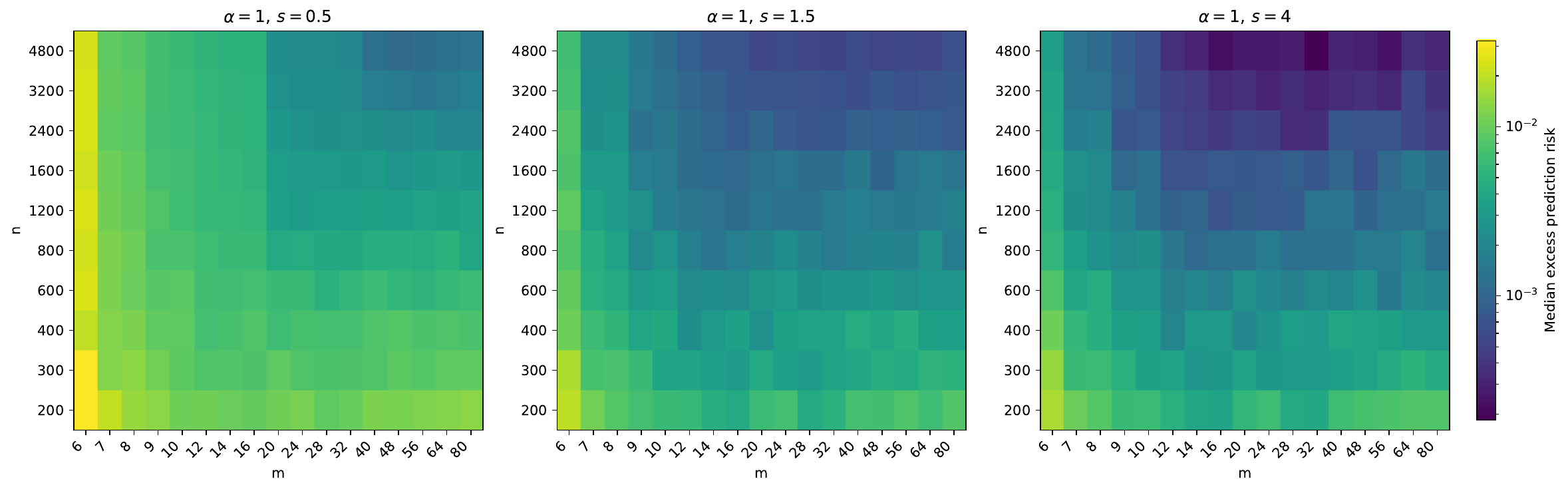}
  \caption{Simulation heatmaps under independent design (top) and common design (bottom).
  Each cell plots the median excess prediction risk over \(50\) repetitions.
  For common design, the three panels correspond to the three regimes in \Cref{subsubsec:common-phase-transitions}.
  }
  \label{fig:simulation-qualitative-heatmaps}
\end{figure}

We report the results for three smoothness pairs \((\alpha,s) = (1,0.5), (1,1.5), (1,4)\) under both independent and common design in \Cref{fig:simulation-qualitative-heatmaps}.
For independent design, the three panels display a similar qualitative pattern.
The risk decreases as one moves from the lower left corner to the upper right corner, reflecting the two-term structure of the rate under independent design in \cref{eq:indep-minimax-rate}.
The first term improves with the number of trajectories, while the second improves with the total number of noisy point evaluations.
Changing \(s\) modifies the theoretical exponents, but the overall geometry of the heatmaps remains comparable across the three smoothness pairs.

The common-design panels show a different behavior.
When \(m\) is small, increasing \(n\) alone produces only limited improvement, because the shared grid leaves an approximation or identification error that depends on \(m\) and cannot be averaged away over subjects.
The three panels also have visibly different decay patterns, in agreement with the three regimes in \cref{subsubsec:common-phase-transitions}.
For \((\alpha,s)=(1,0.5)\), the dominant effect at low resolution is the term \(m^{-(2\alpha+2s)}\) from the fixed grid, so the risk decreases mainly as the grid resolution \(m\) increases.
For \((\alpha,s)=(1,1.5)\) and \((1,4)\), the heatmaps exhibit a more mixed dependence on \(m\) and \(n\), consistent with the interaction between the sampling term, the term for identifying unknown eigenvalues, and the fully observed benchmark.

\begin{figure}[htpb]
  \centering
  \includegraphics[width=1\textwidth]{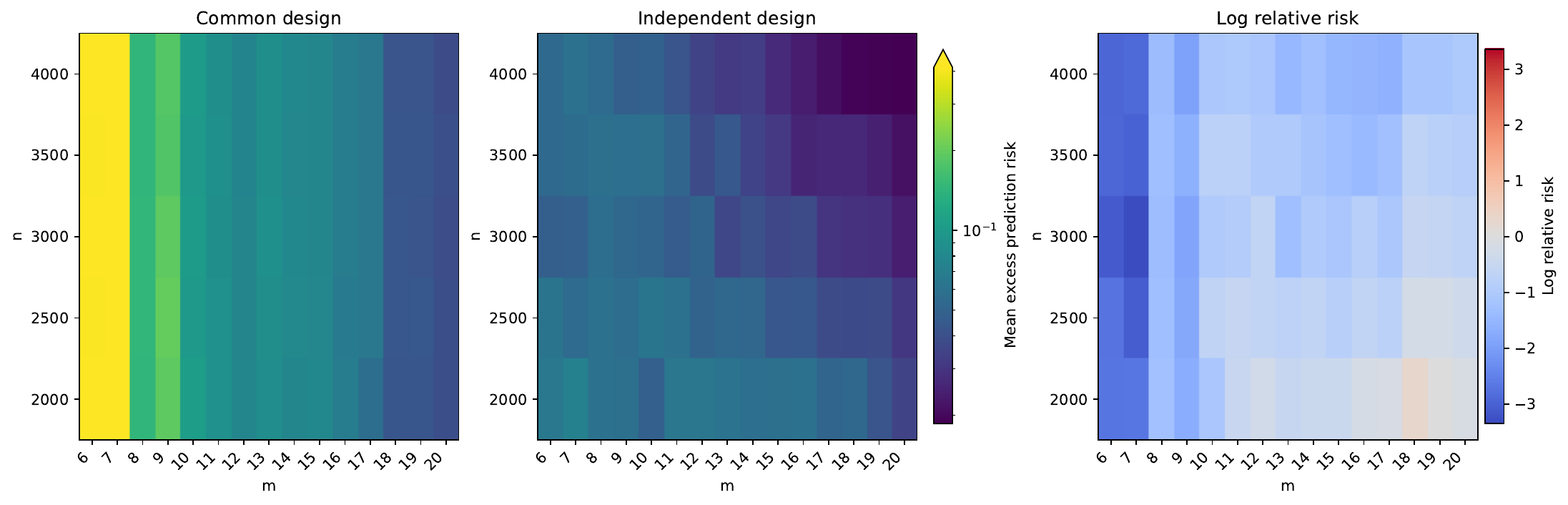}
  \caption{Matched comparison between common and independent design for \((\alpha,s)=(0.6,0.1)\).
  The two left panels show mean excess prediction risk, and the right panel shows \(\log_2(R_{\mathrm{IND}}/R_{\mathrm{CD}})\), so negative values (blue) favor independent design.
  }
  \label{fig:design-comparison-heatmap}
\end{figure}

\Cref{fig:design-comparison-heatmap} compares the two observation schemes on the same grid of $(m,n)$.
We choose relatively large $n$ to show the asymptotic behavior more clearly.
The common-design estimator remains strongly constrained by the grid when \(m=6\) or \(7\): increasing \(n\) alone gives little improvement because the same low-resolution grid is shared by all subjects.
As \(m\) increases, the common-design risk decreases and the gap between the two designs narrows.
The independent-design estimator is nevertheless smaller over most of the displayed grid, reflecting the additional fixed grid error terms under common design that cannot be averaged away across subjects.

\begin{figure}[htpb]
  \centering
  \includegraphics[width=1\textwidth]{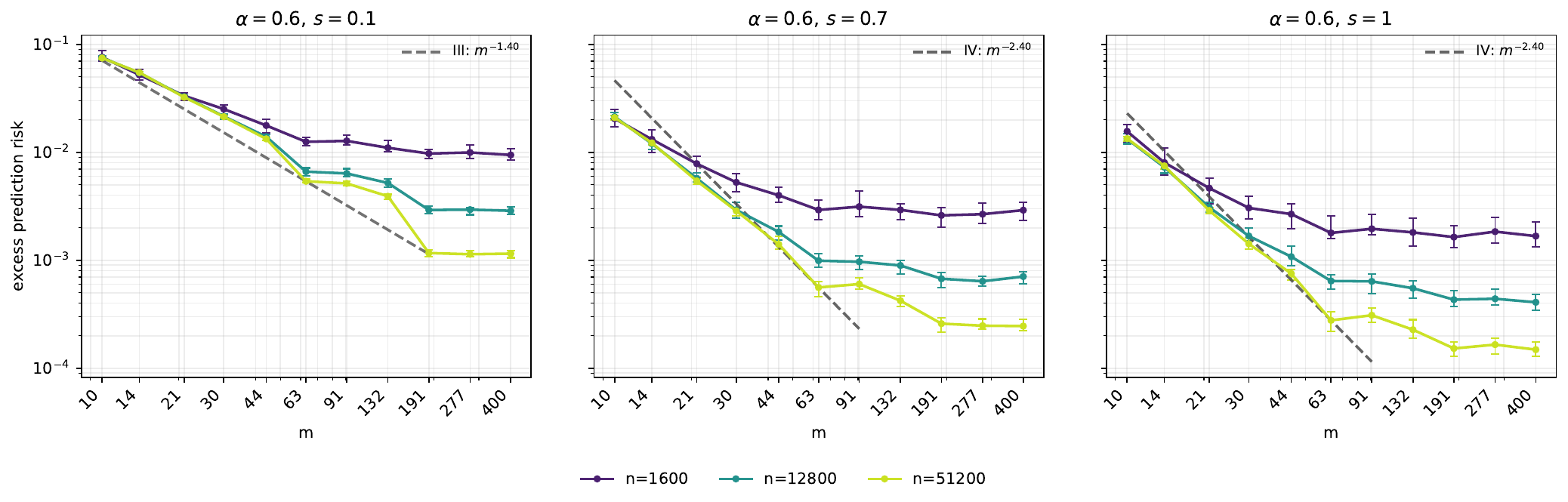}
  \vspace{0.5em}
  \includegraphics[width=1\textwidth]{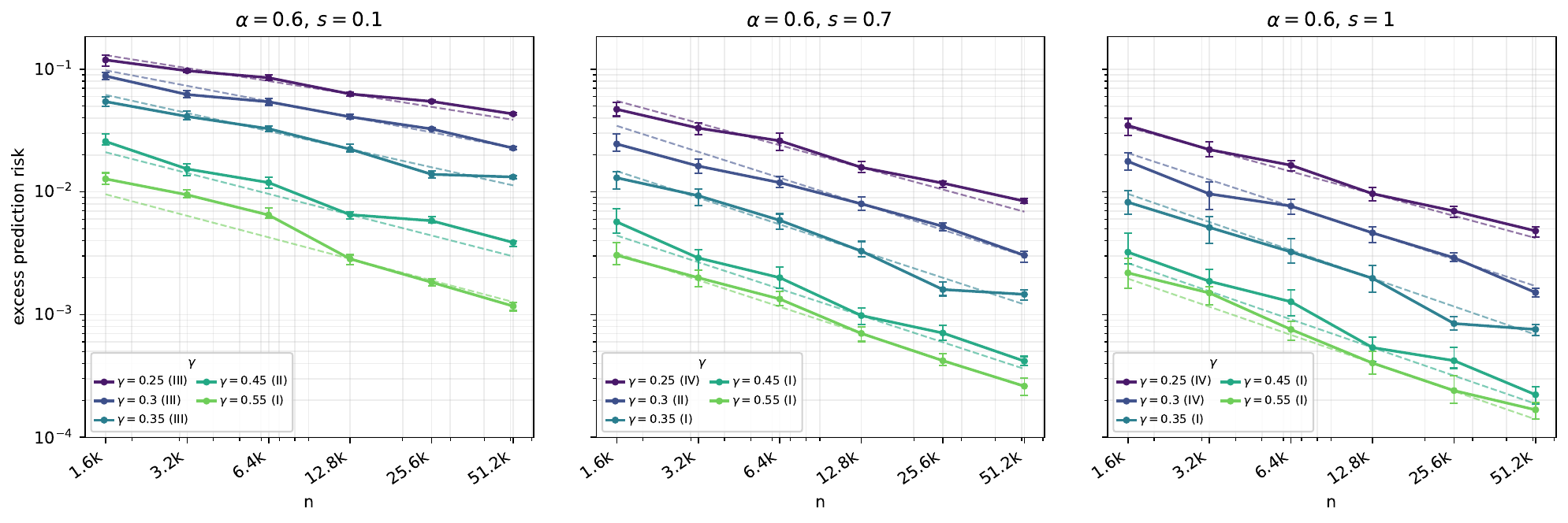}
  \caption{Common design rate comparisons as \(m\) varies (top) and along \(m=n^\gamma\) (bottom).
  Each panel uses \(50\) repetitions and plots the median excess prediction risk with interquartile error bars.
  The dashed reference lines show the low-resolution term predicted by \cref{eq:common-minimax-rate} in the top plots,
  and show the dominant rate in the bottom plots.
  }
  \label{fig:common-design-rate-comparisons}
\end{figure}

To further investigate the rate behavior under common design, we compare the empirical rates with the theoretical prediction in more detail.
To display the qualitative rate behavior more clearly and isolate the effect of biased coefficient and eigenvalue estimation, we use a version of the adaptive estimator with an oracle choice of the truncation level instead of the block thresholding procedure.
We consider three smoothness pairs \((\alpha,s)=(0.6,0.1),(0.6,0.7),(0.6,1)\), which lie in the three regimes in \cref{subsubsec:common-phase-transitions}, and report the results in \Cref{fig:common-design-rate-comparisons}.
The top part of \Cref{fig:common-design-rate-comparisons} plots the risk against \(m\) under various values of \(n\).
For each fixed \(n\), the risk decreases as \(m\) grows until the grid is fine enough that the part of the error depending on \(n\) becomes dominant.
After this point, the curves form a plateau, and the plateau level decreases as \(n\) increases.
This pattern matches the phase transition described by \cref{eq:common-minimax-rate}.
The dashed reference lines show the low-resolution terms predicted by the theory: \(m^{-(2\alpha+2s)}\) in the first panel and \(m^{-4\alpha}\) in the second and third panels, and the empirical curves align closely with these theoretical guides.
The bottom part of \Cref{fig:common-design-rate-comparisons} gives the corresponding comparison along paths \(m=n^\gamma\).
Different values of \(\gamma\) correspond to different regimes in the phase diagram and hence to different dominant terms in the rate.
As \(\gamma\) increases, the curves move from grid-constrained behavior toward the statistical terms, and the transition occurs at different values of \(\gamma\) across the three smoothness pairs.
The empirical curves again align with the theoretical rates shown by the dashed reference lines, supporting the predicted rate behavior in the different regimes.
Together, the two plots support the rate behavior predicted by the theory.

\begin{figure}[htpb]
  \centering
  \includegraphics[width=1\textwidth]{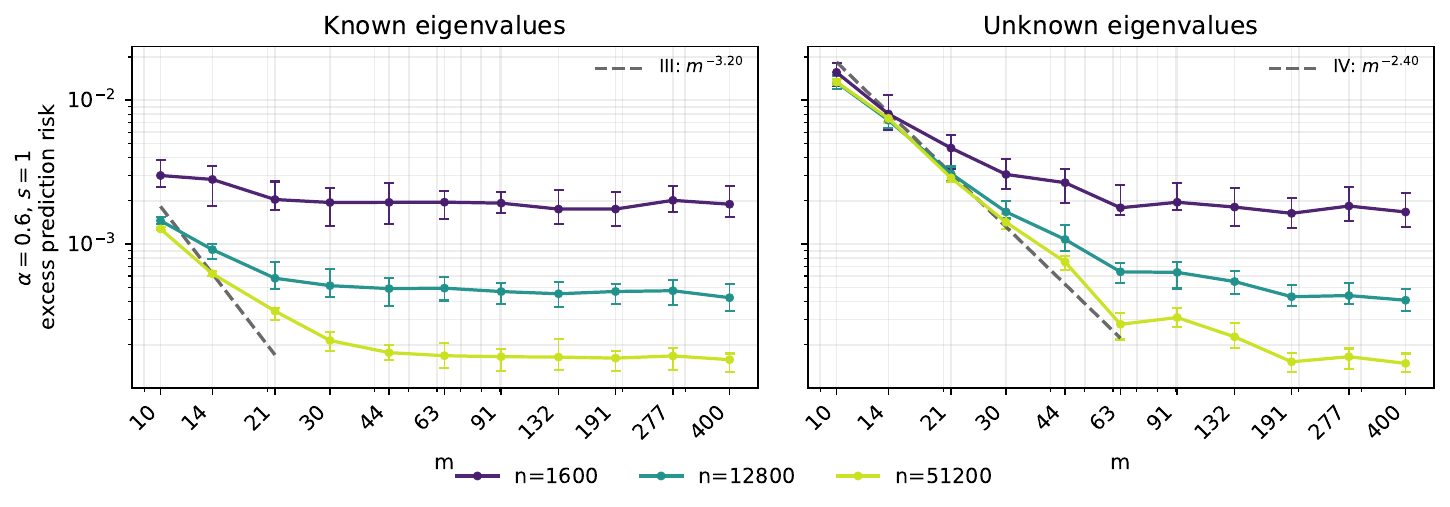}
  \caption[Known versus estimated eigenvalues under common design]{Known versus estimated eigenvalues under common design for \((\alpha,s)=(0.6,1)\).
  The two panels use the same vertical scale.
  }
  \label{fig:common-known-unknown-oracle-loglog-alpha06-s10}
\end{figure}

\Cref{fig:common-known-unknown-oracle-loglog-alpha06-s10} isolates the effect of estimating the eigenvalue sequence under common design.
It compares the oracle estimator, which uses the true eigenvalues, with the corresponding estimator that estimates the eigenvalues from data on the common grid.
The known eigenvalue estimator has smaller prediction error and reaches the dense plateau at a much smaller grid size \(m\).
The plateau levels are nevertheless comparable, as both estimators are then mainly governed by the dense FLR term, which does not depend on the grid resolution.
This contrast is consistent with the additional unknown eigenvalue term in \cref{eq:common-minimax-rate}.

\subsection{Real Data}
\label{subsec:numerics-real-data}

We next compare our estimators with other FLR methods on the wheat data set, a commonly used benchmark in the FLR literature.
The data contain \(100\) near-infrared spectra measured on a common grid of \(701\) wavelengths from \(1100\) nm to \(2500\) nm, with wheat protein content as the scalar response.
Following the experimental setup in \citet{zhou2023_DiscretelyObservedFLR}, we use the first \(80\) samples for training and the last \(20\) samples for testing, and compare performance under sparse observation settings with \(m=10\) and \(m=30\) observed wavelengths per curve.
For the estimator under common design, we use equispaced common subgrids with \(m=10,30\), so no subsampling randomness is introduced.
The estimator under independent design is evaluated under random subsampling within each curve with \(m=10,30\) over \(200\) repetitions.
The results and the published benchmarks from \citet{zhou2023_DiscretelyObservedFLR} are reported in \Cref{tab:wheat-real-data}.
In this evaluation, the estimator under common design has the smallest reported prediction error.
While seemingly contrary to the theory, this finite-sample behavior may reflect smaller hidden constants for equispaced fixed grids than for random subsampling.
The comparison should nevertheless be interpreted with the sampling protocol in mind: the common design estimator uses deterministic subgrids, whereas the independent design estimator and the published benchmarks are evaluated under random subsampling protocols.

\begin{table}[htpb]
  \centering
  \caption{Wheat data prediction error.
  CD for common design and IND for independent design are the two estimators from this paper.
  The benchmarks are quoted from Table 2 in \citet{zhou2023_DiscretelyObservedFLR}:
  Plug-in~\citep{hall2007_MethodologyConvergenceFLR}; IN: approximated least square method with integrated scores;
  PACE~\citep{yao2005_FLRLongitudinal}.
  Numbers in parentheses are standard deviations over repetitions; no standard deviation is reported for CD because the common subgrid is deterministic.
  }
  \label{tab:wheat-real-data}
  \resizebox{\textwidth}{!}{%
  \begin{tabular}{l c c c c c c}
    \hline
    \(m\) & CD & IND & Plug-in & IN & PACE & Zhou et al. \\
    \hline
    \(10\)  & {\(\bm{0.294}\)} & \(0.473\;(0.037)\) & \(0.781\;(0.840)\) & \(0.499\;(0.198)\) & \(1.517\;(9.112)\) & \(0.345\;(0.074)\) \\
    \(30\)  & {\(\bm{0.252}\)} & \(0.396\;(0.044)\) & \(0.513\;(0.925)\) & \(0.398\;(0.183)\) & \(0.363\;(0.244)\) & \(0.278\;(0.066)\) \\
    \hline
  \end{tabular}
  }
\end{table}
   \section{Discussion and Future Directions}
\label{sec:discussion}

This paper establishes sharp minimax rates for functional linear regression with noisy discrete observations under two canonical sampling designs. In the independent design, each functional trajectory is observed on its own random grid; in the common design, all trajectories are observed on a shared fixed grid. These two designs represent substantially different statistical regimes. In the independent design, the randomness of the observation grids provides a form of averaging across trajectories, while in the common design the shared sampling geometry creates additional identifiability and approximation issues. By treating the two settings separately, our results clarify how the within-trajectory resolution \(m\) interacts with the sample size \(n\), the smoothness of the slope function, and the decay of the covariance spectrum.

A central contribution of the paper is to isolate the precise cost of discretization in functional linear regression. Both sampling designs inherit the fully observed benchmark rate, corresponding to the idealized setting in which the entire functional predictor is available without discretization error. In addition, both designs contain a term reflecting the statistical cost of recovering the relevant cross-covariance information from noisy point evaluations. This term captures the effective loss of information caused by observing only \(m\) noisy measurements per trajectory rather than the full curve. Under the common design, however, two further effects appear. One is the deterministic approximation error induced by the fixed grid; the other is the cost of identifying or reliably estimating the unknown covariance eigenvalues on the grid. These additional terms are absent, or substantially attenuated, under the independent design because independent random grids provide more diversified sampling information across subjects.

Although the paper focuses on excess prediction risk, the same arguments extend naturally to the \(L^2\) estimation risk
\[
    \E \|\hat{\beta}-\beta\|_{L^2}^2 .
\]
In Fourier coordinates, prediction risk weights the squared coordinate error by \(\lambda_r\), while \(L^2\) risk removes this weight. Thus high-frequency components become more difficult to estimate under \(L^2\) loss. For \(s>0\), the independent-design estimation rate becomes
\[
    n^{-\frac{2s}{2\alpha+2s+1}}
    +
    (nm)^{-\frac{2s}{4\alpha+2s+1}} .
\]
Under the common design with unknown eigenvalues, the corresponding rate is
\[
    n^{-\frac{2s}{2\alpha+2s+1}}
    +
    (nm)^{-\frac{2s}{4\alpha+2s+1}}
    +
    m^{-2s}
    +
    m^{-4\alpha}.
\]
At the proof level, this change amounts to normalizing the cross-covariance estimation error by \(\lambda_r^2\) rather than \(\lambda_r\), and replacing the prediction-weighted block loss in the van Trees lower bounds by the unweighted block loss. The adaptive procedure also carries over: the eigenvalue screening step is unchanged, while the second thresholding step is based on the empirical \(L^2\) block energy
\[
    \sum_{r\in B_\ell} \frac{|\hat{\gamma}_r|^2}{\check{\lambda}_r^2},
\]
with the corresponding variance proxy adjusted accordingly.

Methodologically, the adaptive estimator is not specific to this FLR model. The main difficulty is that the loss depends jointly on the smoothness of the slope function and on the unknown covariance spectrum. The estimator must therefore identify frequency blocks containing estimable signal while also determining whether the covariance scale on those blocks can be stably inverted. This motivates the two-stage thresholding rule: an eigenvalue threshold first removes unstable blocks, and an empirical energy threshold then selects blocks whose signal exceeds the relevant noise level. This principle may be useful in other inverse or semiparametric problems where the effective loss is governed by an unknown nuisance operator, metric, or covariance structure.

The lower-bound arguments also provide tools beyond the present setting. Under the independent design, the discretization cost is characterized through Fisher information bounds for noisy point evaluations, showing that the loss from observing only finitely many measurements per trajectory is intrinsic rather than estimator-specific. Under the common design, the lower bounds exploit unidentifiability created by the interaction between the covariance function, the slope function, and the shared sampling grid. These mechanisms may be relevant for other functional problems in which discrete sampling and inverse structure interact.

Several limitations and open questions remain. The analysis considers two stylized acquisition schemes: independent random grids and a regular common grid. Extending sharp minimax theory to heterogeneous protocols, such as irregular, nonuniform, or partially deterministic grids, would broaden the scope of the results. Another direction is to relax the fixed trigonometric basis formulation. If the covariance eigenbasis must be estimated from discretely observed data, or if the smoothness scale of \(\beta\) is not aligned with the covariance geometry, additional costs may appear. A complete theory would need to combine the present analysis of the dependence on \(m\) with covariance eigenbasis recovery~\citep{cai2010_NonparametricCovarianceFunction} and the alignment effects between covariance geometry and slope regularity studied in fully observed FLR. Finally, it would be interesting to understand how these discretization phenomena affect downstream goals such as inference, testing, distributed learning, and privacy-constrained functional prediction.

\clearpage
\tableofcontents
\clearpage
\appendix

  \clearpage\section{Additional Notations and Preliminaries}
\label{sec:additional-notations-preliminaries}

This appendix records additional notation and elementary coordinate identities used throughout the appendices.

\subsection{Additional Notation}
\label{subsec:additional-notation}

The letters \(C\) and \(c\) denote positive generic constants whose values may change from line to line.
Unless explicitly stated otherwise, they may depend only on the fixed constants in the model and in the statement being proved.
They do not depend on \(n,m\), frequency indices, block indices, or the particular elements over which a stated uniform bound is taken.
Decorated constants, such as \(C_M\), may also depend on the displayed parameter.
Fixed tuning constants introduced in estimator or event definitions are not generic once chosen.

For a finite index set \(A\) and \(a=(a_r)_{r\in A}\in\R^A\), \(\norm{a}_p\) denotes the usual \(\ell^p(A)\) norm.
In particular,
\[
  \norm{a}_1=\sum_{r\in A} \absx{a_r},
  \qquad
  \norm{a}_2^2=\sum_{r\in A} \absx{a_r}^2,
  \qquad
  \norm{a}_\infty=\max_{r\in A} \absx{a_r}.
\]
For a measurable function \(h\) on \([0,1]\), \(\norm{h}_{L^q([0,1])}\) denotes the usual \(L^q\) norm, with the domain omitted when it is clear.

For a random variable \(W\) and \(\rho>0\), define the Orlicz norm
\[
  \norm{W}_{\psi_\rho}
  \coloneqq
  \inf\left\{
    K>0:
    \E\exp\left[
      \left(\frac{\absx{W}}{K}\right)^\rho
    \right]\le 2
  \right\}.
\]
Finite \(\psi_\rho\) norm is referred to as sub-Weibullity of order \(\rho\).
The cases \(\rho=2\) and \(\rho=1\) are the usual sub-Gaussian and
sub-exponential norms, respectively.
We use the standard moment characterization
\[
  \norm{W}_{\psi_\rho}
  \asymp
  \sup_{q\ge1} q^{-1/\rho} \norm{W}_{L^q},
  \qquad
  \rho>0,
\]
with comparison constants depending only on \(\rho\).
If \(W\in\R^p\), then \(\norm{W}_{\psi_\rho}\) denotes the directional norm
\[
  \norm{W}_{\psi_\rho}
  \coloneqq
  \sup_{u\in\bbS^{p-1}}\norm{\angx{u,W}}_{\psi_\rho},
  \qquad
  \rho>0,
\]
whenever a vector Orlicz norm is used.
For a sub-\(\sigma\)-field \(\mathcal{F}\) and a positive \(\mathcal{F}\)-measurable random variable \(A\), the notation
\[
  \norm{W\mid\mathcal{F}}_{\psi_\rho} \le A
\]
means that
\[
  \E\left[
    \exp\left\{
      \left(\frac{\absx{W}}{A}\right)^\rho
    \right\}
    \middle|
    \mathcal{F}
  \right]\le 2
  \qquad
  \text{almost surely}.
\]

\subsection{Fourier Coordinate Calculations}
\label{subsec:fourier-coordinate-calculations}

This subsection records the elementary identities regarding the Fourier basis used throughout the main paper and the appendices.

\paragraph*{Basis conventions.}
We use the standard trigonometric basis of \(L^2([0,1])\):
\[
  e_0(t)\coloneqq 1,\qquad
  e_r(t)\coloneqq \sqrt{2}\cos(2\pi r t),\qquad
  e_{-r}(t)\coloneqq \sqrt{2}\sin(2\pi r t),
  \quad r\ge1.
\]
If \(f=\sum_{r\in\bbZ} a_r e_r\) and \(h=\sum_{r\in\bbZ} b_r e_r\), then the \(L^2\) pairing satisfies
\[
  \int_0^1 f(t)h(t)\dd t
  =
  \sum_{r\in\bbZ} a_r b_r.
\]
Indeed, this follows from \(\int_0^1 e_r(t)e_s(t)\dd t=\mathbf{1}\{r=s\}\).
Moreover, let $H^s$ denote the periodic Sobolev space of order \(s\) on the unit interval.
Then, it is well known that the Sobolev norm of \(f\) is equivalent to the weighted \(\ell^2\) norm of its Fourier coefficients:
\[
  \norm{f}_{H^s}^2
  \asymp
  \sum_{r\in\bbZ}(1+\absx{r})^{2s} a_r^2.
\]

We recall that the covariance operator \(\Sigma\) is diagonalized by the trigonometric basis with eigenvalues \((\lambda_r)_{r\in\bbZ}\), namely \( \Sigma e_r = \lambda_r e_r \).
Hence, the Karhunen--Loève expansion of the predictor trajectories \(X_i\) is
\[
  X_i(t)=\sum_{r\in\bbZ} x_{ir} e_r(t),\quad x_{ir} \sim \mathcal{N}(0,\lambda_r),
\]
where the coordinates \((x_{ir})_{r\in\bbZ}\) are independent.

\paragraph*{Regression and risk identities.}
With the Fourier expansions of \(X_i\) and \(\beta\), the regression signal can be expressed as
\[
  Y_i
  =
  \int_0^1 X_i(t)\beta(t)\dd t + \varepsilon_i
  =
  \sum_{r \in\bbZ} x_{ir} \theta_r + \varepsilon_i.
\]
Since \(\varepsilon_i\) is centered and independent of \(X_i\), the cross-covariance function is defined as
\( g \coloneqq \Sigma \beta \).
With the definition of \( \Sigma \), we have
\begin{align*}
  g(t) = \int_0^1 K(t,s)\beta(s)\dd s = \E X_i(t) \int_0^1 X_i(s)\beta(s)\dd s = \E[Y_i X_i(t)].
\end{align*}
On the other hand, since \(\Sigma e_r = \lambda_r e_r\), the Fourier expansion of \(g\) is
\begin{equation}
  g = \sum_{r\in\bbZ} \gamma_r e_r = \sum_{r\in\bbZ} \lambda_r \theta_r e_r(t),\quad \gamma_r = \lambda_r \theta_r.
  \label{eq:cross-covariance-derivation}
\end{equation}

The prediction risk can also be expressed in terms of the Fourier expansion.
If \(\hat{\beta}=\sum_{r\in\bbZ} \hat{\theta}_r e_r\), then, conditionally on \(\hat{\beta}\),
\[
  \E\left[
  \left.
  \left(\int_0^1 X_\star(t)(\hat{\beta}-\beta)(t)\dd t\right)^2
  \right|\hat{\beta}
  \right]
  =
  \sum_{r\in\bbZ} \lambda_r (\hat{\theta}_r-\theta_r)^2.
\]
The expectation on the left is over the independent test trajectory \(X_\star\), with \(\hat{\beta}\) held fixed.
Averaging over the training data that produce \(\hat{\beta}\) gives the expected risk formula used in the main text.
Equivalently, if \(\hat{\gamma}_r=\lambda_r \hat{\theta}_r\), then
\begin{equation}
  \label{eq:cd-g-risk-identity}
  \E\caR(\hat{\beta};\theta,\lambda)
  =
  \sum_{r\in\bbZ} \lambda_r \E(\hat{\theta}_r-\theta_r)^2
  =
  \sum_{r\in\bbZ} \frac{\E(\hat{\gamma}_r-\gamma_r)^2}{\lambda_r}.
\end{equation}
If \(\lambda\in\caL_\alpha(c_\lambda,C_\lambda)\), then the same coefficient relation gives
\begin{equation}
  \label{eq:cd-g-sobolev-equivalence}
  \norm{g}_{H^{s+2\alpha}}^2
  \asymp
  \sum_{r\in\bbZ}(1+\absx{r})^{2(s+2\alpha)}\gamma_r^2
  \asymp
  \sum_{r\in\bbZ}(1+\absx{r})^{2s}\theta_r^2.
\end{equation}

\paragraph*{Random-design sampling identities.}
The independent-design analysis repeatedly uses the following consequences of uniform sampling.
If \(T\sim\mathrm{Unif}[0,1]\) is independent of \((X_i,Y_i,\delta)\), \(\delta\) is centered and independent of \((X_i,Y_i,T)\), and \(Z_i(T)=X_i(T)+\delta\), then
\begin{equation}
  \label{eq:random-design-score-projection}
  \E\left[
  Z_i(T)e_r(T)
  \mid X_i
  \right]
  =
  x_{ir}.
\end{equation}
Moreover, the conditional second moment is
\begin{equation}
  \label{eq:random-design-second-moment}
  \E\left[
  Z_i(T)^2
  \mid X_i
  \right]
  =
  \sum_{\ell\in\bbZ} x_{i\ell}^2+\sigma_\delta^2.
\end{equation}
Finally, the random-design cross-covariance score is unbiased:
\begin{equation}
  \label{eq:random-design-gamma}
  \E\left[
  Y_i Z_i(T)e_r(T)
  \right]
  =
  \gamma_r.
\end{equation}
   \clearpage\section{Oracle Upper Bound Under Independent Design}
\label{sec:upper-indep-oracle}

This section proves the baseline independent-design oracle upper bound.
The smoothness indices \((\alpha,s)\), the eigenvalue sequence \((\lambda_r)_{r\in\bbZ}\), and the truncation level are treated as known.
The detailed moment computations are also used by the fully adaptive construction in
\cref{sec:upper-indep-adaptive}.
We recall the oracle construction from the main text, \cref{subsec:indep-oracle-upper}, and restate the relevant quantities to fix notation for the proof.

\subsection[Cross-Covariance Estimation]{Cross-Covariance Coefficient Estimation}
\label{subsec:indep-oracle-gamma-estimation}

By the coordinate identities in \cref{eq:cross-covariance-derivation}, the cross-covariance function has Fourier coefficients \(\gamma_r=\lambda_r \theta_r\).
The cross-covariance coefficient estimator is
\begin{equation}
  \hat{\gamma}_r
  \coloneqq
  \frac{1}{nm}\sum_{i=1}^n \sum_{j=1}^m
  Y_i Z_{ij} e_r(t_{ij}),
  \qquad r\in\bbZ.
  \label{eq:gk-estimator}
\end{equation}

\begin{lemma}[Moment bounds for \(\hat{\gamma}_r\)]
\label{lem:indep-gamma-estimator-moments}
For each \(r \in \bbZ\),
\[
  \E[\hat{\gamma}_r]
  =
  \gamma_r
  =
  \lambda_r \theta_r.
\]
If \(\alpha > 1/2\), then there exists \(C\), independent of \(\sigma_\delta\), such that for all \(r \in \bbZ\),
\begin{equation}
  \E\absx{\hat{\gamma}_r-\gamma_r}^2
  \le
  \frac{C}{n}\left(\lambda_r + \frac{1+\sigma_\delta^2}{m}\right).
  \label{eq:var-gk-bound}
\end{equation}
\end{lemma}

\begin{proof}
  The unbiasedness follows from \cref{eq:random-design-gamma} applied to each point evaluation, followed by averaging over \((i,j)\).
  For the variance bound, write
  \[
    \bar{A}_{ir}
    \coloneqq
    \frac{1}{m}\sum_{j=1}^m
    Y_i Z_{ij} e_r(t_{ij}),
    \qquad
    \hat{\gamma}_r = \frac{1}{n}\sum_{i=1}^n \bar{A}_{ir}.
  \]
  By independence across subjects and the unbiasedness just proved,
  \[
    \E\absx{\hat{\gamma}_r-\gamma_r}^2
    = \frac{1}{n}\E\absx{\bar{A}_{1r}-\gamma_r}^2.
  \]
  We use the total-variance decomposition
  \[
    \Var(\bar{A})
    =
    \E[\Var(\bar{A} \mid x,\varepsilon)]
    +
    \Var(\E[\bar{A} \mid x,\varepsilon]).
  \]

  For the structural term, \cref{eq:random-design-score-projection} gives
  \[
    \E[\bar{A} \mid x,\varepsilon] = Y x_r,
  \]
  and hence
  \[
    \Var(\E[\bar{A} \mid x,\varepsilon])
    =
    \Var(Y x_r)
    \le
    \E\absx{Y}^2 \absx{x_r}^2.
  \]
  Also,
  \[
    \E[Y^2]
    =
    \sum_{\ell \in \bbZ} \lambda_\ell \absx{\theta_\ell}^2
    + \sigma_\varepsilon^2
    \le
    C(\alpha,s,R_0) + \sigma_\varepsilon^2,
  \]
  and Gaussian fourth moments imply
  \[
    \E\absx{Y}^2 \absx{x_r}^2
    \le
    (\E Y^4)^{1/2}(\E\absx{x_r}^4)^{1/2}
    \lesssim
    \lambda_r,
  \]
  so
  \[
    \Var(Y x_r)\le C\lambda_r.
  \]

  For the design-averaging term, conditioned on \((x,\varepsilon)\), the variables \(\{(t_j,\delta_j)\}_{j=1}^m\) are i.i.d. and \(Y\) is fixed.
  Therefore
  \[
    \begin{aligned}
      \Var(\bar{A}\mid x,\varepsilon)
      &=
      \frac{1}{m}\Var(Y Ze_r(t)\mid x,\varepsilon)\\
      \le
      \frac{Y^2}{m}\E[\absx{Z}^2 \mid x].
    \end{aligned}
  \]
  By \cref{eq:random-design-second-moment}, \(\E[Z^2 \mid x]=\sum_{\ell\in\bbZ} x_\ell^2+\sigma_\delta^2\).
  Using \(\sum_\ell \lambda_\ell < \infty\) and Gaussian fourth moments,
  \[
    \E[\Var(\bar{A}\mid x,\varepsilon)]
    \le
    \frac{C(1+\sigma_\delta^2)}{m}.
  \]

  Combining the structural and design-averaging bounds gives
  \[
    \Var(\bar{A}) \le C\lambda_r + \frac{C(1+\sigma_\delta^2)}{m},
  \]
  Dividing by \(n\) gives \cref{eq:var-gk-bound}.
\end{proof}

\subsection[Proof of the Oracle Independent-Design Upper Bound]{Proof of \texorpdfstring{\cref{thm:indep-oracle-upper}}{the oracle independent-design upper bound}}
\label{subsec:proof-indep-oracle-upper}

For an arbitrary cutoff \(d\ge1\), the estimator in \cref{eq:oracle-estimator} has coefficient form
\begin{equation}
  \tilde{\theta}_r
  \coloneqq
  \frac{\hat{\gamma}_r}{\lambda_r}\mathbf{1}\{\absx{r} \le d\},
  \qquad
  \tilde{\beta}
  =
  \sum_{r \in \bbZ} \tilde{\theta}_r e_r.
  \label{eq:oracle-theta-estimator}
\end{equation}
We first show the following bound, keeping the fixed noise variance explicit:
\begin{equation}
  \sup_{\lambda \in \caL_\alpha(c_\lambda,C_\lambda)}
  \sup_{\theta \in \Theta_s(R_0)}
  \E\caR(\tilde{\beta}; \theta, \lambda)
  \lesssim
  \frac{d}{n}
  +
  \frac{(1+\sigma_\delta^2)d^{2\alpha+1}}{nm}
  +
  d^{-(2\alpha+2s)}.
  \label{eq:oracle-risk-upper}
\end{equation}
Decompose the risk as
\[
  \E\caR(\tilde{\beta}; \theta, \lambda)
  =
  \sum_{\absx{r} \le d}\lambda_r \E\absx{\tilde{\theta}_r-\theta_r}^2
  +
  \sum_{\absx{r} > d}\lambda_r \absx{\theta_r}^2.
\]
The tail term is bounded by \cref{eq:upper-two-sided-tail-bias}.
For the estimation part, the unbiasedness in \cref{lem:indep-gamma-estimator-moments} gives
\[
  \sum_{\absx{r} \le d}\lambda_r \E\absx{\tilde{\theta}_r-\theta_r}^2
  =
  \sum_{\absx{r} \le d}
  \frac{\E\absx{\hat{\gamma}_r-\gamma_r}^2}{\lambda_r}.
\]
Using \cref{lem:indep-gamma-estimator-moments} and \(\lambda_r \ge c_\lambda(1+\absx{r})^{-2\alpha}\),
\[
  \frac{\E\absx{\hat{\gamma}_r-\gamma_r}^2}{\lambda_r}
  \le
  \frac{C}{n}\left(1 + \frac{1+\sigma_\delta^2}{m\lambda_r}\right)
  \le
  \frac{C}{n}\left(
  1 + \frac{(1+\sigma_\delta^2)(1+\absx{r})^{2\alpha}}{m}
  \right).
\]
Summing over \(\absx{r} \le d\),
\[
  \begin{aligned}
    \sum_{\absx{r} \le d}
    \frac{\E\absx{\hat{\gamma}_r-\gamma_r}^2}{\lambda_r}
    &\le
    \frac{C}{n}\left(
    d + \frac{1+\sigma_\delta^2}{m}\sum_{k \le d} k^{2\alpha}
    \right)\\
    &\le
    C\left(
    \frac{d}{n} + \frac{(1+\sigma_\delta^2)d^{2\alpha+1}}{nm}
    \right).
  \end{aligned}
\]
Combining the estimation bound with the tail bound gives \cref{eq:oracle-risk-upper}.

Let \(\tilde{\beta}^*\) be \(\tilde{\beta}\) constructed with \(d=d^*\), where
\[
  d^*
  \asymp
  n^{\frac{1}{2\alpha+2s+1}}
  \wedge
  \left(\frac{nm}{1+\sigma_\delta^2}\right)^{\frac{1}{4\alpha+2s+1}}.
\]
Substituting this choice into \cref{eq:oracle-risk-upper} gives
\[
  \sup_{\lambda \in \caL_\alpha(c_\lambda,C_\lambda)}
  \sup_{\theta \in \Theta_s(R_0)}
  \E\caR(\tilde{\beta}^*; \theta, \lambda)
  \lesssim
  n^{-\frac{2\alpha+2s}{2\alpha+2s+1}}
  +
  (1+\sigma_\delta^2)^{\frac{2\alpha+2s}{4\alpha+2s+1}}
  (nm)^{-\frac{2\alpha+2s}{4\alpha+2s+1}},
\]
yielding the desired result.
   \clearpage\section{Adaptive Upper Bound Under Independent Design}
\label{sec:upper-indep-adaptive}

This section proves the fully adaptive independent-design upper bound, \cref{thm:indep-adaptive-upper}.
The estimator does not know \((\lambda_r)_{r\in\bbZ}\), \(\alpha\), or \(s\).
The present proof gives the upper bound uniformly over compact subsets of
\(\{\alpha>1/2,\ s\ge0\}\).
Throughout the section, fix constants
\[
  \underline{\alpha}\le\bar{\alpha},
  \qquad
  \underline{s}\le\bar{s},
  \qquad
  \underline{\alpha}>1/2,
  \qquad
  \underline{s}\ge0,
\]
and work uniformly over
\[
  (\alpha,s)
  \in
  [\underline{\alpha},\bar{\alpha}]
  \times
  [\underline{s},\bar{s}].
\]
All implicit constants in this section may depend on these four endpoints.
The proof separates the oracle block-selection analysis with known \(\lambda\) from the pilot layer that controls the plug-in eigenvalue estimates \(\check{\lambda}_r\).
Auxiliary results shared across designs are collected in \cref{sec:auxiliary}; independent-design estimates are stated in this section.

\noindent\textbf{Proof roadmap.}
The proof is organized around the good pilot event \(\caE_n\).
On this event, blocks satisfying \(\caEelig{\ell}\) have stable denominators, and all blocks up to the oracle cutoff satisfy \(\caEelig{\ell}\).
The local block argument then compares the adaptive selector with the oracle block selector and leaves only an eigenvalue plug-in remainder.
The final aggregation sums the oracle variance up to the oracle cutoff, the Sobolev bias beyond it, the high-frequency tail beyond \(K_n\), the plug-in remainder, the exponential thresholding remainder, and the contribution of \((\caE_n)^c\).

\subsection{Setup}
\label{subsec:adaptive-setup}

The construction uses only the boundary constraints \(\alpha>1/2\)
and \(s\ge0\); the fixed compact rectangle above is used only for uniform
constants.
We recall the adaptive estimator from the main text, \cref{subsec:indep-adaptive-upper}, and restate it to fix notation for the proof.

Split the subjects into two disjoint subsets
\[
  \{1,\dots,n\}
  =
  I_\gamma \cup I_\lambda,
  \qquad
  n_\gamma \coloneqq \absx{I_\gamma},
  \qquad
  n_\lambda \coloneqq \absx{I_\lambda},
  \qquad
  n_\gamma \asymp n_\lambda \asymp n.
\]
The regression group \(I_\gamma\) estimates the cross-covariance coefficients and is reused for both block selection and the retained coefficients, while \(I_\lambda\) is reserved for the pilot covariance layer.
The analysis treats the oracle block selector with known \(\lambda\) and then controls the plug-in error from \(\check{\lambda}_r\).

Estimate frequencies only up to the dyadic cutoff
\begin{equation}
  L_n
  \coloneqq
  \left\lfloor
    \log_2 \left(
             C_L \min\left\{
                       n^{1/2},
      (nm)^{1/3}
    \right\}
  \right)
  \right\rfloor,
  \qquad
  K_n \coloneqq 2^{L_n}.
  \label{eq:indep-pilot-bandwidth}
\end{equation}
Here \(C_L>0\) is a fixed numerical constant.
Then
\[
  K_n
  \asymp
  \min\left\{n^{1/2},(nm)^{1/3}\right\}.
\]
Write
\[
  \bar{Z}_{ik}
  \coloneqq
  \frac{1}{m}\sum_{j=1}^m Z_{ij} e_k(t_{ij}).
\]
For \(\absx{k} \le K_n\), the cross-covariance estimator is
\begin{equation}
  \hat{\gamma}_k
  \coloneqq
  \frac{1}{n_\gamma}\sum_{i \in I_\gamma} Y_i \bar{Z}_{ik}.
  \label{eq:indep-adaptive-gk}
\end{equation}

For each \(i \in I_\lambda\), the observation indices are split into two disjoint
sets
\[
  J_i^{(1)} \cup J_i^{(2)} = \{1,\dots,m\},
  \qquad
  m_i^{(a)} \coloneqq \absx{J_i^{(a)}} \asymp m.
\]
The corresponding pilot half-scores are
\begin{equation}
  U_{ir}^{(a)}
  \coloneqq
  \frac{1}{m_i^{(a)}}\sum_{j \in J_i^{(a)}}
  Z_{ij} e_r(t_{ij}),
  \qquad
  r\in\bbZ,\quad a=1,2,
  \label{eq:indep-pilot-half-scores}
\end{equation}
The pilot covariance estimator is
\begin{equation}
  \check{\lambda}_r
  \coloneqq
  \frac{1}{n_\lambda}\sum_{i \in I_\lambda}
  U_{ir}^{(1)} U_{ir}^{(2)}.
  \label{eq:indep-pilot-estimator}
\end{equation}
The pilot floor is
\begin{equation}
  \zeta_{n,m}
  \coloneqq
  M_0 m^{-1} \sqrt{\frac{\log(nm)}{n_\lambda}},
  \label{eq:indep-pilot-floor}
\end{equation}
where \(M_0 > 0\) is a large constant.
Use the dyadic blocks
\[
  B_0 = \{0,\pm 1\},
  \qquad
  B_\ell = \{r:2^{\ell-1}<\absx{r}\le 2^\ell\},
  \qquad
  \ell \ge 1.
\]
The active blocks are
\[
  B_0,\dots,B_{L_n}.
\]
Frequencies \(\absx{k} > K_n\) are set to zero and handled through the separate tail \(\sum_{\absx{k} > K_n}\lambda_k \absx{\theta_k}^2\).

For each active block \(B_\ell\), the empirical prediction energy and plug-in variance proxy are
\begin{equation}
  \hat{S}_\ell
  \coloneqq
  \sum_{k \in B_\ell}
  \frac{\absx{\hat{\gamma}_k}^2}{\check{\lambda}_k},
  \qquad
  \hat{V}_\ell
  \coloneqq
  \sum_{k \in B_\ell}
  \frac{C_V}{n}\left(1+\frac{1}{m\check{\lambda}_k}\right),
  \label{eq:indep-adaptive-block-statistics}
\end{equation}
Define the block events
\begin{equation}
  \caEelig{\ell}
  \coloneqq
  \left\{
    \min_{k \in B_\ell} \check{\lambda}_k \ge 2\zeta_{n,m}
  \right\},
  \qquad
  \caEthr{\ell}
  \coloneqq
  \left\{\hat{S}_\ell \ge \hat{V}_\ell \right\},
  \qquad
  \caEkeep{\ell}
  \coloneqq
  \caEelig{\ell} \cap \caEthr{\ell}.
  \label{eq:indep-adaptive-events}
\end{equation}
Here \(C_V>0\) is a fixed, sufficiently large constant used in the variance bound and the selection rule.
The selected coefficients are
\begin{equation}
  \hat{\theta}_k
  \coloneqq
  \frac{\hat{\gamma}_k}{\check{\lambda}_k}\mathbf{1}_{\caEkeep{\ell}},
  \qquad
  k \in B_\ell,\quad 0\le\ell\le L_n.
  \label{eq:indep-adaptive-theta}
\end{equation}
Fix a constant \(R\ge R_0\) and let \(\Pi_R\) denote the Euclidean projection onto the closed Euclidean ball of radius \(R\):
\begin{equation}
  \Pi_R(u)
  \coloneqq
  \begin{cases}
    u,
    & \norm{u}_2 \le R,\\[1mm]
    R\dfrac{u}{\norm{u}_2},
    & \norm{u}_2 > R.
  \end{cases}
  \label{eq:indep-adaptive-projection-map}
\end{equation}
For each active block, write
\[
  \hat{\theta}_{B_\ell}
  \coloneqq
  (\hat{\theta}_k)_{k \in B_\ell}.
\]
The blockwise clipped vector is
\begin{equation}
  \bar{\theta}_{B_\ell}
  \coloneqq
  \Pi_R(\hat{\theta}_{B_\ell}),
  \qquad
  0 \le \ell \le L_n,
  \label{eq:indep-adaptive-block-clipping}
\end{equation}
and let \(\bar{\theta}_k\) denote the corresponding coordinates.
The final estimator is
\begin{equation}
  \hat{\beta}(t)
  \coloneqq
  \sum_{\absx{k}\le K_n}\bar{\theta}_k e_k(t).
  \label{eq:indep-adaptive-beta}
\end{equation}
Finally, for each active block, set the oracle signal energy
\begin{equation}
  \Theta_\ell
  \coloneqq
  \sum_{k \in B_\ell} \lambda_k \absx{\theta_k}^2.
  \label{eq:indep-block-signal-energy}
\end{equation}

\subsection{Eigenvalue estimation}
\label{subsec:pilot-estimation-eigenvalues}

In this section we focus on the pilot covariance estimator \(\check{\lambda}_k\) from \cref{eq:indep-pilot-estimator},
where the pilot half-scores \(U_{ik}^{(a)}\) are defined in \cref{eq:indep-pilot-half-scores}.

\begin{lemma}
  \label{lem:indep-pilot-moments}
  For every \(k \in \bbZ\),
  \begin{equation}
    \E(\check{\lambda}_k - \lambda_k)^2
    \le
    \frac{C}{n_\lambda}\left(
                         \lambda_k^2 + \frac{\lambda_k}{m} + \frac{1}{m^2}
    \right).
    \label{eq:indep-pilot-mse}
  \end{equation}
\end{lemma}

\begin{proof}
  Write
  \begin{equation}
    U_{ik}^{(1)} = x_{ik} + \eta_{ik}^{(1)},
    \qquad
    U_{ik}^{(2)} = x_{ik} + \eta_{ik}^{(2)}.
    \label{eq:indep-pilot-decomposition}
  \end{equation}
  Because \(J_i^{(1)}\) and \(J_i^{(2)}\) are disjoint and the design/noise
  pairs are independent across \(j\),
  \[
    \E[U_{ik}^{(1)} \mid X_i] = \E[U_{ik}^{(2)} \mid X_i] = x_{ik},
    \qquad
    \E[\eta_{ik}^{(a)} \mid X_i] = 0.
  \]
  By \cref{assum:indep-design}, conditional on \(X_i\), the pairs
  \(\{(t_{ij},\delta_{ij}) : j \in J_i^{(a)}\}\) are independent and
  \(t_{ij} \sim \mathrm{Unif}[0,1]\). Hence
  \begin{align*}
    \E\bigl[\absx{\eta_{ik}^{(a)}}^2 \mid X_i \bigr]
    &=
    \frac{1}{(m_i^{(a)})^2}
    \sum_{j \in J_i^{(a)}}
    \Var\bigl(Z_{ij} e_k(t_{ij}) \mid X_i \bigr)
    \\
    &\le
    \frac{C}{m_i^{(a)}}\E\bigl[\absx{Z_{ij}}^2 \mid X_i \bigr]
    \\
    &\le
    \frac{C}{m}\E\bigl[\absx{Z_{ij}}^2 \mid X_i \bigr].
  \end{align*}
  Since \(Z_{ij}=X_i(t_{ij})+\delta_{ij}\), boundedness of the basis gives
  \[
    \E\bigl[\absx{Z_{ij}}^2 \mid X_i \bigr]
    \le
    2\int_0^1 \absx{X_i(t)}^2 \dd t + 2\sigma_\delta^2
    =
    2\sum_{\ell \in \bbZ} \absx{x_{i\ell}}^2 + 2\sigma_\delta^2.
  \]
  Therefore
  \begin{equation}
    \E\bigl[\absx{\eta_{ik}^{(a)}}^2 \mid X_i \bigr]
    \le
    \frac{C}{m}
    \left(
      \sum_{\ell \in \bbZ} \absx{x_{i\ell}}^2 + \sigma_\delta^2
    \right),
    \label{eq:indep-pilot-eta-second-moment}
  \end{equation}
  and in particular \(\E\absx{\eta_{ik}^{(a)}}^2 \le C/m\).
  Therefore
  \[
    \E\bigl[U_{ik}^{(1)} U_{ik}^{(2)} \mid X_i \bigr] = x_{ik}^2,
  \]
  hence
  \[
    \E\bigl[U_{ik}^{(1)} U_{ik}^{(2)} \bigr]
    = \E x_{ik}^2 = \lambda_k,
  \]
  so the summands in \(\check{\lambda}_k-\lambda_k\) are centered.

  Expanding,
  \begin{equation}
    U_{ik}^{(1)} U_{ik}^{(2)} - \lambda_k
    =
    (x_{ik}^2 - \lambda_k)
    +
    x_{ik} \eta_{ik}^{(1)}
    +
    x_{ik} \eta_{ik}^{(2)}
    +
    \eta_{ik}^{(1)} \eta_{ik}^{(2)}.
    \label{eq:indep-pilot-expansion}
  \end{equation}
  We bound the four second moments separately. Since \(x_{ik}\) is
  Gaussian with \(\E x_{ik}^2=\lambda_k\),
  \begin{equation}
    \E(x_{ik}^2-\lambda_k)^2 \le C\lambda_k^2.
    \label{eq:indep-pilot-gaussian-square}
  \end{equation}
  Next, using \cref{eq:indep-pilot-eta-second-moment},
  \begin{equation}
    \begin{aligned}
      \E\bigl[\absx{x_{ik}}^2 \absx{\eta_{ik}^{(1)}}^2 \bigr]
      &=
      \E\left[
          \absx{x_{ik}}^2
          \E\bigl[\absx{\eta_{ik}^{(1)}}^2 \mid X_i \bigr]
      \right]\\
      &\le
      \frac{C}{m}
      \E\left[
          \absx{x_{ik}}^2
          \left(
            \sum_{\ell \in \bbZ} \absx{x_{i\ell}}^2 + \sigma_\delta^2
          \right)
      \right].
    \end{aligned}
    \label{eq:indep-pilot-mixed-start}
  \end{equation}
  Because the Fourier coordinates are independent,
  \[
    \E\left[\absx{x_{ik}}^2
        \sum_{\ell \in \bbZ} \absx{x_{i\ell}}^2 \right]
    =
    \sum_{\ell \ne k} \E\absx{x_{ik}}^2 \E\absx{x_{i\ell}}^2
    + \E\absx{x_{ik}}^4
    =
    \lambda_k \sum_{\ell \ne k} \lambda_\ell + C\lambda_k^2
    \le
    C\lambda_k,
  \]
  since \(\sum_{\ell \in \bbZ} \lambda_\ell < \infty\). Thus
  \begin{equation}
    \E\bigl[\absx{x_{ik}}^2 \absx{\eta_{ik}^{(1)}}^2 \bigr]
    \le
    \frac{C\lambda_k}{m}.
    \label{eq:indep-pilot-mixed-bound-1}
  \end{equation}
  The same argument gives
  \begin{equation}
    \E\bigl[\absx{x_{ik}}^2 \absx{\eta_{ik}^{(2)}}^2 \bigr]
    \le
    \frac{C\lambda_k}{m}.
    \label{eq:indep-pilot-mixed-bound-2}
  \end{equation}
  Finally, conditional on \(X_i\), the two halves are independent, so
  \begin{equation}
    \begin{aligned}
      \E\bigl[\absx{\eta_{ik}^{(1)} \eta_{ik}^{(2)}}^2 \bigr]
      &=
      \E\left[
          \E\bigl[\absx{\eta_{ik}^{(1)}}^2 \mid X_i \bigr]
          \E\bigl[\absx{\eta_{ik}^{(2)}}^2 \mid X_i \bigr]
      \right]\\
      &\le
      \frac{C}{m^2}
      \E\left[
          \left(
            \sum_{\ell \in \bbZ} \absx{x_{i\ell}}^2 + \sigma_\delta^2
          \right)^2
      \right]
      \le
      \frac{C}{m^2},
    \end{aligned}
    \label{eq:indep-pilot-product-bound}
  \end{equation}
  where the last step again uses \(\sum_{\ell \in \bbZ} \lambda_\ell < \infty\)
  and Gaussian fourth moments.
  Therefore, combining the four terms,
  \[
    \E\absx{U_{ik}^{(1)} U_{ik}^{(2)}-\lambda_k}^2
    \le
    C\left(
       \lambda_k^2 + \frac{\lambda_k}{m} + \frac{1}{m^2}
    \right),
  \]
  Averaging over \(i \in I_\lambda\) yields
  \cref{eq:indep-pilot-mse}.
\end{proof}

\begin{lemma}
  \label{lem:indep-pilot-halfscore-tail}
  Let \(U_{ik}^{(1)},U_{ik}^{(2)}\) be defined by
  \cref{eq:indep-pilot-half-scores}.
  Then
  \[
    \norm{U_{ik}^{(1)}}_{\psi_1} + \norm{U_{ik}^{(2)}}_{\psi_1}
    \le
    C\left(\lambda_k^{1/2} + m^{-1/2} \right).
  \]
\end{lemma}

\begin{proof}
  It suffices to control \(\norm{U_{ik}^{(1)}}_{\psi_1}\) since the same argument applies to \(U_{ik}^{(2)}\).
  Write \(U_{ik}^{(1)}=x_{ik}+\eta_{ik}^{(1)}\) and decompose
  \[
    m_i^{(1)} \coloneqq \absx{J_i^{(1)}},
    \qquad
    \eta_{ik}^{(1)}
    =
    A_{ik}^{(1)} + B_{ik}^{(1)},
  \]
  where
  \begin{align*}
    A_{ik}^{(1)}
    & \coloneqq
    \frac{1}{m_i^{(1)}}\sum_{j \in J_i^{(1)}}
    \left(
      X_i(t_{ij})e_k(t_{ij}) - x_{ik}
    \right),\\
    B_{ik}^{(1)}
    &\coloneqq
    \frac{1}{m_i^{(1)}}\sum_{j \in J_i^{(1)}}
    \delta_{ij} e_k(t_{ij}).
  \end{align*}

  Conditional on \(X_i\), the summands in \(A_{ik}^{(1)}\) are independent,
  centered, and satisfy
  \[
    \absx{X_i(t_{ij})e_k(t_{ij}) - x_{ik}}
    \le
    \absx{X_i(t_{ij})}
    +
    \absx{x_{ik}}
    \le
    2\norm{X_i}_\infty.
  \]
  Hoeffding's inequality therefore yields
  \[
    \norm{A_{ik}^{(1)} \mid X_i}_{\psi_2}
    \le
    C\frac{\norm{X_i}_\infty}{\sqrt{m_i^{(1)}}}
    \le
    C\frac{\norm{X_i}_\infty}{\sqrt{m}}.
  \]
  Hence, for every \(q \ge 1\),
  \[
    \left(
      \E\left[\absx{A_{ik}^{(1)}}^q \mid X_i \right]
    \right)^{1/q}
    \le
    C\sqrt{q} \frac{\norm{X_i}_\infty}{\sqrt{m}}.
  \]
  Taking expectations and using \cref{lem:indep-process-supnorm},
  \[
    \norm{A_{ik}^{(1)}}_{L^q}
    \le
    C\frac{q}{\sqrt{m}},
    \qquad
    q \ge 1.
  \]
  By the moment characterization of the \(\psi_1\) norm,
  \[
    \norm{A_{ik}^{(1)}}_{\psi_1} \le C m^{-1/2}.
  \]

  Conditional on the design points, \(B_{ik}^{(1)}\) is centered Gaussian with
  variance
  \[
    \Var(B_{ik}^{(1)} \mid t_{i\cdot})
    =
    \frac{\sigma_\delta^2}{(m_i^{(1)})^2}
    \sum_{j \in J_i^{(1)}}\absx{e_k(t_{ij})}^2
    \le
    \frac{C}{m_i^{(1)}}
    \le
    \frac{C}{m},
  \]
  so \(\norm{B_{ik}^{(1)}}_{\psi_2} \le C m^{-1/2}\), hence also
  \(\norm{B_{ik}^{(1)}}_{\psi_1} \le C m^{-1/2}\). Therefore
  \[
    \norm{\eta_{ik}^{(1)}}_{\psi_1} \le C m^{-1/2}.
  \]

  Finally, \(x_{ik} \sim \mathcal{N}(0,\lambda_k)\), so
  \(\norm{x_{ik}}_{\psi_1} \le C\lambda_k^{1/2}\). The triangle inequality
  gives the bound for \(U_{ik}^{(1)}\); the same argument applies to \(U_{ik}^{(2)}\).
\end{proof}

\begin{lemma}[Pilot cross-product concentration]
  \label{lem:indep-pilot-cross-concentration}
  For every \(M > 0\), there exist finite constants \(C_M\) and \(N_M\)
  such that the following holds.
  If \(\absx{k}\le K_n\) and \(n_\lambda \ge N_M\), then
  \[
    \Pr\left\{
         \absx{\check{\lambda}_k - \lambda_k}
         >
         C_M \left(\lambda_k + \frac{1}{m}\right)
         \sqrt{\frac{\log n_\lambda}{n_\lambda}}
    \right\}
    \le
    C_M n_\lambda^{-M}.
  \]
\end{lemma}

\begin{proof}
  Define the centered cross-product
  \[
    X_{ik}^\lambda
    \coloneqq
    U_{ik}^{(1)} U_{ik}^{(2)} - \lambda_k.
  \]
  Combining \cref{lem:indep-pilot-halfscore-tail} and \cref{prop:indep-product-tail} with \(\alpha_A = \alpha_B = 1\) gives
  \[
    \E \exp\zk{\xk{\frac{X_{ik}^\lambda}{C(\lambda_k + 1/m)}}^{1/2}} \le C.
  \]
  Since
  \[
    \check{\lambda}_k - \lambda_k
    =
    \frac{1}{n_\lambda}\sum_{i \in I_\lambda} X_{ik}^\lambda,
  \]
  the result follows from applying
  \cref{lem:upper-subweibull-bernstein} with \(\alpha = 1/2\),
  \(K \asymp \lambda_k + m^{-1}\),
  \(L \asymp (\log n_\lambda)^2\), and
  \[
    u
    \asymp \left(\lambda_k+\frac{1}{m}\right)
    \sqrt{\frac{\log n_\lambda}{n_\lambda}},
  \]
  where the constants depend on \(M\) through the choice of \(C_M\) and \(N_M\).
\end{proof}

With the pilot floor \cref{eq:indep-pilot-floor} and bandwidth
\cref{eq:indep-pilot-bandwidth}, define the good pilot event
\begin{equation}
  \caE_n
  \coloneqq
  \bigcap_{\absx{k}\le K_n}
  \left\{
    \absx{\check{\lambda}_k - \lambda_k}
    \le
    \frac{1}{2}\max\{\lambda_k,\zeta_{n,m}\}
  \right\}.
  \label{eq:indep-good-event}
\end{equation}

\begin{lemma}[Uniform pilot control]
  \label{lem:indep-pilot-uniform-control}
  For \(M_0\) sufficiently large, there exist constants \(C,N_0 < \infty\) such
  that, for all \(n,m \ge N_0\),
  \begin{equation}
    \Pr((\caE_n)^c)\le Cn^{-20}.
    \label{eq:indep-good-event-prob}
  \end{equation}
\end{lemma}

\begin{proof}
  Put
  \[
    a_n \coloneqq \sqrt{\frac{\log n_\lambda}{n_\lambda}},
    \qquad
    b_{n,m} \coloneqq \sqrt{\frac{\log(nm)}{n_\lambda}}.
  \]
  For \(\absx{k}\le K_n\), write
  \[
    t_k \coloneqq \frac{1}{2}\max\{\lambda_k,\zeta_{n,m}\},
    \qquad
    K_k^\lambda \coloneqq \lambda_k+\frac{1}{m}.
  \]
  Since \(m\ge1\), we have \(a_n \le b_{n,m}\), and
  \(\zeta_{n,m}=M_0 m^{-1} b_{n,m}\). By
  \cref{lem:indep-pilot-cross-concentration} with \(M=24\), for all
  sufficiently large \(n\),
  \[
    \Pr\left\{
         \absx{\check{\lambda}_k-\lambda_k}
         >
         C K_k^\lambda a_n
    \right\}
    \le
    Cn_\lambda^{-24}
  \]
  uniformly over \(\absx{k}\le K_n\).
  We now show that, after choosing \(M_0\)
  large and then \(N_0\) large,
  \[
    C K_k^\lambda a_n \le t_k
    \qquad
    \text{for every }\absx{k}\le K_n.
  \]

  If \(\lambda_k < \zeta_{n,m}\), then
  \[
    C K_k^\lambda a_n
    \le
    C m^{-1} \left(1+M_0 b_{n,m} \right)a_n
    \le
    \frac{1}{2}M_0 m^{-1} b_{n,m}
    =
    t_k,
  \]
  where the last inequality uses \(a_n \le b_{n,m}\), \(a_n \to0\), and \(M_0\)
  large. If \(\zeta_{n,m} \le\lambda_k \le m^{-1}\), then
  \[
    C K_k^\lambda a_n
    \le
    C m^{-1} a_n
    \le
    \frac{1}{2}\lambda_k
    =
    t_k,
  \]
  after increasing \(M_0\), because
  \(\lambda_k \ge M_0 m^{-1} b_{n,m}\). Finally, if \(\lambda_k>m^{-1}\), then
  \[
    C K_k^\lambda a_n
    \le
    C\lambda_k a_n
    \le
    \frac{1}{2}\lambda_k
    =
    t_k
  \]
  for all \(n\ge N_0\).

  Hence, for all sufficiently large \(n,m\),
  \[
    \Pr\left\{
         \absx{\check{\lambda}_k-\lambda_k}
         >
         t_k
    \right\}
    \le
    Cn_\lambda^{-24}
  \]
  uniformly over \(\absx{k}\le K_n\). Since \(K_n \le C_L n^{1/2}\) and
  \(n_\lambda \asymp n\), the union bound gives
  \[
    \Pr((\caE_n)^c)
    \le
    C K_n n_\lambda^{-24}
    \le
    Cn^{-20}.
  \]
\end{proof}

\subsection{Eligible Blocks}
\label{subsec:indep-eligible-blocks}

We say that a block \(B_\ell\) is eligible if \(\caEelig{\ell}\) holds, i.e. if the pilot eigenvalue estimates on the block exceed the pilot floor.
In this subsection, we show that the eigenvalues can be well estimated over the whole interested range until the oracle band.
Recall the eligibility event \(\caEelig{\ell}\) from  \cref{eq:indep-adaptive-events}.

\begin{lemma}
  \label{lem:indep-eligible-blocks}
  On the event \(\caE_n\) from \cref{eq:indep-good-event}, the following two
  implications hold for every active block \(B_\ell\):
  \begin{enumerate}[label=(\roman*)]
    \item if \(\min_{k\in B_\ell} \lambda_k \ge 4\zeta_{n,m}\), then
    \(\caEelig{\ell}\) occurs;
    \item if \(\caEelig{\ell}\) occurs, then, for every \(k\in B_\ell\),
    \[
      \frac{1}{2}\lambda_k
      \le
      \check{\lambda}_k
      \le
      \frac{3}{2}\lambda_k.
    \]
  \end{enumerate}
\end{lemma}

\begin{proof}
  Assume \(\caE_n\).
  If \(\min_{k\in B_\ell} \lambda_k \ge4\zeta_{n,m}\), then, for every
  \(k\in B_\ell\), the definition of \(\caE_n\) gives
  \[
    \check{\lambda}_k
    \ge
    \lambda_k-\frac{1}{2}\lambda_k
    =
    \frac{1}{2}\lambda_k
    \ge
    2\zeta_{n,m},
  \]
  so \(\caEelig{\ell}\) occurs.

  Now assume that \(\caEelig{\ell}\) occurs.
  If some \(k \in B_\ell\) satisfied
  \(\lambda_k < \zeta_{n,m}\), then by the definition of \(\caE_n\),
  \[
    \check{\lambda}_k
    \le
    \lambda_k + \frac{1}{2}\zeta_{n,m}
    <
    \frac{3}{2}\zeta_{n,m},
  \]
  contradicting \(\caEelig{\ell}\).
  Hence every \(k \in B_\ell\) satisfies \(\lambda_k \ge \zeta_{n,m}\).
  On \(\caE_n\), this implies, for every \(k\in B_\ell\),
  \[
    \absx{\check{\lambda}_k - \lambda_k}
    \le
    \frac{1}{2}\lambda_k,
  \]
  which proves the second implication.
\end{proof}

\begin{lemma}[Oracle band lies above the pilot floor]
  \label{lem:indep-oracle-band-floor}
  Let
  \[
    d_*(\alpha,s) \asymp n^{1/(2\alpha+2s+1)} \wedge (nm)^{1/(4\alpha+2s+1)}
  \]
  Then, uniformly over the fixed range of \((\alpha,s)\), for all sufficiently large \(n,m\),
  \begin{equation}
    d_*(\alpha,s)\le K_n,
    \label{eq:indep-oracle-band-in-pilot}
  \end{equation}
  Moreover, if \(L_*\) is chosen so that
  \[
    2^{L_*-1}<d_*(\alpha,s)\le 2^{L_*},
  \]
  then
  \begin{equation}
    \min_{0\le \ell\le L_*} \min_{k\in B_\ell} \lambda_k
    \ge
    4\zeta_{n,m}.
    \label{eq:indep-oracle-band-above-floor}
  \end{equation}
\end{lemma}

\begin{proof}
  Since the proof is uniform over the fixed compact rectangle with
  \(\underline{\alpha}>1/2\) and \(\underline{s}\ge0\), there exist
  \(\epsilon_1,\epsilon_2>0\) such that
  \[
    \frac{1}{2\alpha+2s+1}\le \frac{1}{2}-\epsilon_1,
    \qquad
    \frac{1}{4\alpha+2s+1}\le \frac{1}{3}-\epsilon_2
  \]
  throughout the rectangle.
  Therefore
  \[
    d_*(\alpha,s)
    \lesssim
    n^{1/2-\epsilon_1}
    \wedge
    (nm)^{1/3-\epsilon_2}
    \le
    K_n
  \]
  for all sufficiently large \(n,m\), because the ratios
  \(d_*/n^{1/2}\) and \(d_*/(nm)^{1/3}\) tend to zero uniformly, while
  \(C_L>0\) is fixed.

  In the dense regime \(m\ge n^{2\alpha/(2\alpha+2s+1)}\), we have
  \(d_* \asymp n^{1/(2\alpha+2s+1)}\), so
  \[
    \frac{\lambda_{d_*}}{\zeta_{n,m}}
    \gtrsim
    \frac{n^{-2\alpha/(2\alpha+2s+1)}}{m^{-1} \sqrt{\log(nm)/n}}
    =
    \frac{mn^{1/2-2\alpha/(2\alpha+2s+1)}}{\sqrt{\log(nm)}}.
  \]
  In this regime, for all sufficiently large \(n,m\), the lower bound
  \(n^{2\alpha/(2\alpha+2s+1)}\) exceeds \(2\), and the map
  \(m \mapsto m/\sqrt{\log(nm)}\) is increasing on \([2,\infty)\).
  Therefore the last display is bounded below by its value at
  \(m = n^{2\alpha/(2\alpha+2s+1)}\), namely
  \[
    \frac{n^{1/2}}{
      \sqrt{\log\left(n^{1+2\alpha/(2\alpha+2s+1)} \right)}
    }
    \asymp
    \frac{n^{1/2}}{\sqrt{\log n}}
    \to
    \infty.
  \]
  This divergence is uniform over the fixed range of \((\alpha,s)\).

  In the sparse regime \(m<n^{2\alpha/(2\alpha+2s+1)}\), we have
  \(d_* \asymp (nm)^{1/(4\alpha+2s+1)}\), so
  \[
    \frac{\lambda_{d_*}}{\zeta_{n,m}}
    \gtrsim
    \frac{(nm)^{-2\alpha/(4\alpha+2s+1)}}{m^{-1} \sqrt{\log(nm)/n}}
    =
    \frac{
      n^{1/2-2\alpha/(4\alpha+2s+1)}
      m^{1-2\alpha/(4\alpha+2s+1)}
    }{\sqrt{\log(nm)}}.
  \]
  The exponents satisfy
  \[
    \frac{1}{2} - \frac{2\alpha}{4\alpha+2s+1}
    =
    \frac{2s+1}{2(4\alpha+2s+1)} > 0,
    \qquad
    1 - \frac{2\alpha}{4\alpha+2s+1} > 0,
  \]
  and both exponents are bounded away from zero over the fixed range of
  \((\alpha,s)\).
  Since \(\log(nm)\lesssim\log n\) in the sparse regime, this ratio again diverges uniformly.
  Hence, uniformly over the fixed range of \((\alpha,s)\),
  \[
    \frac{\lambda_{d_*(\alpha,s)}}{\zeta_{n,m}}\to\infty.
  \]

  Let \(0 \le \ell \le L_*\) and \(k \in B_\ell\).
  By the definition of \(L_*\),
  we have
  \[
    \absx{k}
    \le
    2^\ell
    \le
    2^{L_*}
    \le
    2d_*.
  \]
  Hence
  \[
    \lambda_k
    \ge
    c_\lambda (2d_*)^{-2\alpha}
    \gtrsim
    \lambda_{d_*}.
  \]
  Since \(\lambda_{d_*}/\zeta_{n,m} \to\infty\) uniformly, for all sufficiently large \(n,m\),
  \[
    \min_{k \in B_\ell} \lambda_k
    \ge
    4\zeta_{n,m}.
  \]
\end{proof}

\subsection{Block estimation}
\label{subsec:indep-oracle-block-statistics}

Now we proceed with the block energy $\hat{S}_\ell$.
Let us define the deterministic noise scales for each coordinate and block:
\begin{equation}
  v_k \coloneqq \frac{C_V}{n}\left(1 + \frac{1}{m\lambda_k}\right),
  \qquad
  V_\ell \coloneqq \sum_{k \in B_\ell} v_k,
  \qquad
  \bar{v}_\ell \coloneqq \max_{k \in B_\ell} v_k.
  \label{eq:indep-block-noise-scale}
\end{equation}
Since \(\lambda_k \asymp (1+\absx{k})^{-2\alpha}\), dyadic blocks satisfy
\begin{equation}
  v_k \asymp \bar{v}_\ell,
  \qquad
  V_\ell \asymp \absx{B_\ell}\bar{v}_\ell,
  \qquad
  k \in B_\ell.
  \label{eq:indep-block-comparability}
\end{equation}

It is easy to see that the estimation variance for each coordinate is bounded by the noise scale \(v_k\):

\begin{lemma}[Single-coordinate variance]
  \label{lem:indep-block-variance-envelope}
  \[
    \lambda_k
    \E\absx{
      \frac{\hat{\gamma}_k}{\lambda_k} - \theta_k
    }^2
    \le
    C v_k.
  \]
\end{lemma}
\begin{proof}
  Apply \cref{lem:indep-gamma-estimator-moments} with \(n\) replaced by \(n_\gamma\).
\end{proof}

For block thresholding, we need a more refined control for a block relative to the block noise scale \(V_\ell\).
Paralleling $\hat{S}_l$ in \cref{eq:indep-adaptive-block-statistics},
we define the block vector
\begin{equation}
  \label{eq:indep-block-noise-coordinate-definition}
  \Delta_{\ell}
  \coloneqq
  \xk{
    \frac{\hat{\gamma}_k - \lambda_k \theta_k}{\sqrt{\lambda_k}}
  }_{k \in B_\ell}.
\end{equation}

\begin{lemma}[Block concentration]
  \label{lem:indep-block-noise-truncation}
  For every fixed \(\tau>0\), after choosing \(C_V\) sufficiently large in
  \cref{eq:indep-block-noise-scale}, there exist constants \(C,c,N_1>0\) such
  that, for all \(n\ge N_1\),
  \begin{equation}
    \E\left[
        \norm{\Delta_\ell}_2^2
        \mathbf{1}\{\norm{\Delta_\ell}_2^2>\tau V_\ell\}
    \right]
    \le
    C V_\ell \omega_\ell,
    \label{eq:indep-block-noise-truncated}
  \end{equation}
  where
  \begin{equation}
    \omega_\ell
    \coloneqq
    e^{-c\absx{B_\ell}}
    +
    e^{-c(nm)^{1/4}}
    +
    e^{-c n^{1/3}}.
    \label{eq:indep-block-truncation-tail}
  \end{equation}
\end{lemma}

Expanding the definition of $\Delta_\ell$, we have
\[
  \Delta_{\ell,k}
  =
  \frac{1}{n_\gamma}\sum_{i\in I_\gamma}
  \xk{
    \frac{Y_i}{m}\sum_{j=1}^m
    (X_i(t_{ij})+\delta_{ij})\frac{e_k(t_{ij})}{\sqrt{\lambda_k}}
    -
    \sqrt{\lambda_k}\theta_k
  }.
\]
We further decompose it into a structural component and a noise component.
Note that \(\int_0^1 X_i(t)e_k(t)\dd t=x_{ik}\).
We use decomposition
\begin{equation}
  \Delta_\ell
  =
  \Delta_\ell^{a}
  +
  \Delta_\ell^{b},
  \label{eq:indep-block-noise-component-decomposition}
\end{equation}
where
summands are
\begin{align}
  \Delta^{a}_{\ell,k}
  &\coloneqq
  \frac{1}{n_\gamma}\sum_{i\in I_\gamma}
  \xk{Y_i \frac{x_{ik}}{\sqrt{\lambda_k}}
  -
  \sqrt{\lambda_k}\theta_k}
  \label{eq:indep-block-structural-component-definition}
  \\
  \Delta^{b}_{\ell,k}
  &\coloneqq
  \frac{1}{n_\gamma}\sum_{i\in I_\gamma}
  \frac{Y_i}{m \sqrt{\lambda_k}} \sum_{j=1}^m
  \zk{
    \xk{X_i(t_{ij}) + \delta_{ij}}e_k(t_{ij}) - x_{ik}
  }.
  \label{eq:indep-block-B-component-definition}
\end{align}

For later componentwise estimates, define the natural block scales
\[
  V_\ell^{a}
  \coloneqq
  \frac{\absx{B_\ell}}{n},
  \qquad
  V_\ell^{b}
  \coloneqq
  \frac{1}{nm}\sum_{k\in B_\ell} \frac{1}{\lambda_k}.
\]
Then \(V_\ell^{a}+V_\ell^{b} \lesssim V_\ell\), with the fixed
constant \(C_V\) in \cref{eq:indep-block-noise-scale} absorbed into the
comparison constant.
We shall repeatedly use the following deterministic reduction: if \(U=U_1+U_2\)
in a Hilbert space, then, for every \(t>0\),
\begin{equation}
  \norm{U}^2 \mathbf{1}\{\norm{U}^2>t\}
  \le
  4\norm{U_1}^2 \mathbf{1}\{\norm{U_1}^2>t/4\}
  +
  4\norm{U_2}^2 \mathbf{1}\{\norm{U_2}^2>t/4\}.
  \label{eq:indep-two-term-truncation-reduction}
\end{equation}
This uses no independence; in applications we may further drop either indicator
on the right-hand side.

\begin{proposition}[Structural block component]
  \label{prop:indep-block-structural-truncation}
  There exist constants \(\tau_A,C,c,N_A>0\), depending only on
  \((R_0,\sigma_\varepsilon,C_\lambda)\), such that, for all \(n\ge N_A\), every
  active block \(B_\ell\), and every \(\tau\ge\tau_A\),
  \begin{equation}
    \E\left[
        \norm{\Delta_\ell^{a}}_2^2
        \mathbf{1}\{\norm{\Delta_\ell^{a}}_2^2>\tau V_\ell^{a}\}
    \right]
    \le
    C V_\ell^{a} e^{-c\absx{B_\ell}}.
    \label{eq:indep-block-structural-truncated}
  \end{equation}
\end{proposition}

\begin{proof}
  Put \(p_\ell=\absx{B_\ell}\). For \(i\in I_\gamma\), define
  \(\eta_{i,\ell} \in\R^{B_\ell}\) by
  \[
    (\eta_{i,\ell})_k
    \coloneqq
    Y_i \frac{x_{ik}}{\sqrt{\lambda_k}} - \sqrt{\lambda_k}\theta_k,
    \qquad k\in B_\ell.
  \]
  Then \(\Delta_\ell^{a}=n_\gamma^{-1} \sum_{i\in I_\gamma} \eta_{i,\ell}\).
  To make the hypercontractive step explicit, fix
  \(u\in\bbS^{p_\ell-1}\) and put
  \[
    W_{i,u}
    \coloneqq
    \sum_{k\in B_\ell} u_k \frac{x_{ik}}{\sqrt{\lambda_k}},
    \qquad
    T_i
    \coloneqq
    Y_i-\varepsilon_i
    =
    \sum_{r\in\bbZ} x_{ir} \theta_r.
  \]
  Then
  \[
    \angx{u,\eta_{i,\ell}}
    =
    \sum_{k\in B_\ell} u_k
    \left[
      Y_i \frac{x_{ik}}{\sqrt{\lambda_k}}
      -
      \sqrt{\lambda_k}\theta_k
    \right]
    =
    Y_i W_{i,u}-\E(Y_i W_{i,u}).
  \]
  The normalized score \(W_{i,u}\) is a centered Gaussian linear form with
  \(\E\absx{W_{i,u}}^2=1\). Also
  \[
    \E\absx{T_i}^2
    =
    \sum_{r\in\bbZ} \lambda_r \absx{\theta_r}^2
    \le
    C_\lambda R_0^2.
  \]
  Since \(T_i\), \(W_{i,u}\), and \(\varepsilon_i\) are jointly Gaussian linear
  forms, Gaussian fourth-moment equivalence gives
  \[
    \E\absx{Y_i W_{i,u}}^2
    \le
    2\bigl(\E\absx{T_i}^4 \bigr)^{1/2}
    \bigl(\E\absx{W_{i,u}}^4 \bigr)^{1/2}
    +
    2\sigma_\varepsilon^2 \E\absx{W_{i,u}}^2
    \le
    C(R_0,\sigma_\varepsilon,C_\lambda).
  \]
  Therefore \(Y_i W_{i,u}-\E(Y_i W_{i,u})\) is a centered Gaussian polynomial of
  degree at most two with \(L^2\)-norm bounded by
  \(C(R_0,\sigma_\varepsilon,C_\lambda)\), uniformly in \(u\) and \(B_\ell\).
  Hypercontractivity for degree-two Gaussian chaoses gives, for \(q\ge2\),
  \[
    \norm{Y_i W_{i,u}-\E(Y_i W_{i,u})}_{L^q}
    \le
    Cq \norm{Y_i W_{i,u}-\E(Y_i W_{i,u})}_{L^2}.
  \]
  By the moment characterization of the \(\psi_1\) norm, uniformly in \(u\),
  \[
    \norm{\angx{u,\eta_{i,\ell}}}_{\psi_1} \le C,
    \qquad u\in\bbS^{p_\ell-1}.
  \]
  Applying \cref{lem:upper-vector-bernstein-psi-one} to
  \((\eta_{i,\ell})_{i\in I_\gamma}\), using
  \(n_\gamma \asymp n\), gives, for a constant \(C_A\) and every \(x\ge1\),
  \[
    \Pr\left\{
         \norm{\Delta_\ell^{a}}_2
      >
      C_A \left(
           \sqrt{\frac{p_\ell+x}{n}}
           +
           \frac{p_\ell+x}{n}
      \right)
    \right\}
    \le
    Ce^{-cx}.
  \]
  For \(x\ge p_\ell\), and since active blocks satisfy \(p_\ell \le n\) for all
  large \(n\), this implies
  \[
    \Pr\left\{
         \norm{\Delta_\ell^{a}}_2^2
      >
      C_A \frac{1}{n}\left(x+\frac{x^2}{n}\right)
    \right\}
    \le
    Ce^{-cx}.
  \]
  Apply \cref{lem:upper-bernstein-tail-integration} with
  \[
    Z=\norm{\Delta_\ell^{a}}_2^2,\qquad
    p=p_\ell,\qquad
    b=\frac{1}{n},\qquad
    M=\frac{1}{n},\qquad
    q=2,\qquad
    V=V_\ell^{a}.
  \]
  Since \(p_\ell \le n\),
  \[
    \frac{1}{n}\left(p_\ell+\frac{p_\ell^2}{n}\right)
    \le
    2V_\ell^{a}.
  \]
  Therefore the hypotheses of
  \cref{lem:upper-bernstein-tail-integration} hold for all
  \(\tau\ge\tau_A\), with \(\tau_A\) sufficiently large. This proves
  \cref{eq:indep-block-structural-truncated}.
\end{proof}

\begin{proposition}[Sampling block component]
  \label{prop:indep-block-design-truncation}
  There exist constants \(\tau_X,C,c,N_X>0\) such that, for all
  \(n\ge N_X\), every active block \(B_\ell\), and every \(\tau\ge\tau_X\),
  \begin{equation}
    \E\left[
        \norm{\Delta_\ell^{b}}_2^2
        \mathbf{1}\{\norm{\Delta_\ell^{b}}_2^2>\tau V_\ell^{b}\}
    \right]
    \le
    C V_\ell^{b}
    \left[
      e^{-c\absx{B_\ell}}
      +
      e^{-c(nm)^{1/4}}
      +
      e^{-c n^{1/3}}
    \right].
    \label{eq:indep-block-design-truncated}
  \end{equation}
\end{proposition}

\begin{proof}
  Put \(p_\ell=\absx{B_\ell}\), \(\lambda_\ell^\circ=\lambda_{2^\ell}\), and
  \(N_\gamma=n_\gamma m\). Throughout the proof we use
  \(n_\gamma \asymp n\) and dyadic comparability of \(\lambda_k\) on
  \(B_\ell\). Define
  \[
    \Phi_\ell(t)
    =
    \left(\frac{e_k(t)}{\sqrt{\lambda_k}}\right)_{k\in B_\ell},
    \qquad
    \bar{\Phi}_\ell=\int_0^1 \Phi_\ell(t)\dd t,
    \qquad
    G_{i,\ell}=\int_0^1 X_i(t)\Phi_\ell(t)\dd t .
  \]
  Then \(G_{i,\ell}=(x_{ik}/\sqrt{\lambda_k})_{k\in B_\ell}\). Decompose
  \[
    \Delta_\ell^{b}=\widetilde\Delta_\ell+Q_\ell,
  \]
  with
  \[
    \widetilde\Delta_\ell
    =
    \frac{1}{N_\gamma}\sum_{i\in I_\gamma} \sum_{j=1}^m
    Y_i \bigl[
      X_i(t_{ij})\Phi_\ell(t_{ij})-G_{i,\ell}
      +\delta_{ij} \bigl(\Phi_\ell(t_{ij})-\bar{\Phi}_\ell \bigr)
      \bigr]
  \]
  and
  \[
    Q_\ell
    =
    \bar{\Phi}_\ell \frac{1}{N_\gamma}
    \sum_{i\in I_\gamma} \sum_{j=1}^m Y_i \delta_{ij}.
  \]
  This is exactly the preceding definition of \(\Delta_\ell^{b}\), with
  \(\delta_{ij} \Phi_\ell(t_{ij})\) split into
  \(\delta_{ij}(\Phi_\ell(t_{ij})-\bar{\Phi}_\ell)+\delta_{ij} \bar{\Phi}_\ell\).

  Since the trigonometric basis is uniformly bounded,
  \begin{equation}
    \norm{\Phi_\ell(t)}_2^2
    =
    \sum_{k\in B_\ell} \frac{1}{\lambda_k}
    \lesssim
    \frac{p_\ell}{\lambda_\ell^\circ},
    \qquad
    \norm{\Phi_\ell(t)-\bar{\Phi}_\ell}_2^2
    \lesssim
    \frac{p_\ell}{\lambda_\ell^\circ}.
    \label{eq:indep-design-component-Phi-envelope}
  \end{equation}
  Condition on
  \[
    \mathcal{F}_{X,\delta}
    =
    \sigma\{X_i,Y_i,\delta_{ij}:i\in I_\gamma,\ 1\le j\le m\}.
  \]
  In this proof, \(\E_t\) and \(\Pr_t\) denote conditional expectation and
  probability over the design points \(t_{ij}\), given \(\mathcal{F}_{X,\delta}\).
  Define
  \[
    \zeta_{ij}
    =
    Y_i \bigl[
      X_i(t_{ij})\Phi_\ell(t_{ij})-G_{i,\ell}
      +
      \delta_{ij} \bigl(\Phi_\ell(t_{ij})-\bar{\Phi}_\ell \bigr)
      \bigr],
    \qquad
    \xi_{ij}=N_\gamma^{-1} \zeta_{ij}.
  \]
  Then
  \(\widetilde\Delta_\ell=\sum_{i\in I_\gamma} \sum_{j=1}^m \xi_{ij}\),
  and, conditionally on \(\mathcal{F}_{X,\delta}\), the
  \(\xi_{ij}\)'s are independent and centered.

  Set
  \[
    H_{n,m}
    =
    \frac{1}{n_\gamma m}\sum_{i\in I_\gamma} \sum_{j=1}^m
    Y_i^2 \left(\norm{X_i}_\infty^2+\delta_{ij}^2 \right),
    \qquad
    L_{n,m}
    =
    \max_{i\in I_\gamma,\ 1\le j\le m}
    \absx{Y_i}\left(\norm{X_i}_\infty+\absx{\delta_{ij}}\right).
  \]
  We first compute the conditional quantities required by Bousquet. Since the
  \(\xi_{ij}\)'s are conditionally independent and centered,
  \[
    \E_t \norm{\widetilde\Delta_\ell}_2^2
    =
    \sum_{i\in I_\gamma} \sum_{j=1}^m
    \E_t \norm{\xi_{ij}}_2^2.
  \]
  By \cref{eq:indep-design-component-Phi-envelope},
  \[
    \sum_{i,j} \E_t \norm{\zeta_{ij}}_2^2
    \le
    C\frac{p_\ell}{\lambda_\ell^\circ}
    \sum_{i,j} Y_i^2 \left(\norm{X_i}_\infty^2+\delta_{ij}^2 \right),
  \]
  where subtracting \(G_{i,\ell}\) can only reduce the \(X_i(t)\Phi_\ell(t)\)
  second moment. Hence
  \begin{equation}
    \E_t \norm{\widetilde\Delta_\ell}_2^2
    \le
    C H_{n,m} \frac{p_\ell}{nm\lambda_\ell^\circ},
    \qquad
    \E_t \norm{\widetilde\Delta_\ell}_2
    \le
    C\sqrt{H_{n,m} \frac{p_\ell}{nm\lambda_\ell^\circ}}.
    \label{eq:indep-design-component-conditional-second}
  \end{equation}

  For the weak variance, for every \(\norm{u}_2=1\),
  define the scalar block function
  \[
    \Psi_{\ell,u}(t)
    =
    \sum_{k\in B_\ell}
    u_k \frac{e_k(t)}{\sqrt{\lambda_k}},
    \qquad
    \bar{\Psi}_{\ell,u}=\int_0^1 \Psi_{\ell,u}(t)\dd t .
  \]
  By Parseval,
  \begin{equation}
    \norm{\Psi_{\ell,u}}_{L^2([0,1])}^2
    \lesssim
    \frac{1}{\lambda_\ell^\circ},
    \qquad
    \norm{\Psi_{\ell,u}-\bar{\Psi}_{\ell,u}}_{L^2([0,1])}^2
    \lesssim
    \frac{1}{\lambda_\ell^\circ}.
    \label{eq:indep-design-component-direction-L2}
  \end{equation}
  Hence
  \[
    \sum_{i,j} \E_t \absx{\angx{u,\xi_{ij}}}^2
    \le
    C\frac{H_{n,m}}{nm\lambda_\ell^\circ}.
  \]
  Therefore
  \begin{equation}
    \sigma_{w,\ell}^2
    \coloneqq
    \sup_{\norm{u}_2=1}
    \sum_{i,j} \E_t \absx{\angx{u,\xi_{ij}}}^2
    \le
    C\frac{H_{n,m}}{nm\lambda_\ell^\circ}.
    \label{eq:indep-design-component-conditional-weak-variance}
  \end{equation}
  Finally,
  \begin{equation}
    \max_{i,j} \norm{\xi_{ij}}_2
    \le
    C\frac{L_{n,m}}{nm}
    \sqrt{\frac{p_\ell}{\lambda_\ell^\circ}}.
    \label{eq:indep-design-component-conditional-envelope}
  \end{equation}

  We now control \(H_{n,m}\) and \(L_{n,m}\). By
  \cref{lem:indep-process-supnorm}, \(\norm{\norm{X_i}_\infty}_{\psi_2} \le C\);
  also \(\norm{Y_i}_{\psi_2} \le C\), and \(\delta_{ij}\) is Gaussian with
  variance \(\sigma_\delta^2\). Thus the summands defining \(H_{n,m}\) are
  subject-level averages of uniformly sub-Weibull variables of order \(1/2\).
  Applying \cref{lem:upper-subweibull-bernstein} at the subject level with
  \(\alpha=1/2\), a fixed deviation level, and truncation parameter
  \(L\asymp n^{2/3}\), followed by Cauchy--Schwarz and the boundedness of
  \(\E(H_{n,m})^2\), gives a constant \(C_H\) such that
  \begin{equation}
    \Pr\{H_{n,m}>C_H\}\le Ce^{-c n^{1/3}},
    \qquad
    \E\left[H_{n,m} \mathbf{1}\{H_{n,m}>C_H\}\right]
    \le
    Ce^{-c n^{1/3}}.
    \label{eq:indep-design-component-H-tail}
  \end{equation}
  Choose \(L_0=K\log(enm)+K(nm)^{1/4}\), with \(K\) sufficiently large.
  A union
  bound gives
  \begin{equation}
    \Pr\{L_{n,m}>L_0\}\le Ce^{-c(nm)^{1/4}}.
    \label{eq:indep-design-component-L-tail}
  \end{equation}
  Let
  \[
    \mathcal{G}
    =
    \{H_{n,m} \le C_H\}\cap\{L_{n,m} \le L_0\}.
  \]

  On \(\mathcal{G}\), the conditional mean, weak variance, and envelope are
  bounded by
  \[
    M_\ell=C\sqrt{\frac{p_\ell}{nm\lambda_\ell^\circ}},
    \qquad
    \sigma_{w,\ell}^2 \le \frac{C}{nm\lambda_\ell^\circ},
    \qquad
    E_\ell=C\frac{(nm)^{1/4}}{nm}
    \sqrt{\frac{p_\ell}{\lambda_\ell^\circ}},
  \]
  for all large \(n\).
  Applying \cref{lem:upper-hilbert-bousquet} conditionally
  on \(\mathcal{F}_{X,\delta}\), and then integrating over
  \(\mathcal{G}\), gives, for every \(x\ge1\),
  \begin{equation}
    \Pr\left\{
         \norm{\widetilde\Delta_\ell}_2
      >
      C\left[
         \sqrt{\frac{p_\ell+x}{nm\lambda_\ell^\circ}}
         +
         \sqrt{
           \frac{(nm)^{1/4}p_\ell}{(nm)^{3/2}\lambda_\ell^\circ}x}
         +
         \frac{(nm)^{1/4}}{nm}
         \sqrt{\frac{p_\ell}{\lambda_\ell^\circ}} x
      \right],
         \ \mathcal{G}
    \right\}
    \le
    Ce^{-cx}.
    \label{eq:indep-design-component-good-tail-raw}
  \end{equation}
  By dyadic comparability,
  \(V_\ell^{b} \asymp p_\ell/(nm\lambda_\ell^\circ)\).
  Hence, after taking
  \(\tau_X\) sufficiently large, \cref{eq:indep-design-component-good-tail-raw}
  implies, with \(\kappa_\ell=p_\ell \wedge(nm)^{1/4}\),
  \begin{equation}
    \Pr\left\{
         \norm{\widetilde\Delta_\ell}_2^2>s V_\ell^{b},\ \mathcal{G}
    \right\}
    \le
    C e^{-c\kappa_\ell \sqrt s},
    \qquad s\ge\tau_X.
    \label{eq:indep-design-component-good-tail}
  \end{equation}
  Layer-cake integration yields, for every \(\tau\ge\tau_X\),
  \begin{equation}
    \E\left[
        \norm{\widetilde\Delta_\ell}_2^2
        \mathbf{1}\{\norm{\widetilde\Delta_\ell}_2^2>\tau V_\ell^{b}\}
        \mathbf{1}_{\mathcal{G}}
    \right]
    \le
    C V_\ell^{b} e^{-c\kappa_\ell}.
    \label{eq:indep-design-component-good-truncated}
  \end{equation}

  On \(\mathcal{G}^c\), \cref{eq:indep-design-component-conditional-second}
  and \cref{eq:indep-design-component-H-tail} control the event
  \(\{H_{n,m}>C_H\}\).
  For the envelope bad event, the same block bounds and
  Gaussian fourth moments give
  \[
    \E\norm{\widetilde\Delta_\ell}_2^4 \le C(V_\ell^{b})^2.
  \]
  Thus, by Cauchy--Schwarz and \cref{eq:indep-design-component-L-tail},
  \begin{equation}
    \E\left[
        \norm{\widetilde\Delta_\ell}_2^2 \mathbf{1}_{\mathcal{G}^c}
    \right]
    \le
    C V_\ell^{b} \left(e^{-c n^{1/3}}+e^{-c\kappa_\ell} \right).
    \label{eq:indep-design-component-bad-truncated}
  \end{equation}
  Combining
  \cref{eq:indep-design-component-good-truncated,eq:indep-design-component-bad-truncated}
  gives
  \begin{equation}
    \E\left[
        \norm{\widetilde\Delta_\ell}_2^2
        \mathbf{1}\{\norm{\widetilde\Delta_\ell}_2^2>\tau V_\ell^{b}\}
    \right]
    \le
    C V_\ell^{b} \left(e^{-c\kappa_\ell}+e^{-c n^{1/3}}\right).
    \label{eq:indep-centered-B-truncated}
  \end{equation}

  It remains to treat \(Q_\ell\).
  Since
  \[
    \int_0^1 e_k(t)\dd t=\mathbf{1}\{k=0\},
  \]
  we have \(\bar{\Phi}_\ell=0\) unless \(\ell=0\).
  If \(\ell=0\), then
  \(\norm{\bar{\Phi}_0}_2^2 \le C\) and, conditioning on \((Y_i)_{i\in I_\gamma}\),
  \[
    \E_\delta \norm{Q_0}_2^2
    \le
    \frac{C}{n_\gamma m}\frac{1}{n_\gamma}\sum_{i\in I_\gamma} Y_i^2.
  \]
  Hence \(\E\norm{Q_\ell}_2^2 \le C V_\ell^{b} \mathbf{1}\{\ell=0\}\), and
  this is bounded by \(C V_\ell^{b} e^{-c\kappa_\ell}\), after adjusting
  constants, because \(\kappa_0\) is bounded.

  Finally, applying
  \cref{eq:indep-two-term-truncation-reduction} to
  \(\Delta_\ell^{b}=\widetilde\Delta_\ell+Q_\ell\), and dropping the indicator
  on the \(Q_\ell\) term, gives
  \[
    \norm{\Delta_\ell^{b}}_2^2
    \mathbf{1}\{\norm{\Delta_\ell^{b}}_2^2>\tau V_\ell^{b}\}
    \le
    4\norm{\widetilde\Delta_\ell}_2^2
    \mathbf{1}\{\norm{\widetilde\Delta_\ell}_2^2>\tau V_\ell^{b}/4\}
    +
    4\norm{Q_\ell}_2^2 .
  \]
  Increasing \(\tau_X\) by a fixed factor and using
  \cref{eq:indep-centered-B-truncated} gives the bound with factor
  \(e^{-c\kappa_\ell}+e^{-c n^{1/3}}\). Since
  \(e^{-c\kappa_\ell} \le e^{-cp_\ell}+e^{-c(nm)^{1/4}}\), this proves
  \cref{eq:indep-block-design-truncated}.
\end{proof}

\begin{proof}[Proof of \cref{lem:indep-block-noise-truncation}]
  The claim follows by applying
  \cref{eq:indep-two-term-truncation-reduction} to
  \(\Delta_\ell=\Delta_\ell^a+\Delta_\ell^b\), then using
  \cref{prop:indep-block-structural-truncation,prop:indep-block-design-truncation}
  after the allowed enlargement of \(C_V\), and absorbing fixed numerical
  constants into \(C\) and \(c\).
\end{proof}

\subsection{Eligible-Block Risk}
\label{subsec:full-adaptive-upper-bound}
This subsection is the local oracle-to-adaptive comparison.
On the good pilot event, eligibility makes the plug-in denominators comparable with the true eigenvalues, so the adaptive block rule can be compared with the oracle block rule.
The lemma below shows that a block satisfying \(\caEelig{\ell}\) pays the usual oracle thresholding cost, an exponentially small selection remainder, and the extra eigenvalue plug-in remainder \(B_\ell^\lambda\).
For an event \(\mathcal{A}\), write
\[
  R_\ell(\mathcal{A})
  \coloneqq
  \sum_{k \in B_\ell} \lambda_k
  \E\bigl[\absx{\hat{\theta}_k-\theta_k}^2 \mathbf{1}_{\mathcal{A}} \bigr].
\]

\begin{lemma}[Eligible-block risk decomposition]
  \label{lem:indep-eligible-block-risk}
  For every active block \(B_\ell\), the localized risk
  \(R_\ell(\caE_n \cap\caEelig{\ell})\)
  satisfies
  \begin{equation}
    R_\ell(\caE_n \cap\caEelig{\ell})
    \le
    C\min\{\Theta_\ell,V_\ell\}
    +
    C V_\ell \omega_\ell
    +
    B_\ell^\lambda,
    \label{eq:indep-eligible-block-risk}
  \end{equation}
  where
  \begin{equation}
    B_\ell^\lambda
    \coloneqq
    C\sum_{k \in B_\ell}
    \absx{\theta_k}^2
    \E\left[
        \frac{(\check{\lambda}_k-\lambda_k)^2}{\lambda_k}
        \mathbf{1}_{\caE_n \cap\caEelig{\ell}}
    \right].
    \label{eq:indep-plugin-remainder}
  \end{equation}
\end{lemma}

\begin{proof}
  By the definitions of \(\caEelig{\ell}\), \(\caEthr{\ell}\), and
  \(\caEkeep{\ell}\) in \cref{eq:indep-adaptive-events}, on
  \(\caE_n \cap\caEelig{\ell}\),
  \begin{equation}
    \caEkeep{\ell}\cap \caE_n \cap\caEelig{\ell}
    =
    \caEthr{\ell}\cap \caE_n \cap\caEelig{\ell},
    \qquad
    (\caEkeep{\ell})^c\cap \caE_n \cap\caEelig{\ell}
    =
    (\caEthr{\ell})^c\cap \caE_n \cap\caEelig{\ell}.
    \label{eq:indep-adaptive-keep-threshold-on-good}
  \end{equation}
  For \(k \in B_\ell\),
  \begin{equation}
    \hat{\theta}_k - \theta_k
    =
    \mathbf{1}_{\caEkeep{\ell}}
    \frac{\hat{\gamma}_k - \lambda_k \theta_k}{\check{\lambda}_k}
    +
    \mathbf{1}_{\caEkeep{\ell}}
    \theta_k \frac{\lambda_k-\check{\lambda}_k}{\check{\lambda}_k}
    -
    \mathbf{1}_{(\caEkeep{\ell})^c}\theta_k.
    \label{eq:indep-adaptive-error-decomposition}
  \end{equation}
  Using \(\absx{a+b+c}^2 \le 3(\absx{a}^2+\absx{b}^2+\absx{c}^2)\), we get
  \[
    R_\ell(\caE_n \cap\caEelig{\ell})\le T_{1\ell}+T_{2\ell}+T_{3\ell},
  \]
  where
  \begin{align*}
    T_{1\ell}
    &\coloneqq
    C\sum_{k \in B_\ell} \lambda_k
    \E\left[
        \mathbf{1}_{\caEkeep{\ell}}
        \frac{\absx{\hat{\gamma}_k-\lambda_k \theta_k}^2}{\check{\lambda}_k^2}
        \mathbf{1}_{\caE_n \cap\caEelig{\ell}}
    \right],\\
    T_{2\ell}
    &\coloneqq
    C\sum_{k \in B_\ell} \lambda_k \absx{\theta_k}^2
    \E\left[
        \mathbf{1}_{\caEkeep{\ell}}
        \frac{(\check{\lambda}_k-\lambda_k)^2}{\check{\lambda}_k^2}
        \mathbf{1}_{\caE_n \cap\caEelig{\ell}}
    \right],\\
    T_{3\ell}
    &\coloneqq
    C\Theta_\ell
    \Pr((\caEkeep{\ell})^c \cap \caE_n \cap\caEelig{\ell}).
  \end{align*}

  Introduce the oracle threshold statistic and signal vector
  \begin{equation}
    S_\ell
    \coloneqq
    \sum_{k \in B_\ell} \frac{\absx{\hat{\gamma}_k}^2}{\lambda_k},
    \qquad
    \mu_\ell
    \coloneqq
    \bigl(\sqrt{\lambda_k}\theta_k \bigr)_{k\in B_\ell}.
    \label{eq:indep-adaptive-oracle-threshold-quantities}
  \end{equation}
  Here \(S_\ell\) is the oracle analogue of the observable statistic
  \(\hat{S}_\ell\) from \cref{eq:indep-adaptive-block-statistics}.
  The definitions in
  \cref{eq:indep-block-signal-energy,eq:indep-block-noise-coordinate-definition,eq:indep-adaptive-oracle-threshold-quantities}
  give
  \begin{equation}
    S_\ell
    =
    \norm{\mu_\ell+\Delta_\ell}_2^2,
    \qquad
    \norm{\mu_\ell}_2^2=\Theta_\ell.
    \label{eq:indep-adaptive-oracle-signal-noise}
  \end{equation}

  We next record the two consequences of eligibility used below. First, the
  second implication in \cref{lem:indep-eligible-blocks} yields, on
  \(\caE_n \cap\caEelig{\ell}\),
  \begin{equation}
    \check{\lambda}_k \asymp \lambda_k,
    \qquad
    \hat{V}_\ell \asymp V_\ell,
    \qquad
    \hat{S}_\ell \asymp S_\ell.
    \label{eq:indep-adaptive-comparability}
  \end{equation}
  In particular, the empirical threshold transfers to the oracle scale as
  \begin{equation}
    \{S_\ell \ge 3V_\ell\}\cap \caE_n \cap\caEelig{\ell}
    \subseteq
    \caEkeep{\ell}\cap \caE_n \cap\caEelig{\ell}
    \subseteq
    \left\{S_\ell \ge \frac{1}{3}V_\ell \right\}
    \cap \caE_n \cap\caEelig{\ell}.
    \label{eq:indep-adaptive-threshold-transfer}
  \end{equation}
  Indeed, on \(\caE_n \cap\caEelig{\ell}\),
  \[
    \frac{2}{3} S_\ell \le \hat{S}_\ell \le 2S_\ell,
    \qquad
    \frac{2}{3} V_\ell \le \hat{V}_\ell \le 2V_\ell,
  \]
  and \cref{eq:indep-adaptive-keep-threshold-on-good} identifies
  \(\caEkeep{\ell}\) with \(\caEthr{\ell}\) on
  \(\caE_n \cap\caEelig{\ell}\).

  Also, we establish a crude bound on \(T_{1\ell}\) that holds in all regimes.
  \cref{lem:indep-block-variance-envelope} gives
  \[
    \E\norm{\Delta_\ell}_2^2
    =
    \sum_{k\in B_\ell}
    \lambda_k
    \E\absx{\frac{\hat{\gamma}_k}{\lambda_k}-\theta_k}^2
    \le
    C V_\ell.
  \]
  On \(\caE_n \cap\caEelig{\ell}\),
  \(\lambda_k/\check{\lambda}_k^2 \le C/\lambda_k\).
  Therefore
  \begin{equation}
    T_{1\ell}
    \le
    C\E\left[
         \norm{\Delta_\ell}_2^2
         \mathbf{1}_{\caEkeep{\ell}}
         \mathbf{1}_{\caE_n \cap\caEelig{\ell}}
    \right]
    \le
    C\E\norm{\Delta_\ell}_2^2
    \le
    C V_\ell.
    \label{eq:indep-adaptive-t1-basic}
  \end{equation}

  Now, we split into three regimes based on the signal strength \(\Theta_\ell\).

  \textbf{Weak block: \(\Theta_\ell \le \gamma_- V_\ell\).}
  If \(\Theta_\ell \le \gamma_- V_\ell\), choose \(C_->0\) and then
  \(\gamma_->0\) so that \(2\gamma_-+2C_- \le 1/3\).
  On the event
  \(\norm{\Delta_\ell}_2^2 \le C_-V_\ell\),
  \cref{eq:indep-adaptive-oracle-signal-noise} gives
  \[
    S_\ell
    \le
    2\Theta_\ell + 2\norm{\Delta_\ell}_2^2
    \le
    (2\gamma_-+2C_-)V_\ell
    \le
    \frac{1}{3}V_\ell,
  \]
  hence, by \cref{eq:indep-adaptive-threshold-transfer},
  \(\caEkeep{\ell}\cap\caE_n \cap\caEelig{\ell}\) implies
  \(\norm{\Delta_\ell}_2^2>C_-V_\ell\).
  By \cref{lem:indep-block-noise-truncation} with \(\tau=C_-\),
  \[
    T_{1\ell}
    \le
    C\E\left[
         \norm{\Delta_\ell}_2^2
         \mathbf{1}\{\norm{\Delta_\ell}_2^2>C_-V_\ell\}
    \right]
    \le
    C V_\ell \omega_\ell.
  \]
  Since \(T_{3\ell} \le C\Theta_\ell\),
  \[
    T_{1\ell}+T_{3\ell}
    \le
    C\Theta_\ell
    +
    C V_\ell \omega_\ell.
  \]

  \textbf{Strong block: \(\Theta_\ell \ge \gamma_+ V_\ell\).}
  Choose \(\gamma_+\ge12\).
  If \(\Theta_\ell \ge \gamma_+ V_\ell\), then on the event
  \(\norm{\Delta_\ell}_2^2 \le \Theta_\ell/4\),
  \cref{eq:indep-adaptive-oracle-signal-noise} gives
  \[
    S_\ell
    =
    \norm{\mu_\ell+\Delta_\ell}_2^2
    \ge
    \frac{1}{4}\Theta_\ell
    \ge
    3V_\ell.
  \]
  Hence, by \cref{eq:indep-adaptive-threshold-transfer},
  \((\caEkeep{\ell})^c \cap \caE_n \cap\caEelig{\ell}\) can only occur if
  \(\norm{\Delta_\ell}_2^2>\Theta_\ell/4\).
  On this event, \(\Theta_\ell \le 4\norm{\Delta_\ell}_2^2\).
  Since \(\Theta_\ell/4\ge (\gamma_+/4)V_\ell\) in the strong regime,
  \cref{lem:indep-block-noise-truncation} with \(\tau=\gamma_+/4\) yields
  \[
    T_{3\ell}
    \le
    C\E\left[
         \norm{\Delta_\ell}_2^2
         \mathbf{1}\{\norm{\Delta_\ell}_2^2>\Theta_\ell/4\}
    \right]
    \le
    C\E\left[
         \norm{\Delta_\ell}_2^2
         \mathbf{1}\{\norm{\Delta_\ell}_2^2>(\gamma_+/4)V_\ell\}
    \right]
    \le
    C V_\ell \omega_\ell.
  \]
  Using \cref{eq:indep-adaptive-t1-basic} to control \(T_{1\ell}\), we get
  \[
    T_{1\ell}+T_{3\ell}
    \le
    C V_\ell
    +
    C V_\ell \omega_\ell.
  \]

  \textbf{Intermediate block:
    \(\gamma_- V_\ell < \Theta_\ell < \gamma_+ V_\ell\).}
  If \(\gamma_- V_\ell < \Theta_\ell < \gamma_+ V_\ell\), then
  \(\Theta_\ell \asymp V_\ell\), and \cref{eq:indep-adaptive-t1-basic} gives
  \[
    T_{1\ell}+T_{3\ell}
    \le
    C(V_\ell+\Theta_\ell)
    \le
    C\min\{\Theta_\ell,V_\ell\}.
  \]

  Combining the three regimes yields
  \begin{equation}
    T_{1\ell}+T_{3\ell}
    \le
    C\min\{\Theta_\ell,V_\ell\}
    +
    C V_\ell \omega_\ell.
    \label{eq:indep-adaptive-t13}
  \end{equation}
  Finally, \cref{eq:indep-adaptive-comparability} implies
  \[
    \frac{\lambda_k}{\check{\lambda}_k^2}
    \mathbf{1}_{\caE_n \cap\caEelig{\ell}}
    \le
    \frac{C}{\lambda_k}
    \mathbf{1}_{\caE_n \cap\caEelig{\ell}},
  \]
  so
  \[
    T_{2\ell}
    \le
    C\sum_{k \in B_\ell}
    \absx{\theta_k}^2
    \E\left[
        \frac{(\check{\lambda}_k-\lambda_k)^2}{\lambda_k}
        \mathbf{1}_{\caE_n \cap\caEelig{\ell}}
    \right]
    =
    B_\ell^\lambda.
  \]
  Combining this with \cref{eq:indep-adaptive-t13} proves
  \cref{eq:indep-eligible-block-risk}.
\end{proof}

\subsection[Proof of the Adaptive Independent-Design Upper Bound]{Proof of \texorpdfstring{\cref{thm:indep-adaptive-upper}}{the adaptive independent-design upper bound}}
\label{subsec:proof-indep-adaptive-upper}

It is enough to prove the bound uniformly over fixed rectangles of the form above.
For \((\alpha,s)\) in this rectangle, write
\[
  \mathfrak{r}_{n,m}(\alpha,s)
  \coloneqq
  n^{-\frac{2\alpha+2s}{2\alpha+2s+1}}
  +
  (nm)^{-\frac{2\alpha+2s}{4\alpha+2s+1}}.
\]
Fix such \((\alpha,s)\), \(\lambda\in\caL_\alpha(c_\lambda,C_\lambda)\), and \(\theta\in\Theta_s(R_0)\).

\textbf{Step 1: oracle cutoff and risk split.}
Let \(d_* = d_*(\alpha,s)\) be the oracle-rate cutoff from
\cref{lem:indep-oracle-band-floor}, and choose \(L_*\) so that
\begin{equation}
  2^{L_*-1} < d_* \le 2^{L_*}.
  \label{eq:indep-adaptive-oracle-level}
\end{equation}
By \cref{lem:indep-oracle-band-floor}, every block with index
\(0 \le \ell \le L_*\) has eigenvalues above the pilot floor.
By the first implication in \cref{lem:indep-eligible-blocks}, every such block is eligible on \(\caE_n\).

We split the risk into its contributions on \(\caE_n\) and \((\caE_n)^c\).
Write
\[
  \caR(\hat{\beta};\theta,\lambda;\caE_n)
  \coloneqq
  \E\left[
      \sum_{k\in\bbZ} \lambda_k
      \absx{\bar{\theta}_k-\theta_k}^2 \mathbf{1}_{\caE_n}
  \right].
\]
Steps 2--6 bound this good-event term; Step 7 controls the complementary
bad-event contribution.

\textbf{Step 2: good-event block reduction.}
For every active block, the Sobolev constraint and \(s\ge0\) imply
\[
  \norm{\theta_{B_\ell}}_2
  \le
  \norm{\theta}_2
  \le
  R_0,
  \qquad
  0\le \ell\le L_n.
\]
Since \(R\ge R_0\), \(\theta_{B_\ell}\) belongs to the Euclidean ball onto which
\(\hat{\theta}_{B_\ell}\) is projected.
Hence
\[
  \norm{\bar{\theta}_{B_\ell}-\theta_{B_\ell}}_2
  \le
  \norm{\hat{\theta}_{B_\ell}-\theta_{B_\ell}}_2,
  \qquad
  0 \le \ell \le L_n.
\]
Since the eigenvalues are comparable inside each dyadic block,
\[
  \sum_{k \in B_\ell}
  \lambda_k \absx{\bar{\theta}_k-\theta_k}^2
  \le
  C
  \sum_{k \in B_\ell}
  \lambda_k \absx{\hat{\theta}_k-\theta_k}^2,
  \qquad
  0 \le \ell \le L_n.
\]
Applying this blockwise and then adding the deterministic tail beyond \(K_n\)
gives
\begin{align*}
  \caR(\hat{\beta};\theta,\lambda;\caE_n)
  &\le
  C
  \sum_{\ell=0}^{L_n}
  \sum_{k \in B_\ell}
  \lambda_k
  \E\bigl[
    \absx{\hat{\theta}_k-\theta_k}^2
    \mathbf{1}_{\caE_n}
    \bigr]
  +
  \sum_{\absx{k} > K_n}\lambda_k \absx{\theta_k}^2
  \\
  &\le
  C
  \sum_{\ell=0}^{L_n} R_\ell(\caE_n)
  +
  \sum_{\absx{k} > K_n}\lambda_k \absx{\theta_k}^2.
\end{align*}
We now insert the local block risk bounds, separating the oracle range from
the above-oracle range.
For \(0 \le \ell \le L_*\),
eligibility holds on \(\caE_n\), so
\(R_\ell(\caE_n)=R_\ell(\caE_n \cap\caEelig{\ell})\).
\cref{lem:indep-eligible-block-risk} gives
\begin{equation}
  R_\ell(\caE_n)
  \le
  C V_\ell
  +
  B_\ell^\lambda
  +
  C V_\ell \omega_\ell.
  \label{eq:indep-adaptive-good-small-blocks}
\end{equation}
For \(L_* < \ell \le L_n\), decompose
\[
  R_\ell(\caE_n)
  =
  R_\ell(\caE_n \cap\caEelig{\ell})
  +
  R_\ell(\caE_n \cap(\caEelig{\ell})^c).
\]
On \((\caEelig{\ell})^c\), the block is killed, so the second term is bounded
by \(\Theta_\ell\). On \(\caE_n \cap\caEelig{\ell}\),
\cref{lem:indep-eligible-block-risk} applies, and the second implication in
\cref{lem:indep-eligible-blocks} ensures
\(\check{\lambda}_k \asymp \lambda_k\). Thus
\[
  B_\ell^\lambda
  \le
  C\sum_{k \in B_\ell} \lambda_k \absx{\theta_k}^2
  =
  C\Theta_\ell.
\]
Therefore, in all cases
\begin{equation}
  R_\ell(\caE_n)
  \le
  C\Theta_\ell
  +
  C V_\ell \omega_\ell,
  \qquad
  L_* < \ell \le L_n.
  \label{eq:indep-adaptive-good-large-blocks}
\end{equation}
Combining the last two displays gives the good-event decomposition
\begin{equation}
  \begin{aligned}
    \caR(\hat{\beta};\theta,\lambda;\caE_n)
    \lesssim
    &\underbrace{\sum_{\ell=0}^{L_*} V_\ell}_{\text{oracle variance}}
    +
    \underbrace{\sum_{\ell=L_*+1}^{L_n} \Theta_\ell}_{\text{oracle bias}}\\
    &+
    \underbrace{\sum_{\ell=0}^{L_*} B_\ell^\lambda}_{\text{plug-in remainder}}
    +
    \underbrace{\sum_{\ell=0}^{L_n} V_\ell \omega_\ell}_{\text{thresholding remainder}}
    +
    \underbrace{\sum_{\absx{k} > K_n}\lambda_k \absx{\theta_k}^2}_{\text{tail beyond }K_n}.
  \end{aligned}
  \label{eq:indep-adaptive-good-event-risk}
\end{equation}

\textbf{Step 3: oracle terms.}
The oracle variance satisfies, by dyadic comparability,
\begin{equation}
  \sum_{\ell=0}^{L_*} V_\ell
  \asymp
  \sum_{\absx{k} \le d_*}
  \frac{1}{n}\left(1+\frac{1}{m\lambda_k}\right)
  \asymp
  \frac{d_*}{n} + \frac{d_*^{2\alpha+1}}{nm}.
  \label{eq:indep-adaptive-variance-term}
\end{equation}
The oracle bias satisfies, by \cref{eq:upper-two-sided-tail-bias},
\begin{equation}
  \sum_{\ell=L_*+1}^{L_n} \Theta_\ell
  \le
  \sum_{\absx{k} > d_*}\lambda_k \absx{\theta_k}^2
  \lesssim
  d_*^{-(2\alpha+2s)}.
  \label{eq:indep-adaptive-bias-term}
\end{equation}
The high-frequency tail is no larger than the oracle bias scale. Since
\(K_n \ge d_*\) by \cref{lem:indep-oracle-band-floor},
\cref{eq:upper-two-sided-tail-bias} gives
\begin{equation}
  \sum_{\absx{k} > K_n}\lambda_k \absx{\theta_k}^2
  \lesssim
  K_n^{-(2\alpha+2s)}
  \lesssim
  d_*^{-(2\alpha+2s)}.
  \label{eq:indep-adaptive-tail-Kn}
\end{equation}

\textbf{Step 4: thresholding remainder.}
By definition of \(V_\ell\) and the eigenvalue bound
\(\lambda_k \gtrsim (1+\absx{k})^{-2\alpha}\), for \(\ell \ge 1\),
\[
  V_\ell
  \lesssim
  \frac{\absx{B_\ell}}{n}
  +
  \frac{1}{nm}\sum_{k \in B_\ell}(1+\absx{k})^{2\alpha}
  \lesssim
  \frac{2^\ell}{n}
  +
  \frac{2^{(2\bar{\alpha}+1)\ell}}{nm},
\]
while \(V_0 \lesssim n^{-1} + (nm)^{-1}\).
Since \(\absx{B_\ell}=2^\ell\)
for \(\ell \ge 1\), we obtain
\begin{equation}
  \begin{aligned}
    \sum_{\ell=0}^{L_n} V_\ell \omega_\ell
    & \le
    \sum_{\ell=0}^{L_n} V_\ell e^{-c\absx{B_\ell}}
    +
    \left(e^{-c(nm)^{1/4}}+e^{-c n^{1/3}}\right)
    \sum_{\ell=0}^{L_n} V_\ell \\
    & \lesssim n^{-1} + (nm)^{-1}
    =
    o\bigl(\mathfrak{r}_{n,m}(\alpha,s)\bigr).
  \end{aligned}
  \label{eq:indep-adaptive-exponential-remainder}
\end{equation}

\textbf{Step 5: plug-in remainder.}
Summing \cref{eq:indep-plugin-remainder} over \(\ell=0,\dots,L_*\) and using
\cref{lem:indep-pilot-moments},
\begin{equation}
  \sum_{\ell=0}^{L_*} B_\ell^\lambda
  \le
  \frac{C}{n_\lambda}
  \sum_{\absx{k} \le C d_*}
  \left(
    \lambda_k \absx{\theta_k}^2
  +
    \frac{\absx{\theta_k}^2}{m}
    +
    \frac{\absx{\theta_k}^2}{m^2 \lambda_k}
  \right)
  =: r_\lambda(d_*).
  \label{eq:indep-adaptive-plugin-remainder}
\end{equation}
We bound the three pieces separately.
First,
\[
  \sum_{k \in \bbZ} \lambda_k \absx{\theta_k}^2
  \le
  C_\lambda \sum_{k \in \bbZ}(1+\absx{k})^{-2\alpha}
  \absx{\theta_k}^2
  \le
  C_\lambda \sum_{k \in \bbZ}(1+\absx{k})^{2s}
  \absx{\theta_k}^2
  \le
  C.
\]
Second,
\[
  \sum_{k \in \bbZ} \absx{\theta_k}^2
  \le
  \sum_{k \in \bbZ}(1+\absx{k})^{2s}\absx{\theta_k}^2
  \le
  R_0^2.
\]
Third, since \(\lambda_k^{-1} \lesssim (1+\absx{k})^{2\alpha}\),
\[
  \sum_{\absx{k} \le d}\frac{\absx{\theta_k}^2}{\lambda_k}
  \lesssim
  d^{(2\alpha-2s)_+}
\]
by \cref{eq:upper-two-sided-partial-sum}.
Therefore
\begin{equation}
  r_\lambda(d_*)
  \lesssim
  \frac{1}{n_\lambda}
  +
  \frac{1}{n_\lambda m}
  +
  \frac{d_*^{(2\alpha-2s)_+}}{n_\lambda m^2}.
  \label{eq:indep-adaptive-plugin-bound}
\end{equation}
In the dense regime \(d_* \asymp n^{1/(2\alpha+2s+1)}\) and
\(m \ge n^{2\alpha/(2\alpha+2s+1)}\), each term in
\cref{eq:indep-adaptive-plugin-bound} is at most a constant multiple of
\(n^{-(2\alpha+2s)/(2\alpha+2s+1)}\).
In the sparse regime
\(d_* \asymp (nm)^{1/(4\alpha+2s+1)}\), each term is at most a constant multiple of
\((nm)^{-(2\alpha+2s)/(4\alpha+2s+1)}\).
Hence
\begin{equation}
  r_\lambda(d_*) \lesssim \mathfrak{r}_{n,m}(\alpha,s).
  \label{eq:indep-adaptive-plugin-rate}
\end{equation}

\textbf{Step 6: good-event bound.}
Combining
\cref{eq:indep-adaptive-good-event-risk,eq:indep-adaptive-variance-term,eq:indep-adaptive-bias-term,eq:indep-adaptive-tail-Kn,eq:indep-adaptive-exponential-remainder,eq:indep-adaptive-plugin-rate},
there exists \(N_0\) such that for all \(n,m \ge N_0\),
\begin{equation}
  \caR(\hat{\beta};\theta,\lambda;\caE_n)
  \le
  C\left(
     \frac{d_*}{n}
     +
     \frac{d_*^{2\alpha+1}}{nm}
     +
     d_*^{-(2\alpha+2s)}
  \right)
  \le
  C\mathfrak{r}_{n,m}(\alpha,s).
  \label{eq:indep-adaptive-good-final}
\end{equation}

\textbf{Step 7: bad-event contribution.}
Since \(\norm{\bar{\theta}_{B_\ell}}_2 \le R\) by construction and
\(\norm{\theta_{B_\ell}}_2 \le R_0\) by the Sobolev constraint and \(s\ge0\),
\[
  \sum_{k \in B_0} \lambda_k \absx{\bar{\theta}_k-\theta_k}^2
  \le
  C
\]
and, for \(1 \le \ell \le L_n\),
\[
  \sum_{k \in B_\ell}
  \lambda_k \absx{\bar{\theta}_k-\theta_k}^2
  \le
  \max_{k \in B_\ell} \lambda_k
  \norm{\bar{\theta}_{B_\ell}-\theta_{B_\ell}}_2^2
  \le
  C_\lambda (R+R_0)^2 2^{-2\alpha(\ell-1)}
  \le
  C 2^{-2\underline{\alpha}(\ell-1)}.
\]
Therefore
\[
  \sum_{k \in \bbZ} \lambda_k \absx{\bar{\theta}_k-\theta_k}^2
  \le
  C\sum_{\ell=0}^{L_n}2^{-2\underline{\alpha}(\ell-1)_+}
  +
  \sum_{\absx{k} > K_n}\lambda_k \absx{\theta_k}^2
  \le
  C.
\]
Hence \cref{lem:indep-pilot-uniform-control} implies
\begin{equation}
  \E\left[
      \sum_{k \in \bbZ} \lambda_k
      \absx{\bar{\theta}_k-\theta_k}^2 \mathbf{1}_{(\caE_n)^c}
  \right]
  \le
  Cn^{-20}
  =
  o\bigl(\mathfrak{r}_{n,m}(\alpha,s)\bigr).
  \label{eq:indep-adaptive-bad-event}
\end{equation}

Therefore, after possibly enlarging \(N_0\), for all \(n,m \ge N_0\),
combining \cref{eq:indep-adaptive-good-final,eq:indep-adaptive-bad-event}
gives
\[
  \E\caR(\hat{\beta};\theta,\lambda)
  \le
  C\mathfrak{r}_{n,m}(\alpha,s).
\]
This proves \cref{thm:indep-adaptive-upper}.
   \clearpage\section{Lower Bound Under Independent Design}
\label{sec:lower-indep}

The main tool for establishing the lower bound is van Trees' inequality as stated below.

\begin{lemma}[Van Trees inequality]
\label{lem:finite-dimensional-van-trees}
Let \(\vartheta\in\R^d\) have a regular prior \(\pi\) with Fisher information matrix \(J_\pi\).
Let \(I_T(\vartheta)\) be the Fisher information matrix of the observation law.
Then every estimator \(\hat{\vartheta}\) satisfies
\begin{equation}
  \E_\pi \E[
    (\hat{\vartheta}-\vartheta)(\hat{\vartheta}-\vartheta)^\top
  ]
  \succeq
  \bigl(\E_\pi I_T(\vartheta)+J_\pi \bigr)^{-1}.
  \label{eq:finite-dimensional-van-trees}
\end{equation}
Consequently, for every coordinate \(r\),
\begin{equation}
  \E_\pi \E\absx{\hat{\vartheta}_r-\vartheta_r}^2
  \ge
  \frac{1}{\E_\pi I_{T,rr}(\vartheta)+(J_\pi)_{rr}}.
  \label{eq:finite-dimensional-van-trees-coordinate}
\end{equation}
\end{lemma}

\begin{proof}
  \Cref{eq:finite-dimensional-van-trees} is the standard matrix Van Trees inequality~\citep{tsybakov2009_IntroductionNonparametric}.
  The coordinate bound follows because, for any positive definite matrix \(A\), \((A^{-1})_{rr}\ge A_{rr}^{-1}\), by the Schur complement.
\end{proof}

\subsection{Block Fisher Information Bound}
\label{subsec:indep-fisher-information-bound}

Let \(B\subset\bbZ\) be any finite block with \(\abs{B}=d\).
Suppose that the true parameter \(\theta\) is supported on this block, namely
\[
  \theta_k=0,\quad \forall k \notin B.
\]
Then the model is a Gaussian block model with nuisance coordinates outside the block.
The following lemma bounds the corresponding block Fisher information.

Let \((x_k)_{k\in\bbZ}\) be the Fourier scores of the predictor \(X\), so
\[
  X(t)=\sum_{k\in\bbZ} x_k e_k(t).
\]
For the block \(B\), write
\[
  x_B=(x_r)_{r\in B}\in\R^d,
  \qquad
  \theta_B=(\theta_{r})_{r\in B}\in\R^d,
  \qquad
  \Lambda_B=\mr{diag}(\lambda_r)_{r\in B}.
\]
For \(t=(t_1,\dots,t_m)\), define
\[
  \varphi_r(t)=\big(e_r(t_1),\dots,e_r(t_m)\big)^\top\in\R^m,
  \qquad
  \Phi_B(t)=[\varphi_r(t)]_{r\in B}\in\R^{m\times d}.
\]
With this notation, the FLR model can be written as
\[
  Y = \sum_{r \in B} \theta_r x_r + \varepsilon
  = \theta_B^\top x_B + \varepsilon,
\]
and the observation vector is
\begin{equation}
  Z=\Phi_B(t)x_B+\xi_B(t),
  \label{eq:indep-block-observation}
\end{equation}
where the nuisance vector \(\xi_B(t)\) is given by
\[
  \xi_B(t)=
  \sum_{k\in\bbZ\setminus B}
  x_k \big(e_k(t_1),\dots,e_k(t_m)\big)^\top+\delta.
\]
Since \(\alpha>\tfrac{1}{2}\) and the trigonometric basis is uniformly bounded,
\(\sum_k \lambda_k<\infty\). Thus \(\xi_B(t)\) is a well-defined centered
Gaussian vector independent of \(x_B\) conditional on \(t\), with covariance
\begin{equation}
  W_B(t)\coloneqq\mr{Cov}(\xi_B(t)\mid t)
  \succeq
  \sigma_\delta^2 I_m.
  \label{eq:indep-block-nuisance-cov}
\end{equation}
Thus, conditional on \(t\), the observation is a fixed Gaussian block
experiment with design matrix \(\Phi_B(t)\) and nuisance covariance \(W_B(t)\).

\begin{lemma}[Noisy Gaussian block Fisher bound]
\label{lem:noisy-gaussian-block-fisher}
Fix a deterministic matrix
\(\Phi_B=[\varphi_r]_{r\in B}\in\R^{m\times d}\) and a covariance matrix
\(W_B \succeq\sigma_\delta^2 I_m\). Consider the Gaussian block experiment
\[
  x_B \sim \mathcal{N}(0,\Lambda_B),
  \qquad
  Z=\Phi_B x_B+\xi_B,
  \qquad
  Y=\theta_B^\top x_B+\varepsilon,
\]
where \(\xi_B \sim \mathcal{N}(0,W_B)\) is independent of \(x_B\), and
\(\varepsilon\sim \mathcal{N}(0,\sigma_\varepsilon^2)\) is independent of
\((x_B,\xi_B)\). Let \(I^{\Phi,W}(\theta_B)\) be the Fisher information for
this single-observation experiment. Assume
\[
  \max_{r\in B} \norm{\varphi_r}_2^2 \le C_\Phi m.
\]
Then, for every \(r\in B\),
\begin{equation}
  I_{rr}^{\Phi,W}(\theta_B)
  \le
  C\min\left\{\lambda_r,\frac{m}{\sigma_\delta^2}\lambda_r^2 \right\}
  +
  C\lambda_r \theta_B^\top \Lambda_B \theta_B.
  \label{eq:noisy-gaussian-block-fisher}
\end{equation}
\end{lemma}

\begin{proof}
  The law of \(Z\) does not depend on \(\theta_B\), so the Fisher information of the joint observation \((Y,Z)\) is the expectation of the conditional Fisher information of \(Y\mid Z\), namely,
  \[
    I^{\Phi,W}(\theta_B)
    =
    \E_Z[I_{\mathrm{cond}}(\theta_B \mid Z)].
  \]
  To compute the conditional distribution, write
  \[
    V_B=\Phi_B \Lambda_B \Phi_B^\top+W_B.
  \]
  Gaussian conditioning gives
  \[
    x_B \mid Z
    \sim
    N\bigl(\mu_B(Z),\Sigma_B \bigr),
  \]
  where
  \[
    \mu_B(Z)=\Lambda_B \Phi_B^\top V_B^{-1} Z,
    \qquad
    \Sigma_B=\Lambda_B-\Lambda_B \Phi_B^\top V_B^{-1} \Phi_B \Lambda_B.
  \]
  Hence
  \[
    Y\mid Z,\theta_B
    \sim
    N\bigl(
      \theta_B^\top \mu_B(Z),
      \sigma_\varepsilon^2+\theta_B^\top \Sigma_B \theta_B
    \bigr).
  \]

  Consequently, the conditional Fisher information is
  \[
    I_{\mathrm{cond}}(\theta_B \mid Z)
    =
    \frac{1}{v_B(\theta_B)}\mu_B \mu_B^\top
    +
    \frac{2}{v_B(\theta_B)^2}
    \Sigma_B \theta_B \theta_B^\top \Sigma_B,
  \]
  where \(v_B(\theta_B)=
  \sigma_\varepsilon^2+\theta_B^\top \Sigma_B \theta_B\). Since
  \(v_B(\theta_B)\ge\sigma_\varepsilon^2\) and
  \(\Sigma_B \preceq\Lambda_B\),
  \begin{equation}
    \E_Z[I_{\mathrm{cond}}(\theta_B \mid Z)]
    \preceq
    \frac{1}{\sigma_\varepsilon^2}\E_Z[\mu_B \mu_B^\top]
    +
    \frac{2}{\sigma_\varepsilon^4}
    \Sigma_B \theta_B \theta_B^\top \Sigma_B
    \label{eq:noisy-block-conditional-fisher-bound}
  \end{equation}
  Since \(\mu_B(Z)=\Lambda_B \Phi_B^\top V_B^{-1} Z\) and
  \(\E[ZZ^\top]=V_B\), a direct calculation gives
  \[
    \E_Z[\mu_B \mu_B^\top]
    =
    \Lambda_B \Phi_B^\top V_B^{-1} \Phi_B \Lambda_B.
  \]
  Thus the mean component contributes at most
  \[
    \frac{1}{\sigma_\varepsilon^2} \xk{\E_Z[\mu_B \mu_B^\top]}_{rr}
    \leq C\lambda_r^2 \varphi_r^\top V_B^{-1} \varphi_r.
  \]
  Since \(V_B \succeq W_B \succeq\sigma_\delta^2 I_m\),
  \begin{equation}
    \varphi_r^\top V_B^{-1} \varphi_r
    \le
    \sigma_\delta^{-2} \norm{\varphi_r}_2^2
    \le
    C\sigma_\delta^{-2} m.
    \label{eq:noisy-block-mean-noise-bound}
  \end{equation}
  On the other hand, since \(V_B \succeq\lambda_r \varphi_r \varphi_r^\top\),
  \(V_B^{-1/2} \lambda_r \varphi_r \varphi_r^\top V_B^{-1/2} \preceq I_m\).
  Taking traces gives
  \begin{equation}
    \lambda_r \varphi_r^\top V_B^{-1} \varphi_r \le1.
    \label{eq:noisy-block-mean-rank-one-bound}
  \end{equation}
  Combining
  \cref{eq:noisy-block-mean-noise-bound,eq:noisy-block-mean-rank-one-bound}
  gives
  \begin{equation}
    \frac{1}{\sigma_\varepsilon^2}\xk{\E_Z[\mu_B \mu_B^\top]}_{rr}
    \le
    C\min\left\{\lambda_r,\frac{m}{\sigma_\delta^2}\lambda_r^2 \right\}.
    \label{eq:noisy-block-mean-component-bound}
  \end{equation}

  For the variance component, the PSD Cauchy--Schwarz inequality gives
  \begin{equation}
    \frac{2}{\sigma_\varepsilon^4}
    e_r^\top \Sigma_B \theta_B \theta_B^\top \Sigma_B e_r
    =
    \frac{2}{\sigma_\varepsilon^4}\absx{(\Sigma_B \theta_B)_r}^2
    \le
    C(e_r^\top \Sigma_B e_r)(\theta_B^\top \Sigma_B \theta_B)
    \le
    C\lambda_r \theta_B^\top \Lambda_B \theta_B.
    \label{eq:noisy-block-variance-component-bound}
  \end{equation}
  Plugging
  \cref{eq:noisy-block-mean-component-bound,eq:noisy-block-variance-component-bound}
  into the \(rr\) entry of
  \cref{eq:noisy-block-conditional-fisher-bound} proves
  \cref{eq:noisy-gaussian-block-fisher}.
\end{proof}

\subsection{Prior Distribution}
\label{subsec:lower-prior-distribution}

For the lower-bound prior, we will consider prior distribution of $\theta$ over the block
\begin{equation}
  B_d \coloneqq\{d+1,\dots,2d\},
  \label{eq:indep-block-definition}
\end{equation}
and set \( \theta_k=0,\quad k\notin B_d. \)
For an estimator \(\hat{\beta}=\sum_{r\in\bbZ} \hat{\theta}_r e_r\), the full prediction risk dominates the block risk
\begin{equation}
  \E\caR(\hat{\beta};\theta,\lambda)
  =
  \sum_{k\in\bbZ} \lambda_k \E\absx{\hat{\theta}_k-\theta_k}^2
  \ge
  \sum_{r\in B_d} \lambda_r
  \E\absx{\hat{\theta}_{r}-\theta_{r}}^2.
  \label{eq:lower-risk-block-reduction}
\end{equation}

Take the one-dimensional prior
\[
  p(u)=\cos^2 \left(\frac{\pi u}{2}\right)\mathbf{1}_{\{\absx{u}\le1\}},
\]
whose Fisher information is \(\pi^2\). Define the product prior \(\pi_d\) on
the block by
\begin{equation}
  \theta_r=a u_r,\quad r\in B_d,\qquad
  u_r \overset{\mathrm{iid}}{\sim}p(\cdot),
  \qquad
  a\coloneqq c_\pi R_0(2d)^{-s-\frac{1}{2}},
  \label{eq:lower-block-prior}
\end{equation}
where \(c_\pi>0\) is a sufficiently small fixed constant depending only on
\(s\). Translated to Fourier coefficients as above, this prior is supported on
\(\Theta_s(R_0)\), since
\begin{equation}
  \sum_{k\in\bbZ}(1+\absx{k})^{2s}\absx{\theta_k}^2
  \le
  C\sum_{r\in B_d} r^{2s} a^2 u_r^2
  \le
  Cd(2d)^{2s}a^2
  \le
  R_0^2.
  \label{eq:lower-block-prior-sobolev}
\end{equation}
Moreover,
\begin{equation}
  (J_{\pi_d})_{rr}=\frac{\pi^2}{a^2}\asymp d^{2s+1},
  \qquad r\in B_d.
  \label{eq:lower-block-prior-fisher}
\end{equation}

\begin{lemma}[Prior energy on the block]
\label{lem:lower-prior-energy}
Under \cref{eq:lower-block-prior},
\begin{equation}
  \E_{\pi_d}[\theta_B^\top \Lambda_B \theta_B]
  \le
  C R_0^2 d^{-(2\alpha+2s)}.
  \label{eq:lower-prior-energy}
\end{equation}
\end{lemma}

\begin{proof}
  On \(B_d\), \(\lambda_r \asymp d^{-2\alpha}\). Hence
  \[
    \E_{\pi_d}[\theta_B^\top \Lambda_B \theta_B]
    =
    \sum_{r\in B_d} \lambda_r a^2 \E u_r^2
    \le
    C a^2 d^{1-2\alpha}
    \le
    C R_0^2 d^{-(2\alpha+2s)}.
  \]
\end{proof}

\subsection[Proof of the independent-design lower bound]{Proof of \cref{thm:indep-lower}}
\label{subsec:proof-indep-lower}

Fix
\[
  \lambda_r=\bar{c}(1+\absx{r})^{-2\alpha}\in
  \caL_\alpha(c_\lambda,C_\lambda),
  \qquad r\in\bbZ,
  \qquad \bar{c}\in[c_\lambda,C_\lambda].
\]
For a block size \(d\), take the prior \(\pi_d\) defined in \cref{eq:lower-block-prior} on \(B=B_d\).
By \cref{eq:lower-block-prior-sobolev}, \(\pi_d\) is supported on \(\Theta_s(R_0)\).
Combining this support property with \cref{lem:finite-dimensional-van-trees,eq:lower-risk-block-reduction,eq:lower-block-prior-fisher} and \(\lambda_r \asymp d^{-2\alpha}\) on \(B_d\) gives
\begin{align}
  \inf_{\hat{\beta}}
  \sup_{\theta\in\Theta_s(R_0)}
  \E\caR(\hat{\beta};\theta,\lambda)
  &\ge
  \inf_{\hat{\beta}}
  \E_{\pi_d}\E\caR(\hat{\beta};\theta,\lambda)
  \nonumber\\
  &\ge
  \sum_{r\in B_d}
  \lambda_r
  \xk{\E_{\pi_d} I_{T,rr}(\theta_B)+(J_{\pi_d})_{rr}}^{-1}
  \nonumber\\
  &\ge
  c d^{1-2\alpha}
  \zk{\sup_{r\in B_d} \E_{\pi_d} I_{T,rr}(\theta_B)
    +
    Cd^{2s+1}}^{-1}
  \label{eq:indep-van-trees-reduction}
\end{align}

It remains to bound \(I_{T,rr}(\theta_B)\).
Since the law of each \(t_i\) is independent of \(\theta_B\), and the subjects are independent, the Fisher information of the full observation is
\begin{equation}
  I_T(\theta_B)
  =
  \sum_{i=1}^n
  \E_{t_i} \left[
    I^{\Phi_B(t_i),W_B(t_i)}(\theta_B)
  \right].
  \label{eq:indep-random-design-fisher}
\end{equation}
For each fixed \(t_i\), the block observation satisfies the assumptions of \cref{lem:noisy-gaussian-block-fisher} with \(\Phi_B=\Phi_B(t_i)\) and \(W_B=W_B(t_i)\).
The trigonometric basis is uniformly bounded, so \(\norm{\varphi_r(t_i)}_2^2 \le Cm\) for every \(r\in B_d\) and \(t_i\).
Thus \cref{lem:noisy-gaussian-block-fisher,eq:indep-random-design-fisher} yield, for every \(r\in B_d\),
\begin{equation}
  I_{T,rr}(\theta_B)
  \lesssim
  n \lambda_r
  \zk{\min(1, m\lambda_r) + \theta_B^\top \Lambda_B \theta_B}.
  \label{eq:indep-block-fisher-min-bound}
\end{equation}

Averaging \cref{eq:indep-block-fisher-min-bound} with respect to \(\pi_d\), using \cref{lem:lower-prior-energy}, and using \(\lambda_r \asymp d^{-2\alpha}\) on \(B_d\) yields
\begin{equation}
  \begin{aligned}
    \E_{\pi_d} I_{T,rr}(\theta_B)
    &\le
    Cn\lambda_r
    \zk{\min(1,m\lambda_r)+\E_{\pi_d}[\theta_B^\top \Lambda_B \theta_B]}
    \\
    &\le
    Cn d^{-2\alpha} \min(1,m d^{-2\alpha})
    +
    Cn d^{-(4\alpha+2s)},
    \qquad r\in B_d.
  \end{aligned}
  \label{eq:indep-n-term-fisher}
\end{equation}

The \(n\)-term follows from \(\min(1,m d^{-2\alpha})\le1\) in \cref{eq:indep-n-term-fisher}.
Substituting the resulting bound \(\sup_{r\in B_d} \E_{\pi_d} I_{T,rr}(\theta_B)\le Cn d^{-2\alpha}\) into \cref{eq:indep-van-trees-reduction} gives
\[
  \inf_{\hat{\beta}}
  \sup_{\theta\in\Theta_s(R_0)}
  \E\caR(\hat{\beta};\theta,\lambda)
  \ge
  c \frac{d^{1-2\alpha}}{n d^{-2\alpha}+d^{2s+1}}.
\]
Choosing
\[
  d\asymp n^{\frac{1}{2\alpha+2s+1}}
\]
therefore yields
\begin{equation}
  \inf_{\hat{\beta}}
  \sup_{\theta\in\Theta_s(R_0)}
  \E\caR(\hat{\beta};\theta,\lambda)
  \ge
  c n^{-\frac{2\alpha+2s}{2\alpha+2s+1}}.
  \label{eq:indep-n-rate-lower}
\end{equation}

The term due to noisy measurements follows from \(\min(1,m d^{-2\alpha})\le m d^{-2\alpha}\) in \cref{eq:indep-n-term-fisher}.
The second term is lower order because \(n d^{-(4\alpha+2s)}/(nm d^{-4\alpha})=d^{-2s}/m\le1\).
Thus
\begin{equation}
  \sup_{r\in B_d} \E_{\pi_d} I_{T,rr}(\theta_B)
  \le
  Cnm d^{-4\alpha}.
  \label{eq:indep-nm-term-fisher}
\end{equation}
Substituting \cref{eq:indep-nm-term-fisher} into \cref{eq:indep-van-trees-reduction} gives
\[
  \inf_{\hat{\beta}}
  \sup_{\theta\in\Theta_s(R_0)}
  \E\caR(\hat{\beta};\theta,\lambda)
  \ge
  c \frac{d^{1-2\alpha}}{nm d^{-4\alpha}+d^{2s+1}}.
\]
Choosing
\[
  d\asymp(nm)^{\frac{1}{4\alpha+2s+1}}
\]
yields
\begin{equation}
  \inf_{\hat{\beta}}
  \sup_{\theta\in\Theta_s(R_0)}
  \E\caR(\hat{\beta};\theta,\lambda)
  \ge
  c (nm)^{-\frac{2\alpha+2s}{4\alpha+2s+1}}.
  \label{eq:indep-nm-rate-lower}
\end{equation}

Combining \cref{eq:indep-n-rate-lower,eq:indep-nm-rate-lower} and decreasing the constant \(c\) if necessary gives the desired lower bound.
   \clearpage\section{Oracle Upper Bound Under Common Design}
\label{sec:upper-common-oracle}

This section proves the oracle upper bound under common design.
Throughout this section, we assume that \(t_{ij} = t_j=(j-1)/m\) without further mention.
The smoothness indices \((\alpha,s)\), the eigenvalue sequence \((\lambda_r)_{r\in\bbZ}\), and the cutoff are treated as known.
We first recall the oracle construction from the main text in \cref{subsec:common-oracle-upper}.
Recall that
\[
  \hat{\gamma}_r
  \coloneqq
  \frac{1}{n}\sum_{i=1}^n Y_i \bar{Z}_{ir},\quad
  \bar{Z}_{ir} = \frac{1}{m}\sum_{j=1}^m Z_{ij} e_r(t_j),\quad
  r \in \bbZ,
\]
and the estimator is given by
\[
  \tilde{\beta}
  =
  \sum_{\absx{r} \leq d} \tilde{\theta}_r e_r,\quad
  \tilde{\theta}_r = \frac{\hat{\gamma}_r}{\lambda_r},
\]
where \(d \le m/4 \) is a cutoff parameter.

\subsection{Fixed Grid Computation}

We start with computing Fourier coefficients related to the equally spaced grid.
Define the grid Gram coefficients
\begin{equation}
  \label{eq:cd-discrete-orthogonality}
  D_{r\ell}^{(m)}
  \coloneqq
  \frac{1}{m}\sum_{j=1}^m e_\ell(t_j)e_r(t_j),
  \qquad r,\ell\in\bbZ.
\end{equation}
Then, for any $f(t) = \sum_{\ell\in\bbZ} a_\ell e_\ell(t)$, we have
\begin{equation}
  \label{eq:cd-discrete-fourier}
  \frac{1}{m} \sum_{j=1}^m f(t_j) e_r(t_j) = \sum_{\ell\in\bbZ} D_{r\ell}^{(m)} a_\ell,\quad \forall r\in\bbZ.
\end{equation}

The following lemma gives Gram coefficients under equally spaced grids, which are the key to the Fourier calculations in the common-design setting.

\begin{lemma}
  \label{lem:grid-gram}
  For \(p,q\ge1\),
  \begin{equation}
    \label{eq:grid-gram}
    D_{p q}^{(m)}
    =
    \mathbf{1}\{p-q\equiv0\}
    +
    \mathbf{1}\{p+q\equiv0\},
    \quad
    D_{-p,-q}^{(m)}
    =
    \mathbf{1}\{p-q\equiv0\}
    -
    \mathbf{1}\{p+q\equiv0\},
  \end{equation}
  where congruences are modulo \(m\), while
  \[
    \begin{aligned}
      D_{p,-q}^{(m)}&=D_{-p,q}^{(m)}=0,\\
      D_{0q}^{(m)}&=D_{q0}^{(m)}=\sqrt{2}\cdot \mathbf{1}\{q\equiv0\},\\
      D_{0,-q}^{(m)}&=D_{-q,0}^{(m)}=0,\qquad D_{00}^{(m)}=1.
    \end{aligned}
  \]
  In particular, \(D_{rr}^{(m)}=1\) for \(\absx{r}<m/2\).
\end{lemma}

\begin{proof}
  Let \(u_k(t)=\exp(2\pi\mathrm{i}kt)\). Since \(t_j=(j-1)/m\), the grid values \(u_k(t_j)\) run through powers of an \(m\)-th root of unity as \(j\) varies.
  Therefore
  \[
    \frac{1}{m}\sum_{j=1}^m u_k(t_j)
    =
    \mathbf{1}\{k\equiv 0\pmod m\}.
  \]
  Taking real and imaginary parts gives
  \[
    \frac{1}{m}\sum_{j=1}^m \cos(2\pi k t_j)
    =
    \mathbf{1}\{k\equiv0\pmod m\},
    \qquad
    \frac{1}{m}\sum_{j=1}^m \sin(2\pi k t_j)
    =
    0.
  \]
  Recall that
  \[
    e_0(t)=1,\quad e_k(t)=\sqrt{2}\cos(2\pi k t),\quad e_{-k}(t)=\sqrt{2}\sin(2\pi k t),\quad k\ge1.
  \]
  For \(p,q\ge1\),
  the product identities then yield
  \[
    \begin{aligned}
      e_p(t)e_q(t)
      &=
      \cos(2\pi(p-q)t)+\cos(2\pi(p+q)t),\\
      e_{-p}(t)e_{-q}(t)
      &=
      \cos(2\pi(p-q)t)-\cos(2\pi(p+q)t),
    \end{aligned}
  \]
  and
  \[
    e_p(t)e_{-q}(t)
    =
    \sin(2\pi(p+q)t)+\sin(2\pi(q-p)t).
  \]
  Averaging these identities over the grid gives the formulas for \(D_{pq}^{(m)}\), \(D_{-p,-q}^{(m)}\), and \(D_{p,-q}^{(m)}=D_{-p,q}^{(m)}=0\).
  Since \(e_0\equiv1\), the same grid averages give \(D_{0q}^{(m)}=D_{q0}^{(m)}=\sqrt{2}\mathbf{1}\{q\equiv0\}\), \(D_{0,-q}^{(m)}=D_{-q,0}^{(m)}=0\), and \(D_{00}^{(m)}=1\).

  Finally,
  if \(0<\absx{r}<m/2\), then \(2\absx{r}\not\equiv0\pmod m\), so the displayed identities give \(D_{rr}^{(m)}=1\).
  The case \(r=0\) is \(D_{00}^{(m)}=1\).
\end{proof}

It is also easy to see from \cref{lem:grid-gram} that the grid Fourier vectors \(\{(e_r(t_j))_{j=1}^m:\absx{r}\le m\}\) span \(\R^m\).
Hence \(z\in\R^m\) is determined by its grid Fourier coefficients \((m^{-1} \sum_{j=1}^m z_j e_r(t_j))_{\absx{r}\le m}\).

Now, since $Z_{ij} = X_i(t_j) + \delta_{ij}$, we can write
\begin{equation}
  \label{eq:grid-zbar}
  \bar{Z}_{ir}
  \coloneqq
  \frac{1}{m}\sum_{j=1}^m Z_{ij} e_r(t_j)
  =
  \widetilde{x}_{ir}+\bar{\delta}_{ir},
\end{equation}
where the aliased score and averaged noise are
\begin{equation}
  \label{eq:grid-x-tilde}
  \widetilde{x}_{ir}
  \coloneqq \frac{1}{m}\sum_{j=1}^m X_i(t_j) e_r(t_j) =  \sum_{\ell\in\bbZ} D_{r\ell}^{(m)} x_{i\ell},
  \qquad
  \bar{\delta}_{ir}
  \coloneqq
  \frac{1}{m}\sum_{j=1}^m \delta_{ij} e_r(t_j),
\end{equation}
where we use \cref{eq:cd-discrete-fourier} in the first equality.

Similarly, using \cref{eq:cd-discrete-fourier}, the aliased cross-covariance coefficient is given by
\begin{equation}
  \label{eq:cd-aliased-gamma}
  \widetilde\gamma_r
  \coloneqq
  \E[Y_i \bar{Z}_{ir}]
  = \frac{1}{m} \sum_{j=1}^m g(t_j) e_r(t_j)
  =
  \sum_{\ell\in\bbZ} D_{r\ell}^{(m)} \gamma_\ell,
\end{equation}
which also shows that
\begin{equation}
  \label{eq:grid-gamma-mean}
  \E[\hat{\gamma}_r] = \widetilde{\gamma}_r.
\end{equation}

\subsection{Variance and Bias Bounds}
\label{subsec:cd-variance-bias-bounds}

We introduce the aliased score variance for \(r \in \bbZ\):
\begin{equation}
  \label{eq:cd-lambda-tilde}
  \widetilde\lambda_r
  \coloneqq
  \sum_{\ell\in\bbZ} \left(D_{r\ell}^{(m)} \right)^2\lambda_\ell.
\end{equation}

\begin{proposition}
  \label{prop:cd-grid-coefficient-covariance}
  For \(\absx{r},\absx{s}<m/2\), we have
  \begin{align}
    \label{eq:grid-x-covariance}
    \E \xk{\widetilde{x}_{ir} \widetilde{x}_{is}} &=  \widetilde\lambda_r \mathbf{1}\{r=s\}, \\
    \label{eq:grid-noise-covariance}
    \E \xk{\bar{\delta}_{ir} \bar{\delta}_{is}} &= \frac{\sigma_\delta^2}{m}\mathbf{1}\{r=s\},
  \end{align}
  and thus
  \begin{equation}
    \label{eq:grid-zbar-covariance}
    \E \xk{\bar{Z}_{ir} \bar{Z}_{is}} = \left(\widetilde\lambda_r+\frac{\sigma_\delta^2}{m}\right)\mathbf{1}\{r=s\}.
  \end{equation}
\end{proposition}
\begin{proof}
  For \( \widetilde{x}_{ir} \) in \cref{eq:grid-x-tilde}, since $x_{i\ell}$'s are independent Gaussian with variance $\lambda_\ell$, we have
  \[
    \E \xk{\widetilde{x}_{ir} \widetilde{x}_{is}} = \sum_{\ell \in \bbZ} D_{r\ell}^{(m)} D_{s\ell}^{(m)} \lambda_\ell.
  \]
  The congruence formulas in \cref{lem:grid-gram} show that, for \(\absx{r},\absx{s}<m/2\), the sets \(\{\ell:D_{r\ell}^{(m)}\ne0\}\) and \(\{\ell:D_{s\ell}^{(m)}\ne0\}\) are disjoint unless \(r=s\).
  Hence, \( D_{r\ell}^{(m)} D_{s\ell}^{(m)} \) is non-zero only if $r = s$ and thus \cref{eq:grid-x-covariance} holds.

  On the other hand, \cref{eq:grid-noise-covariance} follows from the independence of the noise and the fact that \(D_{rr}^{(m)}=1\) for \(\absx{r}<m/2\).
  Finally, \cref{eq:grid-zbar-covariance} follows from the independence of the signal and noise.
\end{proof}

The aliased score \(\widetilde{x}_{ir}\) and its variance \(\widetilde\lambda_r\) are defined in \cref{eq:grid-x-tilde,eq:cd-lambda-tilde}.

\begin{lemma}[Alias-variance equivalence]
  \label{lem:cd-alias-variance-equivalence}
  If \(\absx{r} \le d \le m/4\), then
  \begin{equation}
    \widetilde{\lambda}_r \asymp \lambda_r.
    \label{eq:cd-tilde-comparable}
  \end{equation}
\end{lemma}

\begin{proof}
  The alias variance contains the latent low-frequency variance plus contributions from the other frequencies with the same grid trace.
  Since \(\lambda_k \asymp (1 + \absx{k})^{-2\alpha}\) with \(\alpha > 1/2\), and
  \(D_{rr}^{(m)}=1\) by \cref{lem:grid-gram},
  \[
    \widetilde{\lambda}_r
    =
    \lambda_r
    +\sum_{\ell\ne r} \left(D_{r\ell}^{(m)} \right)^2\lambda_\ell
    \le
    \lambda_r
    +
    C\sum_{\ell\ne r:D_{r\ell}^{(m)} \ne0}
    (1+\absx{\ell})^{-2\alpha}.
  \]
  The congruence formulas in \cref{lem:grid-gram} show that,
  for \(\absx{r} \le m/4\), every off-diagonal aliased index satisfies \(\absx{\ell}\ge m/2\), and each grid shell contains only \(O(1)\) aliases.
  Hence
  \[
    \sum_{\ell\ne r:D_{r\ell}^{(m)} \ne0}
    (1+\absx{\ell})^{-2\alpha}
    \le
    C m^{-2\alpha}.
  \]
  Since \(\absx{r} \le m/4\), we also have \(\lambda_r \asymp (1+\absx{r})^{-2\alpha} \ge c m^{-2\alpha}\).
  Therefore \(\widetilde{\lambda}_r \le C \lambda_r\), while the lower bound \(\widetilde{\lambda}_r \ge \lambda_r\) is immediate from the diagonal term.
\end{proof}

\begin{lemma}
  \label{lem:cd-var-gamma}
  There exists \(C < \infty\), independent of \(\sigma_\delta\), such that, for every \(\absx{r} \le d \le m/4\),
  \begin{equation}
    \E\absx{\hat{\gamma}_r-\widetilde{\gamma}_r}^2
    =
    \Var(\hat{\gamma}_r)
    \le
    \frac{C}{n}\left(\lambda_r + \frac{\sigma_\delta^2}{m}\right).
    \label{eq:cd-var-gamma}
  \end{equation}
\end{lemma}

\begin{proof}
  The identity \(\E[\hat{\gamma}_r]=\widetilde{\gamma}_r\) follows directly from \cref{eq:cd-aliased-gamma}.
  Define \(A_{ir}\coloneqq Y_i\bar{Z}_{ir}\), so that \(\hat{\gamma}_r=n^{-1}\sum_{i=1}^n A_{ir}\).
  Since the subjects are independent,
  \[
    \Var(\hat{\gamma}_r) = \frac{1}{n}\Var(A_{1r}).
  \]
  We use the total-variance decomposition. Conditional on the Fourier scores and \(\varepsilon\),
  \[
    \E[A_{1r} \mid (x_{1\ell})_{\ell\in\bbZ},\varepsilon] = Y \widetilde{x}_r.
  \]
  Therefore
  \[
    \Var\bigl(\E[A_{1r} \mid (x_{1\ell})_{\ell\in\bbZ},\varepsilon]\bigr)
    =
    \Var(Y\widetilde{x}_r)
    \le
    \E\absx{Y}^2 \absx{\widetilde{x}_r}^2.
  \]
  As in the independent-design proof, \(\E Y^4 \le C\). Since \(\widetilde{x}_r\) is Gaussian with variance \(\widetilde{\lambda}_r\),
  \[
    \E\absx{Y}^2 \absx{\widetilde{x}_r}^2
    \le
    C \widetilde{\lambda}_r
    \le
    C \lambda_r
  \]
  by \cref{lem:cd-alias-variance-equivalence}.

  For the conditional variance term, the grid is deterministic, so the only randomness left after conditioning is the measurement noise:
  \[
    A_{1r}
    =
    \frac{Y}{m}\sum_{j=1}^m \bigl(X(t_j) + \delta_j \bigr)e_r(t_j).
  \]
  Hence
  \[
    \Var(A_{1r} \mid (x_{1\ell})_{\ell\in\bbZ},\varepsilon)
    =
    \frac{\absx{Y}^2}{m^2}
    \Var\Bigl(\sum_{j=1}^m \delta_j e_r(t_j) \Bigr)
    =
    \frac{\sigma_\delta^2 \absx{Y}^2}{m^2}
    \sum_{j=1}^m e_r(t_j)^2
    =
    \frac{\sigma_\delta^2}{m}\absx{Y}^2.
  \]
  The last equality uses \(D_{rr}^{(m)}=m^{-1} \sum_{j=1}^m e_r(t_j)^2=1\) for \(\absx{r}\le d\le m/4\).
  Taking expectations and using \(\E Y^2 \le C\),
  \[
    \E[\Var(A_{1r} \mid (x_{1\ell})_{\ell\in\bbZ},\varepsilon)]
    \le
    \frac{C\sigma_\delta^2}{m}.
  \]
  Therefore \(\Var(A_{1r}) \le C\lambda_r + C\sigma_\delta^2/m\), and dividing by \(n\) proves \cref{eq:cd-var-gamma}.
\end{proof}

\begin{lemma}[Aliasing bias]
  \label{lem:cd-aliasing-bias}
  Let
  \[
    a_r \coloneqq \widetilde{\gamma}_r - \gamma_r
    =
    \sum_{\ell\ne r} D_{r\ell}^{(m)} \gamma_\ell.
  \]
  If \(d \le m/4\), then
  \begin{equation}
    \sum_{\absx{r} \le d} \frac{\absx{a_r}^2}{\lambda_r}
    \le
    C R_0^2 m^{-(2\alpha+2s)}.
    \label{eq:cd-alias-sum}
  \end{equation}
\end{lemma}

\begin{proof}
  Fix \(\absx{r} \le d \le m/4\). By Cauchy--Schwarz,
  \begin{equation}
    \absx{a_r}^2
    \le
    \Bigl(
    \sum_{\ell\ne r:D_{r\ell}^{(m)} \ne0}
    \left(D_{r\ell}^{(m)} \right)^2
    (1+\absx{\ell})^{-2(s+2\alpha)}
    \Bigr)
    \Bigl(
    \sum_{\ell\ne r:D_{r\ell}^{(m)} \ne0}
    (1+\absx{\ell})^{2(s+2\alpha)}\absx{\gamma_\ell}^2
    \Bigr).
    \label{eq:cd-alias-bias}
  \end{equation}
  The first factor is a high-frequency alias tail.
  The congruence formulas in \cref{lem:grid-gram} show that,
  for \(\absx{r} \le m/4\), every aliased index \(\ell\ne r\) satisfies \(\absx{\ell}\ge m/2\), with only \(O(1)\) aliases in each grid shell.
  Hence
  \[
    \sum_{\ell\ne r:D_{r\ell}^{(m)} \ne0}
    \left(D_{r\ell}^{(m)} \right)^2
    (1+\absx{\ell})^{-2(s+2\alpha)}
    \le
    C m^{-2(s+2\alpha)}.
  \]
  Also, for \(\absx{r} \le d \le m/4\),
  \[
    \lambda_r^{-1}
    \le
    C(1+\absx{r})^{2\alpha}
    \le
    C m^{2\alpha}.
  \]
  Combining these bounds with \cref{eq:cd-alias-bias},
  \[
    \frac{\absx{a_r}^2}{\lambda_r}
    \le
    C m^{-(2\alpha+2s)}
    \sum_{\ell\ne r:D_{r\ell}^{(m)} \ne0}
    (1+\absx{\ell})^{2(s+2\alpha)}\absx{\gamma_\ell}^2.
  \]
  Summing over \(\absx{r} \le d\), each coefficient index is counted only a bounded number of times.
  Indeed, the congruence formulas in \cref{lem:grid-gram} leave only \(O(1)\) choices of \(r\) with \(\absx{r}\le d<m/2\) for each fixed \(\ell\).
  Therefore
  \begin{align*}
    \sum_{\absx{r} \le d} \frac{\absx{a_r}^2}{\lambda_r}
    & \lesssim
    m^{-(2\alpha+2s)}
    \sum_{\ell \in \bbZ}(1+\absx{\ell})^{2(s+2\alpha)}\absx{\gamma_\ell}^2 \\
    & \lesssim
    m^{-(2\alpha+2s)} \sum_{\ell \in \bbZ}(1+\absx{\ell})^{2(s+2\alpha)} \lambda_\ell^2 \theta_\ell^2 \\
    & \lesssim
    m^{-(2\alpha+2s)} \sum_{\ell \in \bbZ}(1+\absx{\ell})^{2s} \theta_\ell^2 \\
    & \lesssim
    m^{-(2\alpha+2s)},
  \end{align*}
  where we use the conditions $\lambda \in \caL_\alpha(c_\lambda,C_\lambda)$ and \(\theta \in \Theta_s(R_0)\).
\end{proof}

\subsection[Proof of the Oracle Common-Design Upper Bound]{Proof of \texorpdfstring{\cref{thm:common-oracle-upper}}{the oracle common-design upper bound}}
\label{subsec:cd-main-upper-bound}

For an arbitrary cutoff \(1\le d\le m/4\), we first show the following bound, keeping the fixed noise variance explicit:
\begin{equation}
  \sup_{\lambda \in \caL_\alpha(c_\lambda,C_\lambda)}
  \sup_{\theta \in \Theta_s(R_0)}
  \E\caR(\tilde{\beta}; \theta, \lambda)
  \lesssim
  \frac{d}{n}
  +
  \frac{\sigma_\delta^2 d^{2\alpha+1}}{nm}
  +
  d^{-(2\alpha+2s)}
  +
  m^{-(2\alpha+2s)}.
  \label{eq:cd-upper-oracle}
\end{equation}
Decompose the risk as
\[
  \E\caR(\tilde{\beta}; \theta, \lambda)
  =
  \sum_{\absx{r} \le d}
  \lambda_r\E\absx{\tilde{\theta}_r-\theta_r}^2
  +
  \sum_{\absx{r} > d}\lambda_r\absx{\theta_r}^2.
\]
The tail term is bounded by \cref{eq:upper-two-sided-tail-bias}.
For the estimation part, using \(\gamma_r=\lambda_r\theta_r\) gives
\[
  \sum_{\absx{r} \le d}
  \lambda_r\E\absx{\tilde{\theta}_r-\theta_r}^2
  =
  \sum_{\absx{r} \le d}
  \frac{\E\absx{\hat{\gamma}_r-\gamma_r}^2}{\lambda_r}.
\]
Write \(a_r=\widetilde{\gamma}_r-\gamma_r\). Since \(\E[\hat{\gamma}_r]=\widetilde{\gamma}_r\),
\[
  \E\absx{\hat{\gamma}_r-\gamma_r}^2
  \le
  2\E\absx{\hat{\gamma}_r-\widetilde{\gamma}_r}^2
  +
  2\absx{a_r}^2.
\]
The deterministic aliasing contribution is bounded by \cref{lem:cd-aliasing-bias}.
For the stochastic contribution, \cref{lem:cd-var-gamma} and
\(\lambda_r^{-1}\le C(1+\absx{r})^{2\alpha}\) imply
\[
  \begin{aligned}
    \sum_{\absx{r} \le d}
    \frac{\E\absx{\hat{\gamma}_r-\widetilde{\gamma}_r}^2}{\lambda_r}
    &\le
    \frac{C}{n}\sum_{\absx{r} \le d}
    \left(1+\frac{\sigma_\delta^2}{m\lambda_r}\right)\\
    &\le
    C\left(
       \frac{d}{n}
       +
       \frac{\sigma_\delta^2 d^{2\alpha+1}}{nm}
    \right).
  \end{aligned}
\]
Combining the stochastic bound, the aliasing bound, and the tail bound proves \cref{eq:cd-upper-oracle}.

Let \(\tilde{\beta}^*\) be \(\tilde{\beta}\) constructed with \(d=d^*\), where
\[
  d^*
  \asymp
  m
  \wedge
  n^{\frac{1}{2\alpha+2s+1}}
  \wedge
  \left(\frac{nm}{1+\sigma_\delta^2}\right)^{\frac{1}{4\alpha+2s+1}}.
\]
The implicit constant is chosen sufficiently small that \(d^*\le m/4\).
Substituting this cutoff into \cref{eq:cd-upper-oracle} gives
\begin{equation}
  \sup_{\lambda \in \caL_\alpha(c_\lambda,C_\lambda)}
  \sup_{\theta \in \Theta_s(R_0)}
  \E\caR(\tilde{\beta}^*; \theta, \lambda)
  \lesssim
  n^{-\frac{2\alpha+2s}{2\alpha+2s+1}}
  +
  (1+\sigma_\delta^2)^{\frac{2\alpha+2s}{4\alpha+2s+1}}
  (nm)^{-\frac{2\alpha+2s}{4\alpha+2s+1}}
  +
  m^{-(2\alpha+2s)}.
  \label{eq:cd-upper-final}
\end{equation}
Indeed, if the grid constraint is active, then \(d^*\asymp m\), and the inequalities defining this case imply
\[
  \frac{m}{n}\lesssim m^{-(2\alpha+2s)},
  \qquad
  \frac{\sigma_\delta^2 m^{2\alpha+1}}{nm}
  \le
  \frac{(1+\sigma_\delta^2)m^{2\alpha}}{n}
  \lesssim
  m^{-(2\alpha+2s)}.
\]
If the \(n\)-balancing cutoff is active, then \(d^*\asymp n^{1/(2\alpha+2s+1)}\).
The terms \(d^*/n\) and \((d^*)^{-(2\alpha+2s)}\) are both of order
\(n^{-(2\alpha+2s)/(2\alpha+2s+1)}\), while the measurement-noise term is bounded by the same order because
\[
  d^*
  \le
  \left(\frac{nm}{1+\sigma_\delta^2}\right)^{\frac{1}{4\alpha+2s+1}}.
\]
If the measurement-noise balancing cutoff is active, then
\[
  d^*
  \asymp
  \left(\frac{nm}{1+\sigma_\delta^2}\right)^{\frac{1}{4\alpha+2s+1}},
\]
so \((d^*)^{-(2\alpha+2s)}\) and \((1+\sigma_\delta^2)(d^*)^{2\alpha+1}/(nm)\) are both of order
\[
  (1+\sigma_\delta^2)^{\frac{2\alpha+2s}{4\alpha+2s+1}}
  (nm)^{-\frac{2\alpha+2s}{4\alpha+2s+1}},
\]
and the \(d^*/n\) term is bounded by the same order because \(d^*\le n^{1/(2\alpha+2s+1)}\).
This proves \cref{eq:cd-upper-final}.
After absorbing the fixed measurement-noise variance into the implicit constant,
\cref{eq:cd-upper-oracle,eq:cd-upper-final} give
\cref{eq:common-oracle-terms,eq:common-oracle-d,eq:common-oracle-rate}.
   \clearpage\section{Adaptive Upper Bound Under Common Design}
\label{sec:upper-common-adaptive}

This section proves the fully adaptive upper bound for the equal-grid common design.
The construction and proof are parallel to the independent-design adaptive
estimator in \cref{sec:upper-indep-adaptive}.
The differences are the common-design denominator pilot, which subtracts a
fixed high-frequency floor, and the fixed-grid terms caused by aliasing on the
shared grid, already isolated in \cref{sec:upper-common-oracle}.
We first recall the setup in \cref{subsec:cdeq-adaptive-setup}, estimate the
common-design denominator in \cref{subsec:cdeq-pilots}, identify eligible
blocks in \cref{subsec:cdeq-eligible-blocks}, carry out the block-estimation
analysis in \cref{subsec:cdeq-oracle}, establish the eligible-block risk bound
in \cref{subsec:cdeq-eligible-block-risk}, aggregate the proof of the main result in \cref{subsec:cdeq-main-aggregation}.
Moreover, \cref{subsec:cdeq-interlaced-pilot}  gives an alternative denominator pilot that is available only when \(m\) is even.

\subsection{Setup}
\label{subsec:cdeq-adaptive-setup}

We use the Fourier notation recorded in
\cref{subsec:fourier-coordinate-calculations}, including
\(\gamma_r=\lambda_r \theta_r\), and the common-grid quantities from
\cref{sec:upper-common-oracle}: \(\bar{Z}_{ir}\) in \cref{eq:grid-zbar},
\(\widetilde\gamma_r\) in \cref{eq:cd-aliased-gamma}, and
\(\widetilde\lambda_r\) in \cref{eq:cd-lambda-tilde}.
Throughout this section, \cref{assum:common-design} is in force.
The empirical coefficient \(\hat{\gamma}_r\) below estimates the aliased
quantity \(\widetilde\gamma_r\), not the population coefficient \(\gamma_r\).
Set
\[
  H_m \coloneqq\left\lfloor\frac{m-1}{2}\right\rfloor .
\]
The denominator target is
\begin{equation}
  \bar{\lambda}_r
  \coloneqq
  \widetilde\lambda_r-\widetilde\lambda_{H_m}.
  \label{eq:cdeq-denominator-target}
\end{equation}
The active band is kept below the grid scale, so that \(\absx{r}\le K_m\) lies
in the low-frequency region where \(\bar{\lambda}_r\) is positive and
comparable with both \(\lambda_r\) and \(\widetilde\lambda_r\).
The denominator pilot estimates \(\bar{\lambda}_r\) by subtracting the fixed
high-frequency floor \(\bar{Z}_{iH_m}^2\).
The two full-grid coefficients have the same additive measurement-noise
variance, so this subtraction removes the measurement-noise bias.
Fix constants \(\underline\alpha,\bar{\alpha},\underline s,\bar{s}\) and consider
\[
  (\alpha,s)
  \in
  [\underline\alpha,\bar{\alpha}]
  \times
  [\underline s,\bar{s}],
  \qquad
  \underline\alpha>\frac{1}{2},
  \qquad
  \underline s\ge 0.
\]
Write
\[
  \nu\coloneqq 2\alpha+2s,
  \qquad
  \kappa\coloneqq 4\alpha+2s+1.
\]

Split the subjects into two disjoint subsets
\begin{equation}
  \{1,\dots,n\}
  =
  I_\gamma \cup I_\lambda,
  \qquad
  n_\gamma \asymp n_\lambda \asymp n.
  \label{eq:cdeq-subject-split}
\end{equation}
Choose also a sufficiently small constant
\[
  c_L \in(0,1/8),
\]
depending only on
\((\underline\alpha,\bar{\alpha},c_\lambda,C_\lambda)\), and set
\[
  L_m
  \coloneqq
  \left\lfloor
    \log_2 \xk{c_L m\wedge n}
  \right\rfloor,
  \qquad
  K_m \coloneqq 2^{L_m}.
\]
Thus, \(K_m \asymp m\wedge n\) and \(K_m<m/8\).
Set the active-band logarithm
\[
  \ell_{n,m}^{\lambda} \coloneqq \log\bigl(n(K_m+1)\bigr).
\]
Fix once and for all an integer \(m_{\mathrm{ad}}\) large enough that
\[
  c_L m\ge1
  \quad
  \text{for all } m\ge m_{\mathrm{ad}}.
\]
Since the final upper bound is claimed for sufficiently large \(m\), the analysis is carried out on the range \(m\ge m_{\mathrm{ad}}\), and the theorem constant \(m_0\) is chosen at least this large.
For the pilot floor, set
\begin{equation}
  \zeta_{n,m}
  \coloneqq
  M_0 \left(
        \frac{1}{m}\sqrt{\frac{\ell_{n,m}^{\lambda}}{n_\lambda}}
        +
        \frac{\ell_{n,m}^{\lambda}}{n_\lambda}
  \right),
  \label{eq:cdeq-pilot-floor}
\end{equation}
where \(M_0>0\) is a large constant.
For \(\absx{r}\le K_m\), define the denominator pilot
\begin{equation}
  \check{\lambda}_r
  \coloneqq
  \frac{1}{n_\lambda}\sum_{i\in I_\lambda}
  \left(\bar{Z}_{ir}^2-\bar{Z}_{iH_m}^2 \right).
  \label{eq:cdeq-lambda-pilot}
\end{equation}
Use the two-sided dyadic blocks
\[
  B_0=\{0,\pm 1\},
  \qquad
  B_\ell=\{r:2^{\ell-1}<\absx{r}\le 2^\ell\},
  \qquad
  \ell\ge 1.
\]
The active blocks are
\[
  B_0,\dots,B_{L_m}.
\]
Frequencies \(\absx{r}>K_m\) are set to zero.

On the regression group, define
\begin{equation}
  \hat{\gamma}_r
  \coloneqq
  \frac{1}{n_\gamma}\sum_{i\in I_\gamma} Y_i \bar{Z}_{ir},
  \qquad
  \absx{r}\le K_m.
  \label{eq:cdeq-gamma-estimator}
\end{equation}

Define the pilot eligibility event
\begin{equation}
  \caEelig{\ell}
  \coloneqq
  \left\{
    \min_{r\in B_\ell} \check{\lambda}_r
    \ge
    2\zeta_{n,m}
  \right\}.
  \label{eq:cdeq-eligibility-event}
\end{equation}
On eligible blocks, define
\begin{equation}
  \hat{V}_\ell
  \coloneqq
  \frac{C_V}{n}
  \sum_{r\in B_\ell}
  \left(1+\frac{1}{m\check{\lambda}_r}\right),
  \qquad
  \hat{S}_\ell
  \coloneqq
  \sum_{r\in B_\ell}
  \frac{\absx{\hat{\gamma}_r}^2}{\check{\lambda}_r},
  \label{eq:cdeq-block-stats}
\end{equation}
and define the block events
\begin{equation}
  \caEthr{\ell}
  \coloneqq
  \{\hat{S}_\ell \ge \hat{V}_\ell\},
  \qquad
  \caEkeep{\ell}
  \coloneqq
  \caEelig{\ell}\cap\caEthr{\ell},
  \label{eq:cdeq-block-events}
\end{equation}
where \(C_V>0\) is a sufficiently large constant used in the variance bound and the selection rule.
Ineligible blocks are discarded; equivalently, take \(\caEthr{\ell}=\caEkeep{\ell}=\varnothing\) without evaluating the ratios in \(\hat{S}_\ell\), \(\hat{V}_\ell\), or \(\hat{\theta}_r\).
The thresholded coordinates are defined on retained blocks by
\begin{equation}
  \hat{\theta}_r
  \coloneqq
  \frac{\hat{\gamma}_r}{\check{\lambda}_r},
  \qquad
  r\in B_\ell,
  \quad
  \caEkeep{\ell}\ \text{holds},
  \quad
  0\le\ell\le L_m,
  \label{eq:cdeq-theta-thresholded}
\end{equation}
and are set to zero on discarded blocks and for \(\absx{r}>K_m\).
As in the independent-design construction, apply the same blockwise radius
projection \(\Pi_R\), with \(R\ge R_0\), from \cref{eq:indep-adaptive-projection-map}.
Writing \(\hat{\theta}_{B_\ell}=(\hat{\theta}_r)_{r\in B_\ell}\), define
\begin{equation}
  \bar{\theta}_{B_\ell}
  \coloneqq
  \Pi_R(\hat{\theta}_{B_\ell}),
  \qquad
  0\le\ell\le L_m.
  \label{eq:cdeq-theta}
\end{equation}
Let \(\bar{\theta}_r\) denote the corresponding projected coordinates, and set
\(\bar{\theta}_r=0\) for \(\absx{r}>K_m\). The final estimator is
\begin{equation}
  \hat{\beta}(t)
  \coloneqq
  \sum_{\absx{r}\le K_m} \bar{\theta}_r e_r(t).
  \label{eq:cdeq-beta}
\end{equation}

\subsection{Denominator estimation}
\label{subsec:cdeq-pilots}

\begin{lemma}[Denominator comparison]
  \label{lem:cdeq-denominator-comparison}
  Choose \(c_L>0\) sufficiently small. Then, uniformly over
  \(\absx{r}\le K_m\),
  \begin{equation}
    \widetilde\lambda_{H_m} \le C m^{-2\alpha},
    \qquad
    \absx{\bar{\lambda}_r-\lambda_r}\le C m^{-2\alpha},
    \qquad
    \bar{\lambda}_r \asymp\lambda_r \asymp\widetilde\lambda_r.
    \label{eq:cdeq-alias-variance-equivalence}
  \end{equation}
\end{lemma}

\begin{proof}
  By \cref{lem:grid-gram}, \(D_{rr}^{(m)}=1\).
  The congruence formulas there imply that, if \(\absx{r}\le K_m<m/4\), every off-diagonal aliased index has magnitude at least \(cm\), and there are only \(O(1)\) aliases in each grid shell.
  Hence
  \[
    \absx{\widetilde\lambda_r-\lambda_r}
    \le
    C\sum_{\ell\ne r:D_{r\ell}^{(m)} \ne0}(1+\absx{\ell})^{-2\alpha}
    \le
    C\sum_{a\ge1}(am)^{-2\alpha}
    \le
    C m^{-2\alpha}.
  \]
  Again by the same congruence formulas, every index in the support of \(D_{H_m,\ell}^{(m)}\) has magnitude at least \(cm\), since \(H_m \asymp m\) and \(H_m<m/2\).
  Therefore
  \[
    \widetilde\lambda_{H_m}
    =
    \sum_{\ell\in\bbZ} \left(D_{H_m,\ell}^{(m)} \right)^2\lambda_\ell
    \le
    C\sum_{a\ge0}(m(a+1))^{-2\alpha}
    \le
    C m^{-2\alpha}.
  \]
  Combining the two displays with
  \(\bar{\lambda}_r=\widetilde\lambda_r-\widetilde\lambda_{H_m}\) gives
  \[
    \absx{\bar{\lambda}_r-\lambda_r}
    \le
    C m^{-2\alpha}.
  \]
  Finally, \(\lambda_r \ge c_\lambda(1+K_m)^{-2\alpha}\). Choosing \(c_L\)
  sufficiently small makes the \(Cm^{-2\alpha}\) remainder at most
  \(\lambda_r/2\), uniformly over the parameter rectangle. This proves
  \(\bar{\lambda}_r \asymp\lambda_r\), and
  \(\widetilde\lambda_r \asymp\lambda_r\) follows from the first display. The
  bound \(\widetilde\lambda_{H_m} \le C\bar{\lambda}_r\) follows as well from
  \(\widetilde\lambda_{H_m} \lesssim m^{-2\alpha} \lesssim\lambda_r \asymp
  \bar{\lambda}_r\) on the active band.
\end{proof}

\begin{proposition}[Low-frequency denominator pilot]
  \label{prop:cdeq-lambda-pilot}
  There exist constants \(C,c>0\), depending only on
  \(\underline\alpha,\bar{\alpha},c_\lambda,C_\lambda,\sigma_\delta\), such that the
  following hold for every \(\absx{r}\le K_m\).
  \begin{itemize}
    \item The mean-square error obeys
    \begin{equation}
      \E\bigl(\check{\lambda}_r-\bar{\lambda}_r \bigr)^2
      \le
      \frac{C}{n_\lambda}
      \left(
        \bar{\lambda}_r^2
      +
        \frac{\bar{\lambda}_r}{m}
        +
        \frac{1}{m^2}
      \right)
      \label{eq:cdeq-lambda-pilot-mse}
    \end{equation}
    \item After enlarging \(M_0\) if necessary, the event
    \begin{equation}
      \caE_n^\lambda
      \coloneqq
      \bigcap_{\absx{r}\le K_m}
      \left\{
        \absx{\check{\lambda}_r-\bar{\lambda}_r}
        \le
        \frac{1}{2}\max\{\bar{\lambda}_r,\zeta_{n,m}\}
      \right\}
      \label{eq:cdeq-lambda-good-event}
    \end{equation}
    satisfies, for all \(n\ge n_0\) and \(m\ge m_0\),
    \begin{equation}
      \Pr\bigl((\caE_n^\lambda)^c\bigr)
      \le
      C\bigl[n(K_m+1)\bigr]^{-20},
      \label{eq:cdeq-lambda-good-event-prob}
    \end{equation}
    where \(C,n_0,m_0<\infty\) depend only on
    \(\underline\alpha,\bar{\alpha},c_\lambda,C_\lambda,\sigma_\delta\).
  \end{itemize}
\end{proposition}

\begin{proof}
  Let
  \[
    U_{ir} \coloneqq \bar{Z}_{ir},
    \qquad
    V_i \coloneqq \bar{Z}_{iH_m},
    \qquad
    Q_{ir} \coloneqq U_{ir}^2-V_i^2.
  \]
  Since \(\absx{r}\le K_m<H_m<m/2\), \cref{eq:grid-zbar-covariance} gives
  \[
    \E U_{ir}^2=\widetilde\lambda_r+\frac{\sigma_\delta^2}{m},
    \qquad
    \E V_i^2=\widetilde\lambda_{H_m}+\frac{\sigma_\delta^2}{m},
    \qquad
    \E[U_{ir} V_i]=0.
  \]
  Therefore \(\E Q_{ir}=\bar{\lambda}_r\). Isserlis' formula gives
  \[
    \Var(Q_{ir})
    \le
    C\left[
       \left(\widetilde\lambda_r+\frac{1}{m}\right)^2
      +
       \left(\widetilde\lambda_{H_m}+\frac{1}{m}\right)^2
    \right].
  \]
  By \cref{lem:cdeq-denominator-comparison},
  \(\widetilde\lambda_r \asymp\bar{\lambda}_r\) and
  \(\widetilde\lambda_{H_m} \le C\bar{\lambda}_r\) on the active band. Hence
  \[
    \Var(Q_{ir})
    \le
    C\left(
       \bar{\lambda}_r^2
      +
       \frac{\bar{\lambda}_r}{m}
       +
       \frac{1}{m^2}
    \right).
  \]
  The summands are independent over \(i\), so
  \cref{eq:cdeq-lambda-pilot-mse} follows.

  Since \(Q_{ir}-\bar{\lambda}_r\) is a centered quadratic polynomial of a
  Gaussian vector,
  \[
    \norm{Q_{ir}-\bar{\lambda}_r}_{\psi_1}
    \le
    C\left(\Var(U_{ir})+\Var(V_i)\right)
    \le
    C\left(\bar{\lambda}_r+\frac{1}{m}\right).
  \]
  Bernstein's inequality gives, for every \(x\ge1\),
  \[
    \Pr\left\{
         \absx{\check{\lambda}_r-\bar{\lambda}_r}
         >
         C\left[
            \left(\bar{\lambda}_r+\frac{1}{m}\right)\sqrt{\frac{x}{n_\lambda}}
            +
            \left(\bar{\lambda}_r+\frac{1}{m}\right)\frac{x}{n_\lambda}
      \right]
    \right\}
    \le
    2e^{-cx}.
  \]
  Take \(x=M_1 \ell_{n,m}^{\lambda}\) with \(M_1\) sufficiently large and apply a
  union bound over \(\absx{r}\le K_m\). Since \(K_m \le n\) and
  \(n_\lambda \asymp n\), the terms multiplied by \(\bar{\lambda}_r\) are at most
  \(\bar{\lambda}_r/4\) for all sufficiently large \(n\). After enlarging \(M_0\),
  the remaining terms are at most \(\zeta_{n,m}/4\). Therefore
  \cref{eq:cdeq-lambda-good-event-prob} follows.
\end{proof}

\subsection{Eligible Blocks}
\label{subsec:cdeq-eligible-blocks}

\begin{lemma}[Eligible blocks have stable denominators]
  \label{lem:cdeq-eligible-blocks}
  Recall the eligibility event \(\caEelig{\ell}\) from
  \cref{eq:cdeq-eligibility-event}.
  On the event \(\caE_n^\lambda\), for every block satisfying \(\caEelig{\ell}\)
  and every \(r\in B_\ell\),
  \[
    \frac{1}{2}\bar{\lambda}_r
    \le
    \check{\lambda}_r
    \le
    \frac{3}{2}\bar{\lambda}_r.
  \]
\end{lemma}

\begin{proof}
  On \(\caE_n^\lambda\), if \(\bar{\lambda}_r<\zeta_{n,m}\), then
  \[
    \check{\lambda}_r
    \le
    \bar{\lambda}_r+\frac{1}{2}\zeta_{n,m}
    <
    \frac{3}{2}\zeta_{n,m},
  \]
  contradicting \(\caEelig{\ell}\). Hence such an
  index cannot lie inside an eligible block. Therefore eligibility forces
  \(\bar{\lambda}_r \ge \zeta_{n,m}\), and then
  \[
    \absx{\check{\lambda}_r-\bar{\lambda}_r}
    \le
    \frac{1}{2}\bar{\lambda}_r
  \]
  by \cref{eq:cdeq-lambda-good-event}. This gives the displayed comparability.
\end{proof}

\subsection{Block estimation}
\label{subsec:cdeq-oracle}

Define the observable low-frequency parameter
\begin{equation}
  \vartheta_r
  \coloneqq
  \frac{\widetilde\gamma_r}{\bar{\lambda}_r},
  \qquad
  \absx{r}\le K_m.
  \label{eq:cdeq-observable-parameter}
\end{equation}
For each active block, write
\begin{equation}
  \widetilde\Theta_\ell
  \coloneqq
  \sum_{r\in B_\ell} \bar{\lambda}_r \absx{\vartheta_r}^2
  =
  \sum_{r\in B_\ell}
  \frac{\absx{\widetilde\gamma_r}^2}{\bar{\lambda}_r},
  \label{eq:cdeq-block-signal}
\end{equation}
and
\begin{equation}
  V_\ell
  \coloneqq
  \frac{C_V}{n}
  \sum_{r\in B_\ell}
  \left(1+\frac{1}{m\bar{\lambda}_r}\right).
  \label{eq:cdeq-block-variance}
\end{equation}
The corresponding oracle-denominator block statistic is
\[
  S_\ell^\circ
  \coloneqq
  \sum_{r\in B_\ell} \frac{\absx{\hat{\gamma}_r}^2}{\bar{\lambda}_r}.
\]
For later use, write
\[
  \mu_{\ell,r}
  \coloneqq
  \frac{\widetilde\gamma_r}{\sqrt{\bar{\lambda}_r}}
  =
  \sqrt{\bar{\lambda}_r}\vartheta_r,
  \qquad
  \Delta_{\ell,r}
  \coloneqq
  \frac{\hat{\gamma}_r-\widetilde\gamma_r}{\sqrt{\bar{\lambda}_r}},
  \qquad
  r\in B_\ell.
\]
Then
\[
  S_\ell^\circ=\norm{\mu_\ell+\Delta_\ell}_2^2,
  \qquad
  \norm{\mu_\ell}_2^2=\widetilde\Theta_\ell.
\]

\begin{lemma}[Observable parameter versus target parameter]
  \label{lem:cdeq-observable-target-bias}
  There exists \(C<\infty\), depending only on
  \(\alpha,s,R_0,c_\lambda,C_\lambda\), such that for every integer
  \(1\le d\le K_m\),
  \begin{equation}
    \sum_{\absx{r}\le d}\lambda_r \absx{\vartheta_r-\theta_r}^2
    \le
    C\left(
       m^{-(2\alpha+2s)}
       +
       m^{-4\alpha} d^{(2\alpha-2s)_+}
    \right).
    \label{eq:cdeq-observable-target-bias}
  \end{equation}
\end{lemma}

\begin{proof}
  Write
  \[
    a_r
    \coloneqq
    \widetilde\gamma_r-\gamma_r
    =
    \sum_{\ell\ne r} D_{r\ell}^{(m)} \gamma_\ell,
    \qquad
    b_r
    \coloneqq
    \bar{\lambda}_r-\lambda_r.
  \]
  Then
  \[
    \vartheta_r-\theta_r
    =
    \frac{a_r}{\bar{\lambda}_r}
    -
    \theta_r \frac{b_r}{\bar{\lambda}_r}.
  \]
  By \cref{eq:cdeq-alias-variance-equivalence},
  \[
    \lambda_r \absx{\vartheta_r-\theta_r}^2
    \le
    C\frac{\absx{a_r}^2}{\lambda_r}
    +
    C\absx{\theta_r}^2 \frac{b_r^2}{\lambda_r}.
  \]
  Summing the first term over \(\absx{r}\le d\) and using
  \cref{lem:cd-aliasing-bias} with the same \(d\) gives
  \[
    \sum_{\absx{r}\le d}\frac{\absx{a_r}^2}{\lambda_r}
    \le
    C m^{-(2\alpha+2s)}.
  \]
  For the second term, \cref{lem:cdeq-denominator-comparison} implies
  \(b_r \lesssim m^{-2\alpha}\) uniformly over \(\absx{r}\le K_m<m/8\), so
  \[
    \absx{\theta_r}^2 \frac{b_r^2}{\lambda_r}
    \lesssim
    m^{-4\alpha}(1+\absx{r})^{2\alpha}\absx{\theta_r}^2.
  \]
  Summing over \(\absx{r}\le d\) and using
  \cref{eq:upper-two-sided-partial-sum}, the second term is bounded by
  \(Cm^{-4\alpha} d^{(2\alpha-2s)_+}\). Combining the two bounds proves
  \cref{eq:cdeq-observable-target-bias}.
\end{proof}

\begin{lemma}[Observable coefficient sums]
  \label{lem:cdeq-observable-sums}
  There exists \(C<\infty\), depending only on
  \(\alpha,s,R_0,c_\lambda,C_\lambda\), such that for every integer
  \(1\le d\le K_m\),
  \begin{equation}
    \sum_{\absx{r}\le d}\absx{\vartheta_r}^2
    \le
    C,
    \qquad
    \sum_{\absx{r}\le d}\frac{\absx{\vartheta_r}^2}{\bar{\lambda}_r}
    \le
    C d^{(2\alpha-2s)_+}.
    \label{eq:cdeq-observable-sums}
  \end{equation}
\end{lemma}

\begin{proof}
  By Cauchy--Schwarz and the definition of \(\widetilde\lambda_r\),
  \[
    \absx{\widetilde\gamma_r}^2
    \le
    \widetilde\lambda_r
    \sum_{\ell:D_{r\ell}^{(m)} \ne0}
    \left(D_{r\ell}^{(m)} \right)^2\lambda_\ell \absx{\theta_\ell}^2.
  \]
  Therefore
  \[
    \absx{\vartheta_r}^2
    \le
    \frac{C}{\lambda_r}
    \sum_{\ell:D_{r\ell}^{(m)} \ne0}
    \left(D_{r\ell}^{(m)} \right)^2\lambda_\ell \absx{\theta_\ell}^2,
    \qquad
    \frac{\absx{\vartheta_r}^2}{\bar{\lambda}_r}
    \le
    \frac{C}{\lambda_r^2}
    \sum_{\ell:D_{r\ell}^{(m)} \ne0}
    \left(D_{r\ell}^{(m)} \right)^2\lambda_\ell \absx{\theta_\ell}^2.
  \]
  Since \(d\le K_m<m/8\), \cref{eq:cdeq-alias-variance-equivalence}
  gives
  \[
    \lambda_r^{-1} \lesssim (1+\absx{r})^{2\alpha},
    \qquad
    \lambda_r^{-2} \lesssim (1+\absx{r})^{4\alpha},
    \qquad
    \absx{r}\le d.
  \]

  For \(\sum_{\absx{r}\le d}\absx{\vartheta_r}^2\), the diagonal terms
  \(\ell=r\) contribute
  \[
    \sum_{\absx{r}\le d}(1+\absx{r})^{2\alpha}\lambda_r \absx{\theta_r}^2
    \le
    C\sum_{r\in\bbZ} \absx{\theta_r}^2
    \le
    C.
  \]
  For the off-diagonal aliases, \(D_{r\ell}^{(m)} \ne0\), \(\ell\ne r\), and
  \(\absx{r}\le d\le K_m<m/8\) imply \(\absx{\ell}\ge cm\). Hence
  \[
    (1+\absx{r})^{2\alpha}\lambda_\ell
    \le
    C d^{2\alpha} m^{-2\alpha}
    \le
    C.
  \]
  For fixed \(\ell\), the congruence formulas in \cref{lem:grid-gram} leave only \(O(1)\) choices of \(r\) with \(\absx{r}\le d<m/2\).
  Thus the off-diagonal contribution is also \(O(1)\).
  This proves the first inequality.

  For \(\sum_{\absx{r}\le d}\absx{\vartheta_r}^2/\bar{\lambda}_r\), the diagonal terms give
  \[
    \sum_{\absx{r}\le d}(1+\absx{r})^{4\alpha}\lambda_r \absx{\theta_r}^2
    \le
    C\sum_{\absx{r}\le d}(1+\absx{r})^{2\alpha}\absx{\theta_r}^2
    \le
    C d^{(2\alpha-2s)_+}.
  \]
  For the off-diagonal aliases, the same residue property yields
  \[
    (1+\absx{r})^{4\alpha}\lambda_\ell
    \le
    C d^{4\alpha}(1+\absx{\ell})^{-2\alpha-2s}(1+\absx{\ell})^{2s}
    \le
    C d^{(2\alpha-2s)_+}(1+\absx{\ell})^{2s}.
  \]
  Summing again over \(\absx{r}\le d\), using the same bounded multiplicity from the congruence formulas,
  and applying the Sobolev constraint gives the second inequality.
\end{proof}

\begin{lemma}[Variance of the regression coefficients]
  \label{lem:cdeq-gamma-variance}
  For every \(\absx{r}\le K_m\),
  \[
    \Var(\hat{\gamma}_r)
    \le
    \frac{C}{n}\left(\bar{\lambda}_r+\frac{1}{m}\right),
  \]
  where \(C\) depends only on
  \(\alpha,s,R_0,c_\lambda,C_\lambda,\sigma_\varepsilon,\sigma_\delta\).
\end{lemma}

\begin{proof}
  The proof is identical to \cref{lem:cd-var-gamma}, with \(n\) replaced by
  \(n_\gamma \asymp n\) and \(\widetilde\gamma_r\) in place of \(\gamma_r\),
  followed by the comparison \(\widetilde\lambda_r \asymp\bar{\lambda}_r\) from
  \cref{eq:cdeq-alias-variance-equivalence}.
\end{proof}

\begin{lemma}[Directional common block score tails]
  \label{lem:cdeq-block-direction-tail}
  Fix an active block \(B_\ell\), \(0\le\ell\le L_m\), and let \(u=(u_r)_{r\in B_\ell}\in\R^{B_\ell}\) satisfy
  \(\sum_{r\in B_\ell} u_r^2=1\). Define
  \[
    G_{i,\ell}(u)
    \coloneqq
    \sum_{r\in B_\ell} u_r \frac{\bar{Z}_{ir}}{\sqrt{\bar{\lambda}_r}},
    \qquad
    \eta_{i,\ell}(u)
    \coloneqq
    Y_i G_{i,\ell}(u)-\E\bigl[Y_i G_{i,\ell}(u)\bigr].
  \]
  Then, uniformly over \(u\),
  \begin{equation}
    \norm{\eta_{i,\ell}(u)}_{\psi_1}
    \le
    C\left(
       1+\frac{1}{\sqrt{m\bar{\lambda}_{2^\ell}}}
    \right),
    \label{eq:cdeq-block-direction-psi1}
  \end{equation}
  and
  \begin{equation}
    \E\absx{\eta_{i,\ell}(u)}^2
    \le
    C\left(
       1+\frac{1}{m\bar{\lambda}_{2^\ell}}
    \right),
    \label{eq:cdeq-block-direction-variance}
  \end{equation}
  where the constants depend only on
  \((\underline\alpha,\bar{\alpha},R_0,c_\lambda,C_\lambda,\sigma_\varepsilon,\sigma_\delta)\).
\end{lemma}

\begin{proof}
  The vector
  \(
  (\widetilde{x}_{ir},\bar{\delta}_{ir})_{r\in B_\ell}
  \)
  is centered jointly Gaussian, so \(G_{i,\ell}(u)\) is centered Gaussian.
  Since \(B_\ell \subset\{\absx{r}\le K_m\}\) and \(K_m<m/4\), the covariance identities in \cref{eq:grid-x-covariance,eq:grid-noise-covariance} apply on \(B_\ell\).
  Therefore
  \[
    \E\absx{G_{i,\ell}(u)}^2
    =
    \sum_{r\in B_\ell} u_r^2
    \left(
      \frac{\widetilde\lambda_r}{\bar{\lambda}_r}
      +
      \frac{\sigma_\delta^2}{m\bar{\lambda}_r}
    \right).
  \]
  On \(B_\ell\), \cref{eq:cdeq-alias-variance-equivalence} and the dyadic
  comparability of \((\lambda_r)_{r\in\bbZ}\) yields
  \(\bar{\lambda}_r \asymp \bar{\lambda}_{2^\ell}\), hence
  \[
    \E\absx{G_{i,\ell}(u)}^2
    \le
    C\left(
       1+\frac{1}{m\bar{\lambda}_{2^\ell}}
    \right).
  \]
  Therefore
  \[
    \norm{G_{i,\ell}(u)}_{\psi_2}
    \le
    C\left(
       1+\frac{1}{\sqrt{m\bar{\lambda}_{2^\ell}}}
    \right).
  \]
  Also \(Y_i\) is sub-Gaussian with \(\norm{Y_i}_{\psi_2} \le C\), so the product
  rule for Orlicz norms gives
  \[
    \norm{Y_i G_{i,\ell}(u)}_{\psi_1}
    \le
    C\norm{Y_i}_{\psi_2} \norm{G_{i,\ell}(u)}_{\psi_2}
    \le
    C\left(
       1+\frac{1}{\sqrt{m\bar{\lambda}_{2^\ell}}}
    \right).
  \]
  Centering changes the \(\psi_1\) norm by at most a factor \(2\), proving
  \cref{eq:cdeq-block-direction-psi1}.

  Finally, \(\E Y_i^4 \le C\) and \(G_{i,\ell}(u)\) is Gaussian, so
  \[
    \E\absx{\eta_{i,\ell}(u)}^2
    \le
    \E\absx{Y_i G_{i,\ell}(u)}^2
    \le
    (\E Y_i^4)^{1/2}\left(\E\absx{G_{i,\ell}(u)}^4 \right)^{1/2}
    \le
    C\left(
       1+\frac{1}{m\bar{\lambda}_{2^\ell}}
    \right),
  \]
  which proves \cref{eq:cdeq-block-direction-variance}.
\end{proof}

\begin{lemma}[Truncated common block noise second moment]
  \label{lem:cdeq-block-noise-truncated}
  Fix \(a>0\). After choosing \(C_V\) sufficiently large in
  \cref{eq:cdeq-block-variance}, there exist constants \(C,c>0\) such that, for
  every active block \(B_\ell\),
  \begin{equation}
    \E\left[
        \norm{\Delta_\ell}_2^2
        \mathbf{1}\{\norm{\Delta_\ell}_2^2>aV_\ell\}
    \right]
    \le
    C V_\ell e^{-c\absx{B_\ell}}.
    \label{eq:cdeq-block-noise-truncated}
  \end{equation}
\end{lemma}

\begin{proof}
  Let \(p_\ell=\absx{B_\ell}\) and
  \[
    \bar{w}_{\ell,0}
    \coloneqq
    \frac{1}{n}\left(
                 1+\frac{1}{m\bar{\lambda}_{2^\ell}}
    \right),
    \qquad
    V_{\ell,0}
    \coloneqq
    \frac{1}{n}
    \sum_{r\in B_\ell}
    \left(1+\frac{1}{m\bar{\lambda}_r}\right).
  \]
  Then \(V_\ell=C_V V_{\ell,0}\) and
  \(V_{\ell,0} \asymp p_\ell \bar{w}_{\ell,0}\).
  For \(i\in I_\gamma\), define the block vector
  \[
    \eta_{i,\ell}
    \coloneqq
    \left(
      \frac{Y_i \bar{Z}_{ir}-\widetilde\gamma_r}{\sqrt{\bar{\lambda}_r}}
    \right)_{r\in B_\ell}
    \in\R^{p_\ell}.
  \]
  Then
  \[
    \Delta_\ell
    =
    \frac{1}{n_\gamma}\sum_{i\in I_\gamma} \eta_{i,\ell}.
  \]
  For every unit vector \(u\in\bbS^{p_\ell-1}\),
  \[
    \angx{u,\eta_{i,\ell}}
    =
    \sum_{r\in B_\ell} u_r
    \frac{Y_i \bar{Z}_{ir}-\widetilde\gamma_r}{\sqrt{\bar{\lambda}_r}}.
  \]
  Hence \cref{lem:cdeq-block-direction-tail} gives
  \[
    \norm{\angx{u,\eta_{i,\ell}}}_{\psi_1}
    \le
    C\sqrt{n\bar{w}_{\ell,0}}.
  \]
  Applying \cref{lem:upper-vector-bernstein-psi-one} with \(p=p_\ell\) and
  \(K\asymp \sqrt{n\bar{w}_{\ell,0}}\), and using \(n_\gamma \asymp n\),
  gives constants \(A,c>0\) such that for every \(x\ge1\),
  \[
    \Pr\left\{
         \norm{\Delta_\ell}_2
      >
      A\left[
         \sqrt{\bar{w}_{\ell,0}(p_\ell+x)}
         +
         \sqrt{n\bar{w}_{\ell,0}}\frac{p_\ell+x}{n}
      \right]
    \right\}
    \le
    Ce^{-cx}.
  \]
  Since \(p_\ell \le K_m \le n\), after increasing \(A\) by an absolute factor the
  preceding display implies
  \[
    \Pr\left\{
         \norm{\Delta_\ell}_2^2
      >
      A^2 \bar{w}_{\ell,0}
         \left(x+\frac{x^2}{n}\right)
    \right\}
    \le
    Ce^{-cx},
    \qquad x\ge p_\ell.
  \]
  Since
  \(V_\ell=C_V V_{\ell,0} \asymp C_V p_\ell \bar{w}_{\ell,0}\), choosing \(C_V\)
  sufficiently large, depending on \(a\), ensures
  \[
    A^2 \bar{w}_{\ell,0}
    \left(p_\ell+\frac{p_\ell^2}{n}\right)
    \le aV_\ell,
    \qquad
    p_\ell \bar{w}_{\ell,0} \le CV_\ell.
  \]
  Applying \cref{lem:upper-bernstein-tail-integration} with
  \[
    Z=\norm{\Delta_\ell}_2^2,\qquad
    p=p_\ell,\qquad
    b=\bar{w}_{\ell,0},\qquad
    M=\frac{1}{n},\qquad
    q=2,\qquad
    V=V_\ell,
  \]
  and with \(D\) an absolute constant, proves
  \cref{eq:cdeq-block-noise-truncated}.
\end{proof}

\begin{lemma}[Saturated oracle cutoff is above the pilot floor]
  \label{lem:cdeq-oracle-band}
  Let
  \[
    d_*(\alpha,s)
    \coloneqq
    \begin{cases}
      K_m, & m\le n^{1/(4\alpha+2s)},\\[1mm]
      \lfloor c_1(nm)^{1/\kappa}\rfloor,
      & n^{1/(4\alpha+2s)} \le m\le n^{2\alpha/(\nu+1)},\\[1mm]
      \lfloor c_2 n^{1/(\nu+1)} \rfloor,
      & m\ge n^{2\alpha/(\nu+1)},
    \end{cases}
  \]
  where \(c_1,c_2>0\) are sufficiently small constants depending only on
  \((\underline\alpha,\bar{\alpha},\underline s,\bar{s},c_\lambda,C_\lambda)\). Then
  there exist constants \(C,n_0,m_0<\infty\), depending only on
  \((\underline\alpha,\bar{\alpha},\underline s,\bar{s},c_\lambda,C_\lambda)\), such
  that, for all \(n\ge n_0\) and \(m\ge m_0\), we have \(d_*(\alpha,s)\le K_m\)
  and
  \begin{equation}
    \frac{d_*}{n}
    +
    \frac{d_*^{2\alpha+1}}{nm}
    +
    d_*^{-\nu}
    \le
    C\max\left\{
           n^{-\nu/(\nu+1)},
      (nm)^{-\nu/\kappa},
      m^{-\nu}
    \right\},
    \label{eq:cdeq-oracle-band-rate}
  \end{equation}
  Consequently, if \(L_*\) is chosen so that
  \begin{equation}
    2^{L_*-1}<d_*\le 2^{L_*},
    \label{eq:cdeq-oracle-level}
  \end{equation}
  then every active block \(B_\ell\) with \(0\le \ell\le L_*\) is eligible on
  \(\caE_n^\lambda\) for all \(n\ge n_0\) and \(m\ge m_0\).
\end{lemma}

\begin{proof}
  Throughout this proof, \(K_m \le n\), hence
  \(\ell_{n,m}^{\lambda} \le \log(n(n+1))=O(\log n)\).
  In the floor regime \(m\le n^{1/(4\alpha+2s)}\), we have
  \(d_*=K_m \asymp m\) because then \(m\le n\) and
  \(K_m \asymp m\). Therefore
  \[
    \frac{d_*}{n}
    +
    \frac{d_*^{2\alpha+1}}{nm}
    +
    d_*^{-(2\alpha+2s)}
    \lesssim
    \frac{m}{n}
    +
    \frac{m^{2\alpha}}{n}
    +
    m^{-(2\alpha+2s)}
    \lesssim
    m^{-(2\alpha+2s)},
  \]
  because \(n\ge m^{4\alpha+2s}\) in this regime. Also
  \[
    \frac{\lambda_{d_*}}{\zeta_{n,m}}
    \gtrsim
    \frac{m^{-2\alpha}}{
      m^{-1} \sqrt{\ell_{n,m}^{\lambda}/n}
      +
      \ell_{n,m}^{\lambda}/n
    }
    \to\infty.
  \]
  In the sparse regime
  \(n^{1/(4\alpha+2s)} \le m\le n^{2\alpha/(\nu+1)}\), the choice
  \(d_*=\lfloor c_1(nm)^{1/\kappa}\rfloor\) gives
  \cref{eq:cdeq-oracle-band-rate} after adjusting the constant \(c_1\), and
  \[
    \frac{\lambda_{d_*}}{\zeta_{n,m}}
    \gtrsim
    \frac{(nm)^{-2\alpha/\kappa}}{
      m^{-1} \sqrt{\ell_{n,m}^{\lambda}/n}
      +
      \ell_{n,m}^{\lambda}/n
    }
    \to\infty.
  \]
  In the dense regime \(m\ge n^{2\alpha/(\nu+1)}\), the choice
  \(d_*=\lfloor c_2 n^{1/(\nu+1)} \rfloor\) gives
  \cref{eq:cdeq-oracle-band-rate} after adjusting \(c_2\), and
  \[
    \frac{\lambda_{d_*}}{\zeta_{n,m}}
    \gtrsim
    \frac{n^{-2\alpha/(\nu+1)}}{
      m^{-1} \sqrt{\ell_{n,m}^{\lambda}/n}
      +
      \ell_{n,m}^{\lambda}/n
    }
    \to\infty.
  \]
  By taking \(c_1,c_2\) sufficiently small, the statistical branches also satisfy
  \(d_*\le K_m\) whenever \(m\ge m_{\mathrm{ad}}\). For \(r\in B_\ell\) with
  \(\ell\le L_*\), we have
  \(\absx{r}\le 2^{L_*} \le 2d_*\le 2K_m<m/4\), hence
  \[
    \bar{\lambda}_r
    \gtrsim
    \lambda_r
    \gtrsim
    \lambda_{2d_*}
    \asymp
    \lambda_{d_*}.
  \]
  The three displayed lower bounds for \(\lambda_{d_*}/\zeta_{n,m}\) diverge in
  their respective regimes. Therefore, after enlarging \(n_0\) and \(m_0\) if
  necessary, we may ensure \(m_0 \ge m_{\mathrm{ad}}\) and
  \(\bar{\lambda}_r \ge 4\zeta_{n,m}\) for all \(n\ge n_0\) and \(m\ge m_0\),
  and on
  \(\caE_n^\lambda\) we get
  \[
    \check{\lambda}_r
    \ge
    \bar{\lambda}_r-\frac{1}{2}\bar{\lambda}_r
    \ge
    2\zeta_{n,m}.
  \]
  Hence each block \(B_\ell\), \(0\le \ell\le L_*\), is eligible for all
  \(n\ge n_0\) and \(m\ge m_0\).
\end{proof}

\subsection{Eligible-Block Risk}
\label{subsec:cdeq-eligible-block-risk}

\begin{lemma}[Eligible-block risk decomposition]
  \label{lem:cdeq-eligible-block-risk}
  For every active block \(B_\ell\), the localized risk
  \[
    R_\ell(\caE_n^\lambda\cap\caEelig{\ell})
    \coloneqq
    \sum_{r\in B_\ell}
    \lambda_r
    \E\bigl[
      \absx{\hat{\theta}_r-\vartheta_r}^2
      \mathbf{1}_{\caE_n^\lambda\cap\caEelig{\ell}}
      \bigr]
  \]
  satisfies
  \begin{equation}
    R_\ell(\caE_n^\lambda\cap\caEelig{\ell})
    \le
    C\min\{\widetilde\Theta_\ell,V_\ell\}
    +
    C V_\ell e^{-c\absx{B_\ell}}
    +
    B_\ell^\lambda,
    \label{eq:cdeq-eligible-block-risk}
  \end{equation}
  where
  \begin{equation}
    B_\ell^\lambda
    \coloneqq
    C\sum_{r\in B_\ell}
    \absx{\vartheta_r}^2
    \E\left[
        \frac{(\check{\lambda}_r-\bar{\lambda}_r)^2}
        {\bar{\lambda}_r}
        \mathbf{1}_{\caE_n^\lambda\cap\caEelig{\ell}}
    \right].
    \label{eq:cdeq-plugin-remainder}
  \end{equation}
\end{lemma}

\begin{proof}
  On \(\caE_n^\lambda\cap\caEelig{\ell}\), the event \(\caEkeep{\ell}\) coincides with
  \(\caEthr{\ell}\). For \(r\in B_\ell\), on this localized event,
  \[
    \hat{\theta}_r-\vartheta_r
    =
    \mathbf{1}_{\caEthr{\ell}}
    \frac{\hat{\gamma}_r-\widetilde\gamma_r}{\check{\lambda}_r}
    +
    \mathbf{1}_{\caEthr{\ell}}
    \vartheta_r \frac{\bar{\lambda}_r-\check{\lambda}_r}{\check{\lambda}_r}
    -
    \mathbf{1}_{(\caEthr{\ell})^c}\vartheta_r.
  \]
  Using \((a+b+c)^2\le 3(a^2+b^2+c^2)\), we obtain
  \[
    R_\ell(\caE_n^\lambda\cap\caEelig{\ell})\le T_{1\ell}+T_{2\ell}+T_{3\ell},
  \]
  where
  \begin{align*}
    T_{1\ell}
    &\coloneqq
    C\sum_{r\in B_\ell} \lambda_r
    \E\left[
        \mathbf{1}_{\caEthr{\ell}}
        \frac{\absx{\hat{\gamma}_r-\widetilde\gamma_r}^2}{\check{\lambda}_r^2}
        \mathbf{1}_{\caE_n^\lambda\cap\caEelig{\ell}}
    \right],\\
    T_{2\ell}
    &\coloneqq
    C\sum_{r\in B_\ell} \lambda_r \absx{\vartheta_r}^2
    \E\left[
        \mathbf{1}_{\caEthr{\ell}}
        \frac{(\bar{\lambda}_r-\check{\lambda}_r)^2}{\check{\lambda}_r^2}
        \mathbf{1}_{\caE_n^\lambda\cap\caEelig{\ell}}
    \right],\\
    T_{3\ell}
    &\coloneqq
    C\widetilde\Theta_\ell
    \Pr\bigl((\caEthr{\ell})^c \cap \caE_n^\lambda\cap\caEelig{\ell} \bigr).
  \end{align*}

  On \(\caE_n^\lambda\cap\caEelig{\ell}\), \cref{lem:cdeq-eligible-blocks} yields
  \begin{equation}
    \frac{2}{3} S_\ell^\circ
    \le
    \hat{S}_\ell
    \le
    2S_\ell^\circ,
    \qquad
    \frac{2}{3} V_\ell
    \le
    \hat{V}_\ell
    \le
    2V_\ell,
    \qquad
    \check{\lambda}_r \asymp \bar{\lambda}_r.
    \label{eq:cdeq-adaptive-comparability}
  \end{equation}
  The definitions in \cref{subsec:cdeq-oracle} give
  \[
    S_\ell^\circ=\norm{\mu_\ell+\Delta_\ell}_2^2,
    \qquad
    \norm{\mu_\ell}_2^2=\widetilde\Theta_\ell.
  \]
  Also,
  \[
    \E\norm{\Delta_\ell}_2^2
    =
    \sum_{r\in B_\ell}
    \frac{\Var(\hat{\gamma}_r)}{\bar{\lambda}_r}
    \le
    \frac{C}{n}\sum_{r\in B_\ell}
    \left(
      1+\frac{1}{m\bar{\lambda}_r}
    \right)
    \le
    C V_\ell,
  \]
  by \cref{lem:cdeq-gamma-variance,eq:cdeq-alias-variance-equivalence}.
  Also, on \(\caE_n^\lambda\cap\caEelig{\ell}\), the active-band comparison gives
  \(\lambda_r/\check{\lambda}_r^2 \le C/\bar{\lambda}_r\). Therefore
  \[
    T_{1\ell}
    \le
    C\E\left[
         \norm{\Delta_\ell}_2^2
         \mathbf{1}_{\caEthr{\ell}}
         \mathbf{1}_{\caE_n^\lambda\cap\caEelig{\ell}}
    \right]
    \le
    C\E\norm{\Delta_\ell}_2^2
    \le
    C V_\ell
  \]
  Hence
  \begin{equation}
    R_\ell(\caE_n^\lambda\cap\caEelig{\ell})
    \le
    T_{1\ell}
    +
    C\widetilde\Theta_\ell \Pr((\caEthr{\ell})^c \cap \caE_n^\lambda\cap\caEelig{\ell})
    +
    T_{2\ell}.
    \label{eq:cdeq-block-basic-risk}
  \end{equation}

  Consider three regimes.

  \textbf{Weak block: \(\widetilde\Theta_\ell \le \gamma_- V_\ell\).}
  If \(\caEthr{\ell}\), \(\caE_n^\lambda\), and \(\caEelig{\ell}\) all occur, then
  \cref{eq:cdeq-adaptive-comparability} implies
  \[
    S_\ell^\circ \ge \frac{1}{3}V_\ell.
  \]
  On the other hand, if \(\norm{\Delta_\ell}_2^2 \le C_-V_\ell\), then
  the block identity above gives
  \[
    S_\ell^\circ
    =
    \norm{\mu_\ell+\Delta_\ell}_2^2
    \le
    2\widetilde\Theta_\ell+2\norm{\Delta_\ell}_2^2
    \le
    (2\gamma_-+2C_-)V_\ell.
  \]
  Choose \(C_->0\) and then \(\gamma_->0\) so that
  \(2\gamma_-+2C_-\le1/3\).
  Thus \(\caEthr{\ell}\cap\caE_n^\lambda\cap\caEelig{\ell}\) implies
  \(\norm{\Delta_\ell}_2^2>C_-V_\ell\). By
  \cref{lem:cdeq-block-noise-truncated},
  \[
    T_{1\ell}
    \le
    C\E\left[
         \norm{\Delta_\ell}_2^2
         \mathbf{1}\{\norm{\Delta_\ell}_2^2>C_-V_\ell\}
    \right]
    \le
    C V_\ell e^{-c\absx{B_\ell}}.
  \]
  Since \(T_{3\ell} \le C\widetilde\Theta_\ell\), we obtain
  \begin{equation}
    R_\ell(\caE_n^\lambda\cap\caEelig{\ell})
    \le
    C\widetilde\Theta_\ell
    +
    C V_\ell e^{-c\absx{B_\ell}}
    +
    T_{2\ell}.
    \label{eq:cdeq-block-weak}
  \end{equation}

  \textbf{Strong block: \(\widetilde\Theta_\ell \ge \gamma_+ V_\ell\).}
  Choose \(\gamma_+\ge12\).
  If \(\widetilde\Theta_\ell \ge \gamma_+ V_\ell\), then on the event
  \(\norm{\Delta_\ell}_2^2 \le \widetilde\Theta_\ell/4\), the block identity gives
  \[
    S_\ell^\circ
    =
    \norm{\mu_\ell+\Delta_\ell}_2^2
    \ge
    \left(
      \norm{\mu_\ell}_2-\norm{\Delta_\ell}_2
    \right)^2
    \ge
    \frac{1}{4}\widetilde\Theta_\ell.
  \]
  Hence \cref{eq:cdeq-adaptive-comparability} implies
  \[
    \hat{S}_\ell
    \ge
    \frac{2}{3} S_\ell^\circ
    \ge
    \frac{\widetilde\Theta_\ell}{6}
    \ge
    2V_\ell
    \ge
    \hat{V}_\ell.
  \]
  Thus \((\caEthr{\ell})^c \cap \caE_n^\lambda\cap\caEelig{\ell}\) can only occur if
  \(\norm{\Delta_\ell}_2^2>\widetilde\Theta_\ell/4\).
  On this event, \(\widetilde\Theta_\ell \le4\norm{\Delta_\ell}_2^2\).
  Since \(\widetilde\Theta_\ell/4\ge(\gamma_+/4)V_\ell\) in the strong regime,
  \cref{lem:cdeq-block-noise-truncated} with \(a=\gamma_+/4\) yields
  \[
    \begin{aligned}
      \widetilde\Theta_\ell
      &\Pr\bigl((\caEthr{\ell})^c \cap \caE_n^\lambda\cap\caEelig{\ell} \bigr)\\
      &\le
      C\E\left[
           \norm{\Delta_\ell}_2^2
           \mathbf{1}\{\norm{\Delta_\ell}_2^2>\widetilde\Theta_\ell/4\}
      \right]
      \le
      C\E\left[
           \norm{\Delta_\ell}_2^2
           \mathbf{1}\{\norm{\Delta_\ell}_2^2>(\gamma_+/4)V_\ell\}
      \right]
      \le
      C V_\ell e^{-c\absx{B_\ell}}.
    \end{aligned}
  \]
  Using the bound \(T_{1\ell} \le C V_\ell\) above, we get
  \begin{equation}
    R_\ell(\caE_n^\lambda\cap\caEelig{\ell})
    \le
    C V_\ell
    +
    C V_\ell e^{-c\absx{B_\ell}}
    +
    T_{2\ell}.
    \label{eq:cdeq-block-strong}
  \end{equation}

  \textbf{Intermediate block:
    \(\gamma_-V_\ell<\widetilde\Theta_\ell<\gamma_+V_\ell\).}
  Then \(\widetilde\Theta_\ell \asymp V_\ell\), so
  \cref{eq:cdeq-block-basic-risk} immediately gives
  \begin{equation}
    R_\ell(\caE_n^\lambda\cap\caEelig{\ell})
    \le
    C\min\{\widetilde\Theta_\ell,V_\ell\}
    +
    T_{2\ell}.
    \label{eq:cdeq-block-middle}
  \end{equation}

  For the plug-in term, on \(\caE_n^\lambda\cap\caEelig{\ell}\), eligibility and
  \cref{lem:cdeq-eligible-blocks} imply
  \[
    \lambda_r
    \absx{
      \widetilde\gamma_r \left(
                           \frac{1}{\check{\lambda}_r}-\frac{1}{\bar{\lambda}_r}
      \right)
    }^2
    \mathbf{1}_{\caE_n^\lambda\cap\caEelig{\ell}}
    \le
    C
    \absx{\vartheta_r}^2
    \frac{(\check{\lambda}_r-\bar{\lambda}_r)^2}
    {\bar{\lambda}_r}
    \mathbf{1}_{\caE_n^\lambda\cap\caEelig{\ell}},
  \]
  because \(\widetilde\gamma_r=\bar{\lambda}_r \vartheta_r\) and
  \(\lambda_r \asymp \bar{\lambda}_r\) on the active band. Summing over the
  block shows \(T_{2\ell} \le B_\ell^\lambda\).

  Combining \cref{eq:cdeq-block-weak,eq:cdeq-block-strong,eq:cdeq-block-middle}
  with the bound on \(T_{2\ell}\) proves
  \cref{eq:cdeq-eligible-block-risk}.
\end{proof}

\begin{lemma}[Bad-event contribution]
  \label{lem:cdeq-bad-event}
  There exist \(C,n_{\mathrm{bad}},m_{\mathrm{bad}}<\infty\), depending only on
  \[
    \underline\alpha,\bar{\alpha},\underline s,\bar{s},
    R_0, R, c_\lambda, C_\lambda, \sigma_\varepsilon, \sigma_\delta,
  \]
  such that, for all \(n\ge n_{\mathrm{bad}}\) and \(m\ge m_{\mathrm{bad}}\),
  \begin{equation}
    \sum_{r\in\bbZ}
    \lambda_r
    \E\bigl[\absx{\hat{\theta}_r-\theta_r}^2 \mathbf{1}_{(\caE_n^\lambda)^c} \bigr]
    \le
    C n^{-7}.
    \label{eq:cdeq-bad-event}
  \end{equation}
\end{lemma}

\begin{proof}
  By the eligibility condition in the definition of retained blocks,
  every nonzero \(\hat{\theta}_r\) has
  \(\check{\lambda}_r\ge 2\zeta_{n,m}\). Hence
  \[
    \sum_{\absx{r}\le K_m}\lambda_r \absx{\hat{\theta}_r}^2
    \le
    C\zeta_{n,m}^{-2} \sum_{\absx{r}\le K_m}\absx{\hat{\gamma}_r}^2.
  \]
  For each \(\absx{r}\le K_m\), the variables \(Y_i\) and \(\bar{Z}_{ir}\) are
  centered jointly Gaussian with
  \[
    \Var(Y_i)
    =
    \sigma_\varepsilon^2+\sum_{k\in\bbZ} \lambda_k \absx{\theta_k}^2
    \le
    C,
    \qquad
    \Var(\bar{Z}_{ir})
    =
    \widetilde\lambda_r+\frac{\sigma_\delta^2}{m}
    \le
    C,
  \]
  where the variance of \(\bar{Z}_{ir}\) follows from \cref{eq:grid-zbar-covariance}.
  The bounds use \(\alpha>\frac{1}{2}\) and \(\theta\in\Theta_s(R_0)\).
  Hence
  \[
    \E\absx{Y_i \bar{Z}_{ir}}^4
    \le
    \bigl(\E\absx{Y_i}^8 \bigr)^{1/2}
    \bigl(\E\absx{\bar{Z}_{ir}}^8 \bigr)^{1/2}
    \le
    C.
  \]
  Since \(x\mapsto x^4\) is convex,
  \[
    \E\absx{\hat{\gamma}_r}^4
    =
    \E\absx{
      \frac{1}{n_\gamma}\sum_{i\in I_\gamma} Y_i \bar{Z}_{ir}
    }^4
    \le
    \frac{1}{n_\gamma}\sum_{i\in I_\gamma} \E\absx{Y_i \bar{Z}_{ir}}^4
    \le
    C.
  \]
  Therefore
  \[
    \E\left[
        \left(
          \sum_{\absx{r}\le K_m}\absx{\hat{\gamma}_r}^2
        \right)^2
    \right]
    \le
    (2K_m+1)\sum_{\absx{r}\le K_m}\E\absx{\hat{\gamma}_r}^4
    \le
    C K_m^2
    \le
    C m^2.
  \]
  Also, since \(n_\lambda \asymp n\) and
  \(\ell_{n,m}^{\lambda} \ge \log 2\),
  \[
    \zeta_{n,m}
    \ge
    C\frac{\ell_{n,m}^{\lambda}}{n_\lambda}
    \ge
    \frac{c}{n},
  \]
  so \(\zeta_{n,m}^{-1} \le Cn\). Hence
  \[
    \E\left[
        \left(
          \sum_{\absx{r}\le K_m}\lambda_r \absx{\hat{\theta}_r}^2
        \right)^2
    \right]
    \le
    C\zeta_{n,m}^{-4}
    \E\left[
        \left(
          \sum_{\absx{r}\le K_m}\absx{\hat{\gamma}_r}^2
        \right)^2
    \right]
    \le
    Cn^4 K_m^2.
  \]
  Applying Cauchy--Schwarz and \cref{eq:cdeq-lambda-good-event-prob}, we get,
  for all \(n\ge n_0\) and \(m\ge m_0\),
  \[
    \E\left[
        \sum_{\absx{r}\le K_m}\lambda_r \absx{\hat{\theta}_r}^2
        \mathbf{1}_{(\caE_n^\lambda)^c}
    \right]
    \le
    Cn^2 K_m \bigl[n(K_m+1)\bigr]^{-10}
    \le
    Cn^{-8}.
  \]
  The deterministic term
  \[
    \sum_{r\in\bbZ} \lambda_r \absx{\theta_r}^2
    \le
    C
  \]
  is uniformly bounded, so
  \[
    \E\left[
        \sum_{r\in\bbZ} \lambda_r \absx{\theta_r}^2
        \mathbf{1}_{(\caE_n^\lambda)^c}
    \right]
    \le
    C \Pr\bigl((\caE_n^\lambda)^c\bigr)
    \le
    C n^{-20}
    \le
    C n^{-7}.
  \]
  Since \(\hat{\theta}_r=0\) for \(\absx{r}>K_m\), we have
  \[
    \sum_{r\in\bbZ}
    \lambda_r
    \E\bigl[\absx{\hat{\theta}_r-\theta_r}^2 \mathbf{1}_{(\caE_n^\lambda)^c} \bigr]
    \le
    2\E\left[
         \sum_{\absx{r}\le K_m}\lambda_r \absx{\hat{\theta}_r}^2
         \mathbf{1}_{(\caE_n^\lambda)^c}
    \right]
    +
    2\E\left[
         \sum_{r\in\bbZ} \lambda_r \absx{\theta_r}^2
         \mathbf{1}_{(\caE_n^\lambda)^c}
    \right]
    \le
    C n^{-7},
  \]
  which proves \cref{eq:cdeq-bad-event}.
\end{proof}

\subsection[Proof of the Adaptive Common-Design Upper Bound]{Proof of \texorpdfstring{\cref{thm:common-adaptive-upper}}{the adaptive common-design upper bound}}
\label{subsec:cdeq-main-aggregation}

It is enough to prove the bound uniformly over fixed rectangles of the form above.
For \((\alpha,s)\) in this rectangle, write
\[
  \rho_{n,m}(\alpha,s)
  \coloneqq
  n^{-\nu/(\nu+1)}
  +
  (nm)^{-\nu/\kappa}
  +
  m^{-\nu}
  +
  m^{-4\alpha}.
\]
Fix such \((\alpha,s)\), \(\lambda\in\caL_\alpha(c_\lambda,C_\lambda)\), and
\(\theta\in\Theta_s(R_0)\).
For an event \(\caA\), write
\[
  \caR_0(\theta,\lambda;\caA)
  \coloneqq
  \sum_{r\in\bbZ}
  \lambda_r
  \E\bigl[
    \absx{\hat{\theta}_r-\theta_r}^2\mathbf{1}_{\caA}
    \bigr],
\]
for the thresholded coordinates, and define
\(\caR(\hat{\beta};\theta,\lambda;\caA)\) analogously using the projected
coordinates \(\bar{\theta}_r\).

\textbf{Step 1: oracle cutoff and risk split.}
Choose \(m_0\) large enough that \(m_0 \ge m_{\mathrm{ad}}\) and \(m_0 \ge m_{\mathrm{bad}}\).
Choose \(n_0\) large enough that \(n_0 \ge n_{\mathrm{bad}}\) and that the conclusions of \cref{lem:cdeq-oracle-band} and all asymptotic comparisons below hold for \(n\ge n_0\) and \(m\ge m_0\).
Fix \(n\ge n_0\) and \(m\ge m_0\).
Let \(d_*=d_*(\alpha,s)\) and \(L_*\) be given by \cref{lem:cdeq-oracle-band,eq:cdeq-oracle-level}.
By \cref{lem:cdeq-oracle-band}, every block \(B_0,\dots,B_{L_*}\) is eligible on the pilot event \(\caE_n^\lambda\).
We split the risk into its contributions on \(\caE_n^\lambda\) and
\((\caE_n^\lambda)^c\).
Steps 2--6 bound the good-event contribution; Step 7 controls the
complementary bad-event contribution.

\textbf{Step 2: good-event block reduction.}
The blockwise projection argument is the same as in Step 2 of
\cref{subsec:proof-indep-adaptive-upper}.
Since \(R\ge R_0\), the block \(\theta_{B_\ell}\) belongs to the projection ball.
Since \(\lambda_r\) is comparable within each dyadic block,
\begin{equation}
  \caR(\hat{\beta};\theta,\lambda;\caA)
  \le
  C\caR_0(\theta,\lambda;\caA)
  \label{eq:cdeq-projection-reduction}
\end{equation}
for every event \(\caA\).
The common-design estimator targets
\(\vartheta_r=\widetilde\gamma_r/\bar{\lambda}_r\), so
\cref{eq:cdeq-observable-parameter} gives
\[
  \hat{\theta}_r-\theta_r
  =
  (\hat{\theta}_r-\vartheta_r)+(\vartheta_r-\theta_r),
  \qquad
  \absx{r}\le K_m,
\]
and \(\hat{\theta}_r=0\) for \(\absx{r}>K_m\).
For the active-band estimation part and any event \(\caA\), write
\[
  R_\ell(\caA)
  \coloneqq
  \sum_{r\in B_\ell}
  \lambda_r
  \E\bigl[\absx{\hat{\theta}_r-\vartheta_r}^2 \mathbf{1}_{\caA} \bigr].
\]
The reduction over retained and discarded blocks follows the
independent-design proof, with
\((\bar{\lambda}_r,\vartheta_r,\widetilde\Theta_\ell)\) replacing
\((\lambda_r,\theta_r,\Theta_\ell)\) in the block quantities.
Namely, for \(0\le\ell\le L_*\), \cref{lem:cdeq-oracle-band} gives
\(R_\ell(\caE_n^\lambda)=R_\ell(\caE_n^\lambda\cap\caEelig{\ell})\), so
\cref{lem:cdeq-eligible-block-risk} applies. For \(L_*<\ell\le L_m\), decompose
\(R_\ell(\caE_n^\lambda)\) over
\(\caE_n^\lambda\cap\caEelig{\ell}\) and
\(\caE_n^\lambda\cap(\caEelig{\ell})^c\). On ineligible blocks the estimator is
zero, so the active-band comparison bounds that contribution by
\(C\widetilde\Theta_\ell\); on eligible blocks
\(B_\ell^\lambda\le C\widetilde\Theta_\ell\) by \cref{lem:cdeq-eligible-blocks}.
Thus
\begin{equation}
  \begin{aligned}
    \caR_0(\theta,\lambda;\caE_n^\lambda)
    \le
    &C\underbrace{\sum_{\ell=0}^{L_*} V_\ell}_{\text{oracle variance}}
    +
    C\underbrace{\sum_{\ell=L_*+1}^{L_m}
      \widetilde\Theta_\ell}_{\text{observable signal tail}}\\
    &+
    C\underbrace{\sum_{\ell=0}^{L_*}
      B_\ell^\lambda}_{\text{plug-in remainder}}
    +
    C\underbrace{\sum_{\ell=0}^{L_m}
      V_\ell e^{-c\absx{B_\ell}}}_{\text{thresholding remainder}}\\
    &+
    C\underbrace{\sum_{\absx{r}\le K_m}
      \lambda_r\absx{\vartheta_r-\theta_r}^2}_{\text{observable-target bias}}
    +
    C\underbrace{\sum_{\absx{r}>K_m}
      \lambda_r\absx{\theta_r}^2}_{\text{tail beyond }K_m}.
  \end{aligned}
  \label{eq:cdeq-good-risk}
\end{equation}

\textbf{Step 3: oracle and fixed-grid terms.}
By \cref{lem:cdeq-oracle-band} and \cref{eq:cdeq-alias-variance-equivalence},
\[
  \sum_{\ell=0}^{L_*} V_\ell
  \asymp
  \frac{d_*}{n}
  +
  \frac{d_*^{2\alpha+1}}{nm}
  \le
  C\rho_{n,m}(\alpha,s).
\]
The observable signal tail beyond the oracle cutoff contains the Sobolev tail
of \(g\) and the common-grid aliasing term.
Writing \(a_r=\widetilde\gamma_r-\gamma_r\) and using
\cref{eq:upper-two-sided-tail-bias,lem:cd-aliasing-bias},
\begin{align*}
  \sum_{\ell=L_*+1}^{L_m} \widetilde\Theta_\ell
  &\le
  C\sum_{\absx{r}>d_*}\frac{\absx{\gamma_r}^2}{\lambda_r}
  +
  C\sum_{\absx{r}\le K_m}\frac{\absx{a_r}^2}{\lambda_r}
  \\
  &\le
  C d_*^{-(2\alpha+2s)}
  +
  C m^{-(2\alpha+2s)}
  =
  C d_*^{-\nu}+C m^{-\nu}.
\end{align*}
Thus this term is \(O(\rho_{n,m}(\alpha,s))\).
The observable-target bias is controlled by
\cref{lem:cdeq-observable-target-bias} with \(d=K_m\):
\[
  \sum_{\absx{r}\le K_m}
  \lambda_r\absx{\vartheta_r-\theta_r}^2
  \le
  C\left(
     m^{-\nu}
     +
     m^{-4\alpha}K_m^{(2\alpha-2s)_+}
  \right)
  \le
  C\rho_{n,m}(\alpha,s),
\]
since \(K_m\le m\).
Finally,
\[
  \sum_{\absx{r}>K_m}\lambda_r \absx{\theta_r}^2
  \le
  C K_m^{-(2\alpha+2s)}
\]
by \cref{eq:upper-two-sided-tail-bias}.
Since \(d_*\le K_m\) by \cref{lem:cdeq-oracle-band}, this tail is at most
\(Cd_*^{-\nu}\), hence is also \(O(\rho_{n,m}(\alpha,s))\).

\textbf{Step 4: thresholding remainder.}
The thresholding remainder is controlled exactly as in Step 4 of
\cref{subsec:proof-indep-adaptive-upper}.
Using \(\bar{\lambda}_r\asymp\lambda_r\) on the active band and \(K_m\le n\),
\[
  \sum_{\ell=0}^{L_m}V_\ell e^{-c\absx{B_\ell}}
  \lesssim
  \frac{1}{n} + \frac{1}{nm}
  \le C\rho_{n,m}(\alpha,s).
\]

\textbf{Step 5: plug-in remainder.}
By \cref{eq:cdeq-lambda-pilot-mse}, \cref{eq:cdeq-alias-variance-equivalence}, and the definition \(\widetilde\gamma_r=\bar{\lambda}_r \vartheta_r\),
\[
  \bar{d}_*
  \coloneqq
  2^{L_*} \wedge K_m
  \le
  2d_*,
\]
and
\[
  \sum_{\ell=0}^{L_*} B_\ell^\lambda
  \le
  \frac{C}{n}
  \sum_{\absx{r}\le \bar{d}_*}
  \left(
    \bar{\lambda}_r \absx{\vartheta_r}^2
  +
    \frac{\absx{\vartheta_r}^2}{m}
    +
    \frac{\absx{\vartheta_r}^2}{m^2 \bar{\lambda}_r}
  \right).
\]
Since
\[
  \sum_{\absx{r}\le \bar{d}_*}\bar{\lambda}_r \absx{\vartheta_r}^2
  \le
  C\sum_{\absx{r}\le \bar{d}_*}\lambda_r \absx{\theta_r}^2
  +
  C\sum_{\absx{r}\le \bar{d}_*}\lambda_r \absx{\vartheta_r-\theta_r}^2
  \le
  C,
\]
and \cref{lem:cdeq-observable-sums} yields
\[
  \sum_{\absx{r}\le \bar{d}_*}\absx{\vartheta_r}^2
  \le
  C,
  \qquad
  \sum_{\absx{r}\le \bar{d}_*}\frac{\absx{\vartheta_r}^2}{\bar{\lambda}_r}
  \le
  C \bar{d}_*^{(2\alpha-2s)_+}
  \le
  C d_*^{(2\alpha-2s)_+},
\]
we obtain, with \(a\coloneqq(2\alpha-2s)_+\),
\[
  \sum_{\ell=0}^{L_*} B_\ell^\lambda
  \le
  C\left(
     \frac{1}{n}
     +
     \frac{1}{nm}
     +
     \frac{d_*^a}{nm^2}
  \right).
\]
The last display is absorbed by the same three cutoff regimes in
\cref{lem:cdeq-oracle-band}: in the floor regime it is
\(O(m^{-\nu})\), in the sparse regime it is
\(O((nm)^{-\nu/\kappa})\), and in the dense regime it is
\(O(n^{-\nu/(\nu+1)})\).
Hence
\begin{equation}
  \sum_{\ell=0}^{L_*} B_\ell^\lambda
  \le
  C\rho_{n,m}(\alpha,s).
  \label{eq:cdeq-plugin-bound}
\end{equation}

\textbf{Step 6: good-event bound.}
Combining Steps 3--5 with \cref{eq:cdeq-good-risk} gives
\begin{equation}
  \caR_0(\theta,\lambda;\caE_n^\lambda)
  \le
  C\rho_{n,m}(\alpha,s).
  \label{eq:cdeq-good-event-bound}
\end{equation}
By \cref{eq:cdeq-projection-reduction}, the same bound holds for
\(\hat{\beta}\) on \(\caE_n^\lambda\).

\textbf{Step 7: bad-event contribution.}
This is the same final step as in the independent-design proof, using
\cref{eq:cdeq-projection-reduction,lem:cdeq-bad-event} in place of the
independent-design bad-event lemma.
Namely,
\[
  \caR(\hat{\beta};\theta,\lambda;(\caE_n^\lambda)^c)
  \le
  C n^{-7}.
\]
Since \(\nu/(\nu+1)<1<7\), we have
\(n^{-7} \le n^{-\nu/(\nu+1)} \le \rho_{n,m}(\alpha,s)\).
Combining this with \cref{eq:cdeq-good-event-bound} gives
\(\E\caR(\hat{\beta};\theta,\lambda)\le C\rho_{n,m}(\alpha,s)\), which proves
\cref{thm:common-adaptive-upper}.

\subsection{Even-grid interlaced denominator pilot}
\label{subsec:cdeq-interlaced-pilot}

This subsection provides an alternative denominator pilot that applies only when \(m\) is even.
Within this subsection only, the symbols \(\bar{\lambda}_r\), \(\check{\lambda}_r\), and \(\caE_n^\lambda\) refer to the interlaced definitions below.
For \(m=2M\), split the common grid into
\[
  J_0=\{1,3,\ldots,2M-1\},
  \qquad
  J_1=\{2,4,\ldots,2M\}.
\]
For \(a=0,1\), define
\begin{equation}
  \bar{Z}_{ir}^{(a)}
  \coloneqq
  \frac{1}{M}\sum_{j\in J_a} Z_{ij} e_r(t_j),
  \qquad
  \widetilde{x}_{ir}^{(a)}
  \coloneqq
  \frac{1}{M}\sum_{j\in J_a} X_i(t_j)e_r(t_j).
  \label{eq:cdeq-interlaced-zbar}
\end{equation}
Equivalently,
\[
  \widetilde{x}_{ir}^{(a)}
  =
  \sum_{\ell\in\bbZ} D_{r\ell}^{(a)} x_{i\ell},
  \qquad
  D_{r\ell}^{(a)}
  \coloneqq
  \frac{1}{M}\sum_{j\in J_a} e_\ell(t_j)e_r(t_j).
\]
The denominator target is
\begin{equation}
  \bar{\lambda}_r
  \coloneqq
  \Cov\left(\widetilde{x}_{ir}^{(0)},\widetilde{x}_{ir}^{(1)} \right)
  =
  \sum_{\ell\in\bbZ} D_{r\ell}^{(0)} D_{r\ell}^{(1)} \lambda_\ell .
  \label{eq:cdeq-interlaced-lambda}
\end{equation}
For \(\absx{r}\le K_m\), define
\begin{equation}
  \check{\lambda}_r
  \coloneqq
  \frac{1}{n_\lambda}\sum_{i\in I_\lambda}
  \bar{Z}_{ir}^{(0)} \bar{Z}_{ir}^{(1)}.
  \label{eq:cdeq-interlaced-lambda-pilot}
\end{equation}

\begin{lemma}[Interlaced denominator comparison]
  \label{lem:cdeq-interlaced-denominator}
  For \(m=2M\), choose \(c_L>0\) sufficiently small. Then, uniformly over
  \(\absx{r}\le K_m\),
  \begin{equation}
    \bar{\lambda}_r \asymp\lambda_r \asymp\widetilde\lambda_r.
    \label{eq:cdeq-interlaced-alias-variance-equivalence}
  \end{equation}
  Moreover, for
  \[
    \lambda_r^{(a)}
    \coloneqq
    \Var(\widetilde{x}_{ir}^{(a)})
    =
    \sum_{\ell\in\bbZ}\{D_{r\ell}^{(a)}\}^2 \lambda_\ell,
    \qquad a=0,1,
  \]
  we have \(\lambda_r^{(a)} \asymp\bar{\lambda}_r\) on the same band.
\end{lemma}

\begin{proof}
  On \(J_0\),
  \[
    \frac{1}{M}\sum_{q=0}^{M-1} \exp(2\pi\mathrm{i}kq/M)
    =
    \mathbf{1}\{k\equiv0\pmod M\},
  \]
  while on \(J_1\),
  \[
    \frac{1}{M}\sum_{q=0}^{M-1} \exp(2\pi\mathrm{i}k(q+1/2)/M)
    =
    \exp(\pi\mathrm{i}k/M)\mathbf{1}\{k\equiv0\pmod M\}.
  \]
  Writing the trigonometric basis in terms of complex exponentials, the
  half-grid products are constrained by the same congruences
  \(\ell-r\equiv0\pmod M\) and \(\ell+r\equiv0\pmod M\), with only bounded signs
  and phases on the shifted grid. Thus \(D_{rr}^{(a)}=1\) for \(\absx{r}<M/2\),
  \(\absx{D_{r\ell}^{(a)}}\le C\), and if
  \(\absx{r}\le M/4\), \(D_{r\ell}^{(a)} \ne0\), and \(\ell\ne r\), then
  \(\absx{\ell}\ge cM\). Each fixed coefficient index belongs to only \(O(1)\)
  of these half-grid supports.

  Hence, for \(\absx{r}\le K_m \le c_L m\) and \(c_L<1/8\),
  \[
    \absx{\bar{\lambda}_r-\lambda_r}
    \le
    C\sum_{\ell\ne r:D_{r\ell}^{(0)} D_{r\ell}^{(1)} \ne0}\lambda_\ell
    \le
    C\sum_{a\ge1}(aM)^{-2\alpha}
    \le
    C m^{-2\alpha}.
  \]
  The same argument gives
  \[
    \absx{\lambda_r^{(a)}-\lambda_r}\le Cm^{-2\alpha},
    \qquad a=0,1.
  \]
  Since \(\lambda_r \ge c_\lambda(1+K_m)^{-2\alpha}\), choosing \(c_L\)
  sufficiently small makes the preceding \(Cm^{-2\alpha}\) remainders at most
  \(\lambda_r/2\). Thus \(\bar{\lambda}_r \asymp\lambda_r\) and
  \(\lambda_r^{(a)} \asymp\lambda_r\). The full-grid comparison
  \(\widetilde\lambda_r \asymp\lambda_r\) follows from
  \cref{lem:cd-alias-variance-equivalence}.
\end{proof}

\begin{proposition}[Interlaced denominator pilot]
  \label{prop:cdeq-interlaced-lambda-pilot}
  There exist constants \(C,c>0\), depending only on
  \(\underline\alpha,\bar{\alpha},c_\lambda,C_\lambda,\sigma_\delta\), such that the
  following hold for every \(\absx{r}\le K_m\).
  \begin{itemize}
    \item The mean-square error obeys
    \begin{equation}
      \E\bigl(\check{\lambda}_r-\bar{\lambda}_r \bigr)^2
      \le
      \frac{C}{n_\lambda}
      \left(
        \bar{\lambda}_r^2
      +
        \frac{\bar{\lambda}_r}{m}
        +
        \frac{1}{m^2}
      \right).
      \label{eq:cdeq-interlaced-lambda-pilot-mse}
    \end{equation}
    \item After enlarging \(M_0\) if necessary, the event
    \begin{equation}
      \caE_n^\lambda
      \coloneqq
      \bigcap_{\absx{r}\le K_m}
      \left\{
        \absx{\check{\lambda}_r-\bar{\lambda}_r}
        \le
        \frac{1}{2}\max\{\bar{\lambda}_r,\zeta_{n,m}\}
      \right\}
      \label{eq:cdeq-interlaced-lambda-good-event}
    \end{equation}
    satisfies, for all \(n\ge n_0\) and even \(m\ge m_0\),
    \begin{equation}
      \Pr\bigl((\caE_n^\lambda)^c\bigr)
      \le
      C\bigl[n(K_m+1)\bigr]^{-20},
      \label{eq:cdeq-interlaced-lambda-good-event-prob}
    \end{equation}
    where \(C,n_0,m_0<\infty\) depend only on
    \(\underline\alpha,\bar{\alpha},c_\lambda,C_\lambda,\sigma_\delta\).
  \end{itemize}
\end{proposition}

\begin{proof}
  For \(a=0,1\), write
  \[
    \bar{Z}_{ir}^{(a)}
    =
    \widetilde{x}_{ir}^{(a)}
    +
    \bar{\delta}_{ir}^{(a)},
    \qquad
    \bar{\delta}_{ir}^{(a)}
    =
    \frac{1}{M}\sum_{j\in J_a} \delta_{ij} e_r(t_j).
  \]
  The two half-grid noise averages use disjoint pointwise errors, and the noise
  is independent of \(X_i\). Therefore
  \[
    \E\check{\lambda}_r
    =
    \E\left[\bar{Z}_{ir}^{(0)} \bar{Z}_{ir}^{(1)} \right]
    =
    \bar{\lambda}_r.
  \]

  Set \(A_{ir}=\bar{Z}_{ir}^{(0)}\) and \(B_{ir}=\bar{Z}_{ir}^{(1)}\). The pair
  \((A_{ir},B_{ir})\) is centered jointly Gaussian and
  \(\Cov(A_{ir},B_{ir})=\bar{\lambda}_r\). By
  \cref{lem:cdeq-interlaced-denominator},
  \[
    \Var(\widetilde{x}_{ir}^{(a)})\le C\bar{\lambda}_r,\qquad a=0,1.
  \]
  Also,
  \[
    \Var(\bar{\delta}_{ir}^{(a)})
    =
    \frac{\sigma_\delta^2}{M^2}\sum_{j\in J_a} e_r(t_j)^2
    \le
    \frac{C}{m},
    \qquad a=0,1.
  \]
  Thus
  \[
    \Var(A_{ir})+\Var(B_{ir})
    \le
    C\left(\bar{\lambda}_r+\frac{1}{m}\right).
  \]
  For centered jointly Gaussian \(A,B\),
  \[
    \Var(AB)=\Var(A)\Var(B)+\Cov(A,B)^2.
  \]
  The summands are independent over \(i\), so the stated mean-square bound
  follows.

  The centered product \(A_{ir} B_{ir}-\bar{\lambda}_r\) is sub-exponential with
  \[
    \norm{A_{ir} B_{ir}-\bar{\lambda}_r}_{\psi_1}
    \le
    C\left(\bar{\lambda}_r+\frac{1}{m}\right).
  \]
  Bernstein's inequality gives, for every \(x\ge1\),
  \[
    \Pr\left\{
         \absx{\check{\lambda}_r-\bar{\lambda}_r}
         >
         C\left(
            \bar{\lambda}_r \sqrt{\frac{x}{n_\lambda}}
            +
            \frac{1}{m}\sqrt{\frac{x}{n_\lambda}}
            +
            \frac{x}{n_\lambda}
      \right)
    \right\}
    \le
    2e^{-cx}.
  \]
  Take \(x=M_1 \ell_{n,m}^{\lambda}\) with \(M_1\) sufficiently large and apply a
  union bound over \(\absx{r}\le K_m\). Since \(K_m \le n\) and
  \(n_\lambda \asymp n\), the term
  \(\bar{\lambda}_r \sqrt{\ell_{n,m}^{\lambda}/n_\lambda}\) is at most
  \(\bar{\lambda}_r/4\) for all sufficiently large \(n\). After enlarging \(M_0\),
  the remaining two terms are at most \(\zeta_{n,m}/4\). Therefore
  \cref{eq:cdeq-interlaced-lambda-good-event-prob} follows.
\end{proof}
   \clearpage\section{Lower Bound Under Common Design}
\label{sec:lower-common}

In this section, we provide the lower bound results for the common design setting.
We separate the proof into three parts: the statistical lower bound in \cref{prop:cd-lower-fixed-statistical};
the fixed-grid lower bound in \cref{prop:cd-lower-fixed-grid};
and the unknown-eigenvalue identification lower bound in \cref{prop:cd-lower-identification}.
The first two propositions give the fixed-known-eigenvalue lower bound in \cref{thm:common-lower}, while all three propositions give the unknown-eigenvalue lower bound.

\subsection{Statistical lower bound}
\label{subsec:common-lower-statistical}
\begin{proposition}
  \label{prop:cd-lower-fixed-statistical}
  Assume \(\alpha>1/2\) and \(s\ge 0\).
  Let \( \lambda\in\caL_\alpha(c_\lambda,C_\lambda) \) be fixed.
  Under \cref{assum:common-design}, for all sufficiently large \(n\) and \(m\),
  \begin{equation*}
    \inf_{\hat{\beta}}
    \sup_{\theta\in\Theta_s(R_0)}
    \E\caR(\hat{\beta};\theta,\lambda)
    \gtrsim
    \underbrace{n^{-\frac{2\alpha+2s}{2\alpha+2s+1}}}_{\mathrm{I}}
    +
    \underbrace{(nm)^{-\frac{2\alpha+2s}{4\alpha+2s+1}}}_{\mathrm{II}}.
  \end{equation*}
\end{proposition}

\begin{proof}
  The proof is exactly the same as the proof for the independent design lower bound in \cref{subsec:proof-indep-lower}, except that the Fisher information bound is replaced by the fixed-design version.
  Since all subjects share the deterministic grid \(t=(t_1,\dots,t_m)\), the Fisher information of the observation is
  \begin{equation}
    I_T^{\mr{cd}}(\theta_B)
    =
    n I^{\Phi_B(t),W_B(t)}(\theta_B).
    \label{eq:cd-fixed-design-fisher}
  \end{equation}
  This replaces the random-design Fisher information formula \cref{eq:indep-random-design-fisher}.
  For this fixed \(t\), the block observation satisfies the assumptions of \cref{lem:noisy-gaussian-block-fisher} with \(\Phi_B=\Phi_B(t)\) and \(W_B=W_B(t)\).
  The trigonometric basis is uniformly bounded, so \(\norm{\varphi_r(t)}_2^2 \le Cm\) for every \(r\in B_d\).
  Thus \cref{lem:noisy-gaussian-block-fisher,eq:cd-fixed-design-fisher} yield, for every \(r\in B_d\),
  \begin{equation}
    I_{T,rr}^{\mr{cd}}(\theta_B)
    \lesssim
    n\lambda_r
    \zk{\min(1,m\lambda_r)+\theta_B^\top \Lambda_B \theta_B}.
    \label{eq:cd-block-fisher-bound}
  \end{equation}
  This is the same bound as \cref{eq:indep-block-fisher-min-bound}, with \(I_{T,rr}\) replaced by \(I_{T,rr}^{\mr{cd}}\).
\end{proof}

\subsection{Fixed-grid lower bound}
\label{subsec:common-lower-fixed-grid}

\begin{proposition}[Fixed-grid lower bound]
  \label{prop:cd-lower-fixed-grid}
  Assume \(\alpha>1/2\) and \(s\ge 0\).
  Let \( \lambda\in\caL_\alpha(c_\lambda,C_\lambda) \) be fixed.
  Under \cref{assum:common-design}, for
  all sufficiently large \(n\) and \(m\),
  \begin{equation*}
    \inf_{\hat{\beta}}
    \sup_{\theta\in\Theta_s(R_0)}
    \E\caR(\hat{\beta};\theta,\lambda)
    \gtrsim
    \underbrace{m^{-(2\alpha+2s)}}_{\mathrm{III}}.
  \end{equation*}
\end{proposition}

\begin{proof}
  Let \(\Sigma\) be the covariance operator corresponding to the fixed sequence \(\lambda\), namely
  \[
    \Sigma e_r=\lambda_r e_r.
  \]
  Define
  \[
    g\coloneqq \Sigma\beta,
    \qquad
    g=\sum_{r\in\bbZ} \gamma_r e_r,
    \qquad
    \gamma_r=\lambda_r \theta_r.
  \]
  Since \(\lambda\in\caL_\alpha(c_\lambda,C_\lambda)\), we have \(\lambda_r \asymp(1+\absx{r})^{-2\alpha}\).
  Then
  \begin{equation}
    \norm{\beta}_{H^s}^2
    \asymp
    \sum_{r\in\bbZ}(1+\absx{r})^{2(s+2\alpha)}
    \absx{\gamma_r}^2,
    \label{eq:cd-equivalence-beta-g-fixed}
  \end{equation}
  and
  \begin{equation}
    \langle \beta,\Sigma\beta\rangle
    =
    \sum_{r\in\bbZ} \frac{\absx{\gamma_r}^2}{\lambda_r}
    \asymp
    \sum_{r\in\bbZ}(1+\absx{r})^{2\alpha}\absx{\gamma_r}^2.
    \label{eq:cd-equivalence-pred-g-fixed}
  \end{equation}
  Also,
  \begin{equation}
    \mr{Cov}(Y,Z_{\cdot j})=(\Sigma\beta)(t_j)=g(t_j).
    \label{eq:cd-cov-yz-g-fixed}
  \end{equation}

  Because the design has only \(m\) points on the unit circle, at least one gap between consecutive points has length at least \(1/m\).
  Choose such a gap and pick \(a\) inside it so that
  \[
    [a-h,a+h]\subset (t_j,t_{j+1})
  \]
  for some \(j\), with \(h=(8m)^{-1}\).
  Let \(\psi\in C_c^\infty(-1,1)\), \(\psi\not\equiv 0\), and define
  \begin{equation}
    g(t)\coloneqq A\psi\left(\frac{t-a}{h}\right).
    \label{eq:cd-bump-g-fixed}
  \end{equation}
  Then \(g(t_j)=0\) for all \(j=1,\dots,m\).
  For each \(r>0\), bump scaling gives
  \begin{equation}
    c_r A^2 h^{1-2r}
    \le
    \norm{g}_{H^r}^2
    \le
    C_r A^2 h^{1-2r}.
    \label{eq:cd-scaling-hr-fixed}
  \end{equation}
  Set
  \begin{equation}
    A^2=c_1 R_0^2 h^{2(s+2\alpha)-1}
    \label{eq:cd-bump-amplitude-fixed}
  \end{equation}
  with \(c_1\) small enough.
  Then \(\norm{g}_{H^{s+2\alpha}}\le R_0\), and by \cref{eq:cd-equivalence-beta-g-fixed},
  \[
    \beta\coloneqq \Sigma^{-1} g\in\Theta_s(R_0).
  \]

  Let \(\beta_+=\beta\) and \(\beta_-=-\beta\).
  By \cref{eq:cd-cov-yz-g-fixed,eq:cd-bump-g-fixed},
  \[
    \mr{Cov}_{\beta_+}(Y,Z)=\mr{Cov}_{\beta_-}(Y,Z)=0.
  \]
  Also,
  \[
    \mr{Var}_{\beta_+}(Y)=\mr{Var}_{\beta_-}(Y)
    =\langle\beta,\Sigma\beta\rangle+\sigma_\varepsilon^2,
  \]
  and the law of \(Z\) does not depend on \(\beta\).
  Since \((Y,Z)\) is jointly Gaussian, the two experiments are identical:
  \[
    \mr{TV}(P_{\beta_+}^{(n)},P_{\beta_-}^{(n)})=0.
  \]
  Their prediction separation under \(\lambda\) is
  \[
    D^2
    \coloneqq
    \E\left[\langle X_\star,\beta_+-\beta_-\rangle^2 \right]
    =
    4\langle\beta,\Sigma\beta\rangle.
  \]
  Using \cref{eq:cd-equivalence-pred-g-fixed,eq:cd-scaling-hr-fixed,eq:cd-bump-amplitude-fixed} with \(r=\alpha\),
  \begin{equation}
    D^2
    \asymp
    A^2 h^{1-2\alpha}
    \asymp
    R_0^2 h^{2(\alpha+s)}
    \asymp
    m^{-(2\alpha+2s)}.
    \label{eq:cd-m-rate-separation}
  \end{equation}
  Le Cam's two-point method therefore gives
  \begin{equation}
    \inf_{\hat{\beta}}
    \sup_{\theta\in\Theta_s(R_0)}
    \E\caR(\hat{\beta};\theta,\lambda)
    \ge
    c m^{-(2\alpha+2s)}.
    \label{eq:cd-m-rate-lower}
  \end{equation}
\end{proof}

\subsection{Unknown-eigenvalue identification lower bound}
\label{subsec:common-lower-identification}

\begin{proposition}
  \label{prop:cd-lower-identification}
  Assume \(\alpha>1/2\) and \(s\ge 0\).
  Under \cref{assum:common-design}, for
  all sufficiently large \(n\) and \(m\),
  \begin{equation*}
    \inf_{\hat{\beta}}
    \sup_{\lambda\in\caL_\alpha(c_\lambda,C_\lambda)}
    \sup_{\theta\in\Theta_s(R_0)}
    \E\caR(\hat{\beta};\theta,\lambda)
    \gtrsim
    \underbrace{m^{-4\alpha}}_{\mathrm{IV}}.
  \end{equation*}
\end{proposition}

\begin{proof}
  We will first show that the observation laws of \( (Y_i,(Z_{ij})_{j=1}^m)\)
  under two different parameter points can coincide, and then use this to show that the two parameter points must be separated in prediction risk.
  Recall the Fourier notation in \cref{subsec:fourier-coordinate-calculations}.
  The grid coefficients \(\bar{Z}_{ir}\) are those in \cref{eq:grid-zbar}.
  Since the grid Fourier vectors \((e_r(t_j))_{j=1}^m\), \(\absx{r}\le m\),
  span \(\R^m\), the coefficients \((\bar{Z}_{ir})_{\absx{r}\le m}\) determine
  \(Z_i=(Z_{i1},\dots,Z_{im})\).
  Hence it is enough to compare the law of
  \[
    \bar{W}_i
    \coloneqq
    \bigl(Y_i,(\bar{Z}_{ir})_{\absx{r}\le m}\bigr).
  \]
  This vector is centered Gaussian.
  For \(\absx{r},\absx{s}\le m\), put
  \begin{equation}
    \label{eq:cd-identification-zbar-cov-block}
    C_{rs}
    \coloneqq
    \sum_{\ell\in\bbZ}
    D_{r\ell}^{(m)} D_{s\ell}^{(m)} \lambda_\ell
    +
    \frac{\sigma_\delta^2}{m}D_{rs}^{(m)}.
  \end{equation}
  By \cref{eq:cd-discrete-orthogonality,eq:grid-zbar,eq:grid-x-tilde,eq:cd-aliased-gamma},
  its covariance matrix has the block form
  \[
    \Cov(\bar{W}_i)
    =
    \begin{pmatrix}
      \Var(Y_i)
      &
      (\widetilde\gamma_r)_{\absx{r}\le m}^{\top}
      \\
      (\widetilde\gamma_r)_{\absx{r}\le m}
      &
      \left(C_{rs} \right)_{\absx{r},\absx{s}\le m}
    \end{pmatrix},
  \]
  where the last block is the full covariance matrix of the redundant grid
  Fourier coefficients. Therefore the single-observation common-design
  Gaussian experiment is determined by this \(\bar{Z}\)-covariance block, the
  aliased cross-covariances \((\widetilde\gamma_r)_{\absx{r}\le m}\), and
  \(\Var(Y)\). It remains only to check that these three objects agree under
  the two parameter points constructed below.

  Choose \(\bar{c}\in(c_\lambda,C_\lambda)\) and set
  \[
    \bar{\lambda}_r \coloneqq \bar{c}(1+\absx{r})^{-2\alpha},
    \qquad r\in\bbZ.
  \]
  Let
  \[
    h\coloneqq c_0 m^{-2\alpha},
  \]
  where \(c_0>0\) is sufficiently small. Define two eigenvalue sequences by
  \begin{align*}
    \lambda_1^\pm
    &=\bar{\lambda}_1 \pm h,
    &
    \lambda_{m+1}^\pm
    &=\bar{\lambda}_{m+1} \mp h,\\
    \lambda_2^\pm
    &=\bar{\lambda}_2 \mp h,
    &
    \lambda_{m+2}^\pm
    &=\bar{\lambda}_{m+2} \pm h,
  \end{align*}
  and set
  \[
    \lambda_r^\pm=\bar{\lambda}_r,
    \qquad
    r\notin\{1,2,m+1,m+2\}.
  \]
  Since \(\bar{c}\) lies strictly between \(c_\lambda\) and \(C_\lambda\) and
  \(\bar{\lambda}_{m+1} \asymp m^{-2\alpha}\), shrinking \(c_0\) once and for all if
  necessary ensures \(\lambda^\pm \in\caL_\alpha(c_\lambda,C_\lambda)\).

  Choose
  \[
    g_1 \coloneqq \eta\sqrt{\bar{\lambda}_1^2-h^2},
    \qquad
    g_2 \coloneqq \eta\sqrt{\bar{\lambda}_2^2-h^2},
  \]
  where \(\eta>0\) is chosen below, and define
  \[
    \theta_1^\pm \coloneqq \frac{g_1}{\lambda_1^\pm},
    \qquad
    \theta_2^\pm \coloneqq \frac{g_2}{\lambda_2^\pm}.
  \]
  Set \(\theta_k^\pm \coloneqq 0\) for all \(k\notin\{1,2\}\).
  Because \(\lambda_1^\pm,\lambda_2^\pm \asymp 1\), choosing \(\eta\) sufficiently
  small gives \(\theta^\pm \in\Theta_s(R_0)\).
  
  Now let us verify that the block covariance matrix under \((\theta^+,\lambda^+)\) and \((\theta^-,\lambda^-)\) coincide.
  Use superscripts \(\pm\) to denote the two parameter points.
  First, for \( C_{rs} \) in \cref{eq:cd-identification-zbar-cov-block},
  the congruence formulas in \cref{lem:grid-gram} give \(D_{r,m+a}^{(m)}=D_{ra}^{(m)}\) for \(a=1,2\).
  Thus we have
  \begin{align*}
    C_{rs}^+ - C_{rs}^- &= \sum_{\ell\in\bbZ} D_{r\ell}^{(m)} D_{s\ell}^{(m)}(\lambda_\ell^+-\lambda_\ell^-)
    \\
    &= D_{r1}^{(m)} D_{s1}^{(m)} (\lambda_1^+ + \lambda_{m+1}^+ - \lambda_1^- - \lambda_{m+1}^-) \\
    &\qquad + D_{r2}^{(m)} D_{s2}^{(m)} (\lambda_2^+ + \lambda_{m+2}^+ - \lambda_2^- - \lambda_{m+2}^-) \\
    &= 0,
  \end{align*}
  as \(\lambda_1^++\lambda_{m+1}^+=\lambda_1^-+\lambda_{m+1}^-\) and
  \(\lambda_2^++\lambda_{m+2}^+=\lambda_2^-+\lambda_{m+2}^-\).
  Moreover,
  \[
    \gamma_1^\pm=\lambda_1^\pm \theta_1^\pm=g_1,
    \qquad
    \gamma_2^\pm=\lambda_2^\pm \theta_2^\pm=g_2,
  \]
  while \(\gamma_k^\pm=0\) for \(k\notin\{1,2\}\).
  Therefore, by \cref{eq:cd-aliased-gamma},
  the aliased cross-covariances \((\widetilde\gamma_r)_{\absx{r}\le m}\) are the same under the two experiments.
  Finally,
  \begin{align*}
    \Var_\pm(Y)
    & =
    \sigma_\varepsilon^2
    +
    \sum_{r\in\bbZ} \lambda_r^{\pm} (\theta_r^{\pm})^2 \\
    &=
    \sigma_\varepsilon^2
    +
    \frac{g_1^2}{\lambda_1^\pm}
    +
    \frac{g_2^2}{\lambda_2^\pm} \\
    &=
    \sigma_\varepsilon^2
    +
    \eta^2(\bar{\lambda}_1 \mp h)
    +
    \eta^2(\bar{\lambda}_2 \pm h) \\
    &=
    \sigma_\varepsilon^2+\eta^2(\bar{\lambda}_1+\bar{\lambda}_2).
  \end{align*}
  Hence the full single-observation Gaussian laws coincide, and therefore so do
  the \(n\)-sample laws.

  We next show that the two parameter points are separated in prediction risk.
  Set
  \[
    \underline\lambda_j \coloneqq \min\{\lambda_j^+,\lambda_j^-\},
    \qquad
    j=1,2.
  \]
  Since \(\underline\lambda_j \asymp 1\),
  \[
    \theta_1^+-\theta_1^-
    =
    g_1 \left(\frac{1}{\bar{\lambda}_1+h}-\frac{1}{\bar{\lambda}_1-h}\right)
    \asymp
    h,
  \]
  and similarly \(\theta_2^+-\theta_2^-\asymp h\).
  Therefore
  \begin{equation}
    \sum_{j=1}^2 \underline\lambda_j(\theta_j^+-\theta_j^-)^2
    \asymp
    h^2
    \asymp
    m^{-4\alpha}.
    \label{eq:cd-extra-separation}
  \end{equation}
  Let \(P^{(n)}\) denote this common \(n\)-sample law.
  For any estimator
  \(\hat{\beta}\), the equality of experiments implies
  \[
    \E\caR(\hat{\beta};\theta^+,\lambda^+)
    +
    \E\caR(\hat{\beta};\theta^-,\lambda^-)
    \ge
    \sum_{j=1}^2 \underline\lambda_j
    \E_{P^{(n)}}\left[
                  (\hat{\theta}_j-\theta_j^+)^2+(\hat{\theta}_j-\theta_j^-)^2
    \right].
  \]
  Using \((u-v)^2+(u-w)^2\ge \frac{1}{2}(v-w)^2\) pointwise and then
  \cref{eq:cd-extra-separation} gives
  \[
    \E\caR(\hat{\beta};\theta^+,\lambda^+)
    +
    \E\caR(\hat{\beta};\theta^-,\lambda^-)
    \ge
    c m^{-4\alpha},
  \]
  which proves \cref{prop:cd-lower-identification}.
\end{proof}
  \clearpage\section{Auxiliary Results}
\label{sec:auxiliary}

This section collects auxiliary results used in the upper-bound proofs.

\subsection{Concentration Inequalities}
\label{subsec:auxiliary-concentration}

\begin{proposition}[Product tail bound]
\label{prop:indep-product-tail}
Let \(A,B\) be real random variables such that
\[
  \E\exp\left[\left(\frac{\absx{A}}{K_A}\right)^{\alpha_A}\right]
  \le 2,
  \qquad
  \E\exp\left[\left(\frac{\absx{B}}{K_B}\right)^{\alpha_B}\right]
  \le 2
\]
for some \(\alpha_A,\alpha_B \in (0,\infty)\). Define
\[
  \alpha
  \coloneqq
  \left(\frac{1}{\alpha_A} + \frac{1}{\alpha_B}\right)^{-1}.
\]
Then
\[
  \E\exp\left[
  \left(
  \frac{\absx{AB}}{K_A K_B}
  \right)^\alpha
  \right]
  \le 4.
\]
In particular, the product of two sub-exponential random variables is
sub-Weibull of order \(1/2\).
\end{proposition}

\begin{proof}
  Young's inequality gives
  \[
    \left(
    \frac{\absx{AB}}{K_A K_B}
    \right)^\alpha
    =
    \left(\frac{\absx{A}}{K_A}\right)^\alpha
    \left(\frac{\absx{B}}{K_B}\right)^\alpha
    \le
    \frac{\alpha}{\alpha_A}
    \left(\frac{\absx{A}}{K_A}\right)^{\alpha_A}
    +
    \frac{\alpha}{\alpha_B}
    \left(\frac{\absx{B}}{K_B}\right)^{\alpha_B}.
  \]
  Exponentiating and applying Young's inequality once more,
  \[
    \exp\left[
    \left(
    \frac{\absx{AB}}{K_A K_B}
    \right)^\alpha
    \right]
    \le
    \frac{\alpha}{\alpha_A}
    \exp\left[
    \left(\frac{\absx{A}}{K_A}\right)^{\alpha_A}
    \right]
    +
    \frac{\alpha}{\alpha_B}
    \exp\left[
    \left(\frac{\absx{B}}{K_B}\right)^{\alpha_B}
    \right].
  \]
  Taking expectations proves the claim.
\end{proof}

\begin{lemma}[Sub-Weibull tail expectation]
\label{lem:upper-subweibull-tail-expectation}
Let \(X\) be a random element in a normed vector space
\((\mathcal{X},\norm{\cdot})\). Suppose that
\[
  \E\exp\left[\left(\frac{\norm{X}}{K}\right)^\alpha\right]\le C_0
\]
for some \(\alpha>0\), \(C_0>0\), and \(K>0\). Then there exists
\(L_\alpha<\infty\), depending only on \(\alpha\), such that, for every
\(L\ge L_\alpha\),
\[
  \E\left[\norm{X}\mathbf{1}\{\norm{X}>KL\}\right]
  \le
  C_0 K L e^{-L^\alpha}.
\]
Consequently,
\[
  \norm{\E\left[X\mathbf{1}\{\norm{X}>KL\}\right]}
  \le
  C_0 K L e^{-L^\alpha}.
\]
\end{lemma}

\begin{proof}
  Choose \(L_\alpha\) so that the map \(x\mapsto e^{x^\alpha}/x\) is
  increasing on \([L_\alpha,\infty)\). If \(r\ge L\ge L_\alpha\), then
  \[
    r
    \le
    Le^{-L^\alpha} e^{r^\alpha}.
  \]
  Applying this with \(r=\norm{X}/K\), on the event \(\{\norm{X}>KL\}\), gives
  \[
    \norm{X}\mathbf{1}\{\norm{X}>KL\}
    \le
    KLe^{-L^\alpha}
    \exp\left[\left(\frac{\norm{X}}{K}\right)^\alpha\right].
  \]
  Taking expectations proves the first bound. The second follows from the
  triangle inequality for the norm.
\end{proof}

\begin{lemma}[Truncation for sub-Weibull averages]
\label{lem:upper-subweibull-truncation}
Let \(X_1,\dots,X_n\) be i.i.d.\ random elements in a normed vector space
\((\mathcal{X},\norm{\cdot})\). Suppose that
\[
  \E\exp\left[\left(\frac{\norm{X_1}}{K}\right)^\alpha\right]\le C_0
\]
for some \(\alpha>0\), \(C_0>0\), and \(K>0\). For \(L>0\), put
\[
  Y_i \coloneqq X_i \mathbf{1}\{\norm{X_i}\le KL\}.
\]
Then there exists \(L_\alpha<\infty\), depending only on \(\alpha\), such that
for every \(L\ge L_\alpha\),
\[
  \Pr\left\{
  \norm{
  \frac{1}{n}\sum_{i=1}^n(X_i-\E X_i)
  -
  \frac{1}{n}\sum_{i=1}^n(Y_i-\E Y_i)
  }
  >
  C_0 K L e^{-L^\alpha}
  \right\}
  \le
  C_0 n e^{-L^\alpha}.
\]
\end{lemma}

\begin{proof}
  We have
  \[
    \frac{1}{n}\sum_{i=1}^n(X_i-\E X_i)
    -
    \frac{1}{n}\sum_{i=1}^n(Y_i-\E Y_i)
    =
    \frac{1}{n}\sum_{i=1}^n(X_i-Y_i)
    +
    \E(Y_1-X_1).
  \]
  By Markov's inequality,
  \[
    \Pr\{X_i \ne Y_i\}
    =
    \Pr\{\norm{X_i}>KL\}
    \le
    C_0 e^{-L^\alpha},
  \]
  so the first term is zero with probability at least
  \(1-C_0 n e^{-L^\alpha}\). After enlarging \(L_\alpha\) if necessary,
  \cref{lem:upper-subweibull-tail-expectation} gives
  \[
    \norm{\E(Y_1-X_1)}
    \le
    \E\left[\norm{X_1}\mathbf{1}\{\norm{X_1}>KL\}\right]
    \le
    C_0 K L e^{-L^\alpha}.
  \]
  Combining the two bounds proves the claim.
\end{proof}

\begin{lemma}[Sub-Weibull average concentration by truncation]
\label{lem:upper-subweibull-bernstein}
Let \(X_1,\dots,X_n\) be i.i.d.\ mean-zero real random variables satisfying
\[
  \E\exp\left[\left(\frac{\absx{X_1}}{K}\right)^\alpha\right]\le 4
\]
for some \(\alpha>0\) and \(K>0\). Then there exist constants
\(C,c,L_\alpha>0\), depending only on \(\alpha\), such that, for every
\(L\ge L_\alpha\) and \(u>0\),
\[
  \Pr\left\{
  \absx{\frac{1}{n}\sum_{i=1}^n X_i}
  >
  u+CKLe^{-L^\alpha}
  \right\}
  \le
  Cne^{-L^\alpha}
  +
  2\exp\left[
  -cn\min\left\{
  \frac{u^2}{K^2},
  \frac{u}{KL}
  \right\}
  \right].
\]
\end{lemma}

\begin{proof}
  This is the scalar version of the truncation argument used in
  \cref{lem:upper-subweibull-truncation}. Set
  \(Y_i=X_i \mathbf{1}\{\absx{X_i}\le KL\}\). The truncation lemma gives
  \[
    \absx{
    \frac{1}{n}\sum_{i=1}^n X_i
    -
    \frac{1}{n}\sum_{i=1}^n(Y_i-\E Y_i)
    }
    \le
    CKLe^{-L^\alpha}
  \]
  except on an event of probability at most \(Cne^{-L^\alpha}\).

  Moreover \(\E Y_i^2 \le \E X_i^2 \le C_\alpha K^2\), and
  \(\absx{Y_i-\E Y_i}\le 2KL\). Scalar Bernstein's inequality therefore gives,
  for \(u>0\),
  \[
    \Pr\left\{
    \absx{\frac{1}{n}\sum_{i=1}^n(Y_i-\E Y_i)}>u
    \right\}
    \le
    2\exp\left[
    -c n\min\left\{\frac{u^2}{K^2},\frac{u}{KL}\right\}
    \right].
  \]
  Combining this inequality with the truncation bound and increasing \(C\)
  completes the proof.
\end{proof}

\begin{lemma}[Vector Bernstein under directional sub-exponential tails]
\label{lem:upper-vector-bernstein-psi-one}
Let \(X_1,\dots,X_n\) be independent mean-zero random vectors in \(\R^p\).
Assume that, for every \(u \in \bbS^{p-1}\) and \(1\le i\le n\),
\[
  \norm{\angx{u,X_i}}_{\psi_1} \le K.
\]
Then there exist universal constants \(C,c>0\) such that, for every
\(x\ge1\),
\[
  \Pr\left\{
  \norm{\frac{1}{n}\sum_{i=1}^n X_i}_2
  >
  CK\left(
  \sqrt{\frac{p+x}{n}}
  +
  \frac{p+x}{n}
  \right)
  \right\}
  \le
  Ce^{-cx}.
\]
\end{lemma}

\begin{proof}
  Let \(\mathcal{N}\) be a \(1/2\)-net of \(\bbS^{p-1}\) with
  \(\absx{\mathcal{N}}\le 5^p\). For every \(z\in\R^p\),
  \[
    \norm{z}_2 \le 2\sup_{u\in\mathcal{N}}\absx{\angx{u,z}}.
  \]
  Fix \(u\in\mathcal{N}\). The variables
  \(\angx{u,X_i}\) are independent, mean zero, and have \(\psi_1\)-norm at
  most \(K\). The scalar Bernstein inequality for sub-exponential variables
  gives, for every \(t>0\),
  \[
    \Pr\left\{
    \absx{\angx{u,\frac{1}{n}\sum_{i=1}^n X_i}}>t
    \right\}
    \le
    2\exp\left[
    -cn\min\left\{\frac{t^2}{K^2},\frac{t}{K}\right\}
    \right].
  \]
  Taking
  \[
    t=A K\left(
    \sqrt{\frac{p+x}{n}}
    +
    \frac{p+x}{n}
    \right)
  \]
  with \(A\) sufficiently large gives
  \[
    n\min\left\{\frac{t^2}{K^2},\frac{t}{K}\right\}
    \ge
    C_0(p+x).
  \]
  The union bound over \(\mathcal{N}\), followed by the net inequality above
  and an adjustment of constants, proves the claim.
\end{proof}

\begin{lemma}[Variance-sensitive vector Bernstein under directional sub-Weibull tails]
\label{lem:upper-vector-bernstein}
Let \(X_1,\dots,X_n\) be independent mean-zero random vectors in \(\R^p\).
Let \(\alpha>0\) and \(\alpha_*=\min\{\alpha,1\}\). Assume that, for every
\(u \in \bbS^{p-1}\) and \(1\le i\le n\),
\[
  \norm{\angx{u,X_i}}_{\psi_\alpha} \le K,
  \qquad
  \E \angx{u,X_i}^2 \le \sigma^2.
\]
Then there exist constants \(C_\alpha,c_\alpha > 0\), depending only on
\(\alpha\), such that, for every \(x \ge 1\),
\[
  \Pr\left\{
  \norm{\frac{1}{n}\sum_{i=1}^n X_i}_2
  >
  C_\alpha \left(
  \sigma\sqrt{\frac{p+x}{n}}
  +
  K(\log(2n))^{1/\alpha}
  \frac{(p+x)^{1/\alpha_*}}{n}
  \right)
  \right\}
  \le
  3e^{-c_\alpha x}.
\]
\end{lemma}

\begin{proof}
  Let \(\mathcal{N}\) be a \(1/2\)-net of \(\bbS^{p-1}\) with
  \(\absx{\mathcal{N}}\le 5^p\). For every \(z \in \R^p\),
  \[
    \norm{z}_2 \le 2\sup_{u \in \mathcal{N}} \angx{u,z}.
  \]
  Apply \citet[Theorem~3.4]{kuchibhotla2022_MovingBeyondSubGaussianity} to
  the finite-dimensional vectors
  \((\angx{u,X_i})_{u\in\mathcal{N}}\). Since
  \(\log\absx{\mathcal{N}}\le Cp\), the theorem gives, with probability at
  least \(1-3e^{-x}\),
  \[
    \sup_{u\in\mathcal{N}}
    \absx{\angx{u,\frac{1}{n}\sum_{i=1}^n X_i}}
    \le
    C_\alpha \left(
    \sigma\sqrt{\frac{p+x}{n}}
    +
    K(\log(2n))^{1/\alpha}
    \frac{(p+x)^{1/\alpha_*}}{n}
    \right).
  \]
  Combining this display with the net bound and adjusting constants proves the
  claim.
\end{proof}

\begin{lemma}[Hilbert-valued Bousquet inequality]
\label{lem:upper-hilbert-bousquet}
Let \(\mathcal{H}\) be a separable real Hilbert space, and let
\(\xi_1,\dots,\xi_N\) be independent mean-zero \(\mathcal{H}\)-valued random
variables. Suppose that, almost surely,
\[
  \norm{\xi_i}_{\mathcal{H}}\le R,
  \qquad 1\le i\le N.
\]
Put \(S=\sum_{i=1}^N \xi_i\), \(Z=\norm{S}_{\mathcal{H}}\), and
\[
  \sigma_w^2
  \coloneqq
  \sup_{\norm{u}_{\mathcal{H}}\le1}
  \sum_{i=1}^N \E\angx{u,\xi_i}_{\mathcal{H}}^2 .
\]
Then, for every \(x\ge0\),
\[
  \Pr\left\{
  Z>
  \E Z+\sqrt{2\left(\sigma_w^2+2R\E Z\right)x}
  +\frac{Rx}{3}
  \right\}
  \le e^{-x}.
\]
Consequently, for every deterministic \(M\ge\E Z\) and every \(x\ge0\),
\[
  \Pr\left\{
  Z>
  M
  +
  C\left[
  \sqrt{\sigma_w^2 x}
  +
  \sqrt{RMx}
  +
  Rx
  \right]
  \right\}
  \le e^{-x},
\]
where \(C>0\) is a universal constant.
\end{lemma}

\begin{proof}
  The result is the Hilbert-space specialization of Bousquet's inequality for
  suprema of bounded empirical processes. By separability,
  \(Z\) is the supremum of the empirical process indexed by the unit ball of
  \(\mathcal{H}\):
  \[
    Z
    =
    \sup_{\norm{u}_{\mathcal{H}}\le1}
    \sum_{i=1}^N \angx{u,\xi_i}_{\mathcal{H}} .
  \]
  Applying Bousquet's inequality
  \citep{bousquet2002_BennettConcentrationInequality}; see also
  \citet[Theorem~12.5 and Corollary~12.12]{boucheron2013_ConcentrationInequalities},
  to the centered, \(R\)-bounded class
  \(u\mapsto\angx{u,\xi_i}_{\mathcal{H}}\) gives, with
  \(v=\sigma_w^2+2R\E Z\),
  \[
    \Pr\left\{
    Z>\E Z+\sqrt{2vx}+\frac{Rx}{3}
    \right\}
    \le e^{-x}.
  \]
  If \(M\ge\E Z\), then
  \[
    \sqrt{2vx}
    \le
    \sqrt{2\sigma_w^2 x}+2\sqrt{RMx},
  \]
  and the second claim follows by increasing the universal constant.
\end{proof}

\begin{lemma}[Envelope vector Bernstein under sub-Weibull tails]
\label{lem:upper-vector-bernstein-envelope}
Let \(X_1,\dots,X_n\) be independent mean-zero random vectors in \(\R^p\).
Let \(0<\alpha\le1\). Assume that
\[
  \sup_{u\in\bbS^{p-1}}
  \frac{1}{n}\sum_{i=1}^n \E\angx{u,X_i}^2
  \le \sigma^2,
  \qquad
  M_\alpha
  \coloneqq
  \norm{\max_{1\le i\le n} \norm{X_i}_2}_{\psi_\alpha}
  <\infty.
\]
Then there exist constants \(C_\alpha,c_\alpha>0\), depending only on
\(\alpha\), such that, for every \(x\ge1\),
\[
  \Pr\left\{
  \norm{\frac{1}{n}\sum_{i=1}^n X_i}_2
  >
  C_\alpha \left(
  \sigma\sqrt{\frac{p+x}{n}}
  +
  M_\alpha \frac{x^{1/\alpha}}{n}
  \right)
  \right\}
  \le
  C_\alpha e^{-c_\alpha x}.
\]
In particular, if
\[
  \norm{\angx{u,X_i}}_{\psi_\alpha} \le K,
  \qquad
  u\in\bbS^{p-1},\quad 1\le i\le n,
\]
then
\[
  \Pr\left\{
  \norm{\frac{1}{n}\sum_{i=1}^n X_i}_2
  >
  C_\alpha \left(
  \sigma\sqrt{\frac{p+x}{n}}
  +
  K\sqrt{p} (\log(2n))^{1/\alpha}\frac{x^{1/\alpha}}{n}
  \right)
  \right\}
  \le
  C_\alpha e^{-c_\alpha x}.
\]
\end{lemma}

\begin{proof}
  Let \(\mathcal{F}=\{f_u(z)=\angx{u,z}:u\in\bbS^{p-1}\}\). By separability of
  \(\bbS^{p-1}\), it is enough to apply the empirical-process bound on a
  countable dense subset. Put
  \[
    Z=\sup_{f\in\mathcal{F}}\absx{\sum_{i=1}^n f(X_i)}
    =
    \norm{\sum_{i=1}^n X_i}_2.
  \]
  For this class,
  \[
    \sup_{f\in\mathcal{F}}\sum_{i=1}^n \E f(X_i)^2\le n\sigma^2,
    \qquad
    \norm{\max_{1\le i\le n} \sup_{f\in\mathcal{F}}\absx{f(X_i)}}_{\psi_\alpha}
    =
    M_\alpha.
  \]
  Moreover,
  \[
    \E Z
    \le
    \left(\E\norm{\sum_{i=1}^n X_i}_2^2 \right)^{1/2}
    =
    \left(\sum_{i=1}^n \E\norm{X_i}_2^2 \right)^{1/2}
    \le
    \sigma\sqrt{np}.
  \]
  Adamczak's empirical-process inequality
  \citep[Theorem~4]{adamczak2008_TailInequalitySuprema} therefore gives
  \[
    \Pr\left\{
    Z
    >
    C_\alpha \left(
    \sigma\sqrt{n(p+x)}
    +
    M_\alpha x^{1/\alpha}
    \right)
    \right\}
    \le
    C_\alpha e^{-c_\alpha x}.
  \]
  Dividing by \(n\) proves the first claim.

  For the last assertion, let \(e_1,\dots,e_p\) be the standard basis of
  \(\R^p\). For \(r\ge2\), the moment characterization of the
  \(\psi_\alpha\) norm and Minkowski's inequality give
  \[
    \norm{\norm{X_i}_2}_{L_r}
    =
    \norm{\sum_{j=1}^p \angx{e_j,X_i}^2}_{L_{r/2}}^{1/2}
    \le
    \left(\sum_{j=1}^p \norm{\angx{e_j,X_i}}_{L_r}^2 \right)^{1/2}
    \le
    C_\alpha K\sqrt{p} r^{1/\alpha}.
  \]
  The case \(1\le r<2\) follows by monotonicity of \(L_r\) norms, so
  \(\norm{\norm{X_i}_2}_{\psi_\alpha} \le C_\alpha K\sqrt{p}\). The standard
  maximal inequality for sub-Weibull variables then gives
  \(M_\alpha \le C_\alpha K\sqrt{p} (\log(2n))^{1/\alpha}\).
\end{proof}

\begin{lemma}[Tail integration for Bernstein block tails]
\label{lem:upper-bernstein-tail-integration}
Let \(Z\ge0\) be a random variable.
Let \(p\ge1\), \(b>0\), \(M\ge0\), and \(q\ge1\).
Suppose that there exist constants \(A_0,A_1,c_0>0\) such that
\[
  \Pr\left\{
  Z>A_0 b \left(x+Mx^q \right)
  \right\}
  \le A_1 e^{-c_0 x},
  \qquad x\ge p.
\]
If \(V>0\), \(D<\infty\), and \(a>0\) satisfy
\[
  b\left(p+Mp^q \right)\le DV,
  \qquad
  A_0 b \left(p+Mp^q \right)\le aV,
\]
then there exist constants \(C,c>0\), depending only on
\((A_0,A_1,c_0,q)\), such that
\[
  \E\left[Z\mathbf{1}\{Z>aV\}\right]\le CDV e^{-cp}.
\]
\end{lemma}

\begin{proof}
  Define
  \[
    H(x)\coloneqq A_0 b \left(x+Mx^q \right),
    \qquad x\ge p.
  \]
  Since \(H\) is increasing, the tail-integration formula gives
  \[
  \begin{aligned}
    \E\left[Z\mathbf{1}\{Z>H(p)\}\right]
    &\le
    H(p)\Pr\{Z>H(p)\}
    +
    \int_p^\infty \Pr\{Z>H(x)\}H'(x) dx.
  \end{aligned}
  \]
  Using the tail assumption,
  \[
  \begin{aligned}
    \E\left[Z\mathbf{1}\{Z>H(p)\}\right]
    &\le
    Cb\left(p+Mp^q \right)e^{-c_0 p}
    +
    Cb\int_p^\infty
    \left(1+Mx^{q-1} \right)e^{-c_0 x} dx\\
    &\le
    Cb\left(p+Mp^q \right)e^{-cp},
  \end{aligned}
  \]
  after decreasing \(c\) if necessary. Here we used
  \(H'(x)=A_0 b(1+qMx^{q-1})\) and, for \(p\ge1\),
  \[
    \int_p^\infty x^{q-1} e^{-c_0 x} \dd x
    \le C_q p^{q-1} e^{-cp}
    \le C_q p^q e^{-cp}.
  \]
  If \(H(p)\le aV\), then \(\{Z>aV\}\subseteq\{Z>H(p)\}\).
  Combining this inclusion with \(b(p+Mp^q)\le DV\) proves the claim.
\end{proof}

\subsection{Supremum-Norm Control}
\label{subsec:auxiliary-supnorm}

\begin{lemma}[Uniform supremum-norm control]
\label{lem:indep-process-supnorm}
Let
\[
  \frac{1}{2}<\underline{\alpha}<\infty,
  \qquad
  0<c_\lambda \le C_\lambda<\infty.
\]
For each \(\alpha\ge\underline{\alpha}\) and
\(\lambda\in\caL_\alpha(c_\lambda,C_\lambda)\), let
\(X_i(t)=\sum_{r\in\bbZ} x_{ir} e_r(t)\) be the centered Gaussian process with
Fourier variances \((\lambda_r)_{r\in\bbZ}\).
Then there exists \(C<\infty\), depending only on
\((\underline{\alpha},c_\lambda,C_\lambda)\), such that
\[
  \sup_{\alpha \ge \underline{\alpha}}
  \sup_{\lambda \in \caL_\alpha(c_\lambda,C_\lambda)}
  \norm{\norm{X_i}_\infty}_{\psi_2} \le C.
\]
\end{lemma}

\begin{proof}
  Write \(x_{ir}=\lambda_r^{1/2} g_{ir}\), and let
  \(d_\lambda\) be the
  canonical metric of the centered Gaussian process \(X_i\):
  \[
    d_\lambda(s,t)^2
    \coloneqq
    \E\absx{X_i(s)-X_i(t)}^2
    =
    \sum_{r \in \bbZ} \lambda_r \absx{e_r(s)-e_r(t)}^2.
  \]
  By the trigonometric basis fixed in \cref{sec:setup}, there
  exists \(C<\infty\) such that for every \(r\in\bbZ\) and \(s,t\in[0,1]\),
  \[
    \absx{e_r(s)-e_r(t)}
    \le
    C\min\{1,\absx{r}\absx{s-t}\}.
  \]
  Choose
  \[
    \eta_0
    \coloneqq
    \frac{\underline{\alpha}-\frac{1}{2}}{2}
    >0.
  \]
  Since \(\min\{1,x\}^2 \le x^{2\eta_0}\) for all \(x\ge 0\), we obtain
  \[
    \begin{aligned}
      d_\lambda(s,t)^2
      &\le
      C\sum_{r \in \bbZ} \lambda_r
      \min\{1,\absx{r}\absx{s-t}\}^2\\
      &\le
      C\absx{s-t}^{2\eta_0}
      \sum_{r \in \bbZ}(1+\absx{r})^{-2\alpha+2\eta_0}\\
      &\le
      C\absx{s-t}^{2\eta_0},
    \end{aligned}
  \]
  uniformly over \(\alpha\ge\underline{\alpha}\), because
  \[
    -2\alpha+2\eta_0
    \le
    -2\underline{\alpha}+2\eta_0
    =
    -\underline{\alpha}-\frac{1}{2}
    < -1.
  \]
  Therefore
  \[
    d_\lambda(s,t)\le C\absx{s-t}^{\eta_0}.
  \]

  Also, since the trigonometric basis is uniformly bounded,
  \[
    \sup_{t\in[0,1]} \Var(X_i(t))
    =
    \sup_{t\in[0,1]} \sum_{r \in \bbZ} \lambda_r \absx{e_r(t)}^2
    \le
    C\sum_{r \in \bbZ} \lambda_r
    \le
    C,
  \]
  because \(\alpha>\frac{1}{2}\) implies
  \(\sum_{r \in \bbZ} \lambda_r
  \lesssim \sum_{r \in \bbZ}(1+\absx{r})^{-2\underline{\alpha}} < \infty\).

  The metric bound \(d_\lambda(s,t)\le C\absx{s-t}^{\eta_0}\) gives
  \[
    N([0,1],d_\lambda,\varepsilon)
    \le
    C\varepsilon^{-1/\eta_0},
    \qquad
    0<\varepsilon\le 1,
  \]
  so the standard Dudley entropy bound yields
  \[
    \E\sup_{t\in[0,1]} X_i(t)\le C.
  \]
  Applying the same argument to \(-X_i\) gives
  \[
    \E\norm{X_i}_\infty \le C.
  \]

  Let
  \[
    M_i^+\coloneqq \sup_{t\in[0,1]} X_i(t),
    \qquad
    M_i^-\coloneqq \sup_{t\in[0,1]}(-X_i(t)).
  \]
  Since \(X_i\) is centered Gaussian and
  \(\sup_{t\in[0,1]} \Var(X_i(t))\le C\), the Borell--TIS inequality
  \citep{adler2007_RandomFieldsGeometry} gives
  \[
    \Pr\{M_i^\pm > \E M_i^\pm + u\}\le e^{-cu^2},
    \qquad
    u\ge 0.
  \]
  Because \(\norm{X_i}_\infty=\max\{M_i^+,M_i^-\}\), this implies
  \[
    \Pr\{\norm{X_i}_\infty > C + u\}\le 2e^{-cu^2},
    \qquad
    u\ge 0.
  \]
  This tail bound is equivalent to
  \(\norm{\norm{X_i}_\infty}_{\psi_2} \le C\).
\end{proof}

\subsection{Sequence Bounds}
\label{subsec:auxiliary-sequence-bounds}

\begin{lemma}[Two-sided sequence bounds]
\label{lem:upper-two-sided-sequence-bounds}
Assume \(\lambda \in \caL_\alpha(c_\lambda,C_\lambda)\) and
\(\theta \in \Theta_s(R_0)\). Then there exists \(C<\infty\), depending only on
\((\alpha,s,R_0,c_\lambda,C_\lambda)\), such that for every integer \(d \ge 1\),
\begin{equation}
  \sum_{\absx{r} > d}\lambda_r \absx{\theta_r}^2
  \le
  C R_0^2 d^{-(2\alpha+2s)}.
  \label{eq:upper-two-sided-tail-bias}
\end{equation}
and
\begin{equation}
  \sum_{\absx{r} \le d}(1+\absx{r})^{2\alpha}\absx{\theta_r}^2
  \le
  C d^{(2\alpha-2s)_+}.
  \label{eq:upper-two-sided-partial-sum}
\end{equation}
\end{lemma}

\begin{proof}
  Since \(\lambda_r \le C_\lambda(1+\absx{r})^{-2\alpha}\),
  \[
    \sum_{\absx{r} > d}\lambda_r \absx{\theta_r}^2
    \le
    C_\lambda \sum_{\absx{r} > d}
    (1+\absx{r})^{-(2\alpha+2s)}(1+\absx{r})^{2s}\absx{\theta_r}^2
    \le
    C R_0^2 d^{-(2\alpha+2s)},
  \]
  which proves \cref{eq:upper-two-sided-tail-bias}. Likewise,
  \[
    \sum_{\absx{r} \le d}(1+\absx{r})^{2\alpha}\absx{\theta_r}^2
    =
    \sum_{\absx{r} \le d}
    (1+\absx{r})^{(2\alpha-2s)}(1+\absx{r})^{2s}\absx{\theta_r}^2
    \le
    C d^{(2\alpha-2s)_+},
  \]
  which proves \cref{eq:upper-two-sided-partial-sum}.
\end{proof}
   \clearpage\section{Additional Numerical Experiments}
\label{sec:supp-numerics}
\FloatBarrier

This appendix provides additional numerical experiments for \cref{sec:numerics}.
We use the same Fourier power-law data-generating model as in \cref{subsec:numerics-simulations}.

\subsection{Error decay pattern}
\label{subsec:supp-error-decay-pattern}

We first present additional heatmaps under various smoothness profiles in \Cref{fig:supp-common-design-heatmaps,fig:supp-independent-design-heatmap}.

\begin{figure}[htpb]
  \centering
  \includegraphics[width=1\textwidth]{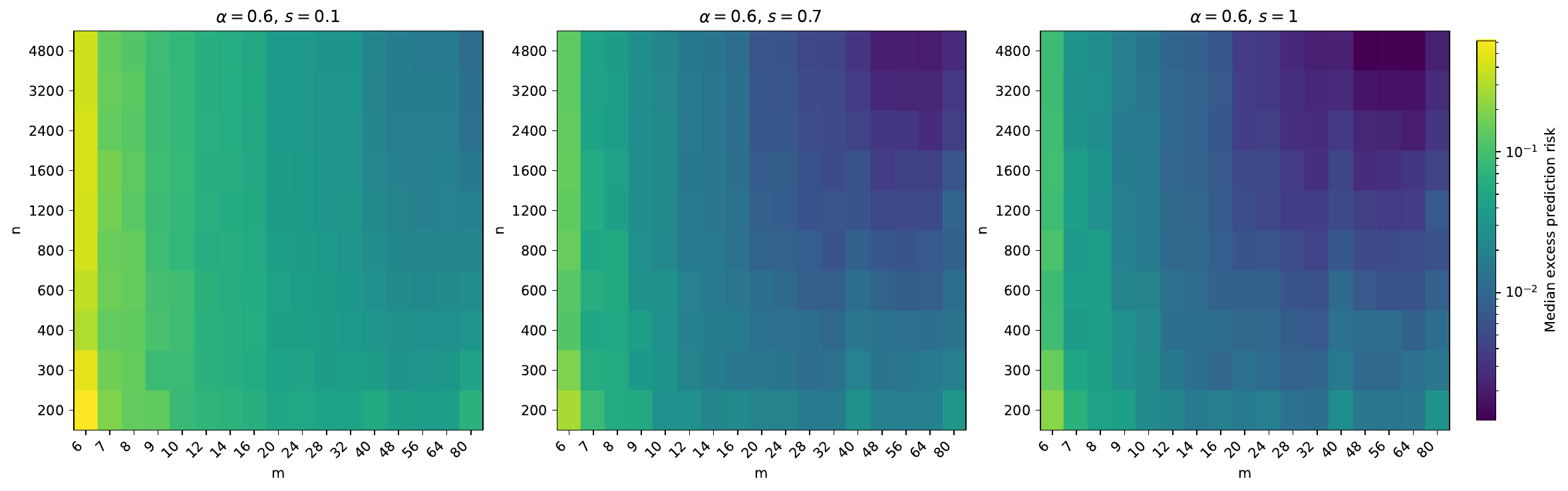}
  \vspace{0.5em}
  \includegraphics[width=1\textwidth]{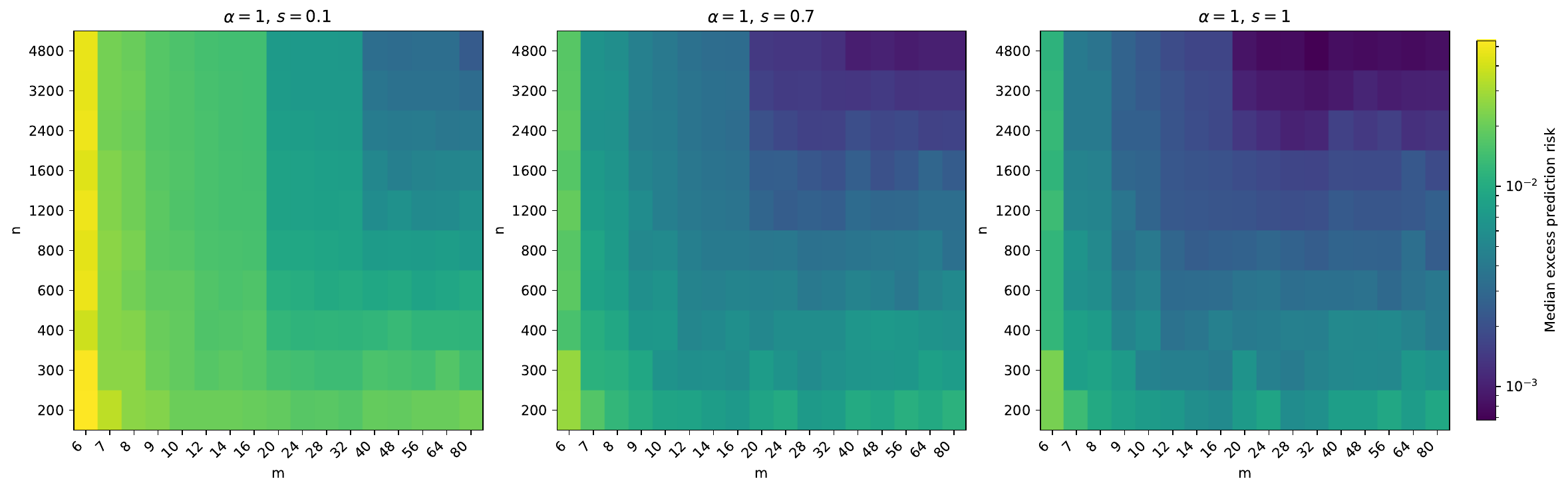}
  \caption{Additional simulation heatmaps under common design.}
  \label{fig:supp-common-design-heatmaps}
\end{figure}

\begin{figure}[htpb]
  \centering
  \includegraphics[width=1\textwidth]{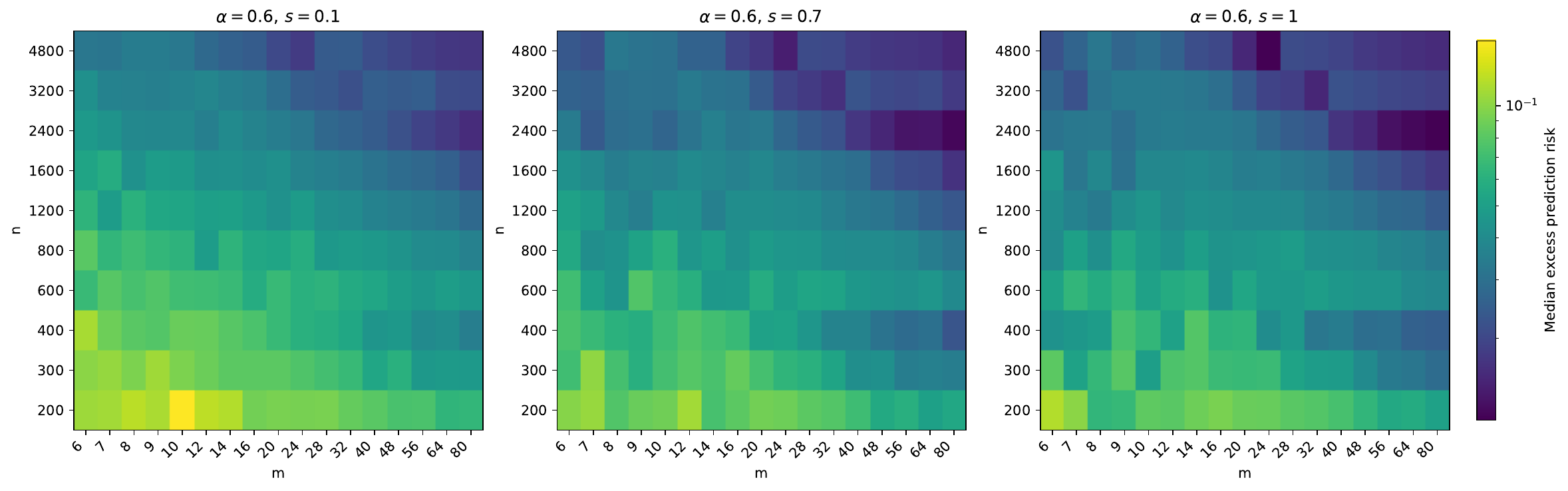}
  \vspace{0.5em}
  \includegraphics[width=1\textwidth]{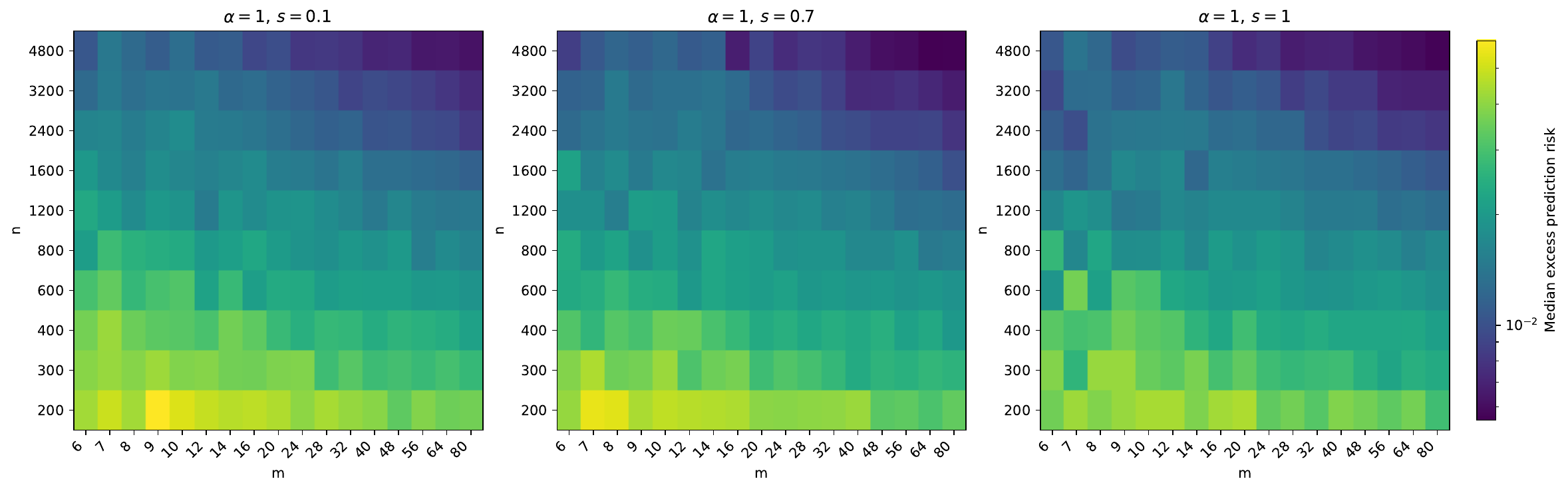}
  \caption{Additional simulation heatmaps under independent design.
  }
  \label{fig:supp-independent-design-heatmap}
\end{figure}

For \Cref{fig:supp-common-design-heatmaps} under common design, changing \(\alpha\) and \(s\) visibly changes the decay pattern in the heatmaps.
When both smoothness parameters are small, the discretization error at low resolution is relatively more prominent, so increasing the grid size \(m\) is the main driver of risk reduction and the decay pattern is mostly left-to-right.
For larger \(\alpha\) or \(s\), the sampling component becomes more prominent on the displayed grid, and the risk decreases mainly when \(m\) and \(n\) grow together.
This produces a more diagonal decay pattern from the lower-left region toward the upper-right region.
These orientations are consistent with the common-design decomposition: the low-resolution terms are primarily controlled by the grid size, whereas the sampling term improves with the joint growth of the number of curves and the number of measurements per curve.

For \Cref{fig:supp-independent-design-heatmap} under independent design, the heatmaps display a simpler decay pattern across the two profile sets.
Most panels decay diagonally from the lower-left region toward the upper-right region, indicating that the risk improves mainly as \(m\) and \(n\) grow together.
This behavior contrasts with the common-design heatmaps, where different low-resolution errors create more varied decay patterns and sharper phase transitions.
Under independent design, the rate decomposition contains only the dense FLR term and the term due to noisy measurements, so the additional fixed-grid and unknown-eigenvalue low-resolution phenomena from common design are absent.
For the \(\alpha=1\) profiles, increasing \(s\) shifts the transition toward smaller \(m\), in line with the independent-design threshold \(\gamma=2\alpha/(2\alpha+2s+1)\) along \(m=n^\gamma\).

\FloatBarrier

\subsection{Convergence rate}
\label{subsec:supp-convergence-rate}

We also present additional rate validation.
To display the qualitative rate more clearly, we use a version of the adaptive estimator with an oracle choice of the truncation level instead of the block thresholding procedure as in the main text.
Results for common design and independent design are given in \Cref{fig:supp-common-design-rate-alpha1} and \Cref{fig:supp-independent-design-rate-comparisons}, respectively.
We also compare the common-design oracle estimator, which uses known eigenvalues, with its counterpart that estimates the covariance eigenvalues in \Cref{fig:supp-common-known-unknown-oracle-loglog}.

\begin{figure}[htpb]
  \centering
  \includegraphics[width=1\textwidth]{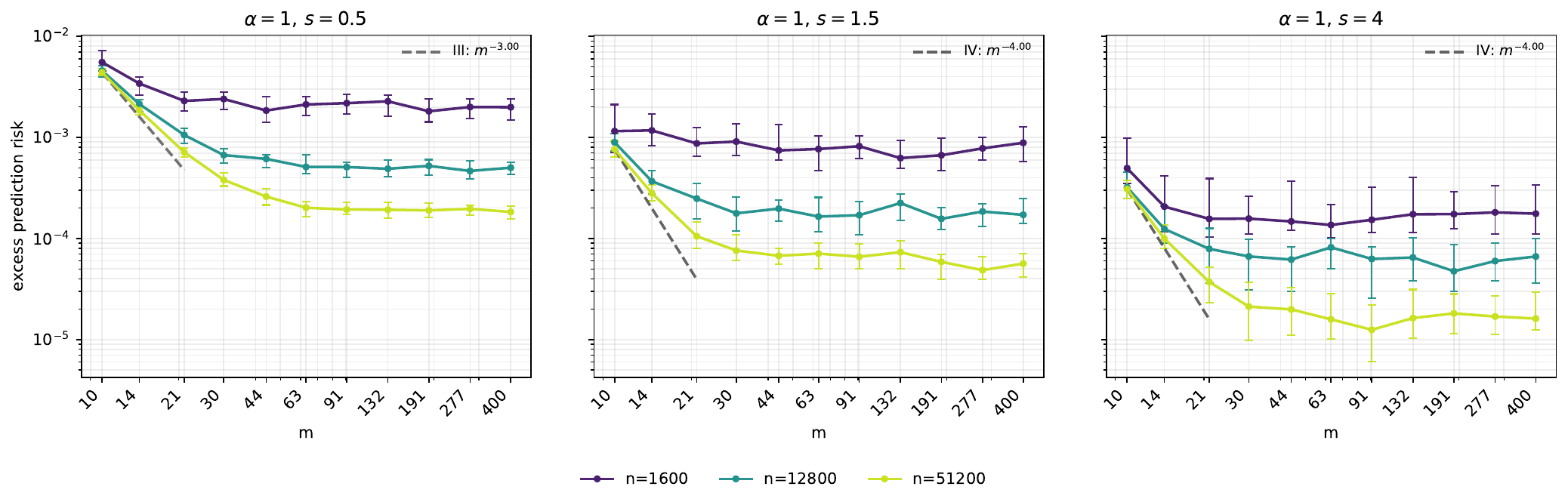}
  \vspace{0.5em}
  \includegraphics[width=1\textwidth]{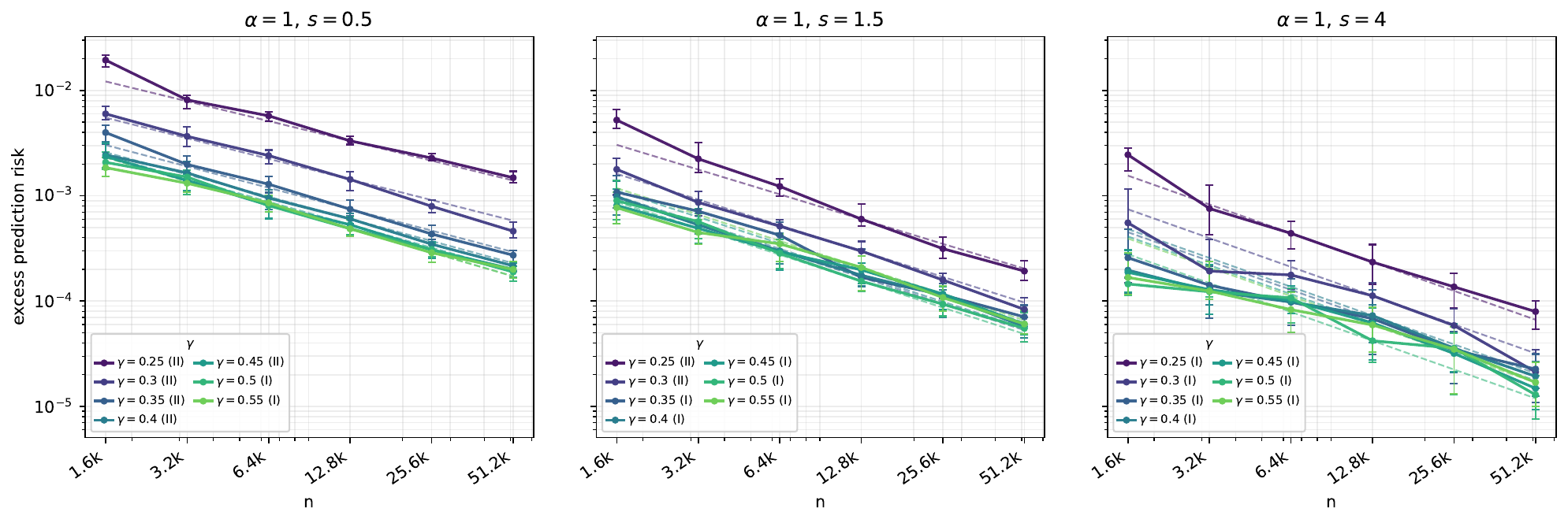}
  \caption{Additional common-design rate comparisons.
  }
  \label{fig:supp-common-design-rate-alpha1}
\end{figure}

\begin{figure}[htpb]
  \centering
  \includegraphics[width=1\textwidth]{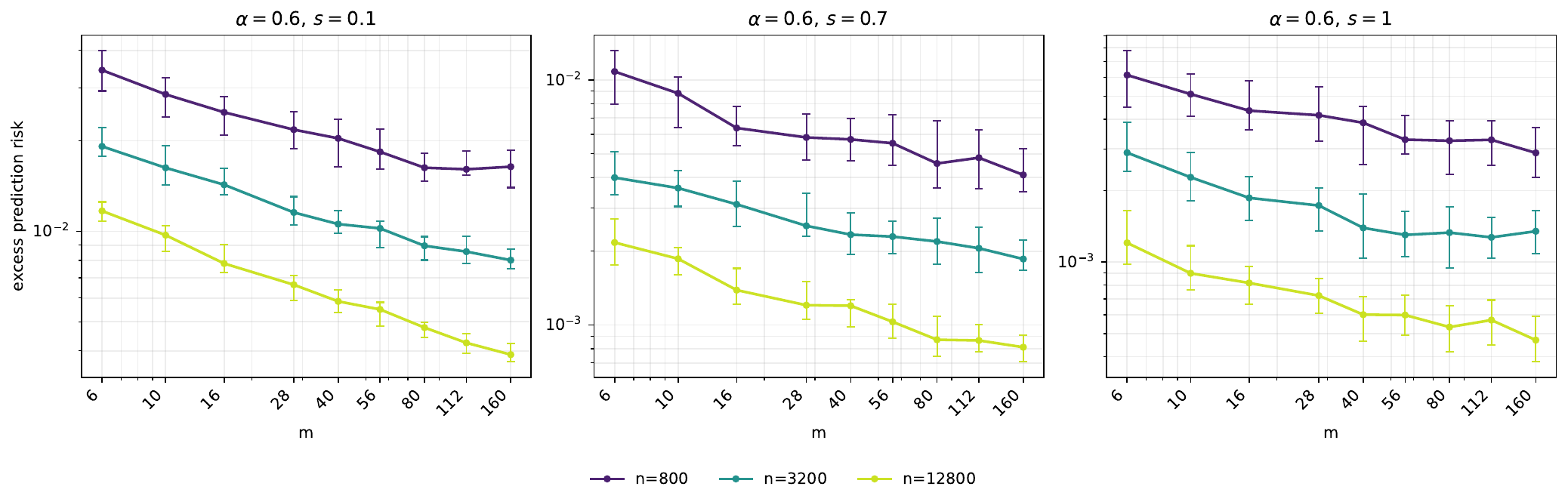}
  \vspace{0.5em}
  \includegraphics[width=1\textwidth]{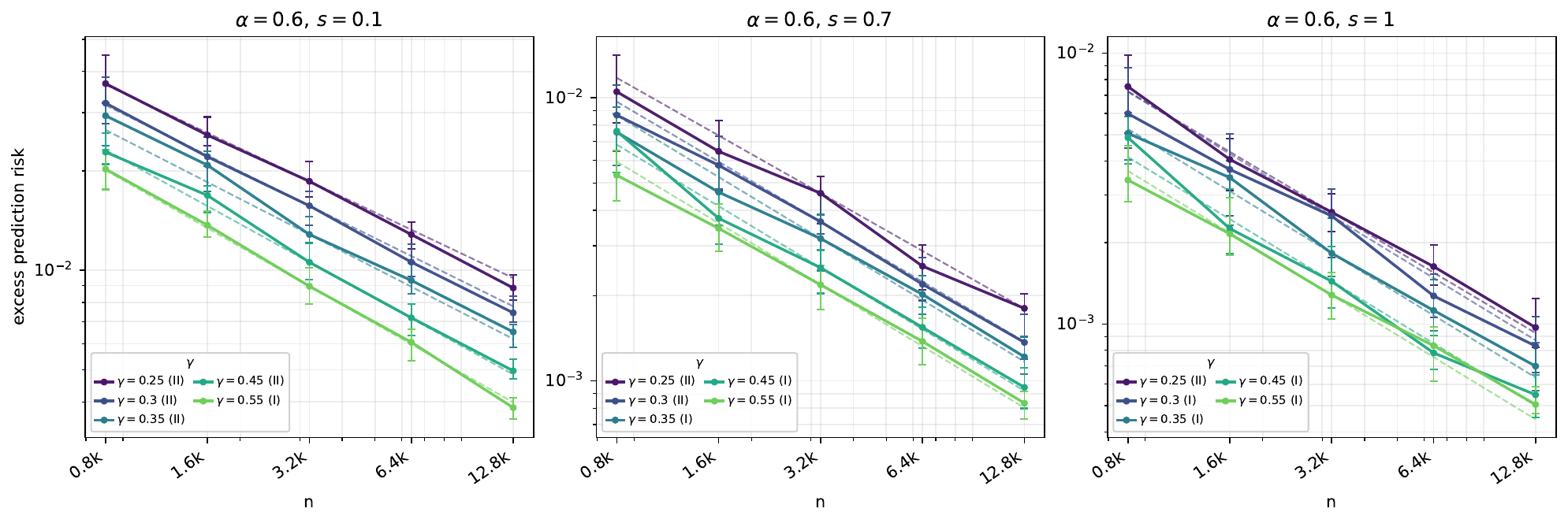}
  \caption{Additional independent-design rate comparisons.}
  \label{fig:supp-independent-design-rate-comparisons}
\end{figure}

\begin{figure}[htpb]
  \centering
  \includegraphics[width=1\textwidth]{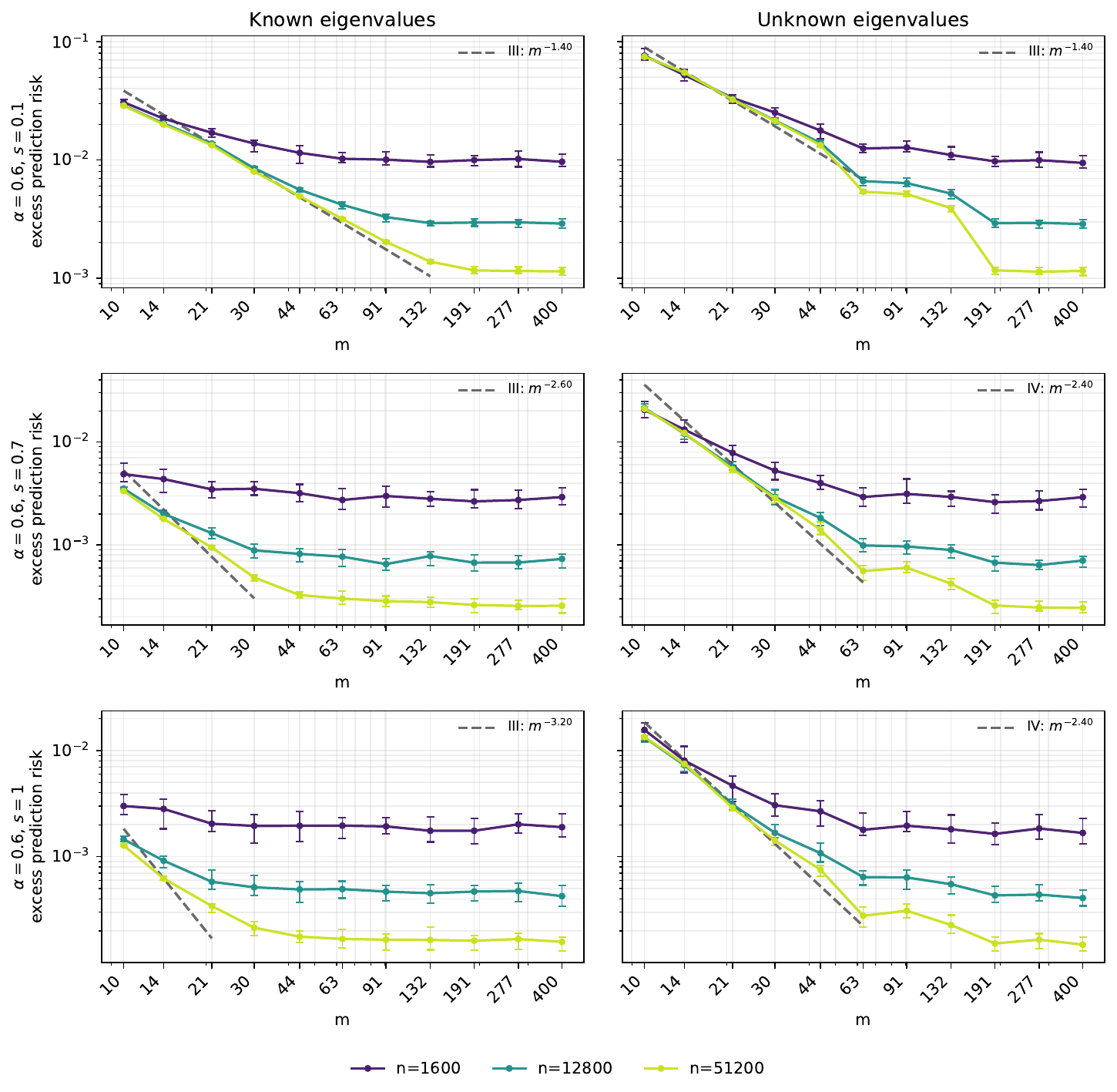}
  \caption[Known versus unknown eigenvalues under common design]{Common-design rate comparisons with known and unknown eigenvalues.
  The left column uses the true eigenvalue sequence in the oracle estimator, while the right column estimates eigenvalues from data on the common grid.
  Panels in the same row use the same vertical scale.
  }
  \label{fig:supp-common-known-unknown-oracle-loglog}
\end{figure}

For \Cref{fig:supp-independent-design-rate-comparisons} under independent design, we observe that when $n$ is fixed, the prediction error gradually decreases as $m$ increases in comparison with the common design case, since it is the $(mn)$-term that dominates the rate under independent design, and the low-resolution phenomena from common design are absent.
The curves also form a plateau when $m$ is large enough, showing the transition to the dense regime.

\Cref{fig:supp-common-known-unknown-oracle-loglog} isolates the contribution of eigenvalue estimation in the common-design rate experiment.
For \(s=0.7\) and \(s=1\), the known-eigenvalue panels follow the steeper fixed-grid guide \(m^{-(2\alpha+2s)}\), while the unknown-eigenvalue panels follow the \(m^{-4\alpha}\) guide over the low-resolution range.
This behavior is consistent with the additional unknown-eigenvalue term in the common-design rate.

\FloatBarrier

\subsection{Impact of noise scale}
\label{subsec:supp-impact-noise-scale}

\Cref{fig:supp-measurement-noise-sensitivity} shows the finite-sample sensitivity to the measurement-noise level under the two sampling designs.
Increasing \(\sigma_\delta\) raises the risk, especially in the lower-resolution and smaller-\(n\) parts of the grid, but the broad orientation of each design-specific heatmap remains similar across the noise levels.
\Cref{fig:supp-observation-noise-sensitivity} gives the analogous sensitivity check for the scalar response error level \(\sigma_\varepsilon\), with \(\sigma_\delta=0.1\) fixed.
Increasing \(\sigma_\varepsilon\) shifts the risk surface upward while preserving the main dependence on \(n\) and \(m\).

\begin{figure}[htpb]
  \centering
  \includegraphics[width=1\textwidth]{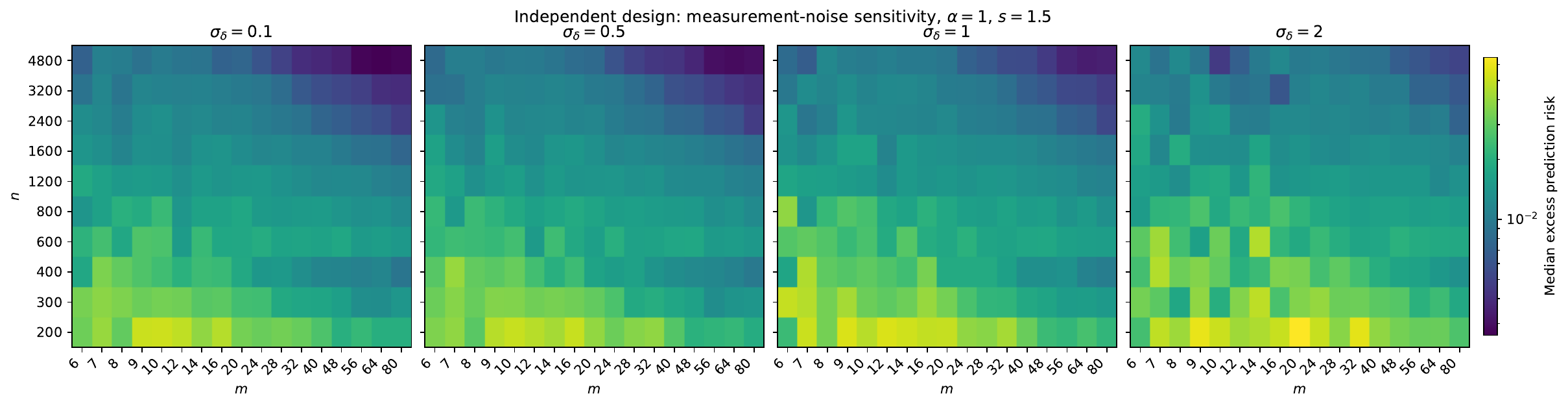}
  \vspace{0.5em}
  \includegraphics[width=1\textwidth]{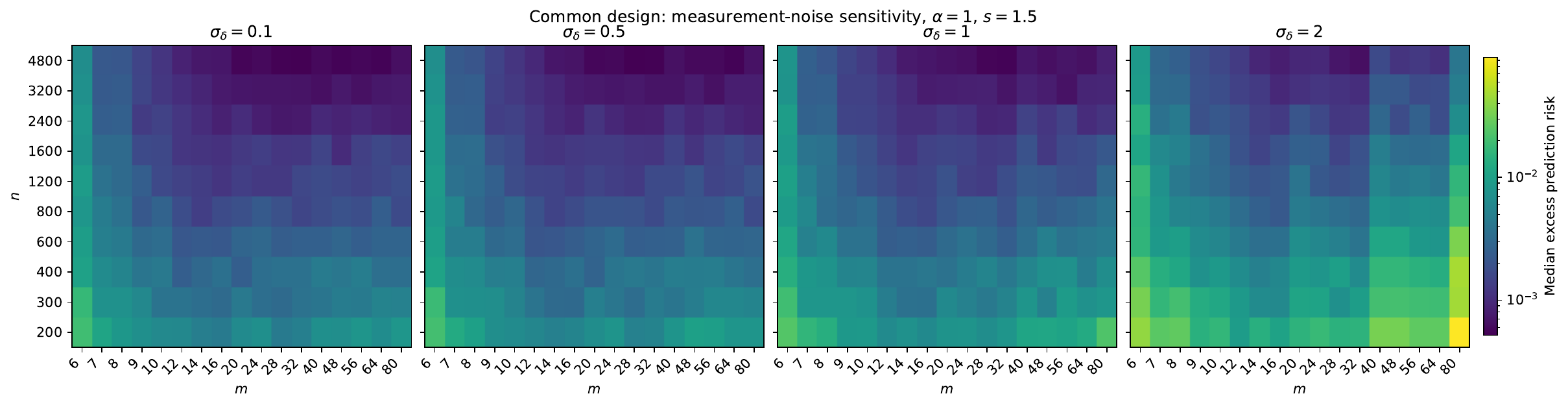}
  \caption[Measurement-noise sensitivity]{Measurement-noise sensitivity under independent design (top) and common design (bottom).
  The panels vary the measurement-error standard deviation \(\sigma_\delta\), with \(\sigma_\varepsilon=0.5\) fixed.
  }
  \label{fig:supp-measurement-noise-sensitivity}
\end{figure}

\begin{figure}[htpb]
  \centering
  \includegraphics[width=1\textwidth]{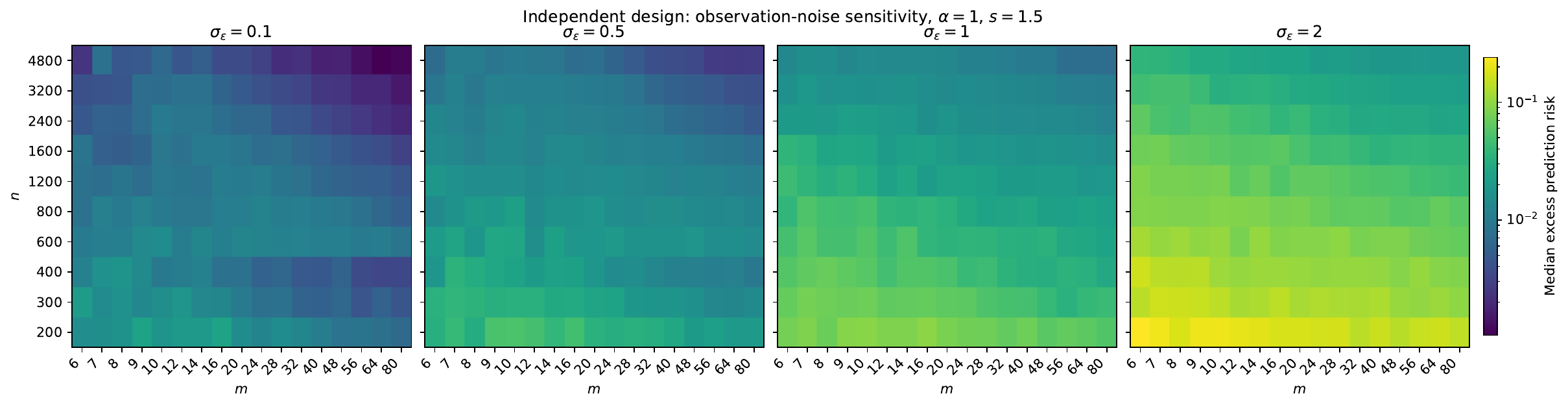}
  \vspace{0.5em}
  \includegraphics[width=1\textwidth]{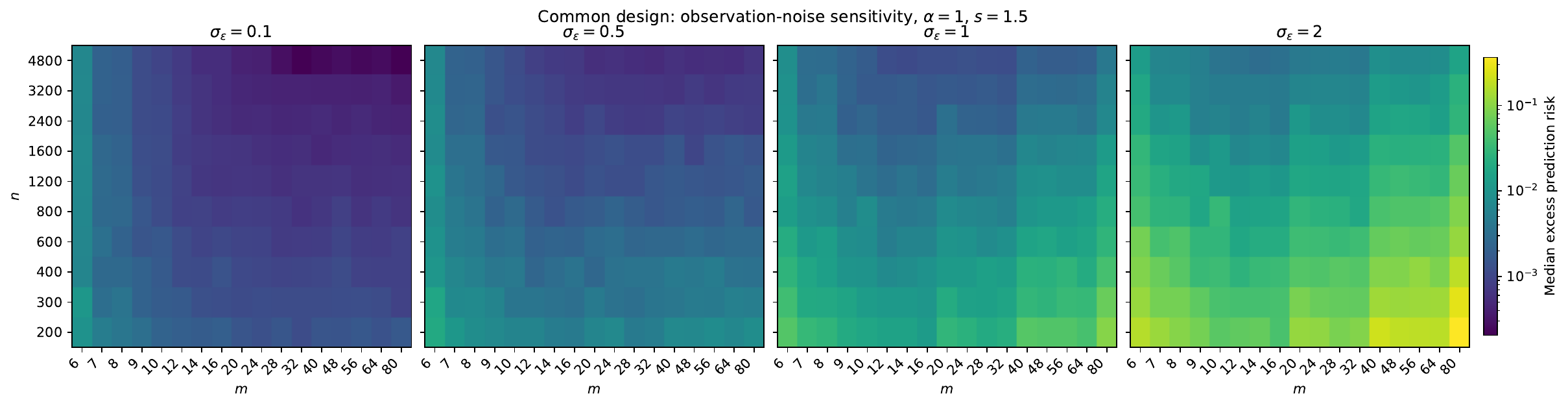}
  \caption[Observation-noise sensitivity]{Observation-noise sensitivity under independent design (top) and common design (bottom).
  The panels vary the scalar response error standard deviation \(\sigma_\varepsilon\), with \(\sigma_\delta=0.1\) fixed.
  }
  \label{fig:supp-observation-noise-sensitivity}
\end{figure}

\FloatBarrier

\bibliographystyle{plainnat}
\bibliography{main}

\end{document}